	\documentclass{amsart}

\usepackage[T1]{fontenc}

\usepackage{cite}
\usepackage{tikz}
\usepackage{tikz-cd}
\usepackage{tikz-layers}
\usetikzlibrary{shapes,arrows}

\usetikzlibrary{calc,
	arrows,decorations.pathmorphing,
	backgrounds,fit,positioning,shapes.symbols,chains}
\usetikzlibrary{decorations.pathmorphing}
\usetikzlibrary{decorations.markings}
\usetikzlibrary{decorations.pathreplacing,angles,quotes}
\usetikzlibrary{positioning,shadows,arrows,trees,fit}
\usetikzlibrary{mindmap}
\usepackage{amssymb,amsmath,amsfonts,latexsym, enumerate}
\usepackage{bm}
\usepackage{yhmath}

\usepackage{multicol}
\usepackage[all,cmtip]{xypic} 
\usepackage{xypic} 
\xyoption{curve}

\usepackage{standalone}

\usepackage{color}
\usepackage{float}

\usepackage{hyperref}
\usepackage[noabbrev, capitalize]{cleveref}
\usepackage{xfrac}    

\numberwithin{equation}{section}

\newtheorem{thm}[equation]{Theorem}
\newtheorem{lem}[equation]{Lemma}
\newtheorem{prop}[equation]{Proposition}
\newtheorem{cor}[equation]{Corollary}

\restylefloat*{figure}

\newtheorem{defn}[equation]{Definition}

\theoremstyle{remark}
\newtheorem{rmk}[equation]{Remark}

\newtheorem{ex}[equation]{Example}

\renewcommand{\emptyset}{\font\cmsy = cmsy10 at 10pt
	\hbox{\cmsy \char 59}
}

\newcommand\R{\mathbb{R}}
\newcommand\N{\mathbb{N}}

\newcommand\T{\mathbb{T}}

\newcommand{\bbb}{\bm{b}}
\newcommand{\ccc}{\bm{c}} 
\newcommand{\ddd}{\bm{d}}

\mathchardef\mathdash="2D
\definecolor{burntsienna}{rgb}{0.91, 0.45, 0.32}
\definecolor{sapgreen}{rgb}{0.31, 0.49, 0.16}
\definecolor{teal}{rgb}{0.0, 0.5, 0.5}

\definecolor{blue}{rgb}{0.38, 0.51, 0.71} 
\definecolor{darkblue}{RGB}{17, 42, 60} 
\definecolor{red}{RGB}{175, 49, 39} 

\definecolor{orange}{RGB}{217, 156, 55} 
\definecolor{green}{RGB}{144, 169, 84} 
\definecolor{palegreen}{RGB}{197, 184, 104} 

\definecolor{yellow}{RGB}{250, 199, 100} 
\definecolor{brokenwhite}{RGB}{218, 192, 166} 
\definecolor{brokengrey}{rgb}{0.77, 0.76, 0.82} 

\newcommand\GS{\mathsf {GS}}
\newcommand\CSM{\mathsf {MO}}
\newcommand{\OOO}{\mathbb {O}}
\newcommand\Klgr{\Xi}
\newcommand{\pr}[1]{\mathsf{psh}(#1)}
\newcommand{\sh}[1]{\mathsf{sh}(#1)}
\newcommand{\TT}{\mathbb T}
\newcommand{\DD}{\mathbb D}
\newcommand{\TTp}{\mathbb T_*}
\newcommand\Set{\mathsf{Set}}
\newcommand\Cat{\mathsf {Cat}} 
\newcommand\fin{\mathsf{Set_f}}
\newcommand\sSet{\mathsf{sSet}}
\newcommand{\mm}[1][m]{\mathbf{#1}}
\newcommand{\nul}{\mathbf{0}}
\newcommand{\one}{\mathbf{1}}
\newcommand{\two}{\mathbf{2}}
\newcommand{\nn}{\mathbf{n}}

\newcommand{\CCat}{ \mathsf{C}}
\newcommand{\DCat}[1][D]{ \mathsf{#1}}
\newcommand{\ElP}[2]{\ensuremath{\mathsf{el}_{#2}(#1)}}
\newcommand{\ov }{/}
\newcommand{\defeq}{\overset{\textup{def}}{=}}

\newcommand{\fiso}{{\mathbf{B}}}
\newcommand{\fisinv}{{\fiso}^{\scriptstyle{\bm \S}}}

\newcommand{\Comm}{{K}}
\newcommand{\CComm}[2][(\CCC,\omega)]{{\Comm}^{#1}_{#2}}

\newcommand\CCC{\mathfrak{C}}
\newcommand\DDD[1][D]{\mathfrak{#1}}

\newcommand\Di{\mathfrak{Di}}
\newcommand{\elG}[1][\mathcal{G}]{\ensuremath{\mathsf{el}{(#1)}}}
\newcommand{\ElS}[1][S]{\ensuremath{\mathsf{el}{(#1)}}}

\tikzset{->-/.style={decoration={
			markings,
			mark=at position #1 with {\arrow{>}}},postaction={decorate}}}
\tikzset{-<-/.style={decoration={
			markings,
			mark=at position #1 with {\arrow{<}}},postaction={decorate}}}

\newcommand\nuCSM{\mathsf {MO}^-}

\newcommand\CO{\mathsf {CO}}

\newcommand\E{\mathsf{E}}

\newcommand{\MM}{\mathbb {M}}

\newcommand{\boffcat}[1][\MM, \DCat]{ \Theta_{#1}}

\newcommand{\Op}{ \mathsf{Op}}

\newcommand{\diag}{\mathsf{\bm D}}

\newcommand{\Gret}{\mathsf{Gr_{et}}}
\newcommand\Gr{\mathsf{C}\Gret}
\newcommand{\Griso}{\mathdash\mathsf{CGr_{iso}}}

\newcommand{\Gretp}{\mathsf{Gr}_{*}}
\newcommand{\Grp}{\ensuremath{\mathsf{C}\Gretp}}
\newcommand{\GSp}{\ensuremath{\mathsf{\GS}_*}}

\newcommand{\fisinvp}{\ensuremath{\fisinv_*}}

\newcommand\G{\mathcal{G}}
\renewcommand\H{\mathcal{H}}
\newcommand\W{\mathcal W}
\newcommand\I{\mathcal{I}}
\newcommand\C{\mathcal{C}}
\newcommand{\Lk}[1][k]{\ensuremath{\mathcal L^{#1}}}
\newcommand{\Wl}[1][m]{\ensuremath{\mathcal W^{#1}}}
\newcommand{\Wm}[1][m]{\ensuremath{\mathcal W^{#1}}}

\newcommand\CX[1][X]{{\mathcal C_{#1}}}

\newcommand{\Fgraph}{ \xymatrix{
	E \ar@(lu,ld)[]_\tau&& H \ar[ll]_s \ar[rr]^t&& V}}
\newcommand{\Fgraphdash}{\xymatrix{
	E \ar@(lu,ld)[]_{\tau } &&H \ar[ll]_{s } \ar[rr]^{t }& &V }}

\newcommand{\Fgraphvar}[6]{\xymatrix{
	*[r] {#1}\ar@(ul,dl)[]_{#6} && {#2} \ar[ll]_-{#4} \ar[rr]^-{#5}&& {#3}}}

\newcommand\Cv[1][v]{{\mathcal C_{\mathbf{#1}}}}

\newcommand{\EI}{\ensuremath{{E_\bullet}}}

\newcommand{\vH}[1][v]{\sfrac{H}{#1}}
\newcommand{\vE}[1][v]{\sfrac{E}{#1}}
\newcommand{\nV}[1][n]{V_{#1}}
\newcommand{\nH}[1][n]{H_{#1}}
\newcommand{\nE}[1][n]{E_{#1}}

\newcommand\shorte[1][\tilde e]{{\shortmid_{#1}}}
\newcommand{\esG}[1][\mathcal{G}]{\ensuremath{\mathsf{es}{(#1)}}}

\newcommand{\esv}[1][v]{\ensuremath{\iota_{{#1}}}}
\newcommand{\ese}[1][\tilde e]{\ensuremath{\iota_{{#1}}}}
\newcommand{\esh}[1][h]{\ensuremath{\delta_{{#1}}}} 

\newcommand{\yet}{\Upsilon}
\newcommand{\yetp}{{\yet_*}}

\newcommand\X{\mathcal X}
\newcommand\Y{\mathcal Y}

\newcommand\Gg{\mathbf \Gamma}
\newcommand\Gdg{\mathbf \Lambda}
\newcommand\Gid[1][\G]{\mathbf{I}^{#1}}
\newcommand\coGg[1][\G]{{\Gg}({#1})}

\newcommand{\elpG}[1][\mathcal{G}]{\ensuremath{\mathsf{el}_*{(#1)}}}

\newcommand{\Gnov}[1][W]{\ensuremath{\mathcal G_{\setminus #1}}}
\newcommand{\Enov}[1][W]{\ensuremath{E_{\setminus #1}}}
\newcommand{\Hnov}[1][W]{\ensuremath{H_{\setminus #1}}}
\newcommand{\Vnov}[1][W]{\ensuremath{V_{\setminus #1}}}
\newcommand{\snov}[1][W]{\ensuremath{s_{\setminus #1}}}
\newcommand{\tnov}[1][W]{\ensuremath{t_{\setminus #1}}}
\newcommand{\taunov}[1][W]{\ensuremath{\tau_{\setminus #1}}}
\newcommand{\delW}[1][W]{\mathsf{del}_{\setminus #1}}

\newcommand\Grsimp{\mathsf{CGr}_{\mathsf{sim}}}
\newcommand\XGrsimp{\mathdash\mathsf{CGr}_{\mathsf{sim}}}

\newcommand{\etap}{\ensuremath{\eta^{\TTp}}}
\newcommand{\Tp}{\ensuremath{T_*}}

\newcommand{\factcat}[1][\beta]{\ensuremath{\mathsf{fact}_*(#1)}}

\newcommand{\listm} {\small{\mathsf{list}}}

\newcommand\V{\mathcal{V}}
\newcommand{\op}{\mathcal O}
 
\newcommand{\CommE}[1][\E]{{K_{#1}}}

\newcommand{\CGSE}[1][\E]{{\GS^{(\CCC,\omega)}_{#1}}}

\newcommand{\GSE}[1][\mathsf E]{{\GS_{#1}}} 
\newcommand{\prE}[2][\E]{\mathsf{psh_{#1}}(#2)}

\newcommand\nuCO{\mathsf {CO}^-}
\newcommand{\GSEp}[1][\E]{\ensuremath{{\mathsf{\GS}_{#1}}_*}}

\newcommand\Gretsimp{\mathsf{Gr}_{\mathsf{sim}}}
\newcommand\XGretsimp{\mathdash\mathsf{Gr}_{\mathsf{sim}}}

\newcommand{\bfom}{\bm{\overleftarrow \omega}}
\newcommand{\WD}{\mathrm{WD}}
\newcommand{\CWD}[1][(\CCC, \omega)]{\WD^{#1}}

\newcommand{\CSME}[1][\E]{{\mathsf {MO}}_{#1}}
\newcommand{\nuCSME}[1][\E]{{\mathsf {MO}}^{-}_{#1}}
\newcommand{\COE}[1][\E]{{\mathsf {CO}}_{#1}}
\newcommand{\nuCOE}[1][\E]{{\mathsf {CO}}^{-}_{#1}}

\newcommand{\LL}{\mathbb L}

\newcommand{\TTk}{{\mathbb{ T}^{\times}}} 
\newcommand{\Tk}{{T^{\times}} }
\newcommand{\etak}{{ \eta^{\mathbb{T}^{\times}}}}

\newcommand{\TTpk}{{\tilde{ \mathbb {L }}\mathbb{ T}_*}} 
\newcommand{\Tpk}{{\widetilde L T_* }} 
\newcommand{\Grisok}{{\mathdash\mathsf{Gr}_{\mathrm{iso}}}}
\newcommand{\CA}{{\mathsf{CA}}}
\newcommand{\Sym}{{\mathsf{Sym}}}

\newcommand{\CCO}[1][(\CCC,\omega)]{{\CO^{#1}}}
\newcommand{\CCA}[1][(\CCC,\omega)]{\CA^{#1}}

\newcommand{\man}{\mathcal M}

\newcommand\Klgrt{{\Xi^\times}}

\newcommand{\GretG}[1][\G]{\ensuremath{\Gret^{(#1)}}}
\newcommand{\GrG}[1][\G]{\ensuremath{\Gr^{(#1)}}}

\newcommand{\ModR}[1][R ]{ {#1}\mathdash\mathsf{Mod}}
\newcommand{\core}[1]{\mathrm{Core}(#1)}
\newcommand{\Coriso}[1][X]{{{#1}\mathdash\mathsf{Cor}_{\boxtimes}}_{\mathrm{iso}}}
\newcommand{\CorGg}[1][\X]{ \mathsf{Cor}_{\boxtimes} ^{(#1)}}
\newcommand{\alg}[1][\mathcal A]{#1}
\newcommand{\Br}[1][\delta]{\mathsf{Br}_{#1}}
\newcommand{\Brd}{\mathsf{Br}_{\downarrow}}
\newcommand{\Bru}{\mathsf{Br}_{\uparrow}}
\newcommand{\CBD}[1][(\CCC, \omega)]{\mathsf{BD}^{#1}}
\newcommand{\BD}{\mathsf{BD}}
\newcommand{\BDop}{\mathring{\mathsf{BD}}} 

\newcommand{\BDd}{\mathsf{BD}_{\downarrow}}

\newcommand{\BDu}{\mathsf{BD}_{\uparrow}}
\newcommand{\vect}[1][\Bbbk]{ \mathsf{Vect}_{\scriptstyle{#1}}}
\newcommand{\modR}[1][R]{ {#1}\mathdash\mathsf{Mod}}
\newcommand{\fopen}[1][f]{\mathring{{#1}}}
\newcommand{\Sf}[1][f]{\mathfrak S(#1)}
\newcommand{\Tf}[1][f]{\mathfrak T(#1)}
\newcommand\dipal{\{\wideparen{+,-}\}}
\newcommand\DiBD{\CBD[\scriptstyle{{\dipal}}]}

\newcommand{\dicomm}{K^{\scriptscriptstyle{\dipal}}} 

\newcommand{\rev}[1]{{\overleftarrow{#1}}}

\newcommand{\kclf}[1][f]{\mathrm{k}_{#1}}
\newcommand{\kcl}{\mathrm{k}}
\newcommand{\colop}[1][\lambda]{{{#1}_{\partial}}}
\newcommand{\colcl}[1][\lambda]{{\widetilde{#1}}}
\newcommand{\coev}[1]{{\lceil{#1}\rceil}}
\newcommand{\ev}[1]{{\lfloor{#1}\rfloor}}

\newcommand{\pal}{\mathsf{Pal}}
\tikzset{
dot/.style = {circle, fill, minimum size=#1,
inner sep=0pt, outer sep=0pt},
dot/.default = 5pt 
}

\title[A nerve theorem for circuit algebras]{Brauer diagrams, modular operads, and a graphical nerve theorem for circuit algebras}
\author{Sophie Raynor}

\thanks{The author acknowledges the support of Australian
	Research Council grants DP160101519 and FT160100393.}

\begin{document}

\maketitle

\begin{abstract}
	Circuit algebras, used in the study of finite-type knot invariants, are a symmetric analogue of Jones's planar algebras. They are very closely related to circuit operads, modular operads equipped with a monoidal product. This paper gives a description of circuit algebras in terms categories of Brauer diagrams, and thereby establishes a link between circuit algebras and the representation theory of classical Lie groups. An abstract nerve theorem for circuit operads -- and hence circuit algebras -- is proved and a graphical calculus is established using an iterated distributive law, and an existing nerve theorem for modular operads.

\end{abstract}

\textbf{Note:} This paper has been superceded by \href{https://arxiv.org/abs/2412.20260}{arXiv:2412.20260} (on circuit algebras, Brauer diagrams, modular operads and invariant theory) and \href{https://arxiv.org/abs/2412.20262}{arXiv:2412.20262} (on a graphical calculus, iterated distributive laws and nerve theorem for circuit algebras). The version below is from November 2022.
\section{Introduction}

{Circuit algebras are a symmetric (or non-planar) version of Jones's planar algebras\cite{Jon94}. They were introduced as a framework for working with virtual tangles in the study of finite type invariants of knotted objects by Bar-Natan and Dansco \cite{BND17}. Recently, Dansco, Halacheva and Robertson have shown that oriented circuit algebras are equivalent to wheeled PROPs \cite{DHR20}, and used this to give descriptions of the (graded) Kashiwara-Vergne group $\mathrm{KV}$ ($\mathrm{KRV}$), and the graded Grothendieck-Teichm\"uller group $\mathrm{GRT}$, as automorphism groups of circuit algebras \cite{DHR21}.

This paper compares circuit algebras with modular operads in the style of \cite{Ray20} (see also \cite{JK11, HRY19a, HRY19b}), and establishes a nerve theorem characterising circuit algebras as presheaves over a category of (disconnected) graphs that satisfy a so-called `Segal condition'.
The basic data of a (coloured) circuit algebra is given by a graded monoid $(A_n)_{n \in \N}$ in a symmetric monoidal category $(\V, \otimes, I)$, together with a levelwise action of the permutation groupoid, and an additional `contraction' operation that satisfies suitable axioms. By \cref{thm. CO CA}, circuit algebras may be viewed as modular operads equipped with an extra external (or monoidal) product. In previous literature \cite{DHR20, DHR21, Hal16, Tub14}, 
circuit algebras have been defined as algebras over an operad of wiring diagrams. The present work describes a family of operads of wiring diagrams in terms of categories of `coloured Brauer diagrams' (or coloured symmetric tangle categories), 
and this leads to a new, equivalent definition of a general class of (coloured) circuit algebras: 
\begin{thm}\label{thm. def CA}
	A (coloured) circuit algebra is a symmetric lax monoidal functor from a category of coloured Brauer diagrams.
\end{thm}

Versions of the categories of Brauer diagrams in \cref{thm. def CA} have been widely studied for nearly a century since Brauer's 1937 paper \cite{Bra37} extending Schur-Weyl duality to describe representations of the finite dimensional orthogonal and symplectic groups (e.g.~{\cite{Wen88, Koi89, LZ15}}). 
Hence, a pleasing consequence of the approach of this paper is that, via \cref{thm. def CA} (and \cref{thm. CO CA}), it establishes a dictionary for translating between categorical (Brauer categories), operadic (circuit algebras) and modular operadic approaches to active research in representation theory and quantum algebra. 

The main technical results of the paper are the construction -- using an iterated distributive law -- of a monad for circuit algebras, and the proof of an abstract `Weber-style' \cite{Web07, BMW12} nerve theorem which provides an explicit description of the graphical combinatorics of circuit algebras.  This is closely related to the graphical nerve theorem for modular operads of \cite[Theorem~8.2]{Ray20} and is based on the observation (\cref{thm. CO CA}) that, in the category of sets, circuit algebras coincide with modular operads equipped with an extra product operation. 
\begin{thm}
	[\cref{iterated law}]
	\label{thm: main intro}
	The category $\CA$ of set-valued circuit algebras is equivalent to the Eilenberg-Moore (EM) category of algebras for a monad $\OOO^\times$ on Joyal and Kock's graphical species category $\GS$ \cite{JK11}. 
\end{thm}

As in \cite{Ray20}, this enables the proof of a corresponding nerve theorem, characterising circuit algebras as functors from a full subcategory $\Klgrt\subset \CA$ whose objects are graphs. 
The nerve theorem for circuit algebras is proved using iterated distributive laws \cite{Che11}, together with the abstract nerve theory of \cite{BMW12}. Let $\pr{\Klgrt}$ denote the category of presheaves (contravariant $\Set$-valued functors) on $\Klgrt$, then:

\begin{thm}[\cref{nerve theorem}]
	\label{thm nerve intro}

	The nerve functor $N \colon \CA \to \pr{\Klgrt}$ induced by the inclusion $\Klgrt \hookrightarrow \CA$ is fully faithful. 
	Moreover, there is a canonical restriction functor $R^* \colon \pr{\Klgrt} \to \GS$, such that a presheaf $P$ on $\Klgrt$ is equivalent to the nerve of a circuit algebra if and only if $P$ is completely described by $R^*P$, in which case $P$ is said to `be Segal', or `satisfy the {Segal condition'}.
\end{thm}

Before outlining the construction, let us consider some reasons why such a result may be of interest. 

As has already been mentioned, the description of circuit algebras in terms of categories of Brauer diagrams indicates a fundamental relationship between circuit algebras, (topological) field theories, and representations of diverse algebraic structures -- including the classical Lie groups -- of interest in mathematics and physics. 
\cref{thm. def CA} enables us to translate circuit algebra and modular operad machinery to other (algebraic) contexts. Moreover, Theorems \ref{thm: main intro} and \ref{thm nerve intro}, establish a new graphical calculus for studying representations of Brauer categories. 

 In \cite[Section~6]{Ray20}, I explained how the so-called `problem of loops' creates an obstacle to constructing the modular operad monad on $\GS$. 
This is also an issue for the circuit algebra monad. 
In each case, the problem is resolved using distributive laws \cite{Bec69}. See  \cite[Section~7]{Ray20} for the construction of the modular operad monad. \cref{sec. iterated} describes the construction of the circuit algebra monad. For circuit algebras, it is even necessary to construct a triple of monads, and an iterated distributive law \cite{Che11}. 
\begin{thm}[\cref{prop. iterated law} {\&} \cref{iterated law}]\label{thm: composite intro}	\label{thm. dist intro}
	There is a triple of monads $\LL, \DD$ and $\TT$ on $\GS$ and distributive laws between them that 
	describe an iterated distributive law \cite{Che11}. The EM category of algebras for the induced composite monad $\LL\DD\TT$ is equivalent (via \cref{thm. CO CA}) to the category $\CA$ of ($\Set$-valued) circuit algebras. 
\end{thm}

Here, the composite $\DD\TT$ is the modular operad monad on $\GS$ described in \cite[Section~7]{Ray20}, and $\TT$ is the non-unital modular operad monad of \cite[Section~5]{Ray20}. Hence \cref{thm. dist intro} implies that circuit algebras are, equivalently, modular operads with extra structure. 

By \cref{unpointed modular operad} and \cref{lem. LT is Tk}, algebras for the monad $\LL \TT$ are non-unital circuit algebras, whose underlying modular operadic multiplication is not required to have a multiplicative unit. And, by \cref{sec: non-unital} these 
are just symmetric monoidal functors from subcategories of `\textit{downward}' Brauer diagrams \cite{SS15, KRW21}. Hence, it follows immediately from \cref{sec: non-unital} and \cite[Sections~3~\&~4]{SS15}, that the 
categories of representations of the infinite orthogonal and symplectic groups (respectively the infinite general linear group) are canonically described by full subcategories of non-unital monochrome circuit algebras (respectively non-unital monochrome oriented circuit algebras). A similar idea has been exploited by 
Kupers and Randall-Williams in \cite{KRW21}  to prove that the Malcev Lie algebras associated to the Torelli groups of surfaces of arbitrary genus are stably Koszul. The same observation appears in \cite{Sto22} where modular operads are obtained as lax functors from the \textit{Brauer properad}.

Unital circuit algebras are more complicated due to the problem of loops \cite{Ray20}. \cref{thm. def CA} identifies representations of so-called Brauer categories -- in the sense of \cite{LZ15}, and recently studied in e.g~\cite{Rey15, RS20, SS20I} -- with circuit algebras satisfying additional dimension and symmetry constraints: for example, there are circuit algebras describing  sequences of representations of finite dimensional orthogonal groups. It is an ongoing project, with R. Street, to further investigate this approach to representations of Brauer (and related) categories.

Weakening the axioms to consider $(\infty, 1)$ or \textit{up-to-homotopy} circuit algebras, provides a second motivation for results of this paper. This is relevant, for example, to the work of Dansco, Halacheva and Robertson comparing the {Grothendieck-Teichm\"uller} and {Kashiwara-Verne} groups $\mathrm{GT}$ and $\mathrm{KV}$. In \cite{DHR21}, they have used circuit algebras to obtain results relating the \textit{graded} {Grothendieck-Teichm\"uller} and {Kashiwara-Verne} groups $\mathrm{GRT}$ and $\mathrm{KRV}$. However, under passage to pro-unipotent completion, the circuit algebra operations are not preserved on the nose. Hence, to obtain an ungraded version of these results relating $\mathrm{GT}$ and $\mathrm{KV}$ directly, it is necessary to adopt a weakened notion of circuit algebras (see \cite[Introduction,~ Remark~1.1]{DHR21}). Since \cref{thm nerve intro} gives a description of ($\Set$-valued) circuit algebras in terms of a \textit{strong} Segal condition, it is natural to define $(\infty,1)$-circuit algebras as 9simplicial set valued) presheaves that satisfy a weakened Segal condition. 
Indeed, \cite{HRY19b} and \cite{Ray20} proved -- via different methods, and using slightly different graphical categories -- that 
weakening the modular operadic Segal condition gives a reasonable notion of $(\infty, 1)$-modular operads, that corresponds with the fibrant objects in a model structure on the relevant category of graphical presheaves. Unfortunately, neither of these models generalises immediately to circuit algebras. Nonetheless, one hopes that the characterisations of circuit algebras, both as algebras of operads of wiring diagrams, and in terms of algebras for the monad $\LL \DD\TT$, should indicate how to transfer an existing model structure on subcategory of \textit{dendroidal sets} to obtain a model on $\Klgrt$-spaces, and thereby obtain $(\infty, 1)$-circuit algebras in terms of the expected weak version of the Segal condition in \cref{thm nerve intro}.

Finally, it is worth highlighting the abstract machinery of this paper -- specifically the combined application of (iterated) distributive laws and abstract nerve theory -- that extends those of \cite{Ray20}. The monad for (set-valued) circuit algebras is constructed in \cref{sec. monad CO} as a composition -- via an iterated distributive law \cite{Che11} -- of three monads, each of which governs a different aspect of the structure of circuit algebras. The proof of \cref{thm nerve intro} uses the abstract nerve theory of \cite{BMW12}, but (as with \cite[Theorem~8.2]{Ray20}) essentially depends on the decomposition of the circuit algebra monad into its constituent parts. The same technique is used again in \cite{Ray20comp} (in preparation). 
An accessible account of the interplay of abstract nerve theory and distributive laws is given in \cref{sec. cat intro}. It would be interesting to explore if this strategy of using distributive laws to broaden the scope for applying abstract nerve theory has wider applications.

\subsection{Details and overview}

In \cite{Ray20}, it was shown that modular operads are algebras for a monad on a certain combinatorial category $\GS$ whose objects are coloured collections called `graphical species' \cite{JK11}, and a graphical category and corresponding nerve theorem for modular operads were obtained, using the nerve machinery of \cite{BMW12}.   

This paper builds on that result by comparing circuit algebras with graded monoids in the category of modular operads, here called `circuit operads'. 
It will be proved in \cref{thm. CO CA} that, in the category of sets, circuit algebras and circuit operads are equivalent. 
More generally, circuit algebras with values in an arbitrary symmetric monoidal category $(\V, \otimes, I)$ are \textit{enrichments} of $\Set$-valued circuit algebras, with a set (or {palette}) of objects (colours), and a set of $\V$-objects encoding (multi-)morphisms between them. 
By contrast, circuit operads satisfy the same axioms, but \textit{internal} to a category $\E$ with finite limits. Hence circuit operads in $\E$ are described by an \textit{object object}, 
and a \textit{graded (multi)-morphism object} in $\E$.

The proof of Theorems \ref{thm: main intro} and \ref{thm nerve intro} relies on an iterated distributive law and the abstract nerve machinery of \cite{BMW12, BG19}. \cref{sec. cat intro} provides a relatively informal and accessible introduction to these notions. 

After establishing notation and terminology for symmetric monoidal categories, \cref{sec. BD} provides a detailed discussion of the categories of (coloured) Brauer diagrams and their functors. These are used, in \cref{sec. CA}, to define circuit algebras.

\cref{ssec operads} provides a quick introduction to operads and their algebras. Operads of wiring diagrams and circuit algebras are introduced in \cref{ssec. wd and ca}.
 
 In \cref{sec: definitions}, circuit operads are introduced as presheaves over the category $\GS$ of graphical species \cite{JK11}. A canonical functor from circuit algebras in $\Set$ to circuit operads is described, as well as a monadic adjunction between circuit operads and modular operads (as they are defined in \cite{Ray20}). 

The technical work of the paper starts at \cref{s. graphs}. This section is largely a review of basic graph constructions, based, as in \cite{Ray20} (and \cite {HRY19a, HRY19b}), on the definition of graphs introduced in \cite{JK11}.

A monad $\TTk$ for non-unital circuit operads is constructed in \cref{sec: non-unital}. This is almost identical to the monad discussed by Kock in \cite{Koc18} and very similar to the monad $\TT$ for non-unital modular operads constructed in \cite{Ray20}. \cref{sec. mod op} reviews (and extends) the construction in \cite{Ray20} of the modular operad monad $\DD\TT$ by composing $\TT$ with a unit monad $\DD$. 

 However, in spite of the close similarities between $\TTk$ and $\TT$, it is not possible to obtain the monad for unital circuit operads directly from a distributive law involving $\TTk$. 
Instead, it is necessary to supplement the construction of \cite{Ray20} by observing that $\TTk$ is, itself a composite monad. 

This happens in \cref{sec. monad CO}, in which an iterated distributive law is described and it is shown that the resulting composite $\LL \DD \TT$ of three monads on $\GS$ is indeed the desired circuit operad monad. Moreover, it is proved that the forgetful functor from $\Set$-valued circuit algebras to $\Set$-valued circuit operads is, in fact, an equivalence of categories. 

Finally, \cref{sec. nerve} contains a description of the graphical category $\Klgrt$ for circuit operads. From there,  the proof of the nerve theorem (\cref{thm nerve intro}) is straightforward. 
Let $\GSp$ be the EM category of algebras for the monad $\DD$ on $\GS$. By the classical theory of distributive laws \cite{Bec69}, there is a lifted monad $\widetilde{ \LL }\TTp$  on $\GSp$ whose algebras are circuit operads. \cref{thm nerve intro} is then implied by the observation that 
the lifted monad $\widetilde{ \LL }\TTp$ on $\GSp$ 
satisfies the conditions of the abstract nerve machinery of \cite{BMW12}: it has `arities' (see \cref{ssec: Weber}). 

\subsection*{Acknowledgements}
I thank Marcy Robertson for taking the time to explain her work on $\mathrm{GT}$, for motivating me to write this paper, and, most of all, for her unfailing encouragement and guidance. I am also grateful to Richard Garner for a small observation that made a big difference to the character and reach of this work.
I am grateful to the members of Centre of Australian Category Theory, Macquarie University for their support, friendship, and mathematical input.	

I worked on this paper at Macquarie in the first half of 2021. I acknowledge my good fortune in being able to enjoy the benefits of working on campus -- whilst accessing a wealth and diversity of mathematical conversations across the globe -- at a time when most of the world was dealing with much more profound consequences of the COVID pandemic.  I acknowledge and pay my respects to the Darug Nation, the traditional owners of the land on which Macquarie is built, and to the Bindal and Wulgurukaba  Peoples upon whose lands I now live and work.

\section{Distributive laws and abstract nerve theory}\label{sec. cat intro}

The main result of this work is the construction of a monad (\cref{iterated law}), and a graphical category and nerve theorem (\cref{thm: CO nerve}) for circuit algebras that give a complete description of their combinatorics. The method uses the machinery of abstract nerve theory \cite{BMW12, Web07}, in combination with (iterated) distributive laws \cite{Bec69, Che11}.

Given its connections with representations of classical groups, this work is aimed at a diverse audience including representation theorists and topologists, as well as category theorists. Therefore, this section is intended to 
provide an accessible explanation of these formal concepts and how they combine to prove \cref{thm: CO nerve}.

\subsection{Iterated distributive laws} \label{ssec. dist} 

Informally, monads are gadgets that encode, via their algebras, {(algebraic) structure} on objects. The key principle underlying this work is that it is possible to build up a description of more complicated algebraic structures on the objects of some category by suitably `combining' monads that describe simpler structures. To do this, we need distributive laws \cite{Bec69}.

For a more detailed review of the general theory of monads and their algebras, see for example {\cite[Chapter~VI]{Mac98}. For now, it suffices that a monad $\MM = (M, \mu^{\MM}, \eta^\MM)$ on a category $\CCat$ is given by an endofunctor $M \colon \CCat \to \CCat$, together with natural transformations $\mu^{\MM} \colon M^2 \Rightarrow M$ (called the \textit{monadic multiplication}), and $\eta^{\MM} \colon \mathrm{id}_{\Cat}\Rightarrow M$ (called the \textit{monadic unit}), that satisfy certain axioms making $\MM$ into an associative monoid in the category of endofunctors on $\CCat$.
	
	 An algebra $(x, \theta)$ for $\MM$ consists of an object $ x$ of $\CCat$, and a morphism $\theta \colon Mx \to x$ that satisfies two axioms that assert the compatibility of $\theta$ with $\mu^{\MM}$ and $\eta^{\MM}$. Algebras are the objects of the \textit{Eilenberg-Moore (EM) category of algebras for $\MM$}, whose morphisms $(x,\theta) \to (y, \rho)$  are given by morphisms $f  \in \CCat(x,y)$ such that $\rho \circ Mf   = f \circ  \theta $.

In general, monads do not compose, since the data of a pair of monads  $\MM = (M, \mu^{\MM}, \eta^{\MM})$ and $ \MM' = ({M'}, \mu^{\MM'}, \eta^{\MM'})$ on a category $\CCat$ is not sufficient to define a multiplication for the composite endofunctor $ M M' \colon \CCat \to \CCat$. 

Observe however that, given any natural transformation $\lambda\colon{M'}M \Rightarrow M{M'}$, there is an induced natural transformation $\mu_\lambda \defeq (\mu^{\mathbb{M}}\mu^{\mathbb{{M'}}})\circ (M\lambda M')  \colon( M{M'})^2 \Rightarrow M{M'}$. 

\begin{defn}\label{defn: dist} 
	A \emph{distributive law} for $\MM$ and $ \MM'$ is a natural transformation $\lambda\colon{M'}M \Rightarrow M{M'}$ such that the triple $(M{M'}, \mu_\lambda , \eta^{\mathbb{M}}\eta^{\mathbb{{M'}}})$ defines a monad $\mathbb{M} \mathbb{{M'}}$ on $\CCat$.
	
\end{defn}

 Distributive laws were first introduced by Beck \cite{Bec69} and are characterised by four axioms  that 
determine how the algebraic structures encoded by the separate monads 
interact with each other to form the structure encoded by their composite. 

\begin{ex}
	\label{ex: category composite} The category of small categories is isomorphic to the EM category of algebras for the category monad on the category of directed graphs. This monad sends a directed graph $G$ to the graph $\widetilde G$ with the same vertices, and a directed edge $v \leadsto w$ for each directed path (of length $n \geq 0$) from $v$ to $w$ in $G$.
	
	This monad may be obtained as a composite of the \textit{semi-category monad} -- which governs associative composition and assigns to each directed graph $G$ the graph with the same vertices as $G$ and edges corresponding to non-trivial directed paths (of length $n >0$) in $G$ -- and the \textit{reflexive graph monad} that adjoins a distinguished loop at each vertex of a directed graph. 
	The corresponding distributive law encodes the property that the adjoined loops provide identities for the (semi-categorical) composition.

\end{ex}
\begin{ex}
In \cite{Ray20}, the monad for modular operads on the category $\GS$ of graphical species (\cref{defn: graphical species}) is obtained, as in  \cref{ex: category composite} for categories, as a composite $\DD\TT$: the monad $\TT$ governs the composition structure (multiplication and contraction), and the monad $\DD$ adjoins distinguished elements that encode the combinatorics of the multiplicative unit. The distributive law ensures that the distinguished elements provide units for the multiplication (see also \cref{sec. mod op}).
\end{ex}

As usual, let $\CCat^\MM$ denote the EM category of algebras for a monad $\MM$ on $\CCat$.

Given monads $\MM$ and $\MM'$ on $\CCat$, and a distributive law $\lambda\colon {M'}M \Rightarrow M{M'}$, there is, by \cite[Section~3]{Bec69}, a commuting square of strict monadic adjunctions:
\begin{equation}\label{eq: adjunction diag}\small
{\xymatrix@C = .5cm@R = .3cm{
		\CCat^{\MM'} \ar@<-5pt>[rr]\ar@<-5pt>[dd]&\scriptsize{\top}&\CCat^{\MM\MM'} \ar@<-5pt>[ll]\ar@<-5pt>[dd]\\
		\vdash&&\vdash\\
		\CCat \ar@<-5pt>[uu] \ar@<-5pt>[rr]&\scriptsize{\top}& \CCat^\MM. \ar@<-5pt>[ll]\ar@<-5pt>[uu]}}\end{equation}

Moreover, by \cite{Che11},  the composite $\MM_n \dots \MM_1$ of an $n$-tuple of monads $\MM_1, \dots,\MM_n$ on $\CCat$ exists if there are pairwise distributive laws $\lambda_{i,j}: M_i M_j \Rightarrow M_j M_i$, for $1\leq i < j \leq n$ such that for each triple $1 \leq i < j< k \leq n$, the corresponding triple of monads and distributive laws satisfies (Yang-Baxter) conditions described in (\ref{YB}). In this case, there is an $n$-cube of adjunctions 
	as illustrated in (\ref{eq. cube}) for the case $n = 3$ (where only the left adjoints are marked):

\begin{equation}\label{eq. cube}
	\begin{tikzcd}[column sep={6em,between origins},row sep=2em]
	&
		\CCat^{\MM_3\MM_2}
	\arrow[rr, rightarrow]
	&&
		\CCat^{\MM_3\MM_2\MM_1}
	\\
		\CCat^{\MM_3}
	\arrow[rr, rightarrow]
	\arrow[ur,rightarrow]
	&&
		\CCat^{\MM_3\MM_1}
	\arrow[ur, rightarrow]
	\\
	&
		\CCat^{\MM_2}
			\arrow[rr, rightarrow]
		\arrow[uu,rightarrow]
	&&
		\CCat^{\MM_2\MM_1}
		\arrow[uu, rightarrow]
	\\
\CCat 
	\arrow[rr, rightarrow]
		\arrow[ur,rightarrow]
			\arrow[uu,rightarrow]
	&&
	\CCat^{\MM_1}
	\arrow[uu,rightarrow,crossing over] 
	\arrow[ur,rightarrow]
	\end{tikzcd}
\end{equation}

\subsection{Presheaves and slice categories}\label{ssec. presheaves}

Before discussing nerve functors, and the abstract nerve theory of \cite{BMW12, Web07}, it is first necessary to review (and establish notation for) certain categories of presheaves.

So, let $\E$ be a category (with finite limits), and let $\CCat$ be an essentially small category. An $\E$-valued \emph{$\CCat$-presheaf} is a functor $ S\colon \mathsf C^\mathrm{op} \to \E$. The corresponding functor category is denoted by $\prE{\mathsf C}$. If $\E = \Set$, the category of sets and all set maps, then $\pr{\CCat} \defeq \prE{\CCat}$. 

When $\E =  \Set$, the category of  elements (or \textit{Grothendieck construction}) of a presheaf $P \colon \CCat^{\mathrm{op}}\to \Set$ is defined as follows:
\begin{defn}
	\label{defn: general element} 
	Objects of the category $\ElP{P}{\CCat}$ of \textit{elements of a presheaf $P \colon \CCat^{\mathrm{op}} \to \Set$} are pairs $(c, x)$ -- called \emph{elements of $P$} -- where $c$ is an object of $\CCat$ and $x \in P(c)$. Morphisms $(c,x) \to (d,y)$ in $\ElP{P}{\CCat}$ are given by morphisms $f \in \CCat (c,d)$ such that $P(f)(y) = x$. 	
\end{defn}

Recall that, if a presheaf $P \colon \CCat^{\mathrm{op}} \to \Set$ is representable (i.e.~ $P =  \CCat(-, c)$, for some $c \in \CCat$), then $\ElP{P}{\CCat}$ is precisely the \textit{slice category} $\CCat\ov c$ whose objects are pairs $(d,f)$ where $f \in \CCat (d,c)$, and morphisms $(d,f) \to (d', f')$ are commuting triangles in $\CCat$:
\[ \xymatrix{ d \ar[rr]^-g \ar[dr]_{f} && d' \ar[dl]^-{f'}\\&c&}\]

In general, a functor $\iota \colon \DCat \to \CCat$ induces a presheaf $\iota^* \CCat(-, c)$, $d \mapsto \CCat(\iota (d),c)$ on $\DCat$. For all $c \in \CCat$, the category $\ElP{\iota^*\CCat(-, c) }{\DCat}$ is called the \textit{slice category of $\DCat$ over $c$} and is denoted by $\iota \ov c$, or simply $\DCat \ov c$. 

 The \textit{slice category of $c$ over $\DCat$}, denoted by $c \ov \DCat$ (or $c \ov \iota$) is the category $\iota^{\mathrm{op}} \ov c$ whose objects are morphisms $f \in \CCat(c, \iota(-))$ and whose morphisms $(d,f)\to (d', f)$ are given by $g \in \DCat (d,d')$ such that $ \iota (g) \circ g = f' \in \CCat(c, \iota(d'))$.

\begin{ex}\label{ex. dense}
	Let $\iota \colon \DCat \hookrightarrow \CCat$ be an inclusion of small categories. By  \cite[Proposition~5.1]{Lam66}, the induced nerve $N \colon \CCat \to \pr{\DCat}$ if fully faithful if and only if every object $c \in \CCat$ is canonically the colimit of the forgetful functor $\DCat \ov c \to \CCat$, $ (f \colon \iota (d) \to c) \mapsto \iota (d)$. In this case, the inclusion $\iota \colon \DCat \hookrightarrow \CCat$ is called \textit{dense}.
\end{ex}

{Observe, in particular, that the Yoneda embedding $\CCat \to \pr{\CCat}$ induces a canonical isomorphism $ \ElP{P}{\CCat} \cong \CCat \ov P$ for all presheaves $P\colon \CCat ^{\mathrm{op}} \to \Set$. These categories will be identified in this work. }

\subsection{Abstract nerve theory}\label{ssec: Weber}

A functor $F\colon \DCat \to \CCat$ (between ordinary categories) induces a \textit{nerve functor}
$N_{\DCat}\colon \CCat \to \pr{\DCat}$ given by $N_{\DCat}(c)(d) = \CCat(F d, c) $ for all $c \in \CCat$ and $ d \in \DCat$. 

\begin{ex}\label{ex: cat nerve}
The nerve of a small category $\CCat$ is described by the image of $\CCat$ under the nerve functor induced by the inclusion $\Delta \hookrightarrow \Cat$, where $\Cat$ is the category of small categories, and $\Delta$ is the (topological) \textit{simplex category} of non-empty finite ordinals and order preserving morphisms. 
\end{ex}

Recall that every functor 
admits an (up to isomorphism) unique \textit{bo-ff factorisation} as a bijective on objects functor followed by a fully faithful functor. For example, if $\MM$ is a monad on a category $\CCat$, and $\CCat^{\MM}$ is the EM category of algebras for $\MM$, then the free functor $ \CCat \to \CCat^\MM$ has bo-ff factorisation $\CCat \rightarrow \CCat_{\MM} \rightarrow \CCat^{\MM}$, where $\CCat_{\MM}$ is the {Kleisli category} of free $\MM$-algebras (see e.g.\ \cite[Section~VI.5]{Mac98}). 

Hence, for any subcategory $\DCat $ of $ \CCat$, the bo-ff factorisation of the canonical functor $\DCat \hookrightarrow \CCat \xrightarrow {\text{ free}} \CCat^\MM$ factors through the full subcategory $\Theta_{\MM, \DCat}$ of $\CCat_\MM $ with objects from $\DCat$. There is 
an induced diagram of functors 
\begin{equation} \label{eq: arities}
\xymatrix{ 
	\boffcat\ar[rr]^{\text{f.f.}}						&& 	\CCat^\mathbb{M}\ar@<2pt>[d]^-{\text{forget}}
	\ar [rr]^-{N_{\mathbb{M}, \DCat}}				&&\pr{	\boffcat}\ar[d]^{j^*}	\\
	\DCat\ar@{^{(}->} [rr]^-{\text{}}_-{\text{}}	 \ar[u]_{j}^{\text{b.o.}}		&&\CCat \ar [rr]_{\text{f.f.}}^-{N_{\DCat}}	 \ar@<2pt>[u]^-{\text{free}}	
	&&\pr{\DCat}.
}
\end{equation} where  $j^*$ is the pullback of the bijective on objects functor $j \colon \DCat \to \Theta_{\MM, \DCat} $. The left square of (\ref{eq: arities}) commutes by definition, and the right square commutes up to natural isomorphism. 

By construction, the defining functor $\Theta_{\MM, \DCat} \to \CCat^{\MM}$ is fully faithful. Abstract nerve theory describes conditions on $\DCat$ and $\MM$ that ensure that the induced nerve $N_{\MM, \DCat} \colon \CCat^{\MM} \to \pr{\Theta_{\MM, \DCat}}$ is also fully faithful \cite{BMW12, Web07}. If there is a subcategory $\DCat \subset \CCat$ such that $\MM$ and $\DCat$ satisfy these conditions, then $\MM$ is called a \textit{monad with arities}, and $\DCat$ \textit{provides arities for $\MM$.}

Moreover, by \cite[Section~4]{Web07}, if $\DCat$ provides arities for $\MM$, then objects in the image of $N_{\MM, \DCat}$ in $ \pr{	\boffcat}$ are (up to equivalence) precisely those that are mapped by $j^*$ to the image of $N_{\DCat}$ in $\pr{\DCat}$. 
This is called the \textit{Segal condition} for the nerve theorem.

\subsection{Using distributive laws to prove nerve theorems}\label{ssec. combine}
Now, let $\DCat \subset \CCat$ be a dense subcategory of an essentially small category, and let $\MM$ be a monad on $\CCat$. 

Even if $\MM$ and $\DCat$ do not satisfy the conditions of \cite[Sections~1~\&~2]{BMW12}, if $\MM  = \MM^n \dots \MM^1$ is a composite monad, then it may still be possible to use abstract nerve theory, together with the distributive laws for the composite $\MM^n \dots \MM^1$, to establish that the induced nerve $N_{\MM, \DCat} \colon \CCat^{\MM} \to \pr{\Theta_{\MM, \DCat}}$ is fully faithful. 

Namely, the existence of pairwise distributive laws means that each of the $n!$ paths of left adjoints from $\CCat$ to $\CCat^\MM$ in the cube (\ref{eq. cube}) describes a distinct factorisation of the diagram (\ref{eq: arities}) into $n$ stages. The nerve $N_{\MM, \DCat}\colon \CCat^\MM\to \pr{\Theta_{\MM, \DCat}}$ is fully faithful if the conditions of \cite{BMW12} hold for any of these stages.

For example, if $\MM = \MM^n \dots \MM^1$, then there is a factorisation of the diagram (\ref{eq: arities}) as
\begin{equation}
\label{eq. iterated arities} 
\xymatrix@R = .3cm{ 
		\boffcat \ar[rr]^{\text{f.f.}}		\ar@{=}[d]				&& 	\CCat^{\MM}\ar@{=}[d]
	\ar [rr]^-{N_{\MM,\DCat} 	}		&&\pr{	\boffcat }\ar@{=}[d]	\\
	\Theta_{\MM_n \dots \MM_1, \DCat}\ar[rr]^{\text{f.f.}}						&& 	\CCat^{\MM_n \dots \MM_1}\ar@<2pt>[dd]^-{\text{forget}}
\ar [rr]^-{N_{\MM_n \dots \MM_1,\DCat} 	}		&&\pr{		\Theta_{\MM_n \dots \MM_1,\DCat}}\ar[dd]^-{j_{n}^*}	\\ &&&&\\
\vdots  \ar[uu]_-{j_n}^{\text{b.o.}}&& \vdots \ar@<2pt>[uu]^-{\text{free}}\ar@<2pt>[d] 
&& \vdots \ar[d]\\	
\Theta_{\MM_1, \DCat}\ar[rr]^{\text{f.f.}}			\ar[u]			&& 	\CCat^{\mathbb{M}_1}\ar@<2pt>[dd]^-{\text{forget}}\ar@<2pt>[u]
\ar [rr]^-{N_{\MM_1, \DCat}}				&&\pr{\Theta_{\MM_1, \DCat}}\ar[dd]^-{j_1^*}	\\&&&&\\
\DCat\ar@{^{(}->} [rr]^-{\text{}}_-{\text{}}	 \ar[uu]_-{j_1}^{\text{b.o.}}		&&\CCat \ar [rr]_{\text{f.f.}}^-{N_{\DCat}}	 \ar@<2pt>[uu]^-{\text{free}}	
&&\pr{\DCat}.
}
\end{equation} 

Hence, to prove that the induced nerve $N_{\MM, \DCat} = N_{\MM_n \dots \MM_1, \DCat}$ on $\CCat^{\MM}$ 
is fully faithful, it is sufficient to consider whether there is an $1 <i < n$ such that the conditions of \cite{BMW12} are satisfied by 
$\Theta_{\MM_n \dots \MM_i, \DCat} $ 
together with the lift of the monad $\MM_n \dots \MM_{i+1}$ on $\CCat^{\MM_i \dots \MM_1}$. The same trick can be applied for every other path in the cube (\ref{eq. cube}).

This is the method used in \cref{sec. nerve} to prove \cref{thm: CO nerve}, the nerve theorem for circuit operads: In \cref{sec. iterated}, the monad for circuit operads is described as a composite $\LL \DD \TT$ of three monads defined on the category $\GS$ of graphical species. On $\GS$, the composite $\LL \DD\TT$ does not satisfy the conditions of \cite{BMW12}. However, the induced monad $\widetilde{ \LL} \TTp$ on the category $\GSp$ of $\DD$-algebras does. Hence, there is a full subcategory $\Klgrt \subset \CO$ of the category of circuit operads that induces a fully faithful nerve functor $\CO \to \pr{\Klgrt}$ (\cref{thm: CO nerve}).

\section{Brauer diagrams}\label{sec. BD}

In \cref{sec. CA}, circuit algebras will be defined in terms of Brauer diagrams. These are an important tool in the representation theory of orthogonal, symplectic and general linear groups \cite{Bra37, SS20I}.

The category $\BD$ of monochrome Brauer diagrams is introduced in \cref{ssec. mono BD}. This definition is generalised to the categories $\CBD$ of coloured Brauer diagrams in \cref{ssec. colour BD}. 

Before defining these categories, \cref{ssec. SMC} provides a brief outline  of the notation and terminology conventions for symmetric monoidal categories 
that will be used in the rest of the paper.

\subsection{Symmetric monoidal categories and duality}\label{ssec. SMC}

For precise definitions and a detailed discussion, see e.g.~\cite[Chapter~8]{EGN15}.

Recall that a \textit{monoidal category} is a category $\V$ together with a bifunctor $\otimes \colon \V \times \V \to \V$ (the \textit{monoidal product}) that is associative up to natural \textit{associator} isomorphism, and for which there is an object $I$ of $\V$ (the \textit{monoidal unit}) that acts as a two-sided identity for $\otimes$ up to natural (left and right) \textit{unitor} isomorphisms. The monoidal product and the associator and unitor isomorphisms are required to satisfy axioms that mean that 
certain sensible diagrams commute. A monoidal category 
is called \textit{strict monoidal} if the associator and unitor isomorphisms are the identity.

A monoidal structure on $\V $ is \textit{braided} if, 
for all objects $x, y \in \V$, there are isomorphisms $\sigma_{x, y} \colon x \otimes y \rightarrow y \otimes x $ that satisfy the familiar braid 
identities 
\begin{equation} \label{eq. braids} (\sigma_{y, z}\otimes id_x)(id_y \otimes\sigma_{x, z})(\sigma_{x, y} \otimes id_z) = (id_z\otimes \sigma_{x, y})(\sigma_{x, z}\otimes id_y)( id_x\otimes\sigma_{y, z})  \ \text{ for all } x, y, z. \end{equation} If $\sigma_{y,x} =  \sigma_{x,y}^{-1}$ for all $x, y$, then the monoidal structure on $\V$ is \textit{symmetric}. 

\begin{rmk}
	
	In this paper, the associators, unitors and symmetry isomorphisms will be ignored in the notation, and (symmetric) monoidal categories will be denoted simply by $\V$, or sometimes $(\V, \otimes, I)$.

\end{rmk}

Symmetric strict monoidal categories are called \textit{permutative categories}. The notation $\oplus$ and $0$ will often be used to designate the monoidal product and unit of a permutative category $(\CCat, \oplus, 0)$. 

A \textit{$\DDD$-coloured PROP} is a small permutative category $\CCat$ whose object monoid is free on a set $\DDD$. When $\DDD = \{1\}$ is a singleton, then $\CCat$ is a (monochrome) PROP in the original sense of \cite{MacL65}.

A \textit{(lax) monoidal functor} $(\V_1, \otimes_1, I_1) \to (\V_2, \otimes_2, I_2)$ is described by a functor 
$\Theta \colon \V_1 \to \V_2$, together with a morphism $\eta_{\Theta} \colon I_2 \rightarrow \Theta (I_1)$ in $\V_2$ and a natural transformation $\theta\colon \Theta(-) \otimes_{2} \Theta(-) \Rightarrow \Theta(- \otimes_1 -)$ such that all the expected structure diagrams commute. A monoidal functor $ (\Theta, \theta, \eta_{\Theta} )$ is called \textit{strong} if $\theta$ and $ \eta_{\Theta}$ are invertible, and \textit{strict} if they are the identity. For convenience, $\eta_{\Theta}$, and $\theta$ will usually be suppressed in the notation.

Symmetric monoidal categories form a category $\Sym$ whose 
morphisms $ \Theta\colon (\V_1, \otimes_1, I_1) \to (\V_2, \otimes_2, I_2)$ 
are {monoidal functors} that preserve the symmetry isomorphisms strictly. 

\begin{ex}\label{ex. Sigma}
For each $n \in \N$, let $ \nn$ denote the set $\{1, 2, \dots, n\}$ (so $\nul = \emptyset$), and let $ \Sigma_n$ be the group of permutations on $\nn$. 

	Let $\Sigma$ be the \emph{symmetric groupoid} with $\Sigma(n,n) = \Sigma_n$ for all $n$, and $\Sigma (m,n) = \emptyset$ when $m \neq n$. There is a permutative structure on $\Sigma$ induced by addition of natural numbers. For any symmetric monoidal category $(\V, \otimes, I)$ and any choice of object $x \in \V$, there is a unique symmetric strict monoidal functor $\Sigma \to \V$ with $0\mapsto I$ and $1 \mapsto x $. 	
\end{ex}

An object $x$ of a symmetric monoidal category $\V$ has a dual object $x^* $ in $\V$ if there are morphisms  $\cap^{\CCat}_x \colon I\to x \otimes x^*$, $\cup^{\CCat}_x \colon x^* \otimes x \to I$ such that the \textit{triangle identities} hold:
\begin{equation} \label{eq. dual}
	(\cup^{\CCat}_x \otimes id_x) \circ (id_x \otimes \cap^{\CCat}_x)= id_x = (id_x \otimes \cup^{\CCat}_{x^*} )\circ (\cap^{\CCat}_{x^*} \otimes id_x).
\end{equation}.

	\begin{figure}	[htb!]

			\begin{tikzpicture}[scale = .16]
				\begin{pgfonlayer}{background}
			
					\node  (2) at (3.75, 5) {};
				
					\node  (4) at (1.25, 5) {};
			
					\node  (7) at (6.25, 5) {};
				
					\node  (10) at (1.25, 10) {};
			
					\node  (13) at (6.25, 0) {};
				
					
					\node  (15) at (25, 5.25) {};
					\node  (16) at (21, 5.25) {};
					\node  (17) at (21.75, 5.25) {};
					\node  (18) at (22.5, 5.25) {};
					\node  (19) at (24.25, 5.25) {};
					\node  (20) at (23.5, 5.25) {};
					\node  (21) at (18.5, 5.25) {};
					\node  (22) at (19.25, 5.25) {};
					\node  (23) at (20, 5.25) {};
					\node  (24) at (25, 10.25) {};
					\node  (25) at (24.25, 10.25) {};
					\node  (26) at (23.5, 10.25) {};
					\node  (27) at (18.5, 0.25) {};
					\node  (28) at (19.25, 0.25) {};
					\node  (29) at (20, 0.25) {};

					\node  (30) at (12.75, 10) {};
					\node  (32) at (12, 10) {};
					\node  (33) at (12, 0) {};
					\node  (35) at (12.75, 0) {};
					
				\end{pgfonlayer}
				\begin{pgfonlayer}{above}
					\draw (10.center) to (4.center);
					\draw (7.center) to (13.center);
					\draw [bend right=90, looseness=2.25] (4.center) to (2.center);
					\draw [bend left=90, looseness=2.50] (2.center) to (7.center);
					\draw (25.center) to (19.center);
					\draw (22.center) to (28.center);
					\draw [bend left=90, looseness=2.00] (19.center) to (17.center);
					\draw [bend right=90, looseness=2.25] (17.center) to (22.center);
					\draw (32.center) to (33.center);

					\draw 
					(8, 4.8)--(9.5,4.8)
					(8, 5.3)--(9.5,5.3);
					
						\draw 
					(14.5, 4.8)--(16,4.8)
					(14.5, 5.3)--(16,5.3);

				\end{pgfonlayer}
		\end{tikzpicture}
	\caption{ The triangle identities. } \label{fig. triangle 1} \end{figure}

\begin{defn}
	\label{defn. compact closed} A \emph{compact closed category} is a symmetric monoidal category  such that every object has a dual \cite{MacL65}.
	
\end{defn}

Let $\V$ be a compact closed category. For all morphisms $f \in \V(x,y)$, there is a corresponding \emph{evaluation morphism} $\ev{f} \in \V(y^* \otimes x, I)$ induced by composition with composition with $\cup_y$ (\cref{fig eval} (a)), and \emph{coevaluation morphism} $\coev{f} \in \V(I, y \otimes x^*)$ induced by composition with $\cap_x$(\cref{fig eval} (b)):
\begin{equation}\label{eq ev coev dual}
	\ev{f} \defeq \cup_y\circ (id_y \otimes f)  \text{ and }  \coev{f} \defeq  (f\otimes  id_{x})\circ \cap_x.
\end{equation}
And there is a \emph{dual morphism} $f^* \in \V(y^*, x^*)$ (\cref{fig eval} (c)) 
given by
\begin{equation}\label{eq. BD dual}
	f^* \defeq (\cup_y \otimes id_x)\circ ( id_{y} \otimes f \otimes  id_x ) \circ ( id_{y} \otimes  \cap_x). 
\end{equation}
(The dual morphism $f^*$ is called the \textit{transpose morphism} in \cite{Sha21}.)

	\begin{figure}	[htb!]

	\begin{tikzpicture}
			\node at (-1,1) {(a)};		
\node at (0,0){\begin{tikzpicture}[scale = .2]
	\begin{pgfonlayer}{background}
		\node  (0) at (-2, 0) {};
		\node  (1) at (2, 0) {};
		\node  (2) at (-2, 5) {};
		\node  (3) at (2, 5) {};
		\draw (2.center) to (0.center);
		\draw (3.center) to (1.center);
		\draw [bend right=90, looseness=5.00] (0.center) to (1.center);
	\end{pgfonlayer}

\filldraw[white] (-3,-1) rectangle ++(2cm,2cm);
\draw (-3, 1)--(-3, -1)--(-1, -1)--(-1, 1)-- (-3,1);
\node at (-2,-0.1){\scriptsize{$id_y$}};

\filldraw[white] (1,-1) rectangle ++(2cm,2cm);
\draw (3, 1)--(3, -1)--(1, -1)--(1, 1)-- (3,1);
\node at (2,-0.1){\scriptsize{$f$}};
\end{tikzpicture}};
\node at (2.5,1) {(b)};

\node at (3.5,0){\begin{tikzpicture}[scale = .2]
		\begin{pgfonlayer}{background}
			
				\node  (0) at (-2, 0) {};
				\node  (1) at (2, 0) {};
				\node  (2) at (-2, -5) {};
				\node  (3) at (2, -5) {};
				
				\draw (0.center) to (2.center);
				\draw (1.center) to (3.center);
				\draw [bend left=90, looseness=5.00] (0.center) to (1.center);
			\end{pgfonlayer}

		\filldraw[white] (-3,-1) rectangle ++(2cm,2cm);
		\draw (-3, 1)--(-3, -1)--(-1, -1)--(-1, 1)-- (-3,1);
		\node at (-2,-0.1){\scriptsize{$f$}};
		
		\filldraw[white] (1,-1) rectangle ++(2cm,2cm);
		\draw (3, 1)--(3, -1)--(1, -1)--(1, 1)-- (3,1);
		\node at (2.1,-0.1){\scriptsize{$id_x$}};
\end{tikzpicture}};
	\node at (6, 1) {(c)};		
\node at (8,0){\begin{tikzpicture}[scale = .2]
		\begin{pgfonlayer}{background}
			
			\node (0) at (-2, 0) {};
		\node (1) at (2, 0) {};
		\node (4) at (-6, 0) {};
		\node (5) at (-6, 6) {};
		\node (6) at (2, -6) {};

\draw [bend left=90, looseness=5.00] (0.center) to (1.center);
\draw (5.center) to (4.center);
\draw [bend right=90, looseness=5.00] (4.center) to (0.center);
\draw (1.center) to (6.center);
		\end{pgfonlayer}
	
		\filldraw[white] (-7,-1) rectangle ++(2cm,2cm);
		\draw (-7, 1)--(-7, -1)--(-5, -1)--(-5, 1)-- (-7,1);
		\node at (-6,-0.1){\scriptsize{$id_y$}};

		\filldraw[white] (-3,-1) rectangle ++(2cm,2cm);
		\draw (-3, 1)--(-3, -1)--(-1, -1)--(-1, 1)-- (-3,1);
		\node at (-2,-0.1){\scriptsize{$f$}};
		
		\filldraw[white] (1,-1) rectangle ++(2cm,2cm);
		\draw (3, 1)--(3, -1)--(1, -1)--(1, 1)-- (3,1);
		\node at (2.1,-0.1){\scriptsize{$id_x$}};
\end{tikzpicture}};
\end{tikzpicture}
\caption{(a) $\ev{f}\colon y^* \otimes x \to I$; (b)$\coev{f}\colon I \to y \otimes  x^*$; (c) $f^*\colon y^*\to x^*$ .} \label{fig eval}
\end{figure}
For composable morphisms $f$ and $g$, $(g \circ f)^* = f^* \circ g^*$ in $\V$.

\subsection{Monochrome Brauer diagrams}\label{ssec. mono BD}

The category $\BD$ of monochrome Brauer diagrams admits a concise definition as the free compact closed category generated by a single self-dual object. This section gives a more concrete description of $\BD$ in terms of pairings on finite sets.

\begin{defn}
	A \emph{pairing} (or \emph{perfect matching}) $\tau$ on a set $X$ is a partition of $X$ into two-element subsets.
\end{defn} Equivalently, a pairing on $X$ is a fixed point free involution $\tau \colon X \to X$:
\[ \tau^2 = id_X, \text{ and } \tau (x) \neq x \text{ for all } x \in X.\]

In particular, a finite set $X$ admits a pairing if and only if it has even cardinality. The empty set has a unique trivial pairing $\varnothing$ by convention. 

\begin{ex}\label{ex. manifold pairing} The boundary $\partial \man$ of a compact 1-manifold $\man$ has a canonical pairing $\tau^\man$, with $x  = \tau^\man y$ if $x, y \in \partial \man$, $x \neq y$, are in the same connected component of $\man$. 
\end{ex}

If $X$ and $Y$ are finite sets, a pairing $\tau$ on their disjoint union $X \amalg Y$ may be represented by a graph whose vertices are indexed by$X \amalg Y$, with elements of $X$ above those of $Y$, and edges connecting vertices $v_1 $ and $v_2$ if and only if the corresponding elements of $X \amalg Y$ are identified by $\tau$. 

\begin{defn}
	\label{def Brauer diag}
	
	A \emph{(monochrome) Brauer diagram} $f$ between natural numbers $m$ and $n$ is given by a pair $(\tau_f, \kclf)$ of 
	a pairing $\tau$ on the disjoint union $ \{ s_1 , \dots, s_m\}\amalg \{t_1, \dots, t_n\}$, and a natural number $\kclf$ called the \emph{number of closed components of $f$}.

	If $\kclf = 0$, then $f$ is called an \emph{open Brauer diagram}.

\end{defn}
Let $\BD(m,n)$ denote the set of Brauer diagrams from $m$ to $n$ and let $\BDop(m,n) \subset \BD(m,n)$ be the subset of open Brauer diagrams.  So $\BD(m,n) \cong \BDop (m,n) \times \N$ for all $m, n \in \N$.

\begin{ex}
	\label{ex symmetric groupoid}
	For all $n$, there is a canonical inclusion $\Sigma_n \hookrightarrow \BDop(n,n)$ that takes $\sigma \in \Sigma_n$ to the open Brauer diagram induced by the pairing $s_i \mapsto t_{\sigma i}$ on $\{s_1, \dots, s_n\} \amalg \{t_1, \dots, t_n\}$, $1 \leq i \leq n$.
	
	In particular, the pairing $s_i \mapsto t_i$, $1 \leq i \leq n$ defines the \emph{identity (open) Brauer diagram} $id_n$ on $n$.
	
\end{ex}

Brauer diagrams $(\tau_1, \kcl_1)\colon m_1 \to n_1$ and $(\tau_2, \kcl_2) \colon m_2 \to n_2$ compose \textit{horizontally} according to 
\[ (\tau_1, \kcl_1) \oplus (\tau_2, \kcl_2) = (\tau_1 \amalg \tau_2, \kcl_1 +\kcl_2) \colon m_1 + m_2 \to n_1 + n_2,\] and the trivial open Brauer diagram $(\varnothing, 0) \colon 0 \to 0$ acts as an identity for this composition. 

Any Brauer diagram $f = (\tau, \kcl)\colon m \to n$ may be written as a horizontal sum $(\tau, 0) \oplus (\varnothing, \kcl)$ of an open Brauer diagram $(\tau, 0) \colon m \to n$ and a scalar $(\varnothing, \kcl) =\bigoplus_{i = 1}^\kcl (\varnothing, 1)\colon 0 \to 0$. 

if the corresponding elements of $X \amalg Y$ are equivalent under $\tau$. 

A Brauer diagram $f = (\tau, \kcl) \colon m \to n$ may be represented by the graph for $\tau$, together with 
$\kcl$ closed circles (called \textit{bubbles} in \cite{RS20}) drawn next to the graph. Horizontal composition of Brauer diagrams is represented by juxtaposition of graphs and closed bubbles.

To define vertical composition of Brauer diagrams, observe that, given finite sets $X,Y, Z$, and pairings $\tau_{X,Y}$ and $\tau_{Y,Z}$ on $\X \amalg Y$ and $Y \amalg Z$ respectively, there is an induced pairing $\tau_{Y, Z} \circ \tau_{X,Y}$ on $X \amalg Z$ obtained, as in \cref{fig. pairing comp}, by stacking $\tau_{X,Y}$ and $\tau_{Y,Z}$: 

Namely, 
$\tau_{X,Y}$ and $\tau_{Y,Z}$ generate an equivalence relation on $X \amalg Y \amalg Z$ 
such that objects $x$ and $y$ are equivalent if and only if they are related by a sequence of (alternating) applications of $\tau_{X,Y}$ and $\tau_{Y,Z}$
(\cref{fig. pairing comp}(b)(i)-(iv)).

Each equivalence class contains precisely zero or two elements of $X \amalg Z$. The classes that contain two elements of $X \amalg Z$ 
-- the \textit{open components} of the composition -- describe the desired pairing $\tau_{Y, Z} \circ \tau_{X,Y}$ on $X \amalg Z$. %

The remaining equivalence classes -- that describe cycles of elements of $Y$ -- are called \textit{closed components formed by the composition} of $\tau_{X,Y}$ and $\tau_{Y,Z}$. 

Write $\Sf  = \{ s_1 , \dots, s_m\}$ for the \textit{source}, and $\Tf = \{t_1, \dots, t_n\}$ for the \textit{target} of a Brauer diagram $f \in \BD (m,n)$.  
Brauer diagrams $f = (\tau_f, \kclf) \in \BD (l,m)$ and $g =  (\tau_g, \kclf[g]) \in \BD (m,n)$  may be composed vertically to obtain a Brauer diagram $ g \circ f = (\tau_{gf}, \kclf[gf]) \in \BD(l,n)$. The pairing $\tau_{g,f}$ is the composition pairing $\tau g \circ \tau f$ obtained by identifying $ \Tf = \Sf[g]$ according to $t^f_i \mapsto s^g_{i}$. The natural number $\kclf[gf]$ satisfies 
\[ \kclf[gf] = \kclf + \kclf[g]\ + \kcl(\tau_{f}, \tau_{g})\] where $\kcl(\tau_{f}, \tau_{g})$ is the number of closed components formed by the composition of $\tau_f$ and $\tau_g$.

	\begin{figure}
		[htb!]
		
		\begin{tikzpicture}
			\node at (-2.5,1){(a)};
			\node at (-.5,0){	\begin{tikzpicture}[scale = .35]
					\begin{pgfonlayer}{above}
						\node at (5,18){\color{sapgreen}$X$};
						\node at (5,16){\color{violet}$Y$};
						\node at (5,14){\color{violet}$Y$};
						\node at (5,12){\color{cyan}$Z$};
						\node [dot, violet]  (12) at (-2, 16) {};
						\node  [dot, violet] (13) at (-1, 16) {};
						\node  [dot, violet] (14) at (0, 16) {};
						\node  [dot, violet] (15) at (1, 16) {};
						\node [dot, violet]  (16) at (2, 16) {};
						\node  [dot, violet] (17) at (3, 16) {};
						\node  [dot, violet] (18) at (4, 16) {};
						\node  [dot, sapgreen] (20) at (-0.5, 18) {};
						\node  [dot, sapgreen](21) at (1, 18) {};
						\node  [dot, sapgreen](22) at (2.5, 18) {};
						\node  [dot, cyan](28) at (-2, 12) {};
						\node   [dot, cyan](29) at (-0.5, 12) {};
						\node   [dot, cyan](30) at (1, 12) {};
						\node   [dot, cyan](31) at (2.5, 12) {};
						\node  [dot, cyan] (32) at (4, 12) {};
						\node [dot, violet]  (33) at (-2, 14) {};
						\node  [dot, violet] (34) at (-1, 14) {};
						\node  [dot, violet] (35) at (0, 14) {};
						\node  [dot, violet] (36) at (1, 14) {};
						\node [dot, violet]  (37) at (2, 14) {};
						\node  [dot, violet] (38) at (3, 14) {};
						\node [dot, violet]  (39) at (4, 14) {};
					\end{pgfonlayer}
					\begin{pgfonlayer}{background}
						\filldraw[draw = white, fill = violet, fill opacity  = .1]	(1,16) ellipse (3.6cm and .6cm); 
						\filldraw[draw = white, fill = violet, fill opacity  = .1]	(1,14) ellipse (3.6cm and .6cm);
						\filldraw[draw = white, fill = sapgreen, fill opacity  = .1]	(1,18) ellipse (3.6cm and .6cm);
						\filldraw[draw = white, fill =cyan , fill opacity  = .1]	(1,12) ellipse (3.6cm and .7cm);
						\draw [red, bend left=300, looseness=0.75] (20.center) to (12.center);
						\draw [red, in=90, out=-150] (21.center) to (14.center);
						\draw [red, bend left=60] (22.center) to (18.center);
						\draw [red, bend left=90, looseness=1.25] (16.center) to (17.center);
						\draw [red, bend left=60] (13.center) to (15.center);
						\draw [blue, bend right=75, looseness=1.25] (36.center) to (37.center);
						\draw [blue, in=135, out=-60] (35.center) to (31.center);
						\draw [blue, bend left=75] (28.center) to (29.center);
						\draw [blue, bend left=60, looseness=0.75] (30.center) to (32.center);
						\draw [blue, bend right] (34.center) to (38.center);
						\draw [blue, bend right] (33.center) to (39.center);
					\end{pgfonlayer}
			\end{tikzpicture}};
			\draw [thick, dashed, -> ] (2,0)--(3,0);
			\node at (5,0){\begin{tikzpicture}[scale = .35]
					\begin{pgfonlayer}{above}
						\node at (5,18){\color{sapgreen}$X$};
						\node at (5,15){\color{violet}$Y$};
						
						\node at (5,12){\color{cyan}$Z$};
						\node [dot, violet]  (12) at (-2, 15) {};
						\node  [dot, violet] (13) at (-1, 15) {};
						\node  [dot, violet] (14) at (0, 15) {};
						\node  [dot, violet] (15) at (1, 15) {};
						\node [dot, violet]  (16) at (2, 15) {};
						\node  [dot, violet] (17) at (3, 15) {};
						\node  [dot, violet] (18) at (4, 15) {};
						\node  [dot, sapgreen] (20) at (-0.5, 18) {};
						\node  [dot, sapgreen](21) at (1, 18) {};
						\node  [dot, sapgreen](22) at (2.5, 18) {};
						\node  [dot, cyan](28) at (-2, 12) {};
						\node   [dot, cyan](29) at (-0.5, 12) {};
						\node   [dot, cyan](30) at (1, 12) {};
						\node   [dot, cyan](31) at (2.5, 12) {};
						\node  [dot, cyan] (32) at (4, 12) {};
						\node [dot, violet]  (33) at (-2, 15) {};
						\node  [dot, violet] (34) at (-1, 15) {};
						\node  [dot, violet] (35) at (0, 15) {};
						\node  [dot, violet] (36) at (1, 15) {};
						\node [dot, violet]  (37) at (2, 15) {};
						\node  [dot, violet] (38) at (3, 15) {};
						\node [dot, violet]  (39) at (4, 15) {};
					\end{pgfonlayer}
					\begin{pgfonlayer}{background}
						\filldraw[draw = white, fill = violet, fill opacity  = .1]	(1,15) ellipse (3.6cm and .6cm); 
						
						\filldraw[draw = white, fill = sapgreen, fill opacity  = .1]	(1,18) ellipse (3.6cm and .6cm);
						\filldraw[draw = white, fill =cyan , fill opacity  = .1]	(1,12) ellipse (3.6cm and .7cm);
						\draw [red, bend left=300, looseness=0.75] (20.center) to (12.center);
						\draw [red, in=90, out=-150] (21.center) to (14.center);
						\draw [red, bend left=60] (22.center) to (18.center);
						\draw [red, bend left=90, looseness=1.5] (16.center) to (17.center);
						\draw [red, bend left=60, looseness=1.5] (13.center) to (15.center);
						\draw [blue, bend right=75, looseness=1.5] (36.center) to (37.center);
						\draw [blue, in=135, out=-60] (35.center) to (31.center);
						\draw [blue, bend left=75] (28.center) to (29.center);
						\draw [blue, bend left=60, looseness=0.75] (30.center) to (32.center);
						\draw [blue, bend right = 75, looseness=1] (34.center) to (38.center);
						\draw [blue, bend right = 90, looseness=1] (33.center) to (39.center);
					\end{pgfonlayer}
				\end{tikzpicture}
			};

	\node at (8.5,1){(b)};
			\node at (10,0){\begin{tikzpicture}[scale = .45]
					\begin{pgfonlayer}{above}
						\node   (12) at (-2, 15) {};
						\node   (13) at (-1, 15) {};
						\node   (14) at (0, 15) {};
						\node   (15) at (1, 15) {};
						\node   (16) at (2, 15) {};
						\node   (17) at (3, 15) {};
						\node   (18) at (4, 15) {};
						\node  [dot, sapgreen] (20) at (-0.5, 16) {};
						\node  [dot, sapgreen](21) at (1, 16) {};
						\node  [dot, sapgreen](22) at (2.5, 16) {};
						\node  [dot, cyan](28) at (-2, 12) {};
						\node   [dot, cyan](29) at (-0.5, 12) {};
						\node   [dot, cyan](30) at (1, 12) {};
						\node   [dot, cyan](31) at (2.5, 12) {};
						\node  [dot, cyan] (32) at (4, 12) {};
						\node   (33) at (-2, 15) {};
						\node   (34) at (-1, 15) {};
						\node   (35) at (0, 15) {};
						\node   (36) at (1, 15) {};
						\node   (37) at (2, 15) {};
						\node  (38) at (3, 15) {};
						\node   (39) at (4, 15) {};
					\end{pgfonlayer}
					\begin{pgfonlayer}{background}
						\draw [ in=90, out=-150] (21.center) to (31.center);
						\draw [ bend right=90, looseness=1.5] (20.center) to (22.center);
						\draw [ bend left=75] (28.center) to (29.center);
						\draw [ bend left=60, looseness=0.75] (30.center) to (32.center);
						\draw [violet] (4, 14) circle (.5cm);
					\end{pgfonlayer}
				\end{tikzpicture}
			};
		\end{tikzpicture}
		\caption{(a) Composition of pairings on $X \amalg Y$ and $Y \amalg Z$; 
				(c) the resulting pairing on $X \amalg Z$, together with the single closed component formed in the composition. }\label{fig. pairing comp}
	\end{figure}

	This composition is associative, with two-sided unit $(id_n, 0)$. Hence, we may define:

	\begin{defn}\label{def. BD}

		Objects of the category $\BD$ of \emph{(monochrome) Brauer diagrams} are natural numbers $n \in \N$. Morphisms in  
		$\BD(m, n)$ are Brauer diagrams $f \colon m \to n$, and composition is vertical composition of Brauer diagrams.

	\end{defn}
	
	The category $\BD$ is a (monochrome) PROP with monoidal product induced by horizontal composition $\oplus$ of Brauer diagrams and monoid of scalars $\BD(0,0) \cong \N$. 
	
	\begin{ex}\label{ex. symm groupoid}
		By \cref {ex symmetric groupoid}, the symmetric groupoid $\Sigma$ embeds canonically in $\BD$. 
	\end{ex}

	By definition, $\BD(m_1, n_1) \cong \BD(m_2, n_2)$ for all $(m_1, n_1)$ and $(m_2, n_2)$ such that $m_1 + n_1 = m_2 + n_2$. 
	
In fact, the category $\BD$ has a compact closed structure, for which all objects are self-dual. Given a Brauer diagram $f \in \BD(m,n)$:
	\begin{itemize}
		\item its {evaluation morphism} $ \ev{f}  =(\ev{\tau}, \kcl) \in \BD (n+m, 0)$ is induced by the canonical isomorphism
			\[\{s_1, \dots, s_m \}\amalg \{t_1, \dots, t_n\} \ {\cong} \  \{t_n, \dots, t_1, s_1, \dots, s_m\}\amalg \emptyset \]
		\item its {coevaluation morphism} $ \coev{f}  =(\coev{\tau}, \kcl) \in \BD (0, n+m)$ is  induced by the canonical isomorphism
		\[	\{s_1, \dots, s_m \}\amalg \{t_1, \dots, t_n\}  \ {\cong} \ \emptyset \amalg \{t_1, \dots, t_n, s_m, \dots, s_1\}\]
		\item its {dual morphism} $f^* = (\tau^*, \kcl ) \in \BD(n,m)$ is induced by the canonical isomorphism \[	\{s_1, \dots, s_m \}\amalg \{t_1, \dots, t_n\} \ {\cong} \  \{t_n, \dots, t_1\} \amalg \{ s_m, \dots, s_1\} .\]
	\end{itemize}

	In particular, the morphisms $id_1 \in \BD(1,1)$, $\cap \in \BD(0,2)$ and $\cup \in \BD(2,0)$ are the open Brauer diagrams induced by the unique pairing on the two-element set $\{s,t\} $. 
	 And, for all $n \in \N$, 
	 $ id_n =  \bigoplus_{i = 1}^n id_1\in \BD (n,n)$ and 
	$ \cap_n \defeq \coev{id_n}  \in \BD(0, 2n)$, and $\cup_n \defeq \ev{id_n} \in \BD(2n, 0)$
	satisfy the $n$-fold triangle identities (\cref{fig. triangle}).
	\begin{equation}\label{eq. n triangle} 
		(\cup_n \oplus id_n) \circ (id_n \oplus \cap_n)= id_n = (id_n\oplus \cup_n )\circ (\cap_n \oplus id_n).
	\end{equation}
	 
	\begin{rmk}\label{rmk. top to bottom} In this work, the direction of morphisms is depicted as top to bottom (source above target). Hence the designation of $\cup$ and $\cap$ may differ from other (Feynman diagram) conventions. \end{rmk}

	It is important to note that the subsets $\BDop(m,n) \subset \BD(m,n)$ of open Brauer diagrams do not describe a subcategory of $ \BD$. Namely, the composition $\cup \circ \cap   $ defines the {unit trace} $\bigcirc   =(\varnothing, 1)\in \BD(0,0)$ which is not open. And, for all $n \geq 0$, 
	\[ \cup_n \circ \cap _n = \bigoplus_{i = 1}^n (\cup \circ \cap) = (\varnothing, n) \in \BD(0,0).\]

	\begin{figure}	[htb!]
		\begin{tikzpicture}

			\node at (7.5,0){
				\begin{tikzpicture}[scale = .16]
					\begin{pgfonlayer}{background}
						\node  (0) at (2, 5) {};
						\node  (1) at (3, 5) {};
						\node  (2) at (3.75, 5) {};
						\node  (3) at (4.5, 5) {};
						\node  (4) at (1.25, 5) {};
						\node  (5) at (0.5, 5) {};
						\node  (6) at (5.5, 5) {};
						\node  (7) at (6.25, 5) {};
						\node  (8) at (7, 5) {};
						\node  (9) at (2, 10) {};
						\node  (10) at (1.25, 10) {};
						\node  (11) at (0.5, 10) {};
						\node  (12) at (5.5, 0) {};
						\node  (13) at (6.25, 0) {};
						\node  (14) at (7, 0) {};

						\node  (15) at (25, 5.25) {};
						\node  (16) at (21, 5.25) {};
						\node  (17) at (21.75, 5.25) {};
						\node  (18) at (22.5, 5.25) {};
						\node  (19) at (24.25, 5.25) {};
						\node  (20) at (23.5, 5.25) {};
						\node  (21) at (18.5, 5.25) {};
						\node  (22) at (19.25, 5.25) {};
						\node  (23) at (20, 5.25) {};
						\node  (24) at (25, 10.25) {};
						\node  (25) at (24.25, 10.25) {};
						\node  (26) at (23.5, 10.25) {};
						\node  (27) at (18.5, 0.25) {};
						\node  (28) at (19.25, 0.25) {};
						\node  (29) at (20, 0.25) {};

						\node  (30) at (12.75, 10) {};
						\node  (31) at (12, 10) {};
						\node  (32) at (11.25, 10) {};
						\node  (33) at (11.25, 0) {};
						\node  (34) at (12, 0) {};
						\node  (35) at (12.75, 0) {};

						\node  (81) at (-9.25, 5.75) {};
						\node  (82) at (-8.25, 5.75) {};
						\node  (83) at (-7.5, 5.75) {};
						\node  (84) at (-6.5, 5.75) {};
						\node  (85) at (-10, 5.75) {};
						\node  (86) at (-10.75, 5.75) {};
						\node  (87) at (-5.5, 4.25) {};
						\node  (88) at (-4.75, 4.25) {};
						\node  (89) at (-4, 4.25) {};
						\node  (90) at (-9.25, 11) {};
						\node  (91) at (-10, 11) {};
						\node  (92) at (-10.75, 11) {};
						\node  (93) at (-5.5, -0.75) {};
						\node  (94) at (-4.75, -0.75) {};
						\node  (95) at (-4, -0.75) {};
						\node  (96) at (-5.75, 5.75) {};
						\node  (97) at (-4.75, 5.75) {};
						\node  (98) at (-4, 5.75) {};
						\node  (99) at (-10.75, 4.25) {};
						\node  (100) at (-10, 4.25) {};
						\node  (101) at (-9, 4.25) {};
						\node  (102) at (-8.25, 4.25) {};
						\node  (103) at (-7.25, 4.25) {};
						\node  (104) at (-6.5, 4.25) {};

						\node  (116) at (35.75, 5.75) {};
						\node  (117) at (29, 5.75) {};
						\node  (118) at (35, 5.75) {};
						\node  (119) at (34.25, 5.75) {};
						\node  (120) at (29, 4.25) {};
						\node  (121) at (29.75, 4.25) {};
						\node  (122) at (30.5, 4.25) {};
						\node  (123) at (35.75, 10.75) {};
						\node  (124) at (35, 10.75) {};
						\node  (125) at (34.25, 10.75) {};
						\node  (126) at (29, -0.75) {};
						\node  (127) at (29.75, -0.75) {};
						\node  (128) at (30.5, -0.75) {};
						\node  (129) at (29.75, 5.75) {};
						\node  (130) at (30.75, 5.75) {};
						\node  (131) at (31.5, 5.75) {};
						\node  (132) at (32.5, 5.75) {};
						\node  (133) at (33.25, 5.75) {};
						\node  (134) at (31.5, 4.25) {};
						\node  (135) at (32.25, 4.25) {};
						\node  (136) at (33.25, 4.25) {};
						\node  (137) at (34, 4.25) {};
						\node  (138) at (35, 4.25) {};
						\node  (139) at (35.75, 4.25) {};
					\end{pgfonlayer}
					\begin{pgfonlayer}{above}
						\draw (11.center) to (5.center);
						\draw (10.center) to (4.center);
						\draw (9.center) to (0.center);
						\draw (6.center) to (12.center);
						\draw (7.center) to (13.center);
						\draw (8.center) to (14.center);
						\draw [bend left=90, looseness=3.25] (3.center) to (6.center);
						\draw [bend right=90, looseness=3.00] (0.center) to (1.center);
						\draw [bend right=90, looseness=2.25] (4.center) to (2.center);
						\draw [bend left=90, looseness=2.50] (2.center) to (7.center);
						\draw [bend right=90, looseness=2.00] (5.center) to (3.center);
						\draw [bend left=90, looseness=2.25] (1.center) to (8.center);
						\draw (26.center) to (20.center);
						\draw (25.center) to (19.center);
						\draw (24.center) to (15.center);
						\draw (21.center) to (27.center);
						\draw (22.center) to (28.center);
						\draw (23.center) to (29.center);
						\draw [bend right=90, looseness=2.25] (18.center) to (21.center);
						\draw [bend left=90, looseness=2.00] (15.center) to (16.center);
						\draw [bend left=90, looseness=2.00] (19.center) to (17.center);
						\draw [bend right=90, looseness=2.25] (17.center) to (22.center);
						\draw [bend left=90, looseness=2.50] (20.center) to (18.center);
						\draw [bend left=270, looseness=2.50] (16.center) to (23.center);
						\draw (32.center) to (33.center);
						\draw (31.center) to (34.center);
						\draw (30.center) to (35.center);
						\draw (92.center) to (86.center);
						\draw (91.center) to (85.center);
						\draw (90.center) to (81.center);
						\draw (87.center) to (93.center);
						\draw (88.center) to (94.center);
						\draw (89.center) to (95.center);
						\draw [bend left=90, looseness=2.25] (82.center) to (83.center);
						\draw [bend left=90, looseness=2.25] (84.center) to (96.center);
						\draw [bend left=90, looseness=2.25] (97.center) to (98.center);
						\draw [bend right=90, looseness=2.25] (99.center) to (100.center);
						\draw [bend right=90, looseness=2.25] (101.center) to (102.center);
						\draw [bend right=90, looseness=2.25] (103.center) to (104.center);
						\draw (86.center) to (99.center);
						\draw (82.center) to (100.center);
						\draw (85.center) to (101.center);
						\draw (102.center) to (84.center);
						\draw (81.center) to (103.center);
						\draw (104.center) to (97.center);
						\draw (83.center) to (87.center);
						\draw (96.center) to (88.center);
						\draw (98.center) to (89.center);
						\draw (125.center) to (119.center);
						\draw (124.center) to (118.center);
						\draw (123.center) to (116.center);
						\draw (120.center) to (126.center);
						\draw (121.center) to (127.center);
						\draw (122.center) to (128.center);
						\draw [bend left=90, looseness=2.25] (117.center) to (129.center);
						\draw [bend left=90, looseness=2.50] (130.center) to (131.center);
						\draw [bend left=90, looseness=2.50] (132.center) to (133.center);
						\draw (117.center) to (120.center);
						\draw (130.center) to (121.center);
						\draw (132.center) to (122.center);
						\draw [bend right=90, looseness=2.50] (134.center) to (135.center);
						\draw [bend right=90, looseness=2.50] (136.center) to (137.center);
						\draw [bend right=90, looseness=2.50] (138.center) to (139.center);
						\draw (116.center) to (139.center);
						\draw (118.center) to (137.center);
						\draw (119.center) to (135.center);
						\draw (129.center) to (134.center);
						\draw (131.center) to (136.center);
						\draw (133.center) to (138.center);
						
						\draw 
						(-2.5, 4.8)--(-1,4.8)
						(-2.5, 5.3)--(-1,5.3);
						
							\draw 
						(8, 4.8)--(9.5,4.8)
						(8, 5.3)--(9.5,5.3);
						
							\draw 
						(14.5, 4.8)--(16,4.8)
						(14.5, 5.3)--(16,5.3);
						
							\draw 
						(26, 4.8)--(27.5,4.8)
						(26, 5.3)--(27.5,5.3);

					\end{pgfonlayer}
				\end{tikzpicture}

			};
			
		\end{tikzpicture}
		\caption{The triangle identities in $\BD$ with $n = 3$.} 
	\label{fig. triangle}
\end{figure}

\begin{ex}\label{ex: Brauer category}
	Let $R$ be a commutative ring, and let $\ModR$ be its category of modules. 
	Let $\underline{\BD}_R$  be the free $\modR$-category on $\BD$. So, for each pair $m, n $ of natural numbers, $\underline{\BD}_R(m,n)$ is the free $R$ module (infinitely) generated by $\BD(m,n)$. 
	For each $\delta \in R$, let $\Br = \Br^R$ be the $\modR$-enriched \textit{Brauer category} \cite{LZ15} whose 
	objects are natural numbers $n \in \N$ and, 
	for all $m, n \in \N$, $\Br(m,n) $ is the free $R$-module (finitely) generated by the open Brauer diagrams in $\BD(m,n)$. Composition in $\Br$ is defined by 
	\[ \tau_g \tau_f =\delta^{\kclf[gf]}\tau_{gf}\in \Br(k,n) \text{ for all generators }\tau_f \in \Br(k,m), \tau_g \in \Br(m,n).\]
	
	If $ R[t]$ is the ring of polynomials in $R$, then the canonical isomorphisms $\underline{\BD}_R(m,n) \cong \Br[t]^{R[t]}(m,n)$ of $R$-modules induced by $(\tau, \kcl) \mapsto t^{\kcl}\tau$ extend to an isomorphism of $\modR$-enriched categories.

	A \emph{Brauer algebra} is an algebra of the form $\Br(n,n)$, for some commutative ring $R$, $\delta \in R$ and $n \in \N$. They were introduced in \cite{Bra37}, where Brauer used them to study the representation theory of the orthogonal and symplectic groups.
	
	In particular, Brauer proved that, if $R  = \Bbbk$ is a field of characteristic $0$ and $V$ is a $k$-dimensional vector space over $\Bbbk$, then representations of $\Br[k]^{\Bbbk}(n,n)$ in $V^{\otimes n}$ are in one to one correspondence with degree $n$ representations of the $k$-dimensional orthogonal group $\mathrm O (k)$ on $V^{\otimes n}$. And, if $ W$ is a $2m$-dimensional vector space over $ \Bbbk$, representations of $\Br[(-m)]^{\Bbbk}(n,n)$ in $W^{\otimes n}$ are in one to one correspondence with degree $n$ representations of the $m$-dimensional symplectic group $\mathrm {Sp} (m)$ on $W^{\otimes n}$. 
	
\end{ex}

\begin{defn}\label{def. dw BD}
	The subcategory $\BDd \subset \BD$ of \emph{downward Brauer diagrams} is the subcategory of open morphisms $  (\tau^{\downarrow}, 0)\in \BD(m,n)$ such that, for all $y \in \Tf$, 
 $\tau^{\downarrow}(y)  \in  \Sf$.  So, $\BDd(m,n) $ is empty whenever $n >m$.
	
	The category $\BDu \subset \BD$ of \emph{upward Brauer diagrams} is the opposite category of $\BDd$. 
\end{defn}

In particular, $\cap \not\in \BDd(0,2)$, $\cup \not \in \BDu(2,0)$.  The intersection of $\BDd$ and $\BDu$ in $BD$ is the symmetric groupoid $\Sigma$. 

\begin{ex}
	\label{ex down Brauer}
For a commutative ring $R$, the \textit{downwards Brauer category} $\Brd $ is the free $\modR$-enriched category on $ \BDd$, and the \textit{upwards Brauer category} $\Bru$ is the free $\modR$-enriched category on $\BDu$.
	
	No closed components are formed by composition in $\BDd$ (or $\BDu$). Hence $\Brd $ is a subcategory of $ \Br$ (and $\Bru \subset \Br$) for all $\delta \in R$. 
	
In particular, for all $n \in \N$, $\Brd(n,n) = \Bru(n,n) \cong R^{\Sigma_n}$ is canonically a $\Sigma_n$-module. It follows from the results of \cite{SS15} that symmetric strict monoidal functors $\BDd \to \vect$ describe representations of the\textit{ infinite orthogonal group} $\mathrm O(\infty) \defeq \bigcup_{n \in \N} \mathrm O(n)$.

\end{ex}

\begin{rmk}
\label{rmk. generators}
By e.g.~\cite{LZ15}, the category $\BD$ is generated under horizontal and vertical composition by the morphisms $id_1$, $\cup$, $\cap$ and the unique non-identity permutation $\sigma_{\two} \in \Sigma_2 \subset \BD(2,2)$. 

The category $\BDd$ is generated by $id_1, \sigma_{\two}$, and $\cup$ (and $\BDu$ is generated by $id_1, \sigma_{\two}$, and $\cap$) under horizontal and vertical composition. 
\end{rmk}

\begin{ex}
\label{ex. manifold components} Let $\I$ denote the unit interval $[0,1]$, and let $\man \cong n_o (\I)\amalg n_c (S^1)$ be a compact 1-manifold with canonical pairing $\tau^\man$ on $\partial \man$ as in \cref{ex. manifold pairing}. 
If $m,n \in \N$ satisfy $m+n = 2 n_o$, and $\phi \colon \{s_1, \dots, s_m\} \amalg \{t_1, \dots, t_n\} \to \partial \man$ is any isomorphism, then 
$(\phi^{-1} \tau^\man \phi, n_c) \in \BD(m,n)$.

Conversely, given a morphism $ f = (\tau, \kcl )\in \BD(m,n)$, there is a (unique up to isomorphism) pair $(\man_f, \phi_f)$ of a compact 1-manifold $\man_f\cong \frac{m+n}{2}(\I) \amalg \kcl (S^1)$, and an isomorphism $\phi_f \colon 
\Sf \amalg \Tf \to \partial \man$ such that $\phi^{-1}_f \tau^{\man_f} \phi_f = \tau$.

\end{ex}

Following \cref{ex. manifold components}, define $ \partial f  \defeq \Sf \amalg \Tf$ to be the \textit{boundary of $f = (\tau, \kcl)\in \BD(m,n)$}. So $\partial f$ consists of $m +n$ points.

The set $\pi_0(f)$ \textit{of components of $f$} is the set of components of a compact manifold $\man_f$ as in \cref{ex. manifold components}. 
So, $|\pi_0(f)|  = \frac{(m+n)}{2} + \kcl$,  and a morphism $f  = (\tau, \kcl) \in \BD(m,n)$ may also be written as a diagram of cospans of finite sets (where the arrow $\partial f \to \pi_0 (f)$ is the canonical map): 
\begin{equation}\label{eq. cospan def BD} \xymatrix{  \Sf \ar@{-->}[rrd] \ar@{>->}[rr]^{S_i \ \mapsto \tau (S_i)}&& \partial f \ar[d] && \ar@{>->}[ll]_-{t_j \ \mapsto \tau(t_j) }\Tf  \ar@{-->}[lld]\\&& \pi_0(f).&&}\end{equation}

\begin{rmk}\label{rmk. pushout} 

By (\ref{eq. cospan def BD}), for composable morphisms $f \in \BD(k,m)$ and $g \in \BD(m,n)$, we may consider the pushout diagram: 
\begin{equation}\label{eq. cospan pushout} 
	\xymatrix@C=.6cm {
		\Sf\ar[drr]_-{\tau_f} &&&& \Tf =  \Sf[g]  \ar[dll]^-{\tau_f} \ar[drr]_-{\tau_g}&&&& \Tf[g] \ar[dll]^-{\tau_g}\\
		&& \partial f \ar@{-->}[drr]\ar[d] &&&&  \partial g\ar@{-->}[dll]\ar[d] &&\\
		&& \pi_0 (f)  \ar@{-->}[drr]&& P(gf)\ar[d] && \pi_0 (g)\ar@{-->}[dll]&& \\
		&&&& \pi^P(gf). &&&&
	}
\end{equation}

In general, $P(gf) \not \cong \partial (gf) = \Sf \amalg \Tf[g]$ and hence composition of morphisms in $\BD$ is not described by compositions (pushouts) of cospans as in (\ref{eq. cospan pushout}). (So $\BD$ is not a \textit{cospan category}.)

For example, in 
the pushout (\ref{eq. cospan pushout}) for the composition $\cup \circ \cap = \bigcirc  \in \BD (0, 0)$, $P(\bigcirc)$ is a two element set. But $\bigcirc \in \BD(0,0)$ so $\partial \bigcirc = \emptyset$. (This is closely related to the \textit{problem of loops} discussed in detail in \cite[Section~6]{Ray20}. See also \cref{rmk. down Brauer graph of graphs}.)

Instead, to represent general composition of all Brauer diagrams in $\BD$ 
in terms of pushouts of cospans, one can replace a diagram of the form (\ref{eq. cospan def BD}) with a diagram of manifolds:
\begin{equation}\label{eq. cospan man BD} \xymatrix{  \Sf \ar[r]^-{\phi\circ \tau}& \partial \man\ar[r]^-{\text{inc }}& \man & \ar[l]_-{\text{inc }} \partial \man &\ar[l]_-{\phi\circ \tau}\Tf , }\end{equation} 
where, $(\man, \phi)$ describes $f\in \BD(m,n)$ as in \cref{ex. manifold components}, and the isomorphism $\pi_0 (\man) \xrightarrow \cong \pi_0 (f)$ is fixed by $\phi$ on open components of $\man$.

Then, if $(\man_f, \phi_f)$ represents $f = (\tau_f, \kclf) \in \BD(k,m)$, and 
$(\man_g, \phi_g) $ represents $g = (\tau_g, \kclf[g]) \in \BD(m,n)$, the composition $gf \in \BD(k,n)$ is obtained from the pushout of cospans of manifolds: 
\begin{equation}\label{eq. manifold pushout} 
	\xymatrix @C = 0.6cm{
		\Sf\ar[drr]_-{\text{inc }\circ\phi_f\circ \tau_f} &&&& \Tf  =  \Sf[g]  \ar[dll]^-{\text{inc }\circ\phi_f\circ \tau_f} \ar[drr]_-{\text{inc }\circ\phi_g\circ \tau_g}&&&& \Tf[g] \ar[dll]^-{\text{inc }\circ\phi_g\circ \tau_g}\\
		&& \man_f \ar@{-->}[drr] &&&& \man_g \ar@{-->}[dll] &&\\
		&&&&\man_f \amalg_{\scriptscriptstyle{ \Tf} } \man_g. && && 
}\end{equation}

\end{rmk}

\subsection{Coloured Brauer diagrams}\label{ssec. colour BD}

Let $\CCC$ be a set, and let $\listm \CCC =\coprod_{n \in \N} \CCC^n$ denote the set of finite ordered sets $\ccc  = (c_1, \dots, c_n)$ of elements of $\CCC$. So $\listm \CCC$ underlies the free associative monoid on $\CCC$. For $\ccc = (c_1, \dots, c_m)$ and $\ddd = (d_1, \dots, d_n)$ is $\listm \CCC$, their (concatenation) product $\ccc \ddd = \ccc\oplus \ddd$ is given by \[  \ccc \ddd  \defeq (c_{1}, \dots, c_{m}, d_1, \dots, d_n) .\] The empty list $\varnothing_{\CCC} \defeq(-)$ is the unit for $\oplus$. 

The symmetric groupoid $\Sigma$ acts on $\listm \CCC$ from the right by $\sigma \colon \ccc   \mapsto (\ccc \sigma) \defeq (c_{\sigma 1}, \dots, c_{\sigma m}) $, for all $\ccc= (c_1, \dots, c_m)$ and $\sigma \in \Sigma_m$.

\begin{defn} \label{def. palette}

A pair $(\CCC, \omega )$ of a set $\CCC$ together with an involution $\omega \colon \CCC \to \CCC$ 
is called an \emph{(involutive) palette}. 
Elements $c \in \CCC$ are called \emph{colours in $(\CCC, \omega )$}. The set of orbits of $\omega$ in $\CCC$ is denoted by $\widetilde \CCC$.

Objects of the category $\pal$ are palettes $(\CCC, \omega)$, and morphisms $(\CCC, \omega) \to (\CCC', \omega')$ are given by morphisms $\lambda \in \Set(\CCC, \CCC')$ that commute with the involution such that $\lambda \circ \omega = \omega' \circ \lambda$.

\end{defn}

For any palette $(\CCC, \omega)$, the induced involutions $ \ccc \mapsto \rev{\ccc} \defeq (c_m , \dots, c_1) $ and $\ccc \mapsto \omega \ccc \defeq (\omega c_1,\dots,  \omega c_m) $ on $\listm \CCC$ commute. Hence their composite $\bfom $ also defines an involution on $\listm \CCC$:
\[\bfom (\ccc) \defeq \rev{\omega \ccc} = \omega (\rev{\ccc})  = (\omega c_m, \dots, \omega c_1) \ \text {for all } \ccc = (c_1, \dots, c_m) \in \listm \CCC. \]

Now let $(\CCC, \omega)$ be any palette, and let $(X, \tau)$ be the palette described by a pairing $\tau$ on a finite set $ X $. 
\begin{defn}

\label{def. colouring}

A \emph{$(\CCC, \omega)$-colouring of $\tau$} is morphism  $\lambda \colon (X, \tau )\to (\CCC, \omega)$ in $\pal$. 

A \emph{$(\CCC, \omega)$-colouring $\lambda$ of a Brauer diagram $f  = (\tau, \kcl)\in \BD(m, n)$} is given by a pair $\lambda = (\colop, \colcl)$ where $\colop$ is a colouring of $\tau$ and $\colcl$ is a map $\pi_0(f)\to \widetilde {\CCC}$ such that the following diagram of sets commutes: 

\begin{equation}\label{eq. colour} \xymatrix{ \partial f \ar[rr]^-{\lambda_\partial }\ar[d]^{\cong}_{\tau}&& \CCC \ar[d]^{\omega}_{\cong}\\
	\partial f  \ar[rr]^-{\lambda_\partial }\ar@{->}[d]&& \CCC \ar@{->>}[d]\\
	\pi_0(f) \ar[rr]^-{\widetilde \lambda} && \widetilde \CCC.}\end{equation} 

The \emph{type of the colouring $\lambda$} 
is the pair $(\ccc, \ddd)\in (\listm \CCC)^2$ -- where $\ccc$ is called the \emph{input type}, and $\ddd$ is called the \emph{output type}, of $(f, \lambda)$ -- defined by:
\begin{equation}\label{eq. type of colouring}
\ddd = (d_1, \dots, d_n)  =  \lambda_{\partial}( \Tf), \text{ and } \ccc = (c_1, \dots, c_m)  =\omega \circ \lambda_{\partial}(\Sf ).
\end{equation}
\end{defn}

\begin{rmk}

The application of $\omega$ in the definition of the input type $\ccc = \omega \circ \colop (\Sf)$ is necessary to define categorical composition of coloured Brauer diagrams in \cref{def. colour BD}.
\end{rmk}

Given $\ccc = (c_1, \dots, c_m)$ and $ \ddd =(d_1, \dots, d_n) $ in $ \listm \CCC$, objects of the set $\CBD(\ccc, \ddd)$ of \emph{$(\CCC, \omega)$-coloured Brauer diagrams from $\ccc$ to $\ddd$} are pairs $(f, \lambda)$ where $f = (\tau, \kcl)$ is a morphism in $\BD(m,n)$, and $\lambda$ is a colouring of $f$ of type $(\ccc, \ddd)$.

Horizontal composition $\oplus$ of coloured Brauer diagrams $(f, \lambda ) \in \CBD(\ccc_1, \ddd_1)$ and $(g, \gamma) \in \CBD (\ccc_2, \ddd_2)$ is given by juxtaposition and concatenation:
\[ (f, \lambda) \oplus (f, \gamma) = (f \oplus g, \lambda \amalg \gamma)\in \CBD (\ccc_1\ccc_2, \ddd_1\ddd_2).\]

Let $(f, \lambda) \in \CBD (\bbb, \ccc)$ and $(g, \gamma) \in \CBD (\ccc, \ddd)$ with $ f = (\tau_f, \kclf) \in \BD(k,m)$ and $g = (\tau_g \kclf[g]) \in \BD(m,n)$ be such that $ gf  = (\tau_{gf}, \kclf[gf]) \in \BD(k,n)$. By definition, $\gamma_\partial (y) = \omega \lambda_\partial (y)$ for each $y \in \Tf  = \Sf $.  
Hence, 
for each $x \in \Sf$, $\lambda_\partial(\tau_{gf}) (x) = \omega \lambda_\partial (x)$. And likewise, for $z \in \Tf[g]$, $\gamma_\partial (\tau_{gf})(z) = \omega \gamma_\partial (z)$. Therefore, there is a well-defined colouring $(\gamma \lambda)_\partial$ of $\tau_{gf}$ that restricts to $\lambda_\partial$ on $\Sf$ and to $\gamma_\partial$ on $\Tf[g]$.

Moreover, for each closed component $[y] = (y, y_1 = \tau_f(y), y_2 = \tau_g (y_1), \dots) \subset \Tf$ created by the composition $g\circ f$, there is a well-defined element $\tilde c \in \widetilde \CCC$ given by the class of $\colcl (y_i) $ ($y_i\in [y]$) in $\widetilde \CCC$. Therefore, the composition $g \circ f$
induces a well-defined map $\widetilde {\gamma \lambda} \colon \pi_0 (gf) \to \widetilde \CCC$ (see \cref{fig. composing colours}). 

It follows that  $ \left ((\gamma \lambda)_\partial, \widetilde {\gamma \lambda}\right )$ describes a colouring  $\gamma \lambda $ on $\tau_{gf}$.

\begin{figure}
[htb!]
\begin{tikzpicture}
\node at (0,0){\begin{tikzpicture}[scale = .6]
		\begin{pgfonlayer}{above}
			\node [dot, red]  (12) at (-2, 15) {};
			\node [left, above] at (-2, 15) {\tiny{$\omega c$}};
			\node  [dot, sapgreen] (13) at (-1, 15) {};
			\node [right, above] at (-1, 15) {\tiny{$d$}};
			\node [right, below] at (-1, 15) {\tiny{$\omega d$}};
			\node  [dot, blue] (14) at (0, 15) {};
			\node  [dot, sapgreen] (15) at (1, 15) {};
			\node [right, above] at (1, 15) {\tiny{$\omega d$}};
			\node [right, below] at (1, 15) {\tiny{$d$}};
			\node [dot, sapgreen]  (16) at (2, 15) {};
			\node [right, above] at (2, 15) {\tiny{$ d$}};
			\node [right, below] at (2, 15) {\tiny{$\omega d$}};
			\node  [dot, sapgreen] (17) at (3, 15) {};
			\node [right, above] at (3, 15) {\tiny{$\omega d$}};
			\node [right, below] at (3, 15) {\tiny{$d$}};
			\node  [dot, red] (18) at (4, 15) {};
			\node [left,above] at (4, 15) {\tiny{$ c$}};
			\node  [dot, red] (20) at (-0.5, 18) {};
			\node [left, below] at (-0.5, 18) {\tiny{$ c$}};
			\node  [dot, blue](21) at (1, 18) {};
			\node  [dot, red](22) at (2.5, 18) {};
			\node [right,below] at (2.5,18) {\tiny{$\omega c$}};
			\node  [dot, violet](28) at (-2, 12) {};
			\node   [dot, violet](29) at (-0.5, 12) {};
			\node   [dot, red](30) at (1, 12) {};
			\node [left, above] at (1,12) {\tiny{$ \omega c$}};
			\node   [dot, blue](31) at (2.5, 12) {};
			\node  [dot, red] (32) at (4, 12) {};
			\node [right, above] at (4,12) {\tiny{$  c$}};
			\node [dot, red]  (33) at (-2, 15) {};
			\node [left,below] at (-2, 15) {\tiny{$ c$}};
			\node  [dot, sapgreen] (34) at (-1, 15) {};
			\node  [dot, blue] (35) at (0, 15) {};
			\node  [dot, sapgreen] (36) at (1, 15) {};
			\node [dot, sapgreen]  (37) at (2, 15) {};
			\node  [dot, sapgreen] (38) at (3, 15) {};
			\node [dot, red]  (39) at (4, 15) {};
			\node [left,below] at (4, 15) {\tiny{$\omega c$}};
			
		\end{pgfonlayer}
		\begin{pgfonlayer}{background}
			\draw [red, bend left=300, looseness=0.75] (20.center) to (12.center);
			\draw [blue, in=90, out=-150] (21.center) to (14.center);
			\draw [red, bend left=60] (22.center) to (18.center);
			
			\draw [sapgreen, bend left=90, looseness=1.5] (16.center) to (17.center);
			\draw [sapgreen, bend left=60, looseness=1.5] (13.center) to (15.center);
			\draw [sapgreen, bend right=75, looseness=1.5] (36.center) to (37.center);
			\draw [blue, in=135, out=-60] (35.center) to (31.center);
			\draw [violet, bend left=75] (28.center) to (29.center);
			\draw [red, bend left=60, looseness=0.75] (30.center) to (32.center);
			fill opacity=0.2,  bend right = 75, looseness=1] (34.center) to (38.center);
			\draw [sapgreen, bend right = 75, looseness=1] (34.center) to (38.center);
			\draw [red, bend right = 90, looseness=1] (33.center) to (39.center);
		\end{pgfonlayer}
	\end{tikzpicture}
};
\draw [thick, dashed, -> ] (3.4,0)--(4.5,0);
\node at (7,0){\begin{tikzpicture}[scale = .5]
		\begin{pgfonlayer}{above}

			\node  [dot, red]  (20) at (-0.5, 16) {};
				\node [left, below] at (-0.7, 16) {\tiny{$ c$}};
			\node  [dot, blue](21) at (1, 16) {};

			\node  [dot, red] (22) at (2.5, 16) {};
				\node [right, below] at (2.8, 16) {\tiny{$\omega c$}};

			\node  [dot, violet](28) at (-2, 12) {};
			
			\node   [dot, violet](29) at (-0.5, 12) {};
			
			\node  [dot, red] (30) at (1, 12) {};
				\node [right, above] at (1, 12) {\tiny{$\omega c$}};
			\node   [dot, blue](31) at (2.5, 12) {};
			
			\node  [dot, red]  (32) at (4, 12) {};
				\node [right, above] at  (4, 12) {\tiny{$c$}};

		\end{pgfonlayer}
		\begin{pgfonlayer}{background}
			\draw [blue, in=90, out=-150] (21.center) to (31.center);
			\draw [red, bend right=90, looseness=1.5] (20.center) to (22.center);
			\draw [violet, bend left=75] (28.center) to (29.center);
			\draw [red, bend left=60, looseness=0.75] (30.center) to (32.center);
			\draw [sapgreen] (4, 14) circle (.5cm);
			\node at (4.8,14){\tiny{$[d]$}};
		\end{pgfonlayer}
	\end{tikzpicture}
};
\end{tikzpicture}
\caption{Composing coloured pairings.} \label{fig. composing colours}
\end{figure}

\begin{defn}
\label{def. colour BD}
Objects of the category $\CBD$ of \emph{$(\CCC, \omega)$-coloured Brauer diagrams} are elements of $\listm \CCC$. Morphisms in $\CBD(\ccc, \ddd)$ are $(\CCC, \omega)$-coloured Brauer diagrams of type $(\ccc, \ddd)$, with composition of morphisms $(f, \lambda) \in \CBD (\bbb, \ccc)$ and $(g, \gamma) \in \CBD (\ccc, \ddd)$ is given by $(gf, \gamma \lambda) \in \CBD(\bbb,\ddd)$.

\end{defn}

The category $\CBD$ is a $\CCC$-coloured PROP (see \cref{ssec. SMC}), with monoidal structure $\oplus$ and unit $O \in \CBD(\varnothing_{\CCC},\varnothing_{\CCC})$, induced by concatenation of object lists and disjoint union of coloured Brauer diagrams. It has a compact closed structure given by $\ccc^* =  \rev{\omega \ccc}$ for all $\ccc$. 

If $f \in \BD(m,n)$ is such that $(f,\lambda)$ is a morphism in $\CBD(\ccc, \ddd)$ then the 
\begin{itemize}
\item \textit{evaluation morphism} $\ev{(f, \lambda)} $ is given by $(\ev{f}, \lambda) \in  \CBD((\rev{\omega \ddd})\oplus\ccc, \varnothing_{\CCC})$; 
\item \textit{coevaluation morphism} $\coev{(f, \lambda)}$ is given by $(\coev{f}, \lambda) \in \CBD(\varnothing_{\CCC}, \ddd\oplus(\rev{\omega \ccc}))$;
\item \textit{dual morphism} $ (f, \lambda)^* $ is given by $(f^*, \lambda) \in \CBD ((\rev{\omega \ddd}), (\rev{\omega \ccc}))$.	
\end{itemize}

\begin{figure}
[htb!]
\begin{tikzpicture}
\node at  (-4, 1){(a)};
\node at (0,0){
	\begin{tikzpicture}[scale = .5]
		\draw[thick, violet]
		(-.3,1)--(-.3,-1);
		\draw[thick, sapgreen]
		(0,1)--(0,-1);
		\draw[thick, red]
		(.3,1)--(.3,-1);
	\end{tikzpicture}
};
\node at (-.8,0){$=$};
\node at (.8,0){$=$};
\node at (-2.5,0){
	\begin{tikzpicture}[scale = .5]
		\draw[ dashed]
		(-.6,0)--(3,0);
		\draw[ violet]
		(-.3,1.5)--(-.3,0);
		\draw[ sapgreen]
		(0,1.5)--(0,0);
		\draw[ red]
		(.3,1.5)--(.3,0);
		
		\draw[ red]
		(.9,0) arc (0:-180:.3);
		\draw[sapgreen]
		(1.2,0) arc(0:-180:.6) ;
		\draw[violet]
		(1.5,0) arc(0:-180:.9) ;
		
		\draw[violet] 
		(2.1,0) arc(0:180:.3) ;
		\draw[sapgreen]
		(2.4,0) arc(0:180:.6) ;
		\draw[ red]
		(2.7,0) arc(0:180:.9);
		\draw[ violet]
		(2.1,-1.5)--(2.1,0);
		\draw[ sapgreen]
		(2.4,-1.5)--(2.4,0);
		\draw[ red]
		(2.7,-1.5)--(2.7,0);
		
\end{tikzpicture}};
\node at (2.5,0){
	\begin{tikzpicture}[scale = .5]
		\draw[ dashed]
		(.6,0)--(-3,0);
		\draw[ red]
		(.3,1.5)--(.3,0);
		\draw[ sapgreen]
		(0,1.5)--(0,0);
		\draw[ violet]
		(-.3,1.5)--(-.3,0);
		
		\draw[ violet]
		(-.9,0) arc (0:180:-.3);
		\draw[sapgreen]
		(-1.2,0) arc(0:180:-.6) ;
		\draw[red]
		(-1.5,0) arc(0:180:-.9) ;
		
		\draw[red] 
		(-2.1,0) arc(0:-180:-.3) ;
		\draw[sapgreen]
		(-2.4,0) arc(0:-180:-.6) ;
		\draw[ violet]
		(-2.7,0) arc(0:-180:-.9);
		\draw[ red]
		(-2.1,-1.5)--(-2.1,0);
		\draw[ sapgreen]
		(-2.4,-1.5)--(-2.4,0);
		\draw[ violet]
		(-2.7,-1.5)--(-2.7,0);
		
\end{tikzpicture}};
\node at  (4.5,1){(b)};
\node at (7.5,0){
	\begin{tikzpicture}[scale = .5]
		
		\draw[ red]
		(4,.2)	--(4,0) arc(0:-180:.3) -- (3.4,.2);
		\draw[sapgreen]
		(4.3,.2)	--	(4.3,0) arc(0:-180:.6) -- (3.1,.2);
		\draw[violet]
		(3.4,1.8)--	(4.6,1)--(4.6,.2)--	(4.6,0) arc(0:-180:.9)--(2.8,1 )--(4,1.8);
		
		\draw[violet] 
		(4,1.8)	--(4,2) arc(0:180:.3) -- (3.4,1.8);
		\draw[sapgreen]
		(4.3,0)	--	(4.3,2) arc(0:180:.6) -- (3.1,0);
		\draw[ red]
		(3.4,.2)--(3.4,0.5)--(4.6,1.8)--	(4.6,2) arc(0:180:.9)--(2.8,1.8 )--(4,0.5)--(4,.2) ;
		\node at (5.5,1){$=$};
		\draw[violet] (6.5,1)circle  (.5cm);
		\draw[sapgreen] (8,1)circle  (.5cm);	
		\draw[red] (9.5,1)circle  (.5cm);
\end{tikzpicture}};
\end{tikzpicture}
\caption{
(a) The coloured triangle identities; (b) some ways of forming the trace $\mathrm{tr}(id_{\ccc}) $.} 
\label{fig. cups colours}
\end{figure}

\begin{ex}\label{ex. directed Brauer diagrams} (See also \cref{ex. wheeled props}.)
Let $\dipal$ be the palette given by the unique non-trivial involution $+ \leftrightarrow -$ on the two element set $\{+,-\}$. 
Then $\DiBD$ is the category of \textit{oriented Brauer diagrams}. This is the free compact closed category with duals on a single object, or, equivalently, the quotient of the \emph{oriented tangle category} \cite{Tur89}, obtained by identifying over and under-crossings. 
Morphisms in $\DiBD$ are represented, as in \cref{fig. composing directions}, by diagrams of oriented intervals and unoriented circles.

Let $\Bbbk$ be a field of characteristic 0 and recall from \cref{ex: Brauer category}, the Brauer categories $\Br$, $\delta \in \Bbbk$. For $ \delta \in\Bbbk$, the \textit{walled Brauer category} $W\Br \subset \Br$ -- studied in the representation theory of general linear groups {\cite {Wen88, SS15}} -- is the subcategory of $\Br$ whose objects are given by pairs $(m,n)  \in \N^2$, and {$W\Br((m_1, n_1), (m_2, n_2)) \subset \Br (m_1 +n_1, m_2 + n_2))$} is subspace spanned by pairings $\tau$ that correspond to bijections \[\{s^+_1, \dots, s^+_{m_1}\} \amalg \{t^-_1, \dots, t^-_{n_2}\}\xrightarrow{\cong} \{s^-_1, \dots, s^-_{n_1}\} \amalg \{ t^+_1, \dots, t^+_{m_1}\}.\]

Since every element of $\listm \{+,-\}$ is isomorphic -- via a canonical shuffle permutation -- to a unique element of the form $(+, \dots, +, - , \dots, -)$, there is a faithful functor $ \DiBD \to W \Br$ that takes open morphisms in $\DiBD$ to generating diagrams in $W\Br$. 
\begin{figure}
[htb!]
\begin{tikzpicture}
	\node at (0,0) {
		
		\begin{tikzpicture}
			\node at (9,-4){\begin{tikzpicture}[scale = .45]
					\begin{pgfonlayer}{above}
						\node [dot ]  (12) at (-2, 15) {};
						\node [above] at (12) {\small$-$};
						\node [below] at (12) {\small$+$};
						\node  [dot ](13) at (-1, 15) {};
						\node [above] at (13) {\small$+$};
						\node [below] at (13) {\small$-$};
						\node  [dot ] (14) at (0, 15) {};
						\node [above] at (14) {\small$-$};
						\node [below] at (14) {\small$+$};
						\node  [dot ](15) at (1, 15) {};
						\node [above] at (15) {\small$-$};
						\node [below] at (15) {\small$+$};
						\node [dot ] (16) at (2, 15) {};
						\node [above] at (16) {\small$+$};
						\node [below] at (16) {\small$-$};
						\node  [dot ] (17) at (3, 15) {};
						\node [above] at (17) {\small$-$};
						\node [below] at (17) {\small$+$};
						\node  [dot ] (18) at (4, 15) {};
						\node [above] at (18) {\small$+$};
						\node [below] at (18) {\small$-$};
						
						\node at (-4, 16.5){$f\colon$};
						\node  [dot ] (20) at (-0.5, 18) {};
						\node [below] at (20) {\small$+$};
						\node [dot ](21) at (1, 18) {};	
						\node [below] at (21) {\small$+$};
						\node  [dot ](22) at (2.5, 18) {};
						\node [below] at (22) {\small$-$};
						\node at (-4, 13.5){$g\colon$};		
						\node  [dot ](28) at (-2, 12) {};
						\node [above] at (28) {\small$-$};
						\node   [dot ](29) at (-0.5, 12) {};
						\node [above] at (29) {\small$+$};
						\node   [dot ](30) at (1, 12) {};
						\node [above] at (30) {\small$+$};
						\node   [dot ](31) at (2.5, 12) {};
						\node [above] at (31) {\small$-$};
						\node  [dot ](32) at (4, 12) {};
						\node [above] at (32) {\small$-$};

					\end{pgfonlayer}
					\begin{pgfonlayer}{background}
						\draw [bend left=300, looseness=0.75, red, ->-=.5] (20.center) to (12.center);
						\draw [ in=90, out=-150, red, ->-=.5] (21.center) to (14.center);
						\draw [bend left=60, red, -<-=.5] (22.center) to (18.center);
						\draw [bend left=90, looseness=1.5, red, ->-=.5] (16.center) to (17.center);
						\draw [ bend left=60, looseness=1.5,  red, ->-=.6] (13.center) to (15.center);
						\draw [ bend right=75, looseness=1.5, red, ->-=.5] (15.center) to (16.center);
						\draw [ in=135, out=-60, red, ->-=.5] (14.center) to (31.center);
						\draw [bend left=75, , red, -<-=.48] (28.center) to (29.center);
						\draw [bend left=60, looseness=1,  red, ->-=.5] (30.center) to (32.center);
						\draw [ bend right = 75, looseness=1, red, -<-=.5] (13.center) to (17.center);
						\draw [bend right = 90, looseness=1, red, ->-=.5] (12.center) to (18.center);
					\end{pgfonlayer}
				\end{tikzpicture}
			};
		\end{tikzpicture}
	};
	
	\node at (6,0) {
		\begin{tikzpicture}[scale = .45]
			\begin{pgfonlayer}{above}
				\node [dot, blue] (0) at (0, 0) {};
				\node [dot, blue] (1) at (-1, 0) {};
				\node [dot, blue] (2) at (-2, 0) {};
				\node [dot, blue] (3) at (-3, 0) {};
				\node [dot, red] (4) at (1, 0) {};
				\node [dot, red] (5) at (2, 0) {};
				\node [dot, red] (6) at (3, 0) {};
				\node [dot, blue] (7) at (-2.25, 3) {};
				\node [dot, blue] (8) at (-0.75, 3) {};
				\node [dot, red] (9) at (1, 3) {};
				\node [dot, blue] (10) at (0, -3) {};
				\node [dot, blue] (11) at (-1.5, -3) {};
				\node [dot, blue] (12) at (-3, -3) {};
				\node [dot, red](13) at (1.5, -3) {};
				\node [dot, red] (14) at (3, -3) {};
			\end{pgfonlayer}
			\begin{pgfonlayer}{background}
				\draw[dashed, gray](.5, 3.5)--(0.5, -3.5);
				\draw [bend right=15] (7.center) to (3.center);
				\draw [bend right] (8.center) to (2.center);
				\draw [bend left] (9.center) to (6.center);
				\draw [bend right=15, looseness=1.25] (2.center) to (11.center);
				\draw [bend right=90] (0.center) to (4.center);
				\draw [bend left=90, looseness=0.75] (10.center) to (14.center);
				\draw [bend right=90, looseness=0.75] (1.center) to (5.center);
				\draw [in=270, out=-45, looseness=0.75] (3.center) to (6.center);
				\draw [bend left=90, looseness=0.75] (0.center) to (5.center);
				\draw [bend left=90, looseness=0.75] (1.center) to (4.center);
				\draw [bend left=90, looseness=0.75] (12.center) to (13.center);
			\end{pgfonlayer}
		\end{tikzpicture}
	};

\end{tikzpicture}
\caption{Composing directed Brauer diagrams $f \colon (-,-,+) \to (-,+,-,-,+,-,+)$ and $g \colon(-,+,-,-,+,-,+)\to (-,+,+,-,-)$. Up to a shuffle permutation, this is equivalent to a composition of walled Brauer diagrams, where horizontal arrows go from left to right.} \label{fig. composing directions}
\end{figure}

More generally, if $\DDD$ is a set, and $\CCC = \DDD \times\dipal$ with involution $ \omega: (c,+)\mapsto (c,-)$, then $\CBD$ is the free compact closed category with objects generated by $\DDD$. 

\end{ex}

\subsection{Functors from categories of Brauer diagrams}
\label{ssec BD func}

A symmetric monoidal functor $ (\alg, \eta_{\alg}, \alpha) \colon \BD \to \V$ induces interesting algebraic structure on its image in $\V$. 

By definition, $( \alg (n) )_{n \in \N}$ describes an $\N$-graded monoid in $\V$ with product $\boxtimes$ and unit induced by the structure morphisms 
\[\boxtimes_{m,n} \defeq  \alpha_{(m,n)} \colon \alg (m) \otimes \alg (n) \to \alg (m+n), \text{ and } \eta_{\alg}\colon I \to \alg (0). \] The product $\boxtimes$ is associative up to associators in $\V$. 

Moreover, for all $n \geq 2$, and all $1 \leq i <  j <n$, there is a canonical \textit{contraction} morphism 
$\zeta^{i \ddagger j}_{n}\colon \alg(n)\to \alg (n-2)$ induced by the open Brauer diagram $(\cup \oplus id_{n-2})\circ \rho^{i,j}_n \in \BD(n, n-2)$ (see \cref{fig. contraction} (a))), where $\rho^{i,j}_n \in \Sigma_n$ is the shuffle permutation 
on $\nn = \{1, \dots, n\} $ 
given by $i \mapsto 1, j \mapsto 2$, whilst leaving the relative order of the remaining elements unchanged. 

The contraction commutes in the sense that, given $1 \leq k <m \leq n$ such that $i,j,k,m$ are all distinct, 
\begin{equation}\label{eq. mono contraction commutes}\zeta^{k' \ddagger m'}_{n-2}\circ \zeta^{i \ddagger j}_{n}  =\zeta^{i' \ddagger j'}_{n-2}\circ \zeta^{k \ddagger m}_{n} \colon \alg (n) \to  \alg (n-4). 
\end{equation}
Here $1 \leq k' < m' \leq n-2$ and $1 \leq i' < j' \leq n-2$ are the obvious adjusted indices 
of $k,m$ and $i,j$ in  $ \{1, \dots, n-2\}$. 
(See \cref{fig. contraction} (b).)

\begin{figure}[htb!]
\begin{tikzpicture}
	\node at (-2,1.5) {(a)};
	\node at (-.5, -2) { $\zeta^{2 \ddagger 5}_5$};
	\node at (-.5,0) {
		\begin{tikzpicture}[scale = .4]
			
			\node  (0) at (-5, 4) {};
			\node  (1) at (-5, 0) {};
			\node  (2) at (-4, 4) {};
			\node  (3) at (-4, 0) {};
			\node  (4) at (-3, 4) {};
			\node  (5) at (-3, 0) {};
			\node  (6) at (-2, 4) {};
			\node  (7) at (-2, 0) {};
			\node  (8) at (-1, 4) {};
			\node  (9) at (-1, 0) {};
			\node  (10) at (-3, -3) {};
			\node  (11) at (-2, -3) {};
			\node  (12) at (-1, -3) {};
			
			\draw (2.center) to (1.center);
			\draw (8.center) to (3.center);
			\draw (0.center) to (5.center);
			\draw (4.center) to (7.center);
			\draw (6.center) to (9.center);
			\draw (5.center) to (10.center);
			\draw (7.center) to (11.center);
			\draw (9.center) to (12.center);
			\draw [bend right=90, looseness=8.00] (1.center) to (3.center);
			
		\end{tikzpicture}
		
	};

	\node at (3.7,1.5) {(b)};
\node at (5,0) {
	\begin{tikzpicture}[scale = .3]
		
		\node (0) at (-6, 4) {};
	\node (1) at (-6, 0) {};
	\node (2) at (-5, 4) {};
	\node (3) at (-5, 0) {};
	\node (4) at (-4, 4) {};
	\node (5) at (-4, 0) {};
	\node (6) at (-3, 4) {};
	\node (7) at (-3, 0) {};
	\node (8) at (-2, 4) {};
	\node (9) at (-2, 0) {};
	\node (10) at (-4, -2) {};
	\node (11) at (-3, -2) {};
	\node (12) at (-2, -2) {};
	\node (13) at (-1, 4) {};
	\node (14) at (-1, 0) {};
	\node (15) at (-1, -2) {};
	\node (16) at (-2, -5) {};
	\node (17) at (-1, -5) {};
	
	\draw (13.center) to (14.center);
	\draw (14.center) to (15.center);
	\draw (15.center) to (17.center);
	\draw (0.center) to (1.center);
	\draw (6.center) to (3.center);
	\draw (2.center) to (5.center);
	\draw (4.center) to (7.center);
	\draw (8.center) to (9.center);
	\draw [bend right=90, looseness=5.75] (1.center) to (3.center);
	\draw (5.center) to (10.center);
	\draw (9.center) to (11.center);
	\draw (7.center) to (12.center);
	\draw (12.center) to (16.center);
	\draw [bend right=90, looseness=5.75] (10.center) to (11.center);

	\end{tikzpicture}
	
};
	\node at (7.5, -2) { $\zeta^{1 \ddagger 3}_4\circ \zeta^{1\ddagger 4}_6=\zeta^{1 \ddagger 3}_4\circ \zeta^{2 \ddagger 5}_6$};
		\draw 
	(6.8, .2)--(7.2,.2)
	(6.8,0)--(7.2,0);
\node at (9,0) {
	\begin{tikzpicture}[scale = .3]

		\node (0) at (-5, 4) {};
		\node (1) at (-5, 0) {};
		\node (2) at (-4, 4) {};
		\node (3) at (-4, 0) {};
		\node (4) at (-3, 4) {};
		\node (5) at (-3, 0) {};
		\node (6) at (-2, 4) {};
		\node (7) at (-2, 0) {};
		\node (8) at (-1, 4) {};
		\node (9) at (-1, 0) {};
		\node (10) at (-3, -2) {};
		\node (11) at (-2, -2) {};
		\node (12) at (-1, -2) {};
		\node (13) at (0, 4) {};
		\node (14) at (0, 0) {};
		\node (15) at (0, -2) {};
		\node (16) at (-1, -5) {};
		\node (17) at (0, -5) {};
		
		\draw (2.center) to (1.center);
		\draw (8.center) to (3.center);
		\draw (0.center) to (5.center);
		\draw (4.center) to (7.center);
		\draw (6.center) to (9.center);
		\draw [bend right=90, looseness=6.00] (1.center) to (3.center);
		\draw (5.center) to (10.center);
		\draw (9.center) to (11.center);
		\draw (7.center) to (12.center);
		\draw (13.center) to (14.center);
		\draw (14.center) to (15.center);
		\draw [bend right=90, looseness=6.00] (10.center) to (11.center);
		\draw (12.center) to (16.center);
		\draw (15.center) to (17.center);

	\end{tikzpicture}
	
};

\end{tikzpicture}
\caption{Brauer diagrams for contraction}\label{fig. contraction}
\end{figure}

\label{ex. unit in Set CA}

Together, the contraction $\zeta$ and monoidal product $\boxtimes$ induce a `multiplication' operation $\diamond$ (see \cref{defn: multiplication}), by  
\begin{equation}\label{eq. bd mult}\diamond_{m,n}^{i \ddagger j} = \zeta^{i \ddagger m+j}_{m+n} \circ \boxtimes_{m,n}  \colon \alg(m) \otimes \alg (n) \to \alg (m+n -2),
\end{equation} for all $m,n \geq 1$ and all $1 \leq i \leq m$, $1 \leq j \leq n$. 

Let $\rho_{n \leadsto i} \in \Sigma_n$ be the permutation $n \mapsto i$ whilst leaving the relative order of the remaining elements of $\nn$ unchanged. By the triangle identities (\ref{eq. n triangle}), 
\[ (\cup \oplus \rho_{n \leadsto i}) \circ \rho _{n+2}^{i, n+1} \circ (id_n \oplus \cap) = id _n \text{ in } \BD.\]

Hence, 
the image of the cap morphism $\cap \in \BD(0,2)$ induces a two-sided unit $\epsilon_{\alg} \colon I \mapsto \alg (2)$ for the multiplication $\diamond$ on $\alg$. 

\begin{rmk}
A symmetric functor $\BDd \to (\V, \otimes, I)$ from the category of downward Brauer diagrams induces contraction and multiplication operations on its image in $\V$. However, since $\cap$ is not a morphism in $\BDd$, there is no unit for this multiplication (see also \cite{Ray20, SS20I}). 
\end{rmk}

\label{ex.contraction of Brauer diagrams}\label{ex. operations on CBD algebras}

The coloured case follows in the same way: If $(\CCC, \omega)$ is a palette, then a functor $( \alg, \alpha, \eta_{\alg}) \colon \CBD \to (\V, \otimes , I)$ equips the $\listm \CCC$-indexed object $(\alg(\ccc))_{\ccc}$ with an, up to associators in $\V$, associative unital graded monoidal product via $\alpha$. The induced product $\boxtimes$ has unit $ I \mapsto \alg (\varnothing_{\CCC})$.

A contraction $\zeta$ and multiplication $\boxtimes$ on $(\alg(\ccc))_{\ccc}$ are obtained exactly as in the monochrome case. And, once again, $\cap_c \in \CBD(\varnothing_{\CCC}, (c, \omega c))$ induces a distinguished morphism $ \alg (\cap_c) \colon I \to \alg(c, \omega c)$ that acts as a two-sided unit for $\diamond$ for all $c \in \CCC$.

\section{Circuit algebras}\label{sec. CA}
Circuit algebras -- defined in \cref{ssec. wd and ca} -- are algebraic structures described by a small family $(A_c)_c$ of algebraic objects (usually vector spaces). Their operations are governed by wiring diagrams, which are, essentially, non-planar versions of Jones's planar diagrams \cite{Jon94}. 

Wiring diagrams are commonly described by partitioning boundaries of 1-manifolds (e.g.~{\cite{DHR20,DHR21}}). However, they 
admit a straightforward description in terms of 
Brauer diagrams, and this is the approach of this paper.

\subsection{Operads preliminaries} \label{ssec operads}

This section summarises the basic theory of (coloured) operads. 
See \cite{Lei04} and \cite{BM07} for more details and precise definitions of coloured operads (called multicategories in \cite{Lei04}) and their algebras.

As in \cref{ssec. colour BD}, let $\listm \DDD$ denote the set of finite lists $\ddd = (d_1, \dots, d_n)$ on a set $\DDD$. Recall that a (symmetric) $\DDD$-coloured operad $\op$ (in the category of sets) is given by a $(\listm \DDD \times \DDD)$-graded set $ (\op(c_1, \dots, c_m; d))_{(c_1, \dots, c_m; d)}$, and 
a family of \textit{composition morphisms} 
\[ \gamma \colon \op(\ccc; d) \times \left( \prod_{ i = 1}^m \op(\bbb_i; c_i)\right) \to \op(\bbb_1 \dots \bbb_m;d)\] defined for all $d \in \DDD$, $ \ccc  = (c_i)_{i = 1}^m $ and all $\bbb_i  \in \listm \DDD$ ($1 \leq i \leq m$).

For $\phi \in \op(c_1, \dots, c_m; d)$, $d$ is called the \textit{output} of $\phi$ and each $c_i$ is an \textit{input} of $\phi$. The symmetric groupoid $\Sigma$ acts on $\op$ by permuting the inputs: each $\sigma \in \Sigma_n$ induces isomorphisms 
\[ \op(c_1, \dots, c_n; d) \xrightarrow{\cong} \op(c_{\sigma 1}, \dots, c_{\sigma n}; d). \]

The composition $\gamma$ is required to be associative and equivariant with respect to the $\Sigma$-action on $\op$. 

For all $d \in D$, there is an element $\nu_d \in \op (d;d)$ such that 
for all $\ccc  = (c_1, \dots, c_m) \in \listm \DDD$, the morphisms given by the composites 
\[\op (\ccc;d)\xrightarrow{ \cong } I \times \op (\ccc;d)\xrightarrow{(\nu_d, id)}  \op(d; d) \times \op (\ccc;d) \xrightarrow{\gamma} \op (\ccc;d), \]
and 
\[
\op (\ccc;d)\xrightarrow{ \cong }\op (\ccc;d)\times I \\ \xrightarrow{(id, \bigotimes_{i =1}^m \nu_{c_i})}  \op (\ccc;d)\times \left(\bigotimes_{ i =1}^m \op (c_i; c_i) \right)\xrightarrow{\gamma}   \op (\ccc;d),\] are the identity on $\op (\ccc;d)$.  

	\begin{ex}\label{ex. first cat operad}
		In particular, the restriction of a $\DDD$-coloured operad $\op$ to the sets $\op(c;d)$( for $c, d \in \DDD$) describes a small category.  Conversely, a small category $\CCat$ with object set $\CCat_0$ describes a $\CCat_0$-coloured operad $C$ such that $C(x;y) = \CCat(x,y)$ for all $x, y \in \CCat$ and $C(x_1, \dots, x_k; y) = \emptyset $ for $k \neq 1$. 
	\end{ex}
	
The following example is somewhat more interesting: 
	\begin{ex}
		\label{ex. operad of a perm cat}
		
		Underlying any (small) permutative category $(\CCat, \oplus, 0)$, with object set $\CCat_0$, is a $\CCat_0$-coloured (symmetric) operad $\op^{\CCat}$, 
		defined by 
		\[\op^{\CCat}(x_1, \dots, x_n; y)  \defeq \CCat(x_1 \oplus \dots \oplus x_n, y ),\] and 
		with operadic composition $\gamma $ in $\op^{\CCat}$ induced by composition in $\CCat$:
		\[ \gamma \colon \op^{\CCat}(x_1, \dots, x_n; y) \times \left(\prod_{i = 1}^k \op^{\CCat}(w_{i,1}, \dots, w_{i, m_i};x_i) \right) \to  \op^{\CCat}(w_{1,1}, \dots, w_{i, m_i}, \dots, w_{n, m_n};y),  \]  
		\[ \left(\overline g, (\overline f_i)_i\right) \mapsto \overline{\left(g \circ  (f_1 \oplus \dots \oplus f_n)\right)}  \] for all $\overline g \in \op^{\CCat}(x_1, \dots, x_n;y)$ described by $g \in \CCat (x_1 \oplus \dots \oplus x_n, y)$, and $\overline f_i \in \op^{\CCat}(w_{i,1}, \dots, w_{i, m_i};x_i)$ described by $f_i \in \CCat (w_{i,1}\oplus \dots \oplus w_{i,m_i}, x_i)$ ($1 \leq i \leq n$). 
		
			(In fact, any monoidal category has an underlying operad \cite[Example~2.1.3]{Lei04}, but that is not needed in this work.)

	\end{ex}

	Objects of the category $\Op$ are (coloured) operads. 
	For $i = 1,2$, let $(\op^i, \gamma^i, \nu^i)$ be a $\DDD_i$-coloured operad. A morphism $\alg \colon (\op^1, \gamma^1, \nu^1) \to  (\op^2, \gamma^2, \nu^2)$ is given by 
	a map of sets $\alg_{0} \colon \DDD_1 \to \DDD_2$, and a $(\listm \DDD_1 \times \DDD_1)$-indexed family of maps \[\alg_{(c_1, \dots, c_k; d)} \colon \op^1(c_1, \dots, c_k; d) \to \op^2(\alg_0 (c_1), \dots,\alg_0 (c_k); \alg_0 (d) )\] that respect units and composition, and are equivariant with respect to the action of $\Sigma$.

	If $\alg_0 = id_{\DDD}$, with $\DDD = \DDD_1= \DDD_2 $, then $\alg \colon \op^1 \to \op^2$ is called \textit{colour-preserving}. The category of $\DDD$-coloured operads and colour-preserving morphisms is denoted by $\Op^{\DDD}$.

	\begin{ex}
		\label{ex. endo operad}
		Let $(\V, \otimes, I)$ be a symmetric monoidal category. Given a set $\DDD$ and a $\DDD$-indexed object $A = (A(c))_{c \in \DDD}$ in $\V$, the $\DDD$-coloured \textit{endomorphism operad $End^A$} is given by
		\[End^A(c_1, \dots, c_k;d)\defeq \V\left(A(c_1) \otimes \dots \otimes A(c_k), A(d)\right),\] together with the obvious composition and units induced by composition and identities in $\V$.
		
	\end{ex}

	Let $(\V, \otimes, I)$ be a symmetric monoidal category and $\op$ a $\DDD$-coloured operad.
	\begin{defn}\label{def operad alg}
		A \emph{$\V$-algebra for $\op$} is a $\DDD$-indexed object $(A(c))_c$ in $\V$, together with a morphism $\alg \colon \op \to End^A$ of $\DDD$-coloured operads. 
		
		The category $\mathsf{Alg}_{\V}(\op)$ of $\V$-algebras for $\op$ is the subcategory of the slice category $\op \ov \Op$ whose objects are of the form $(A, \alg)$ with $ A = (A(c))_c$ a $\DDD$-graded object in $\V$. A morphism $ g \colon (A, \alg) \to (B, \mathcal B)$ in $\op \ov \Op$ is a morphism in $\mathsf{Alg}_{\V}(\op)$ if and only if there are morphisms $ g_c \in \V( A(c) , B(c))$ (for all $c \in \DDD$), such that, for all $ c_1, \dots, c_k,d \in \DDD$ and all $\phi \in End^A (c_1, \dots, c_k;d)$, the following diagram commutes in $\V$:
		
		\[ \xymatrix{ A(c_1) \otimes \dots \otimes  A(c_k) \ar[d]_{\phi}\ar[rr]^-{g_{c_1} \otimes \dots \otimes g_{c_k}}&& 
			 B(c_1) \otimes \dots \otimes B(c_k) \ar[d]^{g(\phi)}\\ 
			 A(d)\ar[rr]_-{g_d} && B(d).}\]

	\end{defn}
	
	\begin{rmk}
		Although 
		\cref{def operad alg} relies on the symmetric monoidal structure on $\V$, it is concerned with operads in the category of sets and does not involve operads enriched in a (closed) symmetric monoidal category. 
		
	\end{rmk}

	If $\op = \op^{\CCat}$ is the operad underlying a small permutative category $\CCat$, then a $\V$-algebra for $\op$ describes a functor $\CCat \to \V$. It follows from e.g.~ \cite[Chapters~2-3,]{Lei04} (and particularly Theorem 3.3.4(b) there) that: 
	\begin{lem}\label{rmk. lax functors and algebras}\label{prop. lax functors and algebras}
		The category $\mathsf{Alg}_{\V} (\op^{\CCat})$ of $\V$-algebras for the operad $\op^{\CCat}$ underlying a small permutative category $(\CCat, \oplus, 0)$ is 
		canonically isomorphic to the category of symmetric lax monoidal functors $\alg \colon (\CCat, \oplus, 0) \to (\V, \otimes , I)$. 
	\end{lem}

	\subsection{Wiring diagrams and circuit algebras}\label{ssec. wd and ca}
	 
	 Circuit algebras are algebras over an operad of wiring diagrams. Wiring diagrams are often defined (for example in {\cite{DHR20,DHR21}}) as equivalence classes of immersions of compact 1-manifolds in punctured 2-discs (see \cref{rmk classic wd}).
	 
	 This work takes a different approach.

\begin{defn}\label{defn. operad of wiring diagrams}
	For a given palette $(\CCC, \omega)$, the \emph{$\listm \CCC$-coloured operad of $(\CCC, \omega)$-wiring diagrams} is the {operad $\CWD \defeq \op^{{\CBD}}$ underlying $\CBD$} (\cref{ex. operad of a perm cat}).

	For each $(\ccc_1, \dots, \ccc_k;\ddd) \in \listm^2 \CCC$, elements of the set
	\[ \CWD (\ccc_1, \dots, \ccc_k; \ddd) \defeq \CBD(\ccc_1 \oplus \dots \oplus \ccc_k; \ddd)\]
	are called \emph{wiring diagrams of type $(\ccc_1, \dots, \ccc_k; \ddd)$}. 
	
	When $\CCC$ is the singleton set, the $\N $-coloured operad $\CWD  \defeq \op^{\BD}$ is denoted by $ \WD$, and called the \emph{operad of (monochrome) wiring diagrams}.  
\end{defn}
\begin{figure}
	[htb!]
	\begin{tikzpicture}
		\node at (-.5,0){
			\begin{tikzpicture}[scale = .33]
				\begin{pgfonlayer}{above}
					\node [dot] (0) at (-9, 2) {};
					\node [dot] (1) at (-8, 2) {};
					\node [dot] (2) at (-7, 2) {};
					\node [dot] (3) at (-6, 2) {};
					\node [dot] (4) at (-8.5, -1) {};
					\node [dot] (5) at (-6.5, -1) {};
					
					\node [dot] (6) at (-5, 2) {};
					\node [dot] (7) at (-4, 2) {};
					\node [dot] (8) at (-3, 2) {};
					\node [dot] (9) at (-2, 2) {};
					
					\node [dot] (10) at (-1, -1) {};
					\node [dot] (11) at (-1, 2) {};
					
					\node [dot] (12) at (1, -1) {};
					\node [dot] (13) at (2.5, -1) {};
					
					\node [dot] (14) at (5, -1) {};
					\node [dot] (15) at (6.5, -1) {};
					\node [dot] (16) at (8, -1) {};
					\node [dot] (17) at (6, 2) {};
					
					\node [dot] (18) at (-3.25, -5) {};
					\node [dot] (19) at (-1.25, -5) {};
					\node [dot] (20) at (0.75, -5) {};
					\node [dot] (21) at (2.75, -5) {};
					\node [dot] (22) at (4.75, -5) {};
					\node [dot] (23) at (-5.25, -5) {};
					
					\draw[thick,red] (-5.5,-1.5)--(-5.5,2.5)
					(0,-1.5)--(0,2.5)
					(4.5,-1.5)--(4.5,2.5)
					(8.5,-1.5)--(8.5,2.5)
					(-9.5,-1.5)--(-9.5,2.5);
					\draw[ultra thick,cyan] 
					(-1.5,2.5)--(-1.5,1.8)
					(3.2,2.5)--(3.2,1.8)
					(-6.5,2.5)--(-6.5, 1.8)
					(-3.5,2.5)--(-3.5, 1.8)
					(1,2.5)--(1, 1.8)
					(2,2.5)--(2, 1.8)
					(6.5,2.5)--(6.5, 1.8)
					;
				\end{pgfonlayer}
				\begin{pgfonlayer}{background}
					\draw [bend right=60, looseness=0.75] (0.center) to (3.center);
					\draw [in=120, out=-105] (1.center) to (5.center);
					\draw [in=60, out=-75, looseness=1.25] (2.center) to (4.center);
					\draw [bend right=90] (7.center) to (9.center);
					\draw [bend right=90] (6.center) to (8.center);
					\draw [in=90, out=-90] (11.center) to (10.center);
					\draw [bend left=90, looseness=1.25] (12.center) to (13.center);
					\draw [in=75, out=-75] (17.center) to (15.center);
					\draw [bend left=90, looseness=0.75] (14.center) to (16.center);
					\draw [in=75, out=-105] (10.center) to (23.center);
					\draw [bend left=90, looseness=0.75] (18.center) to (21.center);
					\draw [bend right=90, looseness=1.50] (12.center) to (13.center);
					\draw [in=30, out=-120, looseness=0.75] (16.center) to (19.center);
					\draw [bend left=90, looseness=0.50] (14.center) to (5.center);
					\draw [in=150, out=-45] (4.center) to (20.center);
					\draw [in=120, out=-105] (15.center) to (22.center);
				\end{pgfonlayer}
		\end{tikzpicture}};
		\draw[gray, dashed, ->, line width = 1](3.7,0)--(4.5,0);
		\node at (8,0){\begin{tikzpicture}[scale = .33]
				\begin{pgfonlayer}{above}
					\node [dot] (0) at (-9, 2) {};
					\node [dot] (1) at (-8, 2) {};
					\node [dot] (2) at (-7, 2) {};
					\node [dot] (3) at (-6, 2) {};
					\node [dot] (6) at (-5, 2) {};
					\node [dot] (7) at (-4, 2) {};
					\node [dot] (8) at (-3, 2) {};
					\node [dot] (9) at (-2, 2) {};
					\node [] (10) at (-1, -1) {};
					\node [dot] (11) at (-1, 2) {};
					\node [] (12) at (8, -1) {};
					\node [] (13) at (9.5, -1) {};
					\node [] (15) at (6.5, -1) {};
					\node [] (16) at (8, -1) {};
					\node [dot] (17) at (6, 2) {};
					\node [dot] (18) at (-3.25, -5) {};
					\node [dot] (19) at (-1.25, -5) {};
					\node [dot] (20) at (0.75, -5) {};
					\node [dot] (21) at (2.75, -5) {};
					\node [dot] (22) at (4.75, -5) {};
					\node [dot] (23) at (-5.25, -5) {};
					\draw[ultra thick,cyan] (-5.5,1.8)--(-5.5,2.5)
					(-1.5,1.8)--(-1.5,2.5)
					(0,1.8)--(0,2.5)
					(3.2,1.8)--(3.2,2.5)
					(4.5,1.8)--(4.5,2.5)
					(8.5,1.8)--(8.5,2.5)
					(-9.5,1.8)--(-9.5,2.5);
					\draw[ultra thick,cyan] 
					(-6.5,2.5)--(-6.5, 1.8)
					(-3.5,2.5)--(-3.5, 1.8)
					(1,2.5)--(1, 1.8)
					(2,2.5)--(2, 1.8)
					(6.5,2.5)--(6.5, 1.8)
					;
				\end{pgfonlayer}
				\begin{pgfonlayer}{background}
					\draw [bend right=60, looseness=0.75] (0.center) to (3.center);
					\draw [bend right=90] (7.center) to (9.center);
					\draw [bend right=90] (6.center) to (8.center);
					\draw [in=90, out=-90] (11.center) to (10.center);
					\draw [bend left=90, looseness=1.25] (12.center) to (13.center);
					\draw [in=75, out=-75] (17.center) to (15.center);
					\draw [in=75, out=-105] (10.center) to (23.center);
					\draw [bend left=90, looseness=0.75] (18.center) to (21.center);
					\draw [bend right=90, looseness=1.50] (12.center) to (13.center);
					\draw [in=120, out=-105] (15.center) to (22.center);
					\draw (2.center) to (20.center);
					\draw (1.center) to (19.center);
				\end{pgfonlayer}
			\end{tikzpicture}
		};
	\end{tikzpicture}
	\caption{Composition in $\CWD$. (See also \cref{fig. pictorial}.) }
	\label{fig. CWD comp}
\end{figure}

	\begin{rmk}\label{rmk classic wd}
	Wiring diagrams are often defined (for example in {\cite{DHR20,DHR21}}) as equivalence classes of immersions of compact 1-manifolds in punctured 2-discs. According to this representation, wiring diagrams form an operad with composition described by inserting discs into punctures as in \cref{fig. disc rep}. From this point of view, the composition in \cref{fig. CWD comp} is represented as in \cref{fig. disc rep}.
	
	On this interpretation, operads of coloured wiring diagrams may be defined by colouring 1-manifolds according to \cref{ex. manifold components} and \cref{def. colouring}.

	This visualisation of wiring diagrams illustrates the relationship between wiring diagrams (and hence circuit algebras) and planar diagrams and algebras \cite{Jon99}. It also clearly exhibits the operad of monochrome wiring diagrams as a suboperad of the operad of wiring diagrams defined in \cite{Spi13}. 
	
	Moreover, the disc representation of wiring diagrams is highly suggestive of the graphical constructions that will follow in Sections \ref{s. graphs}-\ref{sec. nerve}.
	
	However, the definition in terms of Brauer diagrams is more concise and, furthermore, reveals the deep connections between circuit algebras, representations of classical groups and, via Theorems \ref{thm. CO CA} and \ref{thm: CO nerve}, the combinatorics of modular operads.

\end{rmk}

\begin{figure}[htb!]
	
	\includegraphics[width=0.95\textwidth]{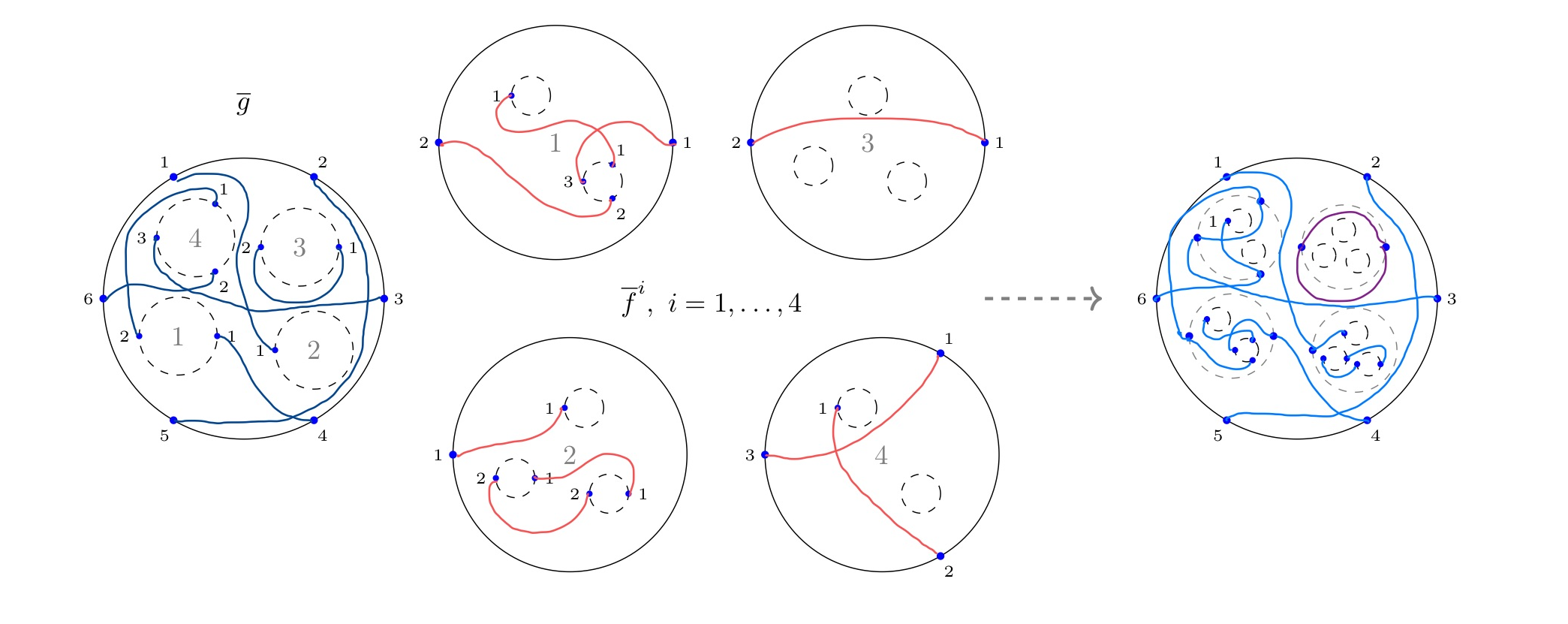}
	
	\caption{Disc representation of the composition in \cref{fig. CWD comp}.}
\label{fig. disc rep}\label{fig. pictorial}
\end{figure}

Let $\V$ be a symmetric monoidal category. The definition of circuit algebras in \cite{BND17, DHR20} is generalised by:
\begin{defn}\label{def CA}A \emph{$(\CCC,\omega)$-coloured $\V$-circuit algebra} is a $\V$-valued algebra for the operad $\CWD$ of $(\CCC, \omega)$-coloured wiring diagrams.
	The full subcategory of $\V$-circuit algebras in $\mathsf{Alg}(\CWD)$ is denoted by $\CCA_\V$.  When $\V = \Set$, $\CCA_\V$ is denoted simply by $\CCA$. 
\end{defn}

\begin{ex}		\label{ex. wheeled props} \label{ex. wheeled props operad} As in \cref{ex. directed Brauer diagrams}, let $\Di$ be the palette described by the unique non-trivial involution on the two element set $\{+,-\}$. For any set $\DDD$, there is a palette $\DDD \times \Di$ given by the involution $(c, + )\leftrightarrow (c,-)$, for all $c \in \DDD$.
	
	If $(\CCC, \omega) = \DDD \times \Di$ for some set $\DDD$, then $\CWD$ is an operad of \emph{$\DDD$-coloured oriented wiring diagrams}, and $\DDD \times \Di$-coloured circuit algebras are \emph{$\DDD$-coloured oriented circuit algebras}. 
	
These are described in detail in \cite{DHR20, DHR21}. In \cite{DHR20}, it is proved that $\DDD$-coloured oriented $\V$-circuit algebras in $\V$ are equivalent to $\DDD$-coloured wheeled PROPs in $\V$. Hence $\Di$-coloured circuit algebras are (monochrome) wheeled PROPs. These have been described in \cite{MMS09, Mer10} and have applications in geometry and deformation theory.
	
\end{ex}

It follows immediately from \cref{rmk. lax functors and algebras} and \cref{defn. operad of wiring diagrams} that:
\begin{thm}\label{thm. lax functor ca}
	The category $\CCA_\V$ of $(\CCC, \omega)$-coloured $\V$-circuit algebras is isomorphic to the category of symmetric monoidal functors $\alg \colon \CBD \to \V$.

\end{thm} 

\begin{ex}
	In particular, if $R$ is a commutative ring, and $\V= \modR$, then a monochrome $\V$-circuit algebra $\alg$ describes a sequence of $R$-modules, each equipped with a pairing $\alg(n)\otimes \alg(n) \to R$ induced by $\alg (\cup_n)$.
	
Moreover, if $\alg (0)= R$ and $\alg (m) \otimes \alg (n) \cong \alg (m +n)$, then $\alg$ corresponds to a symmetric strong monoidal functor $\BD \to \modR$. In this case, the pairing on $\alg (n)$ is non-degenerate for each $n$ and $\alg (n)$ describes a Brauer algebra $\Br(n,n)$ with $\delta  = \alg (\bigcirc) \in R$.

\end{ex}

The notation $\alg$ will be used to denote both a circuit algebra $(A, \alg)$ and the corresponding functor $\alg \colon \CBD \to \V$. 	

The assignment $(\CCC, \omega) \mapsto \CCA_\V$ defines a $\Cat$-valued presheaf $ca_\V$ on the palette category ${\pal}$: a morphism $\phi \colon (\CCC, \omega) \to (\CCC', \omega')$ in $\pal$ induces a strict symmetric monoidal functor $\CBD \to \CBD[(\CCC', \omega')]$, and hence any $(\CCC', \omega')$ coloured circuit algebra $\mathcal A'$ pulls back to a $(\CCC, \omega )$-coloured circuit algebra $\phi^* \mathcal A'$.

For any symmetric monoidal category $\V$, let $\CA_\V$ be the category of \emph{all} $\V$-circuit algebras whose objects are pairs $((\CCC, \omega), \mathcal A)$ of a palette $(\CCC, \omega)$ and a $(\CCC, \omega)$-coloured $\V$-circuit algebra $\mathcal A$, and whose morphisms $((\CCC, \omega), \mathcal A) \to ((\CCC', \omega'), \mathcal A') $ are pairs $(\phi, \alpha )$ where $\phi \colon \CCC\to \CCC'$ satisfies $\phi \omega = \omega' \phi$ and $ \alpha \colon \phi^* \mathcal A' \to \mathcal A$ is a morphism of $(\CCC, \omega)$-coloured circuit algebras.  
 
 In other words $\CA_\V$ is the Grothendieck construction of the functor $ca_A \colon \pal^{\mathrm{op}} \to \Cat$. When $(\V, \otimes, I) = (\Set, \times, *)$, $ \CA_\V$ is denoted by $\CA \defeq \CA_{\Set}$.

	\medspace

	By \cref{thm. lax functor ca}, a $(\CCC, \omega)$-coloured $\V$-circuit algebra 
	consists of 
	a collection $(\alg(\ccc))_{\ccc \in \listm \CCC}$ of objects and, for each $(\ccc_1, \dots, \ccc_k; \ddd)  \in \listm^2 \CCC \times \listm \CCC$, a set of $\V$-morphisms $\alg{(\overline f, \lambda)}\colon \bigotimes _{ i = 1} ^k \alg (\ccc_i) \to \alg(\ddd)$ indexed by Brauer diagrams $(f, \lambda) \in \CBD(\ccc_1 \oplus \dots \oplus \ccc_k , \ddd)$. 
	
	These satisfy:
	\begin{itemize}
		\item for all $\ccc \in \listm \CCC$, $\alg({ \overline {id}_{\ccc}} )= id_{\alg(\ccc)} \in \V(\alg(\ccc),\alg(\ccc))$; 
		\item the morphisms $\alg{( \overline f, \lambda)}$ are equivariant with respect to the $\Sigma$-action on $\listm \CCC$ and on $\CWD$;
		\item given wiring diagrams $(\overline f, \lambda) \in \CWD (\ccc_1, \dots, \ccc_k; \ddd)$, and, for all $1 \leq i \leq k $, $(\overline f^i, \lambda^i)\in \CWD (\bbb_{i, 1}, \dots, \bbb_{i, k_i};\ccc_i)$, the diagram 
		\begin{equation}\label{eq. algebra maps} \xymatrix{ \bigotimes_{i = 1}^k \bigotimes_{j = 1}^{k_i} \alg(\bbb_{i,j}) \ar[rrrrd]_-{ \alg{\left( \gamma \left(( \overline f, \lambda),( \overline f^i, \lambda^i)_{ i}\right)\right)} \qquad {}}\ar[rrrr]^{\bigotimes_{i = 1}^k \alg{(\overline f^i, \lambda^i)}} &&&& \bigotimes_{i = 1}^k \alg(\ccc_i) \ar[d]^{ \alg{(\overline f, \lambda )}}\\ &&&& \alg(\ddd)}\end{equation} commutes in $\V$.
		
	\end{itemize}

	Let $\alg$ be a $(\CCC, \omega)$-coloured $\V$-circuit algebra. 
	\begin{prop}
		\label{prop. product and contraction prop}
		The collection $(\alg({\ccc}))_{\ccc \in \listm \CCC}$ describes a $\listm \CCC$-graded monoid  $(\alg, \boxtimes, \alg(\varnothing_{\CCC} ))$ in $\V$, equipped with
		contraction maps $\zeta_{\ccc}^{i \ddagger j} \colon \alg(\ccc) \to \alg(\ccc_{\widehat{i,j}})$ for all $1 \leq i < j \leq m$ and all $\ccc = (c_1, \dots, c_m)\in \listm \CCC$ such that $c_i  = \omega c_j $.   
		These satisfy:
		\begin{enumerate}[(c1)]
			\item the graded monoidal product $\boxtimes$ on $(\alg(\ccc))_{\ccc \in \listm \CCC}$ is associative up to associators in $\V$;
			\item contractions commute: 
			\[\zeta^{i' \ddagger j'}_{\ccc_{\widehat{k,m}}} \circ \zeta^{k \ddagger m}_{\ccc}= \zeta^{k' \ddagger m'}_{\ccc_{\widehat{i,j}}} \circ \zeta^{i \ddagger j}_{\ccc} \colon \alg({\ccc} )\to \alg({\ccc_{\widehat{i,j,k,m}}}) \ \text{ wherever defined};\] 
			\item contraction commutes with monoidal product:
			\[ \zeta^{ i \ddagger j}_{\ccc\ddd}\  \circ\  \boxtimes_{\ccc, \ddd}  = \boxtimes_{\ccc_{\widehat{i,j}}\oplus \ddd} \ \circ \ (\zeta^{i \ddagger j}_{\ccc} \otimes id_{\alg(\ddd)}) \colon \alg(\ccc)\otimes \alg(\ddd) \to \alg(\ccc_{\widehat{i,j}}\ddd)\]
			for all $\ddd \in \listm \CCC$ and $\ccc = (c_1, \dots, c_m) \in \listm \CCC$ with $c_i  = \omega c_j$, $1 \leq i < j \leq m$.
		\end{enumerate}
		
		Moreover, for each $c \in \CCC$, there is a distinguished morphism $\epsilon_c \in \V(I, \alg{(c, \omega c)})$ such that
		\begin{enumerate}[(e1)]
			\item $\epsilon_c$ satisfies $\zeta_{\ccc \oplus (c, \omega c)}^{i\ddagger 2} ( id_{\ccc} \boxtimes \epsilon_c)  = id_{\ccc}$ for all $\ccc\in \listm \CCC$ and all $i$ such that $c_i = c$. 
		\end{enumerate}
	\end{prop}
	\begin{proof}
		By \cref{thm. lax functor ca}, $\alg$ describes a symmetric monoidal functor $(\alg, \alpha, \eta_\alpha) \colon \CBD \to \V$ and so the proposition follows from \cref{ssec BD func}. In particular, the monoidal product $\boxtimes$ on $(\alg({\ccc}))_{\ccc \in \listm \CCC}$ is induced by $\alpha $ and the maps $\overline{id_{\ccc\ddd} } \in \CWD(\ccc, \ddd; \ccc\ddd) $ corresponding to the identity morphisms $id_{\ccc\ddd}\in \CBD(\ccc\ddd,\ccc\ddd)$. 
	\end{proof}

	Let $\WD_{\downarrow} \subset \WD$ be the suboperad of (monochrome) \textit{downward wiring diagrams}. This is the operad underlying the category $\BDd$ of downward Brauer diagrams. 
	\begin{cor}
		\label{cor. down WD}
		A $\V$-algebra for $\WD_{\downarrow} $ is a graded monoid $(\alg(n)_n, \boxtimes, \alg (0))$ in $\V$ equipped with a contraction operation $\zeta$, such that the monoid multiplication $\boxtimes$ and $\zeta$ satisfy (c1)-(c3). 
	\end{cor}
	\begin{proof}
		By \cref{ex. unit in Set CA}, the graded monoid and contraction structures of a $\WD$-algebra depend only on morphisms in $\BDd$, whence the corollary follows immediately. 
	\end{proof}

	Since $\cap\in \BD(0,2)$ is not a morphism in $\BDd$, algebras for $\WD_{\downarrow}$ do not, in general satisfy (e1), and therefore do not have a unit for the induced multiplication $\diamond$ (\cref{ex. unit in Set CA}). For this reason they will be referred to as \textit{non-unital (monochrome) circuit algebras}.
	
	As in \cref{ex down Brauer}, let 
		$\mathrm O(\infty) \defeq \bigcup_{n \in \N} \mathrm O(n)$ denote the {infinite orthogonal group}. This is isomorphic to the the \textit{infinite symplectic group} $\mathrm{Sp}(\infty) \defeq \bigcup_{n \in \N} \mathrm{Sp}(n)$. Let $\mathrm{GL}(\infty) \defeq \bigcup_{n \in \N} \mathrm{GL}(n)$ denote the infinite general linear group. Since the free $\vect$-category on $\BDd$ is the downward Brauer category $\Brd$ (see \cref{ex: Brauer category}), the following corollary follows immediately from \cite{SS15},  \cref{ex: Brauer category} and \cref{cor. down WD}:
		
		\begin{cor}\label{cor nonunital Brauer}
			There is a canonical inclusion of the category $\mathsf{Rep}(\mathrm O(\infty))$ of (complex) representations of  $\mathrm O(\infty)$ -- which is {isomorphic } to the category $\mathsf{Rep} (\mathrm{Sp}(\infty))$ of representations of $\mathrm{Sp}(\infty)$ -- as the full subcategory of monochrome non-unital $\vect[\mathbb C]$-circuit algebras $\alg$ such that $A_0 \cong \Bbbk$, and $\boxtimes_{m,n}\colon A_m \otimes A_n \to A_ {m+n}$ is an isomorphism for all $m,  n \in \N$.
			
			Likewise, the category $\mathsf{Rep} (\mathrm{GL} (\infty))$ of representations of the infinite general linear group describes a full subcategory of the category of monochrome oriented non-unital $\vect[\mathbb C]$-circuit algebras. %
		\end{cor}

		\begin{ex}\label{ex. free ca}
															Let $S = (S(\ccc))_{\ccc \in \listm \CCC}$ be a collection of sets equipped with an action of $\Sigma$ such that, for each $\ccc$ 
															and $\sigma \in \Sigma_n$, there is an isomorphism $S(\ccc)\xrightarrow{\cong } S(\ccc \sigma)$. 
															The collection $(FS(\ccc))_{\ccc \in \listm \CCC}$ of $(\CCC,\omega)$-coloured wiring diagrams \textit{decorated by $S$} is defined by
															\[ \begin{array}{ll}
																FS(\ddd)& = \coprod_{(\ccc_1, \dots, \ccc_k) \in \listm^2 \CCC}\left( \CWD (\ccc_1, \dots, \ccc_k; \ddd) \times \prod_{i = 1}^k S(\ccc_i)\right)\\
																& = \coprod_{((\ccc\oplus\dots\oplus\ccc_k),(f, \lambda)))\in \CBD \ov \ddd} \left ( \prod_{i = 1}^k S(\ccc_i)\right)
															\end{array}.\]
															
															The set $FS$ underlies the \textit{free circuit algebra $\mathcal{F}S $ on $S$}:
															
															For each $(\overline f, \lambda )\in \CWD(\ccc_1, \dots, \ccc_k; \ddd)$, the morphism $\mathcal F S(\overline f, \lambda) \colon FS(\ccc_1)\times \dots \times FS(\ccc_k) \to FS(\ddd)$ is described by 
															\[ \prod_{i = 1}^k \left( (\overline f^i, \lambda^i ), (x_{j_i}^i)_{j _i= 1}^{m_i}\right) \mapsto \left(\gamma\left((\overline f, \lambda) , \left((\overline f^i, \lambda^i )_{i = 1}^k\right)\right), (x_{j_i}^i)_{{\overset{1 \leq j_i \leq m_i}{1 \leq i \leq k}}} \right).\]
													
														\end{ex}

														In general, circuit algebras, like operads (see \cite[Section~2.1]{GK94}), admit \textit{presentations} in terms of \textit{generators} and \textit{relations} (see \cite[Remark~2.6]{DHR21}). Indeed, every circuit algebra $\alg = (A, \alpha)$ is obtained as a quotient of the free circuit algebra $\mathcal F A$ on its underlying coloured collection $A$. Except to state the following familiar example, presentations of circuit algebras are not discussed in this work, and the interested reader is referred to \cite{DHR21} for more details.

														\begin{ex}\label{ex. tangles} Let $(\CCC, \omega) = \dipal$ and let $T = T(\ccc)_{\ccc \in \listm\{+.-\}} $ be the  $\listm\{+.-\}$-indexed set given such that $T(+,+,-,-)$ is the two element set of an \textit{over}- and {under-crossing} of two oriented intervals embedded in $\R^2 \times [0,1]$, and $T (\ccc) = \emptyset$ for $\ccc \neq (+,+, -,-)$. The circuit algebra $VT$ of virtual oriented tangles (see {\cite{DHR21}}) is the quotient of the free circuit algebra $\mathcal F T$ by relations induced 
															by the virtual Reidemeister identities (see \cite{Kau99}).

															More generally, we may consider circuit algebras of $(\CCC, \omega)$-coloured virtual tangles. 
															This includes, for example, circuit algebras of embedded tangles of mixed dimensions.

														\end{ex}

\section{Graphical species, modular operads, and circuit operads}\label{sec: definitions}

By \cref{prop. product and contraction prop}, a $(\CCC, \omega)$-coloured circuit algebra $\alg$ in a symmetric monoidal category $(\V, \otimes, I)$ has the structure of a $\listm \CCC$-graded monoid with contractions and distinguished morphisms $ \epsilon _c \in \V(I, \alg(c, \omega c))$ satisfying the conditions (c1)-(c3) and (e1) stated there. 

If $\V = \Set$, then the converse also holds (see \cref{thm. CO CA}). A $\listm \CCC$-graded set $A = (A(\ccc))_{\ccc}$, together with morphisms $ \epsilon _c\colon I \to \alg(c, \omega c)$, has the structure of a $(\CCC, \omega)$-coloured circuit algebra if and only if it describes a graded monoid  $(A, \boxtimes, I)$ in $\V$, and is equipped with a family of contraction operations that satisfy the conditions of \cref{prop. product and contraction prop}.

Circuit algebras in an arbitrary monoidal category describe enriched version of these structures.

\subsection{Graphical species}
\label{ssec: graphical species}

Graphical species provide a suitable notion of coloured collection  $(\alg({\ccc}))_{\ccc \in \listm \CCC}$  for describing various generalisations of (coloured) operads. They were introduced in \cite{JK11} to define coloured $\Set$-valued modular operads (called \textit{compact closed categories} in \cite{JK11}) in \cite{JK11, Ray20, HRY19a, HRY19b}. 

This section provides a short discussion on graphical species in an arbitrary category $\E$ with finite limits. For more details, the reader is referred to \cite{Ray20}, where graphical species in $\Set$ are discussed at length. 

Let $\fin \subset \Set$ denote the full subcategory of finite sets, and let $\fiso \subset \fin$ be the groupoid of finite sets and bijections. 

\begin{rmk}\label{rmk. species}
	(See also \cref{ssec. presheaves}.) 

	An $\E$-valued presheaf $P\colon \fiso^\mathrm{op} \to \E$, also called a \emph{(monochrome} or \emph{single-sorted) species in $\E$} \cite{Joy81}, determines a presheaf $\Sigma^{\mathrm{op}} \to \E$ by restriction. 
	
	Conversely, a $\Sigma$-presheaf $Q$ may always be extended to a $\fiso$-presheaf $Q_{\fiso}$, by setting
	\begin{equation} \label{eq. sigma to fiso}Q_{\fiso}(X) \defeq \mathrm{lim}_{(\nn,f) \in \Sigma \ov X} Q(\nn) \ \text{ for all } n \in \N.\end{equation}
\end{rmk}

Let the category $\fisinv$ be obtained from $\fiso $ by adjoining a distinguished object $\S$ that satisfies 
\begin{itemize}
	\item $\fisinv (\S, \S) = \{ id, \tau \} $, and $\tau^2 = id$;
	\item for each finite set $X$ and each element $x \in X$, there is a morphism $ch_x \in \fisinv (\S, X)$ that `chooses' $x$, and $\fisinv(\S, X)= \{ch_x\ ,\ ch_x \circ \tau \}_{x \in X}$;
	\item for all finite sets $X$ and $Y$, $\fisinv(X,Y) = \fiso (X,Y)$, and morphisms are equivariant with respect to the action of $\fiso$. That is, $ ch_{f(x)} = f \circ ch_x \in \fisinv (\S, Y)$ for all $ x \in X$ and all bijections $f\colon X \xrightarrow {\cong} Y$.
\end{itemize}
 
\begin{defn}\label{defn: graphical species}
	A \emph{graphical species $S$ in a category $\E$} is a presheaf $S\colon {\fisinv}^\mathrm{op} \to \E$, and $\GSE\defeq\prE{\fisinv}$ is the category of graphical species in $\E$. When $\E = \Set$, we write $\GS \defeq \GSE[\Set]$.

\end{defn}

Hence, a graphical species $S$ is described by an $\E$-valued species $(S_X)_{X \in \fiso}$ and an involutive object $(S_\S, S_\tau)$ in $\E$, 
together with a family of projections $S(ch_x)\colon S_X \to S_\S$, defined for each finite set $X$, and $x \in X$, and equivariant with respect to the action by $\fiso$. 

Henceforth, $\E$ will always denote a category with all  finite limits.
Let $*$ denote the terminal object of $\E$. 

\begin{ex}\label{Comm}
	The terminal graphical species $\CommE$ is the constant graphical species that sends $\S$ and all finite sets $X$ to the terminal object $*$ in $\E$.
	
\end{ex}

\begin{defn}\label{def: GS palette}

A morphism $ \gamma \in \GSE(S, S' )$ is \emph{colour-preserving} if its component at $\S$ is the identity on ${S_\S}$. 
A graphical species $S \in \GSE$ is called \textit{monochrome} if $S_\S = *$.
\end{defn}

\begin{defn}\label{c arity}
	The element category (\cref{defn: general element}) of a graphical species $S$ in $\Set$ is denoted by $\elG[S] \defeq \ElP{S}{\fisinv}$.
	
	Elements $c \in S_\S$ are called \emph{colours of $S$}. 
		For each $ \underline c = (c_x)_{x \in X} \in {S_\S}^{X}$, the \emph{$\underline c$-(coloured) arity $ {S_{\underline c}}$} is the fibre above $ \underline c \in {S_\S}^{X}$ of the map $(S(ch_x))_{x \in X}\colon {S_X} \to{S_\S}^{X}$.
	
\end{defn}

\begin{ex}\label{ex. terminal coloured species} 
	(Compare \cref{ex. directed Brauer diagrams}, \ref{ex. wheeled props operad}.)
{For any palette $(\CCC, \omega)$, the terminal $(\CCC, \omega)$-coloured graphical species $\CComm{}$ in $\Set$ is given by $\CComm{\underline c}= \{*\}$ for all $\underline c \in \CCC^X$ and all finite sets $X$. This is the terminal object of the category $\CGSE$ of $(\CCC, \omega)$-coloured graphical species in $\Set$, and palette-preserving morphisms.}
	
Recall from Examples \ref{ex. directed Brauer diagrams} and \ref{ex. wheeled props} the category $\DiBD $ 
of directed Brauer diagrams, and the corresponding operad of oriented wiring diagrams and consider the terminal $\dipal$-coloured 
graphical species $\dicomm$.
 
For all finite sets $X$, $\dicomm_{X}=\{+, -\}^X$ is the set of partitions $X =X_{+}\amalg  X_{-}$ of $X$ into \textit{input} and \textit{output} sets, and morphisms $ \dicomm_X \to \dicomm_Y$ are bijections $X \xrightarrow{\cong} Y$ that preserve the partitions. A \textit{monochrome directed graphical species} is a  $\dipal$-coloured graphical species.
	
\end{ex}

\subsection{Modular operads} \label{subs. MO and Comp}

This section generalises the definition of modular operads in $\Set$ from \cite[Section~1.3]{Ray20} to arbitrary categories with finite limits. 

Let $S $ be a graphical species in a category $\E$ with finite limits. For finite sets $X$ and $Y$ with elements $x \in X$ and $y \in Y$, let $(S_X \times S_Y)^{x \ddagger y}  \to S_X \times S_Y$ denote the pullback \[\xymatrix{(S_X \times S_Y)^{x \ddagger y} \ar[rr] \ar[d] && S_X \ar[d]^-{ S(ch_x)}\\
S_Y\ar[rr]_-{S(ch_y \circ \tau)}&& S_\S.}\]
More generally, given distinct elements $x_1, \dots, x_k \in X$ and $y_1, \dots, y_k \in Y$, let $(S_X \times S_Y)^{x_1 \ddagger y_1, \dots, x_k \ddagger y_k}$ be the limit of the collection $ S_X \xrightarrow{S(ch_{x_i})} S_\S \xleftarrow{S(ch_{y_i} \circ \tau)} S_Y$, $1 \leq i \leq k$. 
\begin{defn}\label{defn: multiplication}\label{coloured mult cont} A \emph{multiplication} $\diamond$ on $S$ is a family of morphisms 
\[	\diamond_{X,Y}^{x \ddagger y} \colon (S_X \times S_Y)^{x \ddagger y} \to S_{(X \amalg Y) \setminus \{x,y\}},\] in $\E$, defined for all pairs of finite sets $X$ and $Y$, with elements $x \in X$ and $y \in Y$, and such that 
\begin{enumerate}\item the obvious diagram \[
\xymatrix@R = .5cm{ (S_X \times S_Y)^{x \ddagger y} \ar[rrd]^-{\diamond_{X,Y}^{x \ddagger y}}\ar[dd]_{\cong}&&{}\\
	&&\quad  S_{(X \amalg Y)\setminus\{x,y\}}\\
(S_Y \times S_X)^{y \ddagger x} \ar[urr]_-{ \diamond_{Y,X}^{y \ddagger x}}&& {}}\]
commutes in $\E$,
\item $\diamond $ is equivariant with respect to the $\fisinv$-action on $S$: if $ \hat \sigma \colon X \setminus \{x\} \xrightarrow {\cong } W \setminus \{w\}$ and $ \hat \rho \colon Y\setminus \{y\} \xrightarrow {\cong } Z \setminus \{z\}$ are restrictions of bijections $  \sigma\colon X \xrightarrow{\cong} W$ and $ \rho\colon Y \xrightarrow{\cong} Z$ such that $ \sigma(x) = w$ and $ \rho (y) = z$, then 
		\[ S(\hat \sigma \sqcup  \hat \rho )\  \diamond_{W,Z}^{w \ddagger z} = \diamond_{X,Y}^{x \ddagger y} \ S(  \sigma \sqcup  \rho),\] 
		(where $ \sigma \sqcup \rho\colon X \amalg Y \xrightarrow{\cong} W \amalg Z$ is the blockwise permutation).
\end{enumerate}

	A multiplication $\diamond$ on $S$ is \emph{unital} if $S$ there is a unit-like morphism $\epsilon\in \E( S_\S , S_\two)$ 
	 such that, for all finite sets $X$ and all $x \in X$, the composite
\begin{equation}\label{eq. mult unit} \xymatrix{ S_X \ar[rr]^-{ (id \times ch_x)\circ \Delta}&& S_X \times S_\S \ar[rr] ^-{ id \times \epsilon} &&S_X \times S_\two \ar[rr]^-{\diamond_{X, \two}^{ x \ddagger 2}}&& S_X }\end{equation} is the identity on $S_X$. 
\end{defn}

 By (\ref{eq. mult unit}) and condition (1) of \cref{defn: multiplication}, if a multiplication $\diamond$ on $S$ admits a unit $\epsilon$, then it is unique.
 
 Let $S$ be a graphical species in $\E$, and let $x \neq y$ be distinct elements of a finite set $Z$. Let $(S_Z)^{x \ddagger y} = S_{Z}^{y \ddagger x}$ denote the equaliser 
 \begin{equation}
 	\label{eq. dagger pullback}\xymatrix{S_Z^{x \ddagger y} \ar[rr]&& S_Z \ar@<4pt>[rr]^-{S(ch_x)} \ar@<-4pt>[rr]_-{S(ch_y \circ \tau)} && S_\S.}
 \end{equation}

 More generally, for any morphism $ f \colon E \to S_Z$ in $\E$, let $E^{x \ddagger y}  = E^{y \ddagger x}$ denote the pullback of $f$ along the universal map $S_Z^{x \ddagger y}\to S_Z$.
 
 Given distinct elements $x_1, y_1, \dots, x_k, y_k$ of $S_Z$, and a morphism $ f \colon E \to Z$ in $\E$, then $ E^{x_1 \ddagger y_1, \dots ,x_k \ddagger y_k} \defeq (E^{x_1 \ddagger y_1, \dots ,x_{k-1} \ddagger y_{k-1}} )^{x_k \ddagger y_k}$ is the obvious limit.

 Invariance of $E^{x_1 \ddagger y_1, \dots ,x_k \ddagger y_k}$ under permutations $(x_i, y_i) \mapsto (x_{\sigma i}, y_{\sigma i})$, $\sigma \in Aut (\mathbf {k})$ follows from invariance of the defining morphisms  \[ \xymatrix {E^{x_1 \ddagger y_1, \dots ,x_k \ddagger y_k}\ar[rr]&& S^{x_1 \ddagger y_1, \dots ,x_k \ddagger y_k}_Z \ar[rr]&& S_{Z - \{x_1, y_1, \dots x_k, y_k\}.}}\]

 \begin{defn}\label{defn: contraction}
 	A \emph{(graphical species) contraction} $\zeta$ on $ S$ is a family of maps $\zeta^{x \ddagger y}_X \colon S_X^{x \ddagger y}\to S_{X \setminus \{x,y\}}$ defined for each finite set $X$ and pair of distinct elements $x,y \in X$. 
 	
 	The contraction $\zeta$ is equivariant with respect to the action of $\fisinv$ on $S$: If $\hat \sigma \colon X \setminus \{x,y\}\xrightarrow{\cong }Z \setminus\{w,z\}$ is the restriction of a bijection $ \sigma\colon X\xrightarrow{\cong} Z $ with $\sigma(x) = w$ and $\sigma(y) = z$. Then 
 	\begin{equation} \label{eq. contraction equi} S(\hat \sigma) \circ \zeta^{w\ddagger z}_Z = \zeta ^{x \ddagger y}_X \circ S( \sigma) \colon S_Z^{w \ddagger z} \to S_{X \setminus\{x,y\}}.\end{equation}

 \end{defn}

 In particular, by (\ref{eq. contraction equi}), if $\zeta$ is a contraction on $S$, then  $\zeta_X^{x\ddagger y}= \zeta_X ^{y\ddagger x}$ for all finite sets $X$ and all pairs of distinct elements $x, y \in X$.

\begin{defn}\label{def. mod op} 
	A modular operad $(S, \diamond, \zeta, \epsilon)$ is a graphical species equipped with unital multiplication $(\diamond, \epsilon)$ and a contraction $\zeta$ such that the following four coherence axioms are satisfied:	
	
	\emph{(M1)}	\emph{Multiplication is associative.}\\
For all finite set $X, Y, Z$ and elements $x \in X$, $z \in Z$, and distinct $y_1, y_2 \in Y$, 

\[
\xymatrix@C =.8cm@R =.4cm{
	(	S_{X} \times S_{Y} \times S_{Z} )^{x \ddagger y_1 , y_2 \ddagger z}
	\ar[rr]^-{\diamond_{X, Y}^{x \ddagger y_1}\times id_{S_Z}}
	\ar[dd]_{id_{S_X} \times \diamond_{Y,Z}^{y_2\ddagger z} } &&
	\left (S_{(X \amalg Y)\setminus \{x,y_1\}}\times S_{Z} \right)^{y_2 \ddagger z}
	\ar[dd]^{\diamond^{y_2 \ddagger z} }\\
	&{=}&\\
	\left(S_{X}\times  S_{(Y \amalg Z) \setminus \{y_2,z\}}\right)^{x \ddagger y_1}
	\ar[rr]_-{\diamond^{x\ddagger y_1}} &&
	S_{(X \amalg Y \amalg Z) \setminus \{x, y_1,y_2,z\}}. }
\]

	\emph{(M2)}	\emph{Contraction satisfies (C2).}\\

	\emph{(M3)}\emph{ Multiplication and contraction commute.} \\For finite sets $X$ and $Y$, mutually distinct elements $x_1, x_2$ and $x_3$ in $X$, and $y \in Y$, the following diagram commutes:
	\[
	\xymatrix@C =.8cm@R =.4cm{
		(	S_{X }^{x_1 \ddagger x_2} \times S_{Y})^{x_3 \ddagger y}
		\ar[rr]^-{\zeta^{x_1 \ddagger x_2} \times id}
		\ar[dd]_ {\diamond^{x_3\ddagger y}} &&
		(	S_{X \setminus\{x_1, x_2\}} \times S_{Y)} )^{x_3 \ddagger y}
		\ar[dd]^{\diamond^{x_3 \ddagger y}}\\
		&{=}&\\
		S_{(X \amalg Y)\setminus \{x_3,y\}}^{x_1 \ddagger x_2}
		\ar[rr]_-{\zeta^{x_1 \ddagger x_2}} &&
		S_{(X \amalg Y)\setminus \{x_1, x_2,x_3,y\}}.}\]

	\emph{(M4)} \emph{`Parallel multiplication' of pairs.}\\For finite sets $X$ and $Y$, and distinct elements $x_1, x_2 \in X$ and $y_1 , y_2 \in Y$, the following digram commutes:
	
	\[{
	\xymatrix@C =.8cm@R =.4cm{
		(	S_{X} \times S_{Y})^{x_1 \ddagger y_1, x_2 \ddagger y_2}
		\ar[rr]^-{ \diamond^{x_1 \ddagger y_1}}
		\ar[dd]_{ \diamond{x_2 \ddagger y_2}}&&
		S_{(X \amalg Y) \setminus \{x_1, y_1\}}^{x_2 \ddagger y_2}
		\ar[dd]^{\zeta^{x_2 \ddagger y_2}}\\&{=}&\\
		S_{(X \amalg Y) \setminus \{x_2, y_2\}}^{x_1 \ddagger y_1}\ar[rr]_-{\zeta^{x_1 \ddagger y_1}}&&
		S_{(X \amalg Y) \setminus \{x_1,x_2, y_1,y_2\}}.}}\]
	
\begin{figure}[htb!]
		\begin{tikzpicture}
	\node at (-4.5,3){	\includestandalone[width = .45\textwidth]{CO_subfiles/standalones/M1}};
	\node at (0,0){	\includestandalone[width = .45\textwidth]{CO_subfiles/standalones/M3}};
		\node at (0,6){	\includestandalone[width = .45\textwidth]{CO_subfiles/standalones/C2}};
	\node at (4.5,3){	\includestandalone[width = .45\textwidth]{CO_subfiles/standalones/M4}};
	\node at (-4.5,3){(M1)};
	\node at (0,0){(M3)};
	\node at (0,6){(M2)};
	\node at (4.5,3){(M4)};
	\end{tikzpicture}

	\caption{Modular operad axioms (M1)-(M4).}\label{fig. MO axioms}
\end{figure}
(Axioms (M1)-(M4) are illustrated in \cref{fig. MO axioms}.)

	The category of (non-unital) modular operads in $\E$ is denoted by ($\nuCSME$) $\CSME$.
\end{defn}

\subsection{Circuit Operads}\label{ssec CO}
A circuit operad in a category $\E$ with finite limits is a graphical species $A$ in $\E$ that has contraction and graded monoid structures analogous to circuit algebras 
(\cref{prop. product and contraction prop}). 

In this section, it is shown that circuit operads in $\E$ canonically admit a modular operad structure. This will be used in Sections \ref{s. graphs}-\ref{sec. iterated}, to construct the composite monad for circuit operads by modifying the modular operad monad of 
\cite{Ray20}.

\begin{defn} \label{def. external}
	An \emph{external product} on a graphical species $S$ in $\E$ is given by a $\fiso \times \fiso$-indexed collection of equivariant morphisms \[\boxtimes_{X,Y}: S_X \times S_Y \to S_{X \amalg Y},\] in $\E$ 
	such that, for all elements $x \in X$, the following diagram commutes in $\E$:
	\[ \xymatrix{S_X \times S_Y \ar[rr]^-{S(ch_x^{X \amalg Y}) \circ \boxtimes_{X,Y}}\ar[d] && S_\S \\
		S_X \ar[rru]_-{S(ch^X_x)} && }\]
	
	An \emph{external unit} for the external  product $\boxtimes$ is a distinguished morphism $\circleddash \colon * \to S_\nul$ such that the composite 
	\[ \xymatrix{ S_X \ar[r]^-{\cong} &S_X \times *  \ar[rr]^{id \times \circleddash }&& S_X \times S_\nul \ar[rr]^- {\boxtimes} && S_{(X \amalg \nul )} \ar@{=}[r] &S_X}\] is the identity on $S_X$ for all finite sets $X$.
	
\end{defn}

\begin{rmk} \label{rmk. monoidal unit} The monoidal unit $\circleddash $ for $\boxtimes$ is suppressed in the discussion (and notation) since most of the constructions in Sections \ref{sec: definitions}-\ref{sec. nerve} of this paper carry through without it. 
	To agree with circuit algebra conventions, and since most examples of structures related to circuit operads do involve a unital external product, external products in this paper are unital. See also \cref{rmk. no empty graph}, and \cite[Paragraph~4.7]{Koc18}.
	
\end{rmk}

An external product $\boxtimes$ and contraction $\zeta$ on a graphical species $S$ induce an equivariant multiplication $\diamond$ on $S$ given by, for all finite sets $X$ and $Y$, all $x \in X$ and $y \in Y$, the composition 
\begin{equation}\label{eq. forget mult}
	\diamond_{X,Y}^{x \ddagger y} \colon \xymatrix{(S_X\times S_Y)^{x \ddagger y} \ar[rr]^-{\boxtimes} && (S_X \boxtimes S_Y )^{x \ddagger y} \ar[rr]^-{\zeta^{x\ddagger y}}&& S_{(X \amalg Y)\setminus \{x,y\}}}.
\end{equation} 

Let $\sigma _\two$ denote the unique non-identity involution on $\two$.
\begin{defn}
	\label{defn: formal connected unit}
	A morphism $\epsilon  \in \E(S_\S, S_\two)$ is \emph{unit-like for $S$} if 
	\begin{equation}
		\label{eq: unit compatible with involution}	
		\epsilon \circ  \omega = S(\sigma_\two)  \circ \epsilon\colon S_\S \to S_\two, \text{ and } \ 
		S(ch_1)\circ \epsilon   = id_{S_{\S}}.
	\end{equation} 
	
\end{defn}
It follows immediately from the definition that if $\epsilon$ is unit-like for $S$, then $\epsilon$ is monomorphic.

\begin{defn}\label{def. co}
	A \emph{non-unital circuit operad} in $\E$ is a graphical species $S$ in $\E$, equipped with an external product $\boxtimes$ and a contraction $\zeta$, such that the following three axioms 
	are satisfied: 
	
	\emph{(C1)}	\emph{External product is associative.}\\
	For all finite sets $X, Y, Z$,
	the following square commutes:
	\[\xymatrix@C =.8cm@R =.4cm{
		S_{X} \times S_Y \times S_Z \ar[rr]^-{\boxtimes_{X,Y} \ \times\ id_{S_Z}}\ar[dd]_{id_{S_X}\ \times\ \boxtimes_{Y,Z}} && 
		S_{X \amalg Y} \times S_Z \ar[dd]^{\boxtimes_{(X \amalg Y), Z}}\\
		&{=}&\\
		S_{X} \times S_{Y \amalg Z} \ar[rr]_-{ \boxtimes_{X, (Y \amalg Z)}}&&		S_{X\amalg Y \amalg Z }}.\]

\emph{(C2)} \emph{Order of contraction does not matter.} \\For all finite sets $X$ with distinct elements $w,x,y,z$, the following square commutes:
\[\xymatrix@C =.8cm@R =.4cm{
	S_{X} ^{ w\ddagger x, y \ddagger z}\ar[rr]^-{\zeta^{w \ddagger x}}\ar[dd]_{\zeta^{y \ddagger z}} && 
	S_{X \setminus \{w,x\} }^{y \ddagger z}\ar[dd]^{\zeta^{y \ddagger z}}\\
	&{=}&\\\
	S_{X \setminus \{y,z\} }^{w \ddagger x}\ar[rr]_-{\zeta^{w \ddagger x}}&&	S_{X \setminus \{w,x,y,z\} }}.\]

\emph{(C3)} 
For all finite sets $X$ and $Y$, and distinct elements $x_1, x_2 \in X$, the following square commutes:
\[\xymatrix@C =.8cm@R =.4cm{
	S_{X}^{x_1 \ddagger x_2} \times S_Y \ar[rr]^-{\boxtimes }\ar[dd]_{\zeta^{x_1 \ddagger x_2}\ \times\  id_{S_Y}} && 
	(	S_{X}  \times S_Y)^{ x_1\ddagger  x_2 } \ar[dd]^{\zeta^{x_1 \ddagger x_2}}\\
	&{=}&\\
	S_{X \setminus \{x_1, x_2\} }\times S_Y\ar[rr]_-{\boxtimes}&&		S_{X \setminus \{x_1, x_2\} }\times S_Y .}\]

\begin{figure}[htb!]
	\begin{tikzpicture}
		\node at (-4.5,3){	\includestandalone[width = .45\textwidth]{CO_subfiles/standalones/C1}};
		\node at (4.5,3){	\includestandalone[width = .45\textwidth]{CO_subfiles/standalones/C3}};
		\node at (-4.5,3){(C1)};
		\node at (4.5,3){(C3)};
	\end{tikzpicture}
	\caption{Circuit operad axioms (C1) and (C2).}\label{fig. CO axioms}
\end{figure}
(Axioms (C1) and (C3) are illustrated in \cref{fig. CO axioms}. Axiom (C2) corresponds to (M2) and is illustrated in \cref{fig. MO axioms}.)

Let $\nuCO$ denote the category of non-unital circuit operads and morphisms of graphical species that preserve the external product and contraction. 

A \emph{(unital) circuit operad} is given by a non-unital circuit operad $(S, \boxtimes, \zeta)$ together with a unit-like morphism $\epsilon \colon S_\S \to S_\two $, such that for all finite sets $X$ and all $x \in X$, the composite morphism 
\begin{equation}
	\label{eq. conn unit} \xymatrix@C =.7cm{ S_X \ar[rr]^-{ (id \times ch_x)\circ \Delta}&& (S_X \times S_\S)^{x \ddagger 2} \ar[rr] ^-{ id \times \epsilon} &&(S_X \times S_\two)^{x \ddagger 2} \ar[rr]^-{\boxtimes}&&{ (S_X \boxtimes S_\two)^{x \ddagger 2} }\ar[rr]^-{ \zeta^{x \ddagger 2}_{X \amalg \two} }&& S_X }
\end{equation} is the identity on $S_X$. (Here $\Delta \colon S_X \to S_X \times S_X$ is the diagonal, and the last map makes use of the canonical isomorphism $( X \amalg \{1\})\setminus\{x\} \cong X$.)

Morphisms in the circuit operad category $\CO$ are morphisms in $\nuCO$ that respect the unit. 
\end{defn}

The following proposition, that follows immediately from \cref{def. co} and \cref{prop. product and contraction prop}, implies that $\Set$-valued circuit algebras describe circuit operads with the obvious product and contraction:
\begin{prop}	\label{prop. CA CO}
There is a canonical faithful functor $\CA \to \CO$. 
\end{prop}

In fact, we have the following proposition, that will follow from \cref{unpointed modular operad} and \cref{lem. LT is Tk}, together with the classical theory of distributive laws \cite{Bec69}:

\begin{prop}	\label{prop: CO to MO} There are canonical monadic adjunctions 
	\begin{equation}\label{eq: adjunction diag}\small
		{\xymatrix@C = .5cm@R = .3cm{
				\nuCOE \ar@<-5pt>[rr]\ar@<-5pt>[dd]&\scriptsize{\top}&\COE\ar@<-5pt>[ll]\ar@<-5pt>[dd]\\
				\vdash&&\vdash\\
				\nuCSME \ar@<-5pt>[uu] \ar@<-5pt>[rr]&\scriptsize{\top}& \CSME \ar@<-5pt>[ll]\ar@<-5pt>[uu]}}\end{equation} between the categories of (non-unital) modular and product modular operads in $\E$.
	
\end{prop}

The right adjoints of the pairs $\CSME \rightleftarrows \COE$ (and $\nuCSME \rightleftarrows \nuCOE$) in \cref{prop: CO to MO} are induced by the forgetful functor $(S, \boxtimes, \zeta)\mapsto (S, \diamond, \zeta)$ where the multiplication $\diamond$ is defined by suitable compositions $\zeta \circ \boxtimes$ as in (\ref{eq. forget mult}). (See also \cref{ex. unit in Set CA}.)

The left adjoint is induced by the free graded monoid monad. If $S$ is a graphical species, then there is a graphical species $LS$ defined by
\[ LS_X = \mathrm{colim}_{(Y, f) \in \core{X \ov \fin}} \coprod_{y \in Y} S_{f^{-1}(y)}\] where $\core{X \ov \fin}$ is the maximal subgroupoid of the slice category $X \ov \fin$ whose objects are morphisms of finite sets $f \colon X \to Y$, and whose morphisms $(Y,f) \to (Y', f')$ are isomorphisms $g \cong Y \to Y'$ such that $ gf = f'$.

The free monoid structure induces the external product on $LS$, and, if $(S, \diamond, \zeta)$ is a (non-unital) modular operad, then the modular operad axioms imply that the pair of operations $\zeta$ and $\diamond$, define a contraction on $LS$. (See \cref{sec. iterated}.)

In particular, since units for the modular operadic multiplication are unique:
\begin{cor}\label{cor. inclusion CO nuCO}
	The canonical forgetful functors $\COE \to \nuCOE$ and $\CSM \to \nuCSM$ are inclusions of categories.
\end{cor} 

In \cref{sec. iterated}, we shall see that the converse to \cref{prop. CA CO} is also true, and there is an equivalence of categories $\CA \simeq \CO$ (\cref{thm. CO CA}).

\begin{ex}\label{ex. wheeled prop CO}
(Recall Examples \ref{ex. directed Brauer diagrams}, \ref{ex. wheeled props} and \ref{ex. terminal coloured species}.)
By \cite{DHR20}, wheeled PROPs are equivalent to circuit algebras. Hence, by \cref{prop. CA CO}, wheeled PROPs in $\Set$ describe $\Set$-valued circuit operads.

The image of a wheeled PROP under the forgetful functor $\CO \to \CSM$ is its underlying \textit{wheeled properad} (see \cite{HRY15,JY15}).

\end{ex}

\newcommand{\defretract}[5]{\xymatrix{*[r]{#1} \ar@<1ex>[rrr]^-{#3} \ar@(ul,dl)[]_{#5} &&& #2 \ar@<1ex>[lll]^-{#4}}}

\section{Background on graphs}\label{s. graphs}

The combinatorics of modular operads have been fully described in \cite[Section~8]{Ray20} in terms of a category $\Klgr$ of graphs. The remainder of the present paper is concerned with an analogous construction for circuit operads.

This section reviews the definition of 
the category $\Gret$ of graphs and \'etale morphisms -- introduced in \cite{JK11}, and used in \cite{Koc18, HRY19a, HRY19b, Ray20} -- and describes how the graphs in this category are related to wiring diagrams. 
The interested reader is referred to \cite[Sections~3~\&~4]{Ray20} for more details and explicit proofs related to the graphical formalism.

\subsection{The \'etale category of graphs}\label{subs:feynman graphs}

Recall (from \cref{ex. Sigma}) that $\nn\defeq\{1,\dots, n\}$. Let $(m_1, \dots, m_k; n) \in \listm( \N) \times \N$.

Let $\overline f \in \WD (m_1, \dots, m_k; n)$ be a monochrome wiring diagram with underlying (open) Brauer diagram $f = (\tau_f,0) \in \BD(\sum_{i = 1}^k m_i, n)$ where $\tau_f$ is a pairing on the set $\partial f  = \nn \amalg (\coprod_{i = 1}^k \mm_i)$. 

Then $\overline f$ may equivalently be written as 
\begin{equation}\label{eq. wiring graph}
\Fgraphvar{\partial f}{\coprod_{i =1}^k \mm_i}{\mathbf{k}  }{\text{ inc}}{p}{\tau_{f}}
\end{equation} where $p \colon \coprod_{i =1}^k \mm_i \to \mm[k] $ is the canonical projection. 

By forgetting the orderings on $\mm[k]$ and each $\mm_i$, $1 \leq i \leq k$, we obtain the definition, originally due to Joyal and Kock \cite{JK11}, of a (Feynman) graph.

\begin{defn}\label{defn: graph}
	
	A \emph{graph} $\G$ is a diagram of finite sets
	\begin{equation} \label{eq. graph} \G = \Fgraph,\end{equation}
	such that $s\colon H \to E$ is injective, and $\tau\colon E\to E $ is an involution without fixed points.
	\end{defn}

Elements of $V$ are vertices of $\G$ and elements of $E$ are called \textit{edges} of $\G$. The set $\tilde E$ of $\tau $-orbits in $E$, where $\tilde e \in \tilde E$ is the orbit of an edge $e \in E$, is the set of connections in $\G$. Elements $h$ of the set $H$ -- of \textit{half-edges} of $\G$ -- will also be written as ordered pairs $h = (s(h), t(h)) \in E \times V$.

\begin{rmk}\label{geom real} By \cref{ex. manifold components}, a graph is a rule for gluing closed intervals along their boundaries. The geometric realisation of a graph $\G$ is the one-dimensional space $|\G|$, obtained by taking the discrete space $\{ *_v\}_{v \in V}$, and, for each $e \in E$, a copy $[0,\frac{1}{2}]_e$ of the interval $[0,\frac{1}{2}]$ subject to the identifications
	\begin{itemize}
		\item 
		$0_{s(h)}\sim *_{t(h)}$ for $h \in H$,
		\item 
		$(\frac{1}{2})_{e} \sim (\frac{1}{2})_{\tau e}$ for all $e \in E$. 
		
	\end{itemize}

\end{rmk}
\begin{defn}\label{defn: inner edges} \label{defn. stick comp}
	An \emph{inner edge} of $\G$ is an element $e \in E$ such that $\{e, \tau e\} \subset im(s)$. 
The set $\EI \subset E$ of {inner edges} of $\G$ is the maximal subset of $im(s) \subset E$ that is closed under $\tau$. The set of \emph{inner $\tau$-orbits} $\tilde e \in \widetilde E$ with $e \in \EI$ is denoted by $\widetilde {\EI}$. 
Elements of the set $E_0 = E\setminus im(s)$ are \emph{ports} of $\G$. 

A \emph{stick component} of a graph $\G$ is a pair $\{e, \tau e\}$ of edges of $\G$ such that $e$ and $\tau e$ are both ports.

A graph $\G$ that is obtained, as in (\ref{eq. wiring graph}), from a wiring diagram in $\WD$ is called \emph{ordered}.

	\end{defn}

\begin{ex}\label{stick} 
Recall that there are unique open 
Brauer diagrams $id_1 \in \BD (1,1)$, $ \cap  = \coev{id_1} \in \BD(0, 2)$ and $ \cup  = \ev{id_1} \in \BD(2, 0)$, and hence also wiring diagrams $\overline{id_1} \in \WD(1;1)$, $\overline \cap \in \WD(-;2)$ and $\overline \cup_{(\two; \nul)} \in \WD(2;0)$ and $\overline \cup _{(\one, \one; \nul)}\in WD(1,1;0)$. These correspond, by (\ref{eq. wiring graph}), these describe the following graphs, each of which has one edge orbit:
	\begin{itemize}
		\item for $\overline{id_1} \in \WD(1;1)$, the corresponding graph is the `1-corolla' 
			\[ \xymatrix{
			 \C_{\one}\defeq&\mathbf 2\ar@(lu,ld)[] &  \one \ar[l] \ar[r]&\one }\] with one vertex, no inner edges, and one port  (see \cref{fig. one orbit}(a));
			\item for $\overline \cap \in \WD(-;2)$, the corresponding graph is the \textit{stick graph} $(\shortmid)$ 
		\[ \xymatrix{
		( \shortmid)\defeq&\mathbf 2\ar@(lu,ld)[] & \mathbf 0 \ar[l] \ar[r]&\mathbf  0}\] with no vertices, no inner edges and two ports (see \cref{fig. one orbit}(b)); 
		\item for $\overline \cup_{(\two; \nul)} \in \WD(2;0)$, the corresponding graph is the \textit{wheel graph} $\W$ 
	\[ \xymatrix{
		\W \defeq&\mathbf 2\ar@(lu,ld)[] & \two  \ar[l]_{id} \ar@{->>}[r]&\mathbf 1}\] with one vertex, one inner edge orbit and no ports (see \cref{fig. one orbit}(c));
		\item the graph corresponding to $\overline \cup _{(\one, \one; \nul)}\in WD(1,1;0)$ has two vertices, one inner edge orbit and no ports (\cref{fig. one orbit}(d)):
	\[ \xymatrix{
&\mathbf 2\ar@(lu,ld)[] & \two  \ar[l]_{id} \ar[r]^{id}&\two}.\] 
	\end{itemize}

The stick and wheel graphs $(\shortmid)$ and $\W$ are particularly important in what follows.

	\begin{figure}
		[htb!]
		\begin{tikzpicture}[scale =.85]
		
			\node at (-1,2.5) {(a)};
		\draw (5,1)--(5,2);
		\node at (4.7,1.3) { \small{$2$}};
		\node at (4.7, 1.8) {\small{$ 1$}};

		\node at (3,2.5) {(b)};
		\draw (0,1)--(0,2);
	\node at (-.3,1.2) { \small{$1^\dagger$}};
		\node at (-.3, 1.9) {\small{$ 1$}};
			\filldraw 
		(0,1)circle (2pt);

		\node at (7,2.5) {(c)};
		\node at (9,1.5)
		{\begin{tikzpicture}
			\draw (0,0) circle (15pt);
		\filldraw (.53,0) circle (2pt);
		\end{tikzpicture}};	
			\node at (12,2.5) {(d)};
			\draw (14,1)--(14,2);
				\filldraw 
			(14,1)circle (2pt);
				\filldraw 
			(14,2)circle (2pt);

		\end{tikzpicture}
		\caption{(Realisations of) graphs with one edge orbit: (a) the 1-corolla $\C_\one$; (b) the stick graph $(\shortmid)$; (c) the wheel graph $\W$ consists of a single inner edge orbit with one end vertex; (d) a single inner edge orbit with distinct end vertices.} \label{fig. one orbit}
	\end{figure}

\end{ex}

For any set $ X$, let $X^\dagger \cong X$ denote its formal involution. 
\begin{ex}\label{corolla}(See also \cref{fig. one orbit}(d).) 
		The \emph{$X$-corolla} $\CX$ associated to a finite set $X$ has the form
	\[ \xymatrix{
		\CX: && *[r] {X \amalg X^\dagger }
		\ar@(lu,ld)[]_{\dagger}&& 
		X^\dagger \ar@{_{(}->}[ll]_-{\text{inc}}\ar[r]& \{*\}.
	}\]

\end{ex}

\label{Hv Ev} 
Let $\G$ be a graph with vertex and edge sets $V$ and $E$ respectively. For each vertex $v$, define $\vH \defeq t^{-1}(v) \subset H$ to be the fibre of $t$ at $v$, and let $\vE\defeq s(\vH) \subset E$.
\begin{defn}\label{defn: valency} 
	
	Edges in the set $\vE$ are said to be \emph{incident on $v$}.

	The map $|\cdot|\colon V \to \N$, $v \mapsto |v| \defeq |\vH|$, defines the \emph{valency} of $v$ and $\nV \subset V$ is the set of \emph{$n$-valent} vertices of $\G$. 
	
	A vertex $v$ is \emph{bivalent} if $|v| = 2$. An \emph{isolated vertex} of $\G$ is a vertex $v \in V(\G)$ such that $|v| = 0$. A \emph{bivalent graph} is a graph $\G$ with $V = \nV[2]$.

\end{defn}

Bivalent and isolated vertices are particularly important in \cref{ss. pointed}.

\label{Hn En} Vertex valency induces an $\N$-grading on the edge set $E$ (and half-edge set $H$) of $\G$: For $n \geq 1$, define $\nH \defeq t^{-1}(\nV)$ and $\nE \defeq s(\nH)$. Since $s(H) = E\setminus E_0 = \coprod_{n \geq 1} \nE$, 
\[E = \coprod_{n \in \N } \nE.\]

\begin{ex}
	The stick graph $(\shortmid)$ (\cref{stick}) has $E_0(\shortmid) = E(\shortmid) = \two$, and $\EI(\shortmid)   = \emptyset$. Conversely, the wheel graph $\W$ has $ E_0(\W) = \emptyset$ and $\EI(W) = E_2 (\W) = E(\W) \cong \two. $
	
	For all finite sets $X$, the $X$-corolla $\CX$ (see \cref{corolla}) with vertex $*$ has set of ports $E_0 (\CX) = X$ and $\vE[*]  = X^\dagger$. If $X  = \nn$ for some $n \in \N$, then $|*| = n$, so $ V = \nV$, and $E = \nE \amalg \nE[0]$ with $\nE[i] \cong X$ for $i = 0, n$.\end{ex}

Let $\diag$ be the small category $\label{grdiag}
\xymatrix{
	\bullet \ar@(lu,ld)[]_{}&\bullet \ar[l] _{}\ar[r]^-{}&\bullet}$. We are interested in the subcategory $\Gret$ of $\pr{\diag}$ whose objects are graphs and whose morphisms are described by: 
\begin{defn}
	\label{defn. etale morphism} 
	An \emph{\'etale morphism} $f \colon \G \to \G'$ of graphs is a commuting diagram of finite sets 
	\begin{equation}\label{morphism}
	\xymatrix{
		\G\ar[d]&&   E \ar@{<->}[rr]^-\tau\ar[d]&&E \ar[d]&& H \ar[ll]_s\ar[d] \ar[rr]^t&& V\ar[d]\\
		\G'&&  E' \ar@{<->}[rr]_-{\tau'}&&E'&& H' \ar[ll]^{s'} \ar[rr]_{t'}&& V' }\end{equation}
such that the right-hand square is a pullback. The category of graphs and \'etale morphisms is denoted by $\Gret$.
\end{defn}

It is straightforward to verify that the right hand square in (\ref{morphism}) is a pullback if and only if, for all $v \in V$, the map $\vE \to \sfrac{E'}{f(v)}$ induced by the restriction of the map $V\to V'$ is bijective. Hence, \'etale morphisms describe local isomorphisms, as the name suggests. 

\begin{ex}\label{ex: choose}
	For any graph $\G$ with edge set $E$, there is a canonical (up to unique isomorphism) bijection $\Gret(\shortmid, \G) \cong E$. The morphism $1 \mapsto e$ in $\Gret(\shortmid, \G)$ that \emph{chooses} an edge $e$ is denoted $ ch_e$, or $ch^\G_e$. 
	
\end{ex}

Pointwise disjoint union defines a symmetric strict monoidal structure on $\Gret$ with unit given by the empty graph $\oslash\defeq  \xymatrix{
{ \ \emptyset \ }
&
	\ \emptyset \ \ar[l]_-{}\ar[r]&\ \emptyset
}$. 

\begin{defn}\label{defn. conn graph}
	A graph is \emph{connected} if it cannot be written as a disjoint union of at least two non-empty graphs. A \textit{connected component} $\H$ of $\G$ is a maximal connected subdiagram of $\G$.
\end{defn}
In particular, 
any graph $\G$ is the disjoint union of its connected components, and the inclusion $\H \hookrightarrow \G$ of a connected component $\H$ of $\G$ describes a pointwise injective \'etale morphism. The empty graph $\oslash$ is not connected.

Disjoint union commutes with geometric realisation: for all graphs $\G$ and $\H$, $|\G \amalg \H| = |\G| \amalg |\H|$, and a graph $\G$ is connected if and only if it's geometric realisation $|\G|$ is a connected space.

\begin{ex}
	Let $X$ and $Y$ be finite sets. Write $\mathcal D_{X,Y} \defeq \CX \amalg \CX[Y]$ for the disjoint union of $X$ and $Y$-corollas:
	\[ \xymatrix{
		*[r] {(X\amalg Y) \amalg (X\amalg Y)^\dagger }
		\ar@(lu,ld)[]_{\dagger}&& &
		(X \amalg Y)^\dagger \ar@{_{(}->}[lll]_-{\text{inc}}\ar[r]^{t_{X,Y}}& \{*_X, *_Y\}
	}\] where the arrow $t_{X,Y}$ is the obvious projection $X^\dagger \to *_X$, $Y^\dagger \to *_Y$. 	
This has ports $X \amalg Y = E_0(\CX)\amalg E_0(\CX[Y])$ and no internal edges. The canonical inclusions $\iota_X \colon \CX \to \mathcal D_{X, Y}  \leftarrow \CX[Y] \colon \iota_Y$ are \'etale.

\end{ex}

\begin{ex}\label{Shrubs}
	Up to isomorphism, $(\shortmid)$ is the only connected graph with no vertices.
	As in \cite{Ray20} (and following \cite{Koc16}) a \textit{shrub} $\mathcal S $ is a disjoint union of stick graphs.  
	Any commuting diagram of the form (\ref{morphism}) where $\G = \mathcal S$ is a shrub trivially satisfies the pullback condition, and hence defines a morphism in $\Gret$.

\end{ex}

\begin{ex}
\label{deg loop}   

The pair of parallel morphisms $id, \tau\colon (\shortmid) \rightrightarrows (\shortmid)$ given by the endomorphisms of $(\shortmid)$ in $\Gret$ has coequaliser $* \leftarrow \emptyset \rightarrow \emptyset$ in $\pr{\diag}$. Since this is not a graph, the morphisms $id_{(\shortmid)}$ and $ \tau$ do not admit a coequaliser in $\Gret$. This provided the key motivating example for \cite{Ray20}.

The two endomorphisms $id_\W, \tau_\W \colon \W \to \W$, viewed as morphisms of presheaves on $\diag$, have coequaliser $* \leftarrow * \rightarrow *$ which is terminal in $\pr{\diag}$, and is not a graph. In particular, $\Gret$ does not have a terminal object and is therefore not closed under (finite) limits.

\end{ex}

As \cref{deg loop} shows, $\Gret$ does not, in general, admit finite colimits. However, if $\G $ is a graph and $ e_1, e_2 \in E_0$ are ports of $\G$ with $e_1 \neq \tau e_2$, then the colimit $\G^{e_1 \ddagger e_2}$ of the parallel morphisms \begin{equation}
\label{eq. dagger graph}
ch_{e_1}, ch_{e_2} \circ \tau \colon (\shortmid) \rightrightarrows \G
\end{equation} is the graph described by identifying $e_1 \sim \tau e_2$ and $e_2 \sim \tau e_1$ in the diagram (\ref{eq. graph}) defining $\G$.

\begin{ex}
	The graph $\C_\two^{e\ddagger \tau e}$ obtained by identifying the ports of the $\two$-corolla is isomorphic to the wheel graph $\W$. 
\end{ex}

\begin{ex}\label{ex: N graph} More generally, for $X$ a finite set with distinct elements $x$ and $y$, the graph $
	\C_X^{x \ddagger y}$ 
	has ports in $E_0 = X\setminus \{x,y\}$, one inner $\tau$-orbit $\{x,y\}$ (bold-face in \cref{fig: MN}), and one vertex $v$:
	\[\xymatrix{
		*[r] { 
		(	X \setminus \{x,y\}) \amalg X^\dagger}
		\ar@(lu,ld)[]_{\tau} &&&  
			X^\dagger \ar@{_{(}->}[lll]_-{s} \ar[r]^-{t}& \{v\} , }\]
		where $\tau x^\dagger = y^\dagger $ and $\tau z = z^\dagger$ for $z \in X \setminus \{x,y\}$. 
	Graphs of the form 	$\C_X^{x \ddagger y}$ encode formal contractions in graphical species.

\end{ex}

Since the monoidal structure on $\Gret$ is induced by the cocartesian monoidal structure on diagrams of finite sets, colimits in $\Gret$, when they exist, commute with $\amalg$. In particular, 
for any graphs $\G$ and $ \H$ and distinct ports $e_1, e_2$ of $\G$, 
\begin{equation}
\label{eq. C3 motivation} \G^{e_1 \ddagger e_2} \amalg \H = (\G \amalg \H)^{e_1 \ddagger e_2}.
\end{equation}

\begin{ex}\label{ex: M graph}
In particular, recall that $\mathcal D_{X,Y}$ is the disjoint union $\CX \amalg \CX[Y]$. Let $x \in X$, and $y \in Y$. By identifying the edges $x \sim \tau y, y \sim \tau x$, we obtain a graph $\mathcal D_{X,Y}^{x\ddagger y} \defeq \left(\CX \amalg \CX[Y]\right)^{x \ddagger y}$, that has two vertices and one inner edge orbit $\{x,y\}$, highlighted in bold-face in \cref{fig: MN}.
\[\xymatrix{
	*[r] { \left(\small
		{
			(X \amalg Y)\setminus \{x,y\} \amalg (X \amalg Y)^\dagger 
		}\right)
	}
	\ar@(lu,ld)[]_{\tau} &&&&
	{ \left(\small
		{ (X \amalg Y)^\dagger 
			}\right)
	}\ar@{_{(}->}[llll]_-{s} \ar[rr]^-{t}&& \small{\{v_X, v_Y\}}}\] with $s$ the obvious inclusion and the involution $\tau$ described by $ z \leftrightarrow z^\dagger$ for $z \in (X\amalg Y)\setminus\{x,y\} $ and $x^\dagger \leftrightarrow y^\dagger$. The map $t$ is described by $t^{-1}(v_X ) = \left(X\setminus\{x\}\right)^\dagger \amalg \{y\}$ and $t^{-1}(v_Y) = \left(Y\setminus\{y\}\right)^\dagger\amalg \{x\}$. 

In the construction of modular operads, graphs of the form $\mathcal D_{X,Y}^{x\ddagger y}$ are used to encode formal multiplications 
 in graphical species. \end{ex}

\begin{figure}[htb!] 
\begin{tikzpicture}

\node at(0,0){
	\begin{tikzpicture}{
		\node at (6,0) 
		{\begin{tikzpicture}[scale = 0.5]
			\filldraw(0,0) circle(3pt);
			\foreach \angle in {0,120,240} 
			{
				\draw(\angle:0cm) -- (\angle:1.5cm);

			}
			\draw [ ultra thick] (0,0) -- (1.5,0);
			\node at (.5, -0.3) {\tiny $x$};
			\node at (2.2, -0.3) {\tiny$ y$};

			\end{tikzpicture}};
		
		\node at (7,0){\begin{tikzpicture}[scale = 0.5]
			\filldraw(0,0) circle(3pt);
			\foreach \angle in {-0,90,180,270} 
			{
				\draw(\angle:0cm) -- (\angle:1.5cm);

			}
			\draw [ ultra thick] (0,0) -- (-1.5,0);

			\end{tikzpicture}};

		\node at (6.5,-2){ |$\mathcal D_{X,Y}^{x\ddagger y}$|};
}\end{tikzpicture}};

\node at (5, 0){	\begin{tikzpicture}

\node at (5,0) 
{\begin{tikzpicture}[scale = 0.5]
	\filldraw(0,0) circle(3pt);
	\draw (0,-1.5)--(0,1.5);
	\draw[ultra thick] 
	(0,0)..controls (3,3) and (-3,3)..(0,0);
	\node at (-1,.6) {\tiny $x$};
	\node at (1, .6) {\tiny $y$};
	\end{tikzpicture}};

\node at (5,-2){ |$\C_{X}^{x \ddagger y}$|};

\end{tikzpicture}
};
\end{tikzpicture}
\caption{ Realisations of $\mathcal D_{X,Y}^{x\ddagger y}$ and $\C_{X}^{x\ddagger y}$ for $X \cong \mathbf{4}, \ Y \cong \mm[3]$.} 
\label{fig: MN}
\end{figure}
\label{Hv Ev}    

In general, given ports $e_1, e_2$ of a graph $\G$ such that $e_1 \neq \tau e_2$, then $\G^{e_1 \ddagger e_2}$ is formed by `gluing' the ports $e_1$ and $e_2$ (see \cref{fig: MN}). As the following examples shows, the result of forming $\G^{e_1 \ddagger e_2 }$ is somewhat surprising when the ports $e_1 = \tau e_2$ are the edges of a stick component:

\begin{ex}\label{ex. stick colimit} 
	 Observe that, 
	 $ch_1^{(\shortmid)} \colon (\shortmid) \to (\shortmid)$, $1 \mapsto 1$ is the identity on $(\shortmid)$ and $ch_2^{(\shortmid)} \colon (\shortmid) \to (\shortmid)$, $1 \mapsto 2$ is precisely $\tau$. Hence 
		$ ch_2^{(\shortmid)} \circ \tau = \tau^2 = id_{\shortmid}, $
 and therefore \[ (\shortmid)^{1 \ddagger 2} = (\shortmid).\]
 
 It follows that, if a pair of ports $e_1, e_2 = \tau e_1 \in E_0(\G)$ form a stick component of a graph $\G$ (\cref{defn. stick comp}), then $\G^{e_1 \ddagger e_2} \cong \G$.
 
	\end{ex}

\begin{defn}
	\label{lem. mono}
A graph \emph{embedding} is a morphism $f \in \Gret (\G, \H)$ that factors 
as 
\[ \G \to \G^{e_1 \ddagger e'_1, \dots, e_k \dagger e'_k} \hookrightarrow \H \] where the first morphism describes a colimit $\G \to \G^{e_1 \ddagger e'_1, \dots, e_k \dagger e'_k}$ of pairs of parallel maps 
\[ch_{e_i}, ch_{e'_i} \circ \tau \colon (\shortmid) \rightrightarrows \G,\] (with $e_1, e'_1, \dots, e_k, e'_k$ distinct ports of $\G$), and the second morphism is an inclusion (pointwise injection) in $\Gret$. 
\end{defn}

Let $\G = \G' \amalg \mathcal S$, where $\mathcal S$ is a shrub and each connected component of $\G'$ has non-empty vertex set.  The following follows immediately from the definition of embeddings. 

\begin{lem}\label{lem. embedding}
A morphism $f \in \Gret (\G, \H)$ is an embedding if and only if:
\begin{itemize}
	\item the images $f(\G')$ and $f(\mathcal S)$ are disjoint in $\H$;
	\item $f$ is injective on vertices and half edges;
	\item the restriction of $f$ to $\mathcal S$ is injective.
\end{itemize}
Embeddings form a subcategory in $\Gret$.
\end{lem}
 
\begin{rmk}
The name `embedding' comes from \cite{HRY19a, HRY19b}. In \cite[Proposition~4.8]{Ray20}, I mistakenly identified the graph morphisms satisfying the conditions of \cref{lem. embedding} with the monomorphisms in $\Gret$, and used the term `monomorphism' to refer to `embeddings' throughout the paper. Whilst all monomorphisms are embeddings, the converse does not hold, as explained in \cite[Section~1.3]{HRY19a} (see, in particular, their Example 1.23). The last part of the given proof of \cite[Proposition~4.8]{Ray20} holds only when we consider parallel morphisms from a graph with no stick components.

As a consequence, statements about graph `monomorphisms' in \cite{Ray20} should actually refer to graph `embeddings'. The substance is unchanged. 

\end{rmk}

\begin{ex}\label{def Lk} 
	For $k \geq 0$, the \emph{line graph} $\Lk$ is the connected bivalent graph (illustrated in \cref{fig. line and wheel}) with ports in $ E_0 = \{1_{\Lk}, 2_{\Lk}\}$, and 
	\begin{itemize}
		\item ordered set of edges $E (\Lk)= (l_j)_{j = 0}^{2k+1}$ where $l_0 = 1_{\Lk} \in E_0$ and $l_{2k+1} = 2_{\Lk} \in E_0$, and the involution is given by $\tau (l_{2i}) = l_{2i +1}$, for $0 \leq i \leq k$, 
		\item ordered set of $k$ vertices $ V(\Lk) = (v_i)_{i = 1}^k $, such that $ \vE[v_i] = \{ l_{2i -1}, l_{2_i}\}$ for $1 \leq i \leq k$.
	\end{itemize}
	So, $\Lk$ is described by a diagram of the form 
	$\Fgraphvar{\ \two \amalg 2(\mathbf k)\ }{2(\mathbf k)}{\mathbf k.}{}{}{} $
	
	The line graphs $(\Lk)_{k \in \N}$ may be defined inductively by gluing copies of $\CX[\two]$:
	\[ \begin{array}{llll}
	\Lk[\nul]& = & (\shortmid) &\\
	\Lk[k+1] & \cong  & (\Lk \amalg \C_\two)^{ 2_{\Lk}\ddagger 1_{\C_\two}}, & k \geq 0.
	\end{array}\]
\end{ex}

\begin{ex}%
	
	\label{wheels}
For $m \geq 1$, the {wheel graph} $\Wm$ (illustrated in \cref{fig. line and wheel}) is the connected bivalent graph with no ports, obtained as the coequaliser in $\pr{\diag}$ of the morphisms $ch_{ 1_{\Lk[m]}}, ch_{2_{\Lk[m]}} \circ \tau \colon (\shortmid) \rightrightarrows \Lk[m]$. So, $\Wm$ has
	\begin{itemize}
		\item $2m$ cyclically ordered edges $E(\Wl) = (a_j)_{j = 1}^{2m}$, such that the involution satisfies $\tau (a_{2i}) = a_{2i +1}$ for $0 \leq i \leq m$ (where $a_0 = a_{2m}$),
		\item $ m$ {cyclically} ordered vertices $V(\Wl) = (v_i)_{i = 1}^m$, that $ \vE[v_i] = \{ a_{2i -1}, a_{2i }\}$ for $1 \leq i \leq m$, %
	\end{itemize} 
and is described by a diagram of the form $\Fgraphvar{\ 2(\mathbf m)\ }{2(\mathbf m)}{\mathbf m.}{}{}{}$

	\begin{figure}[htb!]
	
		\begin{tikzpicture}
		\node at (0,-1.5) {$\mathcal L^4$};
		\node at (0,0) { \begin{tikzpicture}[scale = 0.4]
			\draw (0,0) -- (10,0);
			\filldraw (2,0) circle (3pt);
			\filldraw (4,0) circle (3pt);
			\filldraw		(6,0) circle (3pt);
			\filldraw		(8,0) circle (3pt);
			\node at (0,-.7) {\tiny $l_0$};
			\node at (1.7,-.7) {\tiny $l_1$};
			\node at (10,-.7) {\tiny $l_9$};

			\end{tikzpicture}
		};
		\node at (6,-1.5) {$\mathcal W^4$};
		\node at(6,0){\begin{tikzpicture}[scale = 0.4]
			\draw (0,0) circle (2cm);
			\filldraw (2,0) circle (3pt);
			\filldraw (0,2) circle (3pt);
			\filldraw		(-2,0) circle (3pt);
			\filldraw		(0,-2) circle (3pt);
			
			\node at (-.6,2.2) {\tiny{$a_8$}};	
			\node at (.6,2.2) {\tiny{$a_0$}};
			\end{tikzpicture}};
		\node at (10,-1.5) {$\W = \Wl[1]$};
		\node at(10,0){\begin{tikzpicture}
			\draw (0,0) circle (15pt);
			\filldraw (.53,0) circle (1.5pt);
			
			\end{tikzpicture}};
		
	\end{tikzpicture}
	
	\caption{Line and wheel graphs.}
	\label{fig. line and wheel}
\end{figure}

\end{ex}

By \cite[Proposition~4.23]{Ray20}, a connected bivalent graph is isomorphic to $\Lk$ or $\Wl$ for some $k, m$. 

\subsection{The element category of a graph}

Intuitively, a graph may be thought of as a collection of corollas glued together along inner edges. This idea is made precise by the element category $\elG$ of a graph $\G$.

	Observe first that $\Gret(\shortmid, \shortmid) \cong \fisinv(\S, \S) $ canonically, so the maps $\S \mapsto (\shortmid)$ and $X \mapsto \CX$ describe a full inclusion of categories $\fisinv \hookrightarrow \Gret$. 
	The restriction of the Yoneda embedding to the image of $\fisinv$ in $\GS$ induces a functor $ \yet \colon \Gret \to \GS$, given by 
 \begin{equation}\label{eq. yet}
 \yet \G  \colon \C \longmapsto \Gret (\C, \G) \text{ for } \C \in \fisinv.
 \end{equation}
	

\begin{defn}
	\label{def. elements of a graph} The \emph{element category} $\elG$ of a graph $\G$ is the category of elements $\elG \defeq \elG[\yet \G] \cong \fisinv \ov \G$. 
\end{defn}

	Let $\G$ be a graph with edge set $E$. There is a maximal shrub $\mathcal S(E) \subset \G$:		\[ \xymatrix{
		     \mathcal S (E)\ar[d]_f&&   E \ar@{<->}[rr]^-\tau\ar@{=}[d]&&E \ar@{=}[d]&& \emptyset \ar[ll]\ar[d]\ar[rr]&& \emptyset \ar[d]\\
		       \G&&  E \ar@{<->}[rr]_-{\tau}&&E&& H' \ar[ll]^{s} \ar[rr]_{t}&& V}\]
	
whose connected components are indexed by the set of orbits $\tilde E$ of $\tau$ in $E$. For each orbit $\tilde e \in \tilde E$, let $\ese \colon (\shorte) \hookrightarrow \G$ be the inclusion of the stick graph $\ \xymatrix{
		\{e, \tau e\}\ar@(lu,ld)[] & \emptyset \ar[l] \ar[r]&\emptyset.} $ 

	Recall that, for each $v \in V$,  $\vE \defeq s(t^{-1}(v))$ is the set of edges incident on $v$. Let $\mathbf{v} = (\vE)^\dagger$ denote its formal involution. Then the corolla $\Cv$ is given by 
	\[ \Cv  =  \qquad  \xymatrix{	
		*[r] { \left(\vE \amalg ({\vE})^\dagger \right)}\ar@(lu,ld)[] && \vH \ar[ll]_-s \ar[r]^t&{\{v\}}. 
	}\] 
	
	The inclusion $\vE \hookrightarrow E$ induces a morphism $\esv^\G$ or $\esv \colon \Cv \to \G$. 
	Observe that, whenever there is an edge $e$ such that $e$ and $\tau e $ are both incident on the same vertex $v$ -- so $\{e, \tau e\}$ forms a \textit{loop} at $v$ -- then $\esv$ is not injective on edges.
	
	If $ \vE$ is empty -- so $\Cv$ is an isolated vertex -- then $ \Cv \hookrightarrow \G$ is a connected component of $\G$.
	
	\begin{defn}\label{def. neighbourhood v}
		A \emph{neighbourhood} of a vertex $v \in V$ is an embedding $u \colon \H \to \G$ such that $\esv \colon \Cv \to \G$ factors through $u$. 
	\end{defn}

\begin{defn}
\label{def: essential}
Let $\G$ be a graph. An \emph{essential morphism} (or \emph{essential subgraph of $\G$}) is a morphism of the form $\ese \colon (\shorte) \to \G$ ($\tilde e \in \widetilde E$) or $\esv \colon \Cv \to \G$ ($ v \in V$). 
The full subcategory $\esG \hookrightarrow \Gret \ov \G$ on the 
essential subgraphs is called the \emph{essential category of $\G$}. 

\end{defn}

 Non-identity morphisms in $\esG$ are in canonical bijection with half-edges of $\G$. 
 
 By definition, $\esG(\ese, \esv) \subset \Gret (\shorte, \Cv)$ for all edges $e$ and vertices $v$ of $\G$. If
 $h = (e, v) \in H$, then $e  \in E(\shorte) \cap E( \Cv)$, and $\esh \colon \ese \to  \esv$ denotes the (unique) morphism in $\esG$ that fixes $ e $. 
 
The following is proved in \cite[Lemma~4.17~\&~Lemma~4.25]{Ray20}:
\begin{lem}\label{lem: essential cover}
 There is an equivalence of categories $\elG \simeq \esG$ and, for all graphs $\G$,
\[ \G  = \mathrm{colim}_{ (\C, b) \in \esG} \C  =  \mathrm{colim}_{ (\C, b) \in \elG} \C,\] canonically. It follows that the embedding $\yet \colon \Gret \hookrightarrow \GS$ of (\ref{eq. yet}) is fully faithful.
\end{lem}

In particular, if $\G$ has no stick components, then, for each $\tilde e \in \widetilde{\EI}$, the essential morphism $\ese: (\shorte) \hookrightarrow \G$ factors in precisely two ways through the canonical morphism $\coprod_{v \in V} \Cv \to \G$ that restricts to $\esv$ on each $\Cv$:
\begin{equation}\label{eq: parallel maps}
\xymatrix{ (\shorte) \ar@<4pt>[rrr]^{(e, \tau e) \mapsto (e, e^\dagger)}  \ar@<-4pt>[rrr]_{(e, \tau e) \mapsto ((\tau e)^\dagger, \tau e)}  &&&\coprod_{v \in V} \Cv \ar[rr]&& \G.  } \end{equation}
Hence, there exist parallel morphisms $\coprod_{\tilde e \in \widetilde{\EI}} (\shorte) \rightrightarrows \coprod_{v \in V} \Cv$ such that the diagram
\begin{equation}\label{eq: general coequaliser}
\xymatrix{ \coprod_{\tilde e \in \widetilde{\EI}} (\shorte) \ar@<4pt>[rr]\ar@<-4pt>[rr]&&\coprod_{v \in V} \Cv \ar[rr] && \G  } \end{equation}
describes a coequaliser in $\Gret$. Of course, the pair of parallel morphisms in (\ref{eq: general coequaliser}) is not unique (there are $2^{|\widetilde{\EI}|-1}$ choices), but they are unique up to isomorphism.

\subsection{Evaluating graphical species on graphs}

By  \cite[Section~VII.2]{MM94}, the inclusion of presheaves $\GSE \hookrightarrow \prE{\Gret}$:
\begin{equation}\label{eq. S graph} S \longmapsto \left(\G \mapsto \mathrm{lim}_{ (\C, b) \in \elG} S(\C)\right),\end{equation}
 induced by the fully faithful inclusion $\fisinv \hookrightarrow \Gret$, is also fully faithful. Henceforth, the same notation will be used for graphical species viewed as presheaves on $\fisinv$ and for their image as presheaves on $\Gret$. 
 
 Let $\G$ be a graph. A family of morphisms $\mathfrak U = \{f_i \in \Gret (\G_i , \G)\}_{i \in I} $ is \emph{jointly surjective} on $\G$ if $\G = \bigcup_{i \in I} im (f_i)$. 
Families of jointly surjective morphisms on $\G$ 
 define the {covers at $\G$} for a canonical \emph{\'etale topology} {$J$} on $\Gret$. Since $\elG$ refines every cover at $\G$, the category $\GS$ of graphical species in $\Set$ is equivalent to the category $\sh{\Gret,J}$ of \'etale sheaves on $\Gret$. It follows, in particular, that 
for all graphical species $S $ in $\Set$, and all graphs $\G$, there is a canonical isomorphism 
$S(\G) \cong \GS (\yet \G, S)$ (see \cite[Section~4.4]{Ray20}) .

\begin{defn}\label{def. S structured graphs} If $S$ is a graphical species in $\Set$, an element $\alpha \in S(\G)$ is called an \emph{$S$-structure on $\G$}. 
Objects of the category $\Gret (S)$, of \emph{$S$-structured graphs}, are pairs $(\G, \alpha)$ with $\alpha \in S(\G)$, and morphisms $(\G, \alpha) \to (\G', \alpha')$ in $\Gret (S) $ are given by morphisms $f \in \Gret (\G, \G')$ such that $S(f) (\alpha') = \alpha$.	\end{defn}

\begin{ex}\label{ex. product union}
Let $S$ be a graphical species in $\E$. By (\ref{eq. S graph}),  \[ S(\G) \times S(\H) = S(\G \amalg \H),\]for  any graphs $\G$ and $\H$. In particular, for finite sets $X$ and $Y$, 
\begin{equation}
\label{eq. product union}S_X \times S_Y = S(\CX) \times S(\CX[Y]) = S(\CX \amalg \CX[Y]).
\end{equation}
\end{ex}

\begin{ex}\label{ex. N-structures}
	Let $X$ be a finite set with distinct elements $x$ and $y$. Recall from (\ref{eq. dagger pullback}) that $S_X^{x \ddagger y} \in \E$ is the equaliser of the morphisms $ S(ch_x), S(ch_y \circ \tau)  \colon S_X \to S_\S $. Since the graph $\C_X^{x \ddagger y}$ is defined as the coequaliser of the morphisms $ ch_x, ch_y \circ \tau  \colon (\shortmid) \to \CX$ (see \cref{ex: N graph} ), it follows from (\ref{eq. S graph}) that
		\begin{equation}
	\label{eq. N-structures}S^{x \ddagger y}_X = S(\C_X^{x \ddagger y}).	\end{equation}

So, a contraction $\zeta$ on a graphical species $S$ in $\E$ is defined by a collection of maps $S(\C^{x \ddagger y}_X) \to S_{ X \setminus \{x,y\}}$ where $X$ is a finite set with distinct elements $x, y \in X$.

\end{ex}

 \section{Non-unital circuit operads}\label{a free monad}\label{sec: non-unital}\label{nonunital csm section}
 Henceforth, it will always be assumed that the category $\E$ has sufficient limits and colimits for all given constructions.

This section gives a short outline, closely following \cite{JK11} and \cite[Section~5]{Ray20}, of the construction of the non-unital circuit operad monad $\TTk$ on $\GSE$.

\subsection{Gluing constructions and labelled graphs}\label{subs: gluing}

 In \cref{ssec non-unital CO}, the monad $\TTk$ whose algebras are non-unital circuit operads will be defined in terms of graphs. In \cref{sec. iterated}, \cref{thm. CO CA} will establish an equivalence between the categories $\CA$ and $\CO$ of circuit algebras and circuit operads in $\Set$. It follows that, since circuit algebras are algebras for operads of wiring diagrams (\cref{def CA}), there must be graph constructions that are analogous to operadic composition in $\WD$.

The following terminology is based on \cite{Koc16}:
\begin{defn} \label{graph of graphs}
	Let $\G$ be a graph. A \emph{$\G$-shaped graph of graphs} is a functor $ \Gg\colon \elG \to \Gret$ (or $\Gg^\G$) such that
	\[\begin{array}{ll}
	\Gg(a) = (\shortmid) & \text{ for all } (\shortmid, a) \in \elG, \\
	E_0(\Gg(b)) = X &  \text{ for all } (\CX, b) \in \elG,
	\end{array}\]
	and, for all $(\C_{X_b},b) \in \elG $ and all $ x \in X_b$,
	\[ \Gg( ch_x) = ch^{\Gg(b)}_x \in \Gret(\shortmid, \Gg(b)).\]
	\label{defn degenerate}
	A $\G$-shaped graph of graphs $ \Gg\colon \elG \to \Gret$ is \emph{non-degenerate} if, for all $  v \in V$, $\Gg ( \esv)$ has no stick components. Otherwise, $\Gg$ is called \emph{degenerate}.

\end{defn}

In particular, the empty $\C_\nul$-indexed graph of graphs $\C_\nul \mapsto \oslash$ is non-degenerate.

\begin{figure}[!htb]
	
	\includestandalone[width = .4\textwidth]{CO_subfiles/standalones/substitutionstandalone}
	\caption{A $\G$ shaped graph of graphs $\Gg$ describes \textit{graph substitution} in which each vertex $v$ of $\G$ is replaced by a graph $\G_v$ according to a bijection $E_0(\G_v)  \xrightarrow{\cong} (\vE)^\dagger$. When $\Gg$ is non-degenerate, taking its colimit corresponds to erasing the inner (blue) nesting. } \label{fig: graph nesting}
\end{figure}

Informally, a non-degenerate $\G$-shaped graph of graphs may be viewed as a rule for substituting graphs into vertices of $\G$ as in \cref{fig: graph nesting}.

\begin{prop}\label{colimit exists}
	
	A 
	$\G$-shaped graph of graphs $\Gg$ admits a colimit $\Gg(\G)$ in $\Gret$, and, for each $(\C,b) \in \elG$, the universal map $b_{\Gg}  \colon \Gg(b) \to \coGg$ is an embedding in $\Gret$.
	
	If $\Gg$ is non-degenerate and $\Gg(b) \neq \oslash$ for all $(\C, b) \in \elG$, then 
	\begin{itemize}
		\item$\Gg$ describes a bijection $ E(\coGg) \cong E(\G) \amalg \coprod_{v \in V} \EI (\Gg(\esv))$ of sets;
		\item $\Gg$ induces an identity $E_0(\G) \xrightarrow{=} E_0(\coGg)$ of ports;
		\item there is a canonical surjective map $V(\Gg(\G)) \twoheadrightarrow V(\G)$ induced by the embeddings  $\Gg(\esv)\colon \Gg(\esv) \to \Gg(\G)$ for each vertex $v $ of $\G$. 
	\end{itemize}

\end{prop}

\begin{proof}
	If $\Gg$ is non-degenerate, then $\Gg$ admits a colimit $\Gg(\G) $ in $\Gret$ by \cite[Proposition~5.16]{Ray20}. In case $\Gg$ is degenerate, the existence of a colimit in $\Gret$ follows from \cite[Proposition~7.15]{Ray20}.

	The other statements follow directly from \cite[Corollary~5.19]{Ray20}.
\end{proof}

\begin{rmk}
	\label{rmk. down Brauer graph of graphs} Let $\G$ be a graph with $n$ ports, $k$ vertices, and such that each vertex $v_i \in V$ has $m_i = |\vE[v_i]|$ adjacent edges. Up to ordering, $\G$ describes a wiring diagram $\overline g \in \WD (\mm_1, \dots, \mm_k; \nn)$ with no closed components, and a $\G$-shaped graph of graphs $\Gg$ such that $ \Gg(b) \neq \oslash$ for all $(\C,b) \in \elG$ describes a choice of $k$ wiring diagrams $ \overline f^i \in \WD (\bm l _1, \dots, \bm l_{k_i};\mm_i)$ for $( l _1, \dots,  l_{k_i}) \in \listm \N$ and $1 \leq i \leq k$.
	
	Non-degeneracy of $\Gg$ corresponds to the condition, discussed in \cref{rmk. pushout}, that, for each of the $ \overline f^i$, the underlying morphism $f^i $ of Brauer diagrams is a morphism of the downward Brauer category $ f^i  \in \BDd (l_1 + \dots + l_{k_i}, m_i)$. This means, informally, that $f^i$ does not contain caps.  If $\Gg$ is non-degenerate, then its colimit $\Gg(\G)$ is described by the composition $\gamma\left( \overline g ,(\overline f_i)_i \right)$ in $\WD$. (Colimits of degenerate graphs of graphs are discussed in \cref{ss. pointed}.) 
	
{	In both the non-degenerate and degenerate cases, the colimit of $\Gg$ is induced by the pushout diagram (\ref{eq. cospan pushout}). \cref{colimit exists} may also be proved using this construction. }
\end{rmk}

\begin{ex}\label{ex. id Gg}
	By \cref{lem: essential cover}, every graph $\G$ is the colimit of the (non-degenerate) \textit{identity $\G$-shaped graph of graphs} $\Gid$ given by the forgetful functor $\elG \to \Gret$, $(\C, b) \mapsto \C$. It follows from (5.9) in \cite[Section~5.2]{Ray20} that, if $\G$ has no stick components, this is equivalent to the statement that $\G$ is the coequaliser of the canonical diagram
	\begin{equation}\label{eq: graph data}
	\xymatrix{ \mathcal S(\EI) \ar@<4pt>[rr]\ar@<-4pt>[rr]&&\coprod_{v \in V} \Cv \ar[rr]^-{\bigcup (\esv)}&& \G.} \end{equation}
	
	The identity $\G$-shaped graph of graphs corresponds, as in \cref{rmk. down Brauer graph of graphs} to a composition of wiring diagrams of the form $\gamma \left(\overline g , (\overline{ id}_{m_i})_{i = 1}^k \right)$.

\end{ex}

\begin{defn}\label{defn: graphs of graphs cat}
	The category of \emph{non-degenerate $\G$-shaped graphs of graphs $\GretG$} is the full subcategory of non-degenerate $\G$-shaped graphs of graphs $\Gg: \elG \to \Gret$ in the category of functors $\elG \to \Gret$ and natural transformations. 
\end{defn}

\begin{defn}
	\label{def. X graph}
	
	Let $X $ be a finite set. A generalised \emph{$X$-graph} (henceforth $X$-graph) is a $\CX$-shaped graph of graphs. An $X$-graph is \emph{admissible} if it is non-degenerate.  
	
	The groupoid of \emph{(admissible) $X$-graphs} and \emph{$X$-isomorphisms} is the maximal subgroupoid $\core{\GretG[\CX]}\subset \GretG[\CX]$ and is denoted by $X\Grisok$. 
	
	More generally, for any graph $\G$ and any $(\C,b) \in \elG$, $[b]\Grisok$ denotes the groupoid of non-degenerate $\C$-shaped graphs of graphs and isomorphisms.
\end{defn}
	
Hence, a generalised $X$-graph is given by a pair $\X = (\G, \rho)$, where $\G $ is a graph and $\rho\colon E_0\xrightarrow{\cong} X$. An $X$-graph is \textit{admissible} if it has no stick components, and an $X$-isomorphism in $ X\Grisok (\X, \X')$ is an isomorphism $ g \in \Gr(\G, \G')$ that preserves the $X$-labelling: $\rho' \circ g_{E_0} = \rho\colon E_0 \to X.$

\begin{ex}\label{ex. Xv graph of graphs}
	For any graph $\G$, a (non-degenerate) $\G$-shaped graph of graphs $\Gg \colon \elG \to \Gret$ is precisely a choice of (admissible) $X_b$-graph $\X_b$ for each $(\CX[X_b], b) \in \elG$. 
\end{ex}

\begin{ex}\label{ex. external graphs}
	If $X$ and $Y$ are finite sets, then the disjoint union of corollas $\mathcal D_{X,Y}  = \CX \amalg \CX[Y]$ is an admissible $X \amalg Y$-graph. 
\end{ex}

Let $S$ be a graphical species in $\E$, and $\G$ a graph. If $\Gg \colon \elG \to \Gret$ is a non-degenerate $\G$-shaped graph of graphs with colimit $\Gg(\G)$ in $\Gret$, then the composition 
\[ \elG^{\mathrm{op}} \xrightarrow{\Gg^{\mathrm{op}}} \Gret^{\mathrm{op}}\xrightarrow {S} \E\]
defines a functor $ \elG^{\mathrm{op}} \to \E$ with limit $S(\Gg(\G))$.

\begin{ex} Recall \cref{def. S structured graphs}. If $S$ is a graphical species in $\Set$ and $\G$ is any graph, then a non-degenerate $\G$-shaped graph of graphs $\Gg$ and choice of $S$-structure $\alpha \in S(\Gg(\G))$ determines a functor $\Gg_S \colon \elG \to \Gret(S)$ by $b \mapsto S(b_{\Gg})(\alpha)$ (where, for all $(\C, b) \in \elG$, $b_{\Gg}$ is the universal map described in \cref{colimit exists}). In this case $\Gg_{S}$ is  called a \textit{$\G$-shaped graph of $S$-structured graphs}.
\end{ex}

\subsection{A monad for non-unital circuit operads}\label{ssec non-unital CO}

It is now straightforward to construct a monad 
on $\GSE$, whose algebras are non-unital circuit operads in $\E$.

Let $\Tk  = \Tk_{\E}\colon \GSE \to \GSE$ be the endofunctor defined, for all $S \in \GSE$, by 

\begin{equation}\label{free}
\begin{array}{llll}
\Tk S_\S &= & S_\S, &\\
\Tk S_X &= & \mathrm{colim}_{\X\in X{\Grisok}}  S(\X) & \text{ for all finite sets } X
\end{array}
	\end{equation}

 On isomorphisms in $\fisinv$, $\Tk S$ acts by relabelling graph ports. To see that $\Tk S$ extends to all morphisms in $\fisinv$ and therefore defines a graphical species in $\E$, let $g \colon \X' \to \X'$ be a port-preserving isomorphism in $X\Grisok$, and for each $x \in X$, let $ch_x^{\X} \in \Gr(\shortmid, \G)$ be the map $ch_{\rho^{-1}(x)}$ defined by $ 1 \mapsto \rho^{-1}(x) \in E_0 (\G)$. 
 Then, \[S(g \circ ch^\X_x) = S(ch^\X_x)S(g) = S(ch^{\X'}_x). \]  Hence, the projections $\Tk S(ch_x)\colon \Tk S_X \to S_\S$ induced by
$S({ch^{\X}_x}) \colon S(\X) \to S(\shortmid) = S_\S$, 
are well-defined. The assignment $S \mapsto \Tk S$ is clearly natural on $S \in \GSE$, and so $\Tk$ defines an endofunctor on $\GSE$.

\begin{lem}
	\label{lem. Tk mult}
	For any graphical species $S$ in $\E$ and any finite set $\X$, there is a canonical isomorphism
	\[({\Tk})^2 S_X \xrightarrow {\cong} \mathrm{colim}_{ \X \in X \Grisok} \left(\mathrm{lim}_{\Gg \in \Gret^{(\X)} }S (\Gg(\X))\right),\] where, as usual, 
	$\Gg(\X)$ is the colimit of the $\X$-shaped graph of graphs $\Gg$ in $\Gret^{(\X)}$.
\end{lem}
\begin{proof}
	
By definition,
 \[\begin{array}{lll}
({\Tk})^2 S_X&=&\mathrm{colim}_{ \X \in X \Grisok} \Tk S(\X)\\
& = &\mathrm{colim}_{ \X \in X \Grisok} \mathrm{lim}_{ (\C, b) \in \elG[\X]} \Tk S(\C)\\
& = &\mathrm{colim}_{ \X \in X \Grisok} \mathrm{lim}_{ (\C, b) \in \elG[\X]} \mathrm{colim}_{ \mathcal Y \in [b]\Grisok} S(\mathcal Y). \end{array}\]

Since $ [b]\Grisok$ is a groupoid by definition,  $\mathrm{colim}_{ \mathcal Y \in [b]\Grisok} S(\mathcal Y) =\mathrm{lim}_{ \mathcal Y \in [b]\Grisok} S(\mathcal Y)$, for all $(\C, b) \in \elG[\X]$. 
By \cref{defn: graphs of graphs cat}, \[ \mathrm{lim}_{\Gg \in \Gret^{(\X)} }S (\Gg(\X)) \cong \mathrm{lim}_{ (\C, b) \in \elG[\X]}\left( \mathrm{lim}_{ \mathcal Y \in [b]\Grisok} S(\mathcal Y)\right)\] canonically, and therefore, by the universal property of (co)limits, there is a unique isomorphism
 \begin{equation}\label{eq. T2 iso}\begin{array}{lll}
\Tk^2 S_X&
\cong & \mathrm{colim}_{ \X \in X \Grisok} \left(\mathrm{lim}_{ \Gg \in \Gret^{(\X)} } S(\Gg(\X))\right). 
\end{array}\end{equation}
\end{proof}

By \cref{colimit exists}, each $\Gg(\X)$ inherits the structure of an $X$-graph from $\X$. Therefore, using \cref{lem. Tk mult}, there is a morphism $\mu^{\TTk} S_X: ({\Tk})^2 S_X \to \Tk S_X$ in $\E$ given by 
\begin{equation}\label{eq. T mult} \ \Tk^2 S_X =  \mathrm{colim}_{ \X \in X \Grisok}\left( \mathrm{lim}_{ \Gg \in \Gret^{(\X)} } S(\Gg(\X))\right)\to \mathrm{colim}_{ \X' \in X \Grisok} S(\X') = \Tk S_X.\end{equation}

This is  natural in $X$ and in $S$ and hence defines a natural transformation $\mu^{\TTk}\colon \Tk^2 \Rightarrow \Tk$.

\label{eta defn} 
A unit for the endofunctor $\Tk$ on $\GSE$ is provided by the natural transformation $ \etak  \colon id_{\GS} \Rightarrow \Tk$ induced by the inclusion $\fisinv \hookrightarrow \Gret \colon S_X \xrightarrow = S(\CX) \to \Tk S_X$.

It is straightforward to verify that the triple $\TTk = (\Tk, \mu^{\TTk}, \eta^{\TTk})$ defines a monad.

\subsection{$\TTk$-algebras are non-unital circuit operads.}

As usual, let $\E$ be an arbitrary category with sufficient limits and colimits.

The purpose of this section is to establish an isomorphism of categories between the category $\GSE^{\TTk}$ of $\TTk$-algebras, and the category $\nuCOE$ of non-unital circuit operads.

As a first example, let us describe the circuit operad structure on $\Tk S $ where $S$ is an arbitrary graphical species in $\E$. 
\begin{ex}\label{ex. free circuit operad}

	Let $X$ and $Y$ be finite sets. Given an $X$-graph $\X$ and a $Y$-graph $\Y$, the graph $\X \amalg \Y$ has the structure of an $X \amalg Y$-graph and is therefore trivially the colimit of a $\C_{X \amalg Y}$-shaped graph of graphs. It is also, by definition, the colimit of a $\CX \amalg \CX[Y]$-shaped graph of graphs. 
	
	Disjoint union of graphs induces a external product $\boxtimes$ on $\Tk S$. Since $\oslash \in \nul \Grisok$, there is a morphism $ *  = S(\oslash) \to \Tk S_\nul$ from the terminal object $* $ in $\E$, and this provides an external unit (see \cref{def. external}) for $\boxtimes$. For all admissible $X$-graphs $\X$ and admissible $Y$-graphs $\Y$, the following diagram -- in which the unlabelled maps are the defining morphisms for $\Tk$ -- commutes in $\E$:
	\[\xymatrix{ S(\X \amalg \Y ) \ar[rrd] \ar@{=}[r]& S(\X) \times S(\Y) \ar[r]& \Tk S_X \times \Tk S_Y \ar@{=}[r] \ar[d]_-{\boxtimes} & \Tk S(\CX \amalg \C_Y) \ar[r] & \Tk^2 S_{X \amalg Y} \ar[dll]^-{\mu^{\TTk}} \\
		&& \Tk S_{X \amalg Y}.&&}\]

	Similarly, $\Tk S$ admits a canonical contraction $\zeta$ such that, 
for all finite sets $X$, distinct elements $x, y \in X$, and all admissible $X$-graphs $\X$, the following diagram commutes in $\E$:
\[\xymatrix{ S(\X^{x \ddagger y})\ar[rr] \ar[rrd] && \Tk S(\C^{x \ddagger y}_X) \ar[rr] \ar[d]_-{\zeta} && \Tk^2 S_{X\setminus \{x,y\}} \ar[dll]^-{\mu^{\TTk}} \\
	&& \Tk S_{X\setminus \{x,y\}.}&&}\]

It follows immediately from the definitions that $\boxtimes $ and $\zeta$ satisfy (C1)-(C3) of \cref{def. co}.

\end{ex}

\begin{prop}\label{unpointed modular operad} 
	There is a canonical isomorphism of categories $\GS^\TTk  \cong \nuCO$.
\end{prop}

\begin{proof} 
Let $(A,h)$ be an algebra for $\TTk$. 

For finite sets $X$ and $Y$, the morphism $\boxtimes_{X,Y} \colon A_X \times A_Y \to A_{X \amalg Y}$ in $\E$ is obtained by composing the morphism $ A_X \times A_Y  = A(\C_X \amalg \C_Y) \to {\Tk} A_{X \amalg Y}$ with ${h} \colon {\Tk} A_{X \amalg Y} \to A_{X \amalg Y}$. Moreover, the composition
\[ *  = A(\oslash) \to \Tk A_\nul \xrightarrow {h} A_\nul\] provides a unit for the external product $\boxtimes $ so defined.

And, if a finite set $X$ has distinct elements $x$ and $y$, then we may define a morphism $\zeta_{X}^{x \ddagger y} \colon A_X \to A_{X\setminus \{x,y\}}$  in $\E$ by composing 
$ A_X^{x\ddagger y}   = A(\mathcal{ N }_X^{x\ddagger y} )\to {\Tk} A_{X\setminus \{x,y\}}$ with $h$. 

The proof that $ (A, \boxtimes, \zeta)$ satisfies the axioms for a circuit operad proceeds by representing each side of the equations in (C1)-(C3) of \cref{def. co} in terms of graph of graphs, and then showing that these have the same colimit in $\Gret$. 

For (C1), the external product $\boxtimes_{X,Y}$ is represented by the graph $\C_X \amalg \C_Y$, and so \[\boxtimes_{X \amalg Y, Z} \circ  (\boxtimes_{X , Y}\  \times \  id_{A _Z})\] is represented by the $ \C_{X \amalg Y} \amalg \C_Z$-shaped graphs of graphs $\left (\iota_{X \amalg Y}\mapsto  \CX \amalg \CX[Y], \ \iota_Z \mapsto \CX[Z]\right )$, 
with colimit $\CX \amalg \CX[Y] \amalg \CX[Z]$. This is also the colimit of the $\CX \amalg \C_{Y \amalg Z}$-shaped graph of graphs $\left(\iota_X \mapsto \CX, \  \iota_{Y \amalg Z} \mapsto \CX[Y] \amalg \CX[Z]\right)$. 

Precisely, it follows from the monad algebra axioms that, for all finite sets $X, Y$ and $Z$, the diagram
\[\small{\xymatrix@C=.43cm@R = .7cm{
&&& A_{\scriptscriptstyle X} \times A_{\scriptscriptstyle Y }\times A_{\scriptscriptstyle Z}\ar[dl]_-{=} \ar[d]^{=}\ar[dr]^-{=}&&\\
&A_{\scriptscriptstyle {X \amalg Y}} \times A_{\scriptscriptstyle Z } \ar@{=}[d]&\ar@{-->}[l]_-{(\boxtimes,id)}A(\C_X \amalg \C_Y) \times A(\C_Z)\ar[r]^-{=}\ar[d]& A(\C_X \amalg \C_Y \amalg \C_Z)\ar[dd]& A(\C_X) \times A(\C_Y \amalg \C_Z)\ar[d] \ar[l]_-{=}\ar@{-->}[r]^-{(id, \boxtimes)}& A_{\scriptscriptstyle  X }\times A_{\scriptscriptstyle { Y \amalg Z}} \ar@{=}[d]\\
&A(\C_{\scriptstyle {X \amalg Y}} \amalg \C_Z)  \ar[d]&{\Tk} A (\C_{\scriptstyle {X \amalg Y}} \amalg \C_Z) \ar[l]^-{h}\ar[d]&& {\Tk} A(\CX \amalg \C_{\scriptstyle { Y \amalg Z}}) \ar[d]\ar[r]_-{h}\ar[d]& A (\C_X \amalg \C_{\scriptstyle {Y\amalg Z}}) \ar[d]\\
&{\Tk} A_{\scriptscriptstyle{ X \amalg Y \amalg Z}}\ar[rrdd]_-{h}& \ar[l]^{{\Tk} h}{\Tk} ^2A_{ \scriptscriptstyle{ X \amalg Y \amalg Z}} \ar[r]_-{ \mu^{\TTk} k A}& {\Tk} A_{\scriptscriptstyle {X \amalg Y \amalg Z}}\ar[dd]^-{h}& \ar[l]^-{ \mu^{\TTk} k A}{\Tk} ^2A_{ \scriptscriptstyle {X \amalg Y \amalg Z}}\ar[r]_{{\Tk} h} &  {\Tk} A_{\scriptscriptstyle {X \amalg Y \amalg Z}}\ar[ddll]^-{h}\\
&&&&&\\
&&& A_{X \amalg Y \amalg Z}, &&}}\] where the maps $ A_{X \amalg Y} \times A_Z \rightarrow  A_{X \amalg Y \amalg Z} \leftarrow A_X \times A_{ Y \amalg Z}$ are precisely $\boxtimes_{ X \amalg Y, Z}$ and  $\boxtimes_{X,Y \amalg Z}$, commutes.   
Hence, $\boxtimes $ satisfies (C1) of \cref{def. co}.

For (C2), if $w,x,y,z$ are mutually distinct elements of $X$, then, by the monad algebra axioms, the diagram 
	\[\small{\xymatrix@C=.43cm@R = .65cm{
		&& (A_{\scriptscriptstyle X})^{\scriptscriptstyle {w \ddagger x, y\ddagger z} } \ar[dl]_-{=} \ar[d]^{=}\ar[dr]^-{=}&&\\
	\left	(	A_{\scriptscriptstyle {X \setminus \{w,x\}} }\right )^{\scriptscriptstyle {y \ddagger z}} \ar@{=}[d]&\ar@{-->}[l]_-{\zeta^{\scriptscriptstyle {w \ddagger x}}}\left(A(\C_{\scriptscriptstyle X}^{ \scriptscriptstyle {w \ddagger x}})\right)^{\scriptscriptstyle {y \ddagger z}} \ar[r]^-{=}\ar[d]& A\left ((\C_X \amalg \C_Y) ^{\scriptscriptstyle {w \ddagger x, y \ddagger z}}\right )\ar[dd]& \left(A(\C_{\scriptscriptstyle X}^{ \scriptscriptstyle {y \ddagger z}})\right)^{\scriptscriptstyle {x \ddagger y}}\ar[d] \ar[l]_-{=}\ar@{-->}[r]^-{\zeta^{y \ddagger z}}& \left	(	A_{\scriptscriptstyle {X \setminus \{y,z\}}}\right )^{\scriptscriptstyle{ w \ddagger x}}\ar@{=}[d]\\
		A(\C_{\scriptscriptstyle {X \setminus \{w,x\}}}^{\scriptscriptstyle{ y \ddagger z}} ) \ar[d]&{\Tk} A (\C_{\scriptscriptstyle {X\setminus \{w,x\}}}^{\scriptscriptstyle {y \ddagger z}}) \ar[l]_-{h}\ar[d]&&{\Tk} A (\C_{\scriptscriptstyle {X\setminus \{y,z\}}}^{\scriptscriptstyle {w \ddagger x}}) \ar[d]\ar[r]^-{h}\ar[d]& 	A(\C_{\scriptscriptstyle {X \setminus \{y,z\}}}^{\scriptscriptstyle {w \ddagger x}})  \ar[d]\\
		{\Tk} A_{\scriptscriptstyle {X\setminus \{w,x,y,z\}}}\ar[rrdd]_-{h}& \ar[l]_{{\Tk} h}{\Tk} ^2A_{\scriptscriptstyle {X\setminus \{w,x,y,z\}}} \ar[r]^-{ \mu^{\TTk} k A}& {\Tk} A_{\scriptscriptstyle {X\setminus \{w,x,y,z\}}}\ar[dd]^-{h}& \ar[l]_-{ \mu^{\TTk} k A}{\Tk} ^2A_{\scriptscriptstyle {X\setminus \{w,x,y,z\}}}\ar[r]^{{\Tk} h} &  {\Tk} A_{\scriptscriptstyle {X\setminus \{w,x,y,z\}}}\ar[ddll]^-{h}\\
	&&&&\\
		&& A_{\scriptscriptstyle {X\setminus \{w,x,y,z\}}}&&}}\]
	commutes in $\E$, and hence $\zeta$ satisfies (C2).

Finally, if $X$ and $Y$ are finite sets and $x_1$ and $x_2$ are distinct elements of $X$, then the diagram
	\[\scriptsize{\xymatrix@C=.4cm@R = .7cm{
			&& (A_{\scriptscriptstyle {X} }\times A_{\scriptscriptstyle {Y}})^{\scriptscriptstyle{ x_1 \ddagger x_2} }\ar[dl]_-{=}\ar[dr]^-{=}&&\\
			\left	(	A_{\scriptscriptstyle{ X \amalg Y}} \right )^{\scriptscriptstyle {x_1 \ddagger x_2}} \ar@{=}[d]&\ar@{-->}[l]_-{\boxtimes}\left(A(\C_X \amalg \C_Y)\right)^{\scriptscriptstyle {x_1\ddagger x_2}} \ar[ddr]\ar@{=}[rr]\ar[d]&& A(\C_X^{\scriptscriptstyle { x_1 \ddagger x_2}})\times A_Y \ar[ddl]\ar[d] \ar@{-->}[r]^-{\zeta^{x_1 \ddagger x_2}}& 	A_{\scriptscriptstyle {X \setminus \{x_1,x_2\}}}\times A_Y\ar@{=}[d]\\
			A(\C_{\scriptstyle {X \amalg Y}}^{\scriptscriptstyle {x_1 \ddagger x_2}} ) \ar[d]
			&{\Tk} 	A(\C_{\scriptstyle{ X \amalg Y}}^{\scriptscriptstyle {x_1 \ddagger x_2}} ) \ar[l]_-{h}\ar[d]&&{\Tk} A (\C_{\scriptstyle {X\setminus \{x_1, x_2\}}}\amalg \C_Y) \ar[d]\ar[r]^-{h}\ar[d]& 	A (\C_{\scriptstyle {X\setminus \{x_1, x_2\}}}\amalg \C_Y)   \ar[d] \\
						{\Tk} A_{\scriptscriptstyle {(X \amalg Y)\setminus \{x_1, x_2\}}}\ar[rrdd]_-{h}& \ar[l]_-{{\Tk} h} {\Tk} ^2A_{\scriptscriptstyle {(X \amalg Y)\setminus \{x_1, x_2\}}}  \ar[r]^-{ \mu^{\TTk} k A}& {\Tk} A_{\scriptscriptstyle {(X \amalg Y)\setminus \{x_1, x_2\}}}  \ar[dd]^{h}& {\Tk} ^2A_{\scriptscriptstyle{ (X \amalg Y)\setminus \{x_1, x_2\}}}  \ar[r]^-{{\Tk} h} \ar[l]_-{ \mu^{\TTk} k A} &  {\Tk} A_{\scriptscriptstyle {(X \amalg Y)\setminus \{x_1, x_2\}}} \ar[ddll]^-{h} \\
			&&&&\\
			&& A_{\scriptscriptstyle{ (X \amalg Y)\setminus \{x_1, x_2\}}} &&}}\] commutes.
		Therefore, (C3) is satisfied and $(A, \boxtimes, \zeta)$ is a non-unital circuit operad. 
		
		Conversely, let $(S, \boxtimes, \zeta)$ be a non-unital circuit operad. For finite sets $X_1, \dots, X_n$, let $\boxtimes_{ X_1, \dots, X_n}: \prod_{i = 1}^n S_{X_i} \to S_{ \coprod_i X_i}$ be the obvious morphism induced by $\boxtimes$. This is well defined since $ (S, \boxtimes, \zeta)$ satisfies (C1)-(C3). Recall from \cref{ssec CO} that, if $x_1, y_1, \dots, x_k, y_k$ are mutually distinct elements of $\coprod_i X_i$, then $(\prod_i S_{X_i})^{x_1 \ddagger y_1, \dots, x_k \ddagger y_k }$ is the limit of the diagram of parallel morphisms 
		\[\boxtimes_{ X_1, \dots, X_n} \circ S(ch_{x_i}), \boxtimes_{ X_1, \dots, X_n} \circ S(ch_{y_i}\circ \tau) \colon \prod_{i = 1}^n S_{X_i} \to S_{ \coprod_i X_i} \to S_\S, \] and that this is independent of the ordering of the pairs $(x_i, y_i)$.

	Let $X$ be a finite set and let $\X$ be an admissible 
		$X$-graph. For each $v \in V$, let $(\CX[Z_v], b_v) \in \elG[\X]$ be a {neighbourhood} of $v$, and let $Z_v \cong |v|$ be such that $\coprod_v Z_v = X \amalg \{y_1, z_1,\dots, y_k, z_k\}$.
		has no stick components, it follows from (\ref{eq: general coequaliser}) that 
		\begin{equation}\label{eq. SX resolution}
		S(\X) = \left(\prod_{v \in V(\X)} S_{Z_v}\right )^{y_1 \ddagger z_1, \dots, y_k \ddagger z_k} .
		\end{equation} 

\begin{figure}[htb!]
	 \begin{tikzpicture} [scale =  0.9]
	
	\node at(0,3){
		\begin{tikzpicture}[scale = 0.35]\begin{pgfonlayer}{above}
		\node [dot, red] (0) at (0, 0) {};
		\node  [dot, red](1) at (2, 0) {};
		\node [dot, red] (2) at (-2, 1) {};
		\node [dot, red] (3) at (-1, -1) {};
		\node  (4) at (-4.5, 2.25) {};
		\node  (5) at (-3, -3) {};
		\node  (6) at (0, -4) {};
		\node  (7) at (4, 2) {};
		\node  (8) at (5, -1) {};
		\node  (9) at (5, -3) {};
		\node  (10) at (-1, 2) {};
		
		\end{pgfonlayer}
		\begin{pgfonlayer}{background}
		\draw [bend left=45, looseness=1.75] (0.center) to (1.center);
		\draw [bend right=60, looseness=1.25] (2.center) to (3.center);
		\draw [in=165, out=15, looseness=1.25] (4.center) to (2.center);
		\draw [in=60, out=-105, looseness=1.25] (3.center) to (5.center);
		\draw [bend left=15] (3.center) to (6.center);
		\draw [bend left=45] (1.center) to (7.center);
		\draw [in=105, out=-30] (1.center) to (8.center);
		\draw [in=-180, out=-45, looseness=1.25] (1.center) to (9.center);
		\draw [in=135, out=90, looseness=1.50] (2.center) to (10.center);
		\draw [in=315, out=-15] (2.center) to (10.center);
		\draw [bend right=60, looseness=1.50] (0.center) to (1.center);
		\end{pgfonlayer}\end{tikzpicture}};
	
	\node at (-6.5,0){\begin{tikzpicture}[scale = 0.2]\begin{pgfonlayer}{above}
		\node [dot, red] (0) at (-0.25, 0) {};
		\node [dot, red] (1) at (1, 0) {};
		\node [dot, red] (2) at (-1.5, 0.5) {};
		\node  [dot, red](3) at (0, -1) {};
		\node  (4) at (-7.5, 3) {};
		\node  (5) at (-3.5, -4.75) {};
		\node  (6) at (-0.5, -5.75) {};
		\node  (7) at (7, 2) {};
		\node  (8) at (7, -1.75) {};
		\node  (9) at (6.75, -4) {};
		\node  (10) at (-5, 6.5) {};
		\draw[draw = blue, fill = blue, fill opacity = .2] (-.2,.1) circle (2cm);
		\end{pgfonlayer}
		\begin{pgfonlayer}{background}
		\draw [bend left=120, looseness=22.00] (0.center) to (1.center);
		\draw [bend right=60, looseness=1.25] (2.center) to (3.center);
		\draw [in=165, out=15, looseness=1.25] (4.center) to (2.center);
		\draw [in=60, out=-105, looseness=1.25] (3.center) to (5.center);
		\draw [bend left=15] (3.center) to (6.center);
		\draw [bend left=45] (1.center) to (7.center);
		\draw [in=105, out=-30] (1.center) to (8.center);
		\draw [in=-180, out=-45, looseness=1.25] (1.center) to (9.center);
		\draw [in=-180, out=90, looseness=0.75] (2.center) to (10.center);
		\draw [in=0, out=75, looseness=0.50] (2.center) to (10.center);
		\draw [bend left=105, looseness=16.25] (0.center) to (1.center);
		\end{pgfonlayer}\end{tikzpicture}};
	\draw  [->, dashed,  line width = 1.5, draw= gray](-5,0)--(-2,0);
	\node at (-3.5,.5){\tiny{apply $\boxtimes$ locally at vertices}
	};
	
	\node at (0,0){ 
		\begin{tikzpicture}[scale = 0.2]\begin{pgfonlayer}{above}
		\node [dot,blue] (0) at (0, 0) {};
		\node  (1) at (0, 0) {};
		\node  (2) at (0, 0) {};
		\node  (3) at (0, 0) {};
		\node  (4) at (-7, 0.25) {};
		\node  (5) at (-3, -4.25) {};
		\node  (6) at (0.25, -5) {};
		\node  (7) at (6.25, 3.25) {};
		\node  (8) at (5.25, -1.75) {};
		\node  (9) at (5, -4) {};
		\node  (10) at (-2.25, 1) {};
		\node  (11) at (0, 2.25) {};
		\node  (12) at (0, 1.5) {};
		\draw[draw =green, fill = green, fill opacity = .2] (-.2,.1) circle (2.5cm);
		\end{pgfonlayer}
		\begin{pgfonlayer}{background}
		\draw [bend left=135, looseness=12.00] (0.center) to (1.center);
		\draw [bend right=60, looseness=1.25] (2.center) to (3.center);
		\draw [in=-150, out=15, looseness=1.25] (4.center) to (2.center);
		\draw [in=60, out=-105, looseness=1.25] (3.center) to (5.center);
		\draw [bend left=15] (3.center) to (6.center);
		\draw [in=150, out=15] (1.center) to (7.center);
		\draw [in=105, out=-30] (1.center) to (8.center);
		\draw [in=-180, out=-45, looseness=1.25] (1.center) to (9.center);
		\draw [in=-135, out=-165] (2.center) to (10.center);
		\draw [in=45, out=75, looseness=0.50] (2.center) to (10.center);
		\draw [bend left=120, looseness=7.25] (0.center) to (1.center);
		\draw [in=135, out=165, looseness=0.75] (12.center) to (3.center);
		\draw [in=60, out=-15, looseness=0.75] (12.center) to (3.center);
		\draw [in=150, out=180] (11.center) to (3.center);
		\draw [in=30, out=0] (11.center) to (3.center);
		\end{pgfonlayer}\end{tikzpicture}};
	\draw  [->, dashed, line width = 1.5, draw= gray](2,0)--(5,0);
	\filldraw[white](6,-4) circle (30pt);
	\node at (3.5,.5){\tiny{apply $\zeta$ on inner edges}
	};

	\node at (6.5,0){ 
		\begin{tikzpicture}[scale = 0.2]	\begin{pgfonlayer}{above}
		\node [dot, green] (0) at (0, 0) {};
		\node  (1) at (0, 0) {};
		\node  (2) at (0, 0) {};
		\node  (3) at (0, 0) {};
		\node  (4) at (-7, 0.25) {};
		\node  (5) at (-3, -4.25) {};
		\node  (6) at (0.25, -5) {};
		\node  (7) at (6.25, 3.25) {};
		\node  (8) at (5.25, -1.75) {};
		\node  (9) at (5, -4) {};
		\end{pgfonlayer}
		\begin{pgfonlayer}{background}
		\draw [bend right=60, looseness=1.25] (2.center) to (3.center);
		\draw [in=-150, out=15, looseness=1.25] (4.center) to (2.center);
		\draw [in=60, out=-105, looseness=1.25] (3.center) to (5.center);
		\draw [in=75, out=-75] (3.center) to (6.center);
		\draw [in=150, out=15] (1.center) to (7.center);
		\draw [in=105, out=-30] (1.center) to (8.center);
		\draw [in=-180, out=-45, looseness=1.25] (1.center) to (9.center);
		\end{pgfonlayer}\end{tikzpicture}};\end{tikzpicture}
\caption{$S(\X) \to \Tk S_X$ and (below) the corresponding structure map $h_\X \colon S(\X)\to S_X$.} \label{fig. COnu structure map}
\end{figure}

In particular, the morphism $\prod_v S_{Z_v} \to S_{\coprod_v Z_v}$ induces a morphism \[\boxtimes^V \colon S(\X) \to S_{ \coprod_{v} Z_v}^{y_1 \ddagger z_1, \dots, y_k \ddagger z_k}. \]

Let $h_\X \colon S(\X) \to S(\C_X)  = S_X$ be the map (illustrated in \cref{fig. COnu structure map}) defined, using (\ref{eq. SX resolution}), by the composition 
\begin{equation}
\label{eq. prod eq} \xymatrix{ S(\X) \ar[rr]^-{\boxtimes^{V}} && 
	S_{ \coprod_{v} Z_v}^{y_1 \ddagger z_1, \dots, y_k \ddagger z_k}\ar[rr]^- {(\zeta)^{k-1}}&&S_X}
\end{equation}

where $(\zeta)^{k-1}$ is obtained by the obvious iterated contractions. Since $(S, \boxtimes, \zeta)$ satisfies (C1) and (C2), these are independent of choices of ordering.

The maps $h_\X$ so defined are the components of a morphism $h \colon \Tk S \to S$ in $\GSE$.

Moreover, for any $\X$-shaped graph of graphs $\Gg$, $\Gg(b_v)$ is a $Z_v$-graph. So, by (C3), we obtain morphisms $h_{ \Gg(b_v)} \colon S(\Gg(b_v)) \to S_{Z_v}$, and by iterating (\ref{eq. prod eq}), it is straightforward to show that $h \colon \Tk S \to S$ satisfies the axioms for a monad morphism. (\cref{fig. COnu h alg} for coherence of $h$ and $\mu^{\TTk}$.)
\end{proof}

\begin{figure}[htb!]
\begin{tikzpicture} [scale = .9]
\node at(-5,0){ 
	\begin{tikzpicture}[scale = .4]
	\begin{pgfonlayer}{above}
	\node [dot, red] (0) at (0, 0) {};
	\node  [dot, red](1) at (2, 0) {};
	\node [dot, red] (2) at (-2, 1) {};
	\node [dot, red] (3) at (-1, -1) {};
	\node  (4) at (-4.5, 2.25) {};
	\node  (5) at (-3, -3) {};
	\node  (6) at (0, -4) {};
	\node  (7) at (4, 2) {};
	\node  (8) at (5, -1) {};
	\node  (9) at (5, -3) {};
	\node  (10) at (-1, 2) {};
	\draw[draw = blue, fill = blue, fill opacity =.2] (-1.6,1.2) circle (1.2cm);
	\draw[draw = blue, fill = blue, fill opacity =.2] (-1,-1) circle (.6cm);
	\draw[draw = blue, fill = blue, fill opacity =.2] (1,0)ellipse (1.7cm and .9cm);
	\end{pgfonlayer}
	\begin{pgfonlayer}{background}
	\draw [bend left=45, looseness=1.75] (0.center) to (1.center);
	\draw [bend right=60, looseness=1.25] (2.center) to (3.center);
	\draw [in=165, out=15, looseness=1.25] (4.center) to (2.center);
	\draw [in=60, out=-105, looseness=1.25] (3.center) to (5.center);
	\draw [bend left=15] (3.center) to (6.center);
	\draw [bend left=45] (1.center) to (7.center);
	\draw [in=105, out=-30] (1.center) to (8.center);
	\draw [in=-180, out=-45, looseness=1.25] (1.center) to (9.center);
	\draw [in=135, out=90, looseness=1.50] (2.center) to (10.center);
	\draw [in=315, out=-15] (2.center) to (10.center);
	\draw [bend right=60, looseness=1.50] (0.center) to (1.center);
	\end{pgfonlayer}
	\end{tikzpicture}
};

\draw  [->, line width = 1.22, draw= gray, dashed](-2,0)--(2,0);
\node at (0,.5){
	$\Tk h$
};
\draw  [->, line width = 1.2, draw= gray](-5,-2)--(-5,-3);
\node at (-5.8,-2.5){$\mu^{\TTk}S$
};

\node at (5,0){ 
	\begin{tikzpicture}[scale = .4]
	\begin{pgfonlayer}{above}
	\node [dot, blue] (0) at (0, 0) {};
	\node  [dot, blue](1) at (0, 0) {};
	\node [dot, blue] (2) at (-1, 1) {};
	\node [dot, blue] (3) at (-1, -1) {};
	\node  (4) at (-4.5, 2.25) {};
	\node  (5) at (-3, -3) {};
	\node  (6) at (0, -4) {};
	\node  (7) at (4, 2) {};
	\node  (8) at (5, -1) {};
	\node  (9) at (5, -3) {};
	\node  (10) at (-1, 2) {};
	\draw[draw = green, fill = green, fill opacity =.2] (-.5,0) circle (1.8cm);
	\end{pgfonlayer}
	\begin{pgfonlayer}{background}
	\draw [bend left=45, looseness=1.75] (0.center) to (1.center);
	\draw [bend right=60, looseness=1.25] (2.center) to (3.center);
	\draw [in=165, out=15, looseness=1.25] (4.center) to (2.center);
	\draw [in=60, out=-105, looseness=1.25] (3.center) to (5.center);
	\draw [bend left=15] (3.center) to (6.center);
	\draw [bend left=45] (1.center) to (7.center);
	\draw [in=105, out=-30] (1.center) to (8.center);
	\draw [in=-180, out=-45, looseness=1.25] (1.center) to (9.center);
	\draw [bend right=60, looseness=1.50] (0.center) to (1.center);
	\end{pgfonlayer}
	\end{tikzpicture}
};
\draw  [->, line width = 1.2, draw= gray](5,-2)--(5,-3);

\node at (5.8,-2.5){
	$h$
};

\node at (-5,-4.5){ 
	\begin{tikzpicture}[scale = .4]
	\begin{pgfonlayer}{above}
	\node [dot, red] (0) at (0, 0) {};
	\node  [dot, red](1) at (2, 0) {};
	\node [dot, red] (2) at (-2, 1) {};
	\node [dot, red] (3) at (-1, -1) {};
	\node  (4) at (-4.5, 2.25) {};
	\node  (5) at (-3, -3) {};
	\node  (6) at (0, -4) {};
	\node  (7) at (4, 2) {};
	\node  (8) at (5, -1) {};
	\node  (9) at (5, -3) {};
	\node  (10) at (-1, 2) {};
	\end{pgfonlayer}
	\begin{pgfonlayer}{background}
	\draw [bend left=45, looseness=1.75] (0.center) to (1.center);
	\draw [bend right=60, looseness=1.25] (2.center) to (3.center);
	\draw [in=165, out=15, looseness=1.25] (4.center) to (2.center);
	\draw [in=60, out=-105, looseness=1.25] (3.center) to (5.center);
	\draw [bend left=15] (3.center) to (6.center);
	\draw [bend left=45] (1.center) to (7.center);
	\draw [in=105, out=-30] (1.center) to (8.center);
	\draw [in=-180, out=-45, looseness=1.25] (1.center) to (9.center);
	\draw [in=135, out=90, looseness=1.50] (2.center) to (10.center);
	\draw [in=315, out=-15] (2.center) to (10.center);
	\draw [bend right=60, looseness=1.50] (0.center) to (1.center);
	\end{pgfonlayer}
	\end{tikzpicture}
};
\draw  [->, line width = 1.2, draw= gray, dashed](-2,-4)--(2,-4);
\node at (0,-4.5){
	$h$
};
\node at (5,-4){ 
	\begin{tikzpicture}[scale = 0.35]	\begin{pgfonlayer}{above}
	\node [dot, green] (0) at (0, 0) {};
	\node  (1) at (0, 0) {};
	\node  (2) at (0, 0) {};
	\node  (3) at (0, 0) {};
	\node  (4) at (-7, 0.25) {};
	\node  (5) at (-3, -4.25) {};
	\node  (6) at (0.25, -5) {};
	\node  (7) at (6.25, 3.25) {};
	\node  (8) at (5.25, -1.75) {};
	\node  (9) at (5, -4) {};
	\end{pgfonlayer}
	\begin{pgfonlayer}{background}
	\draw [bend right=60, looseness=1.25] (2.center) to (3.center);
	\draw [in=-150, out=15, looseness=1.25] (4.center) to (2.center);
	\draw [in=60, out=-105, looseness=1.25] (3.center) to (5.center);
	\draw [in=75, out=-75] (3.center) to (6.center);
	\draw [in=150, out=15] (1.center) to (7.center);
	\draw [in=105, out=-30] (1.center) to (8.center);
	\draw [in=-180, out=-45, looseness=1.25] (1.center) to (9.center);
	\end{pgfonlayer}
	\end{tikzpicture}};

\end{tikzpicture}
\caption{The structure map $h \colon \Tk S \to S$ is compatible with the monad multiplication $\mu^{\TTk}$ on $(\Tk)^2 S$.}\label{fig. COnu h alg}
\end{figure}

\begin{rmk}\label{rmk. no empty graph}
	The monad in \cite{Koc18} is defined identically to $\TTk$ on $\GS$ except the empty graph $\oslash$ is not included as a $\nul$-graph. Consequently, algebras for the monad in \cite{Koc18} are equipped with a non-unital external product. See also \cref{rmk. monoidal unit}.
\end{rmk}

\begin{rmk}
It is straightforward to verify that $\TTk$ and $\Gret \subset \GS$ satisfy the conditions of \cite[Sections~1~\&~2]{BMW12} and hence there is an abstract nerve theorem, analogous to \cref{thm: CO nerve}, for $\Set$-valued non-unital circuit operads. This construction is not described here since the nerve theorem for unital circuit operads is far more interesting (and challenging).
\end{rmk}

\section{Review of modular operads and pointed graphs}\label{sec. mod op}
This section reviews the construction of the monad for $\Set$-valued modular operads in \cite[Sections~5-7]{Ray20} --  in terms of monads $\TT$ and $\DD$ on $\GS$ and a distributive law $\lambda_{\DD\TT}:TD \Rightarrow DT$ for composing them -- in the more general setting of a category $\E$ that is assumed to have sufficient limits and colimits. 

\subsection{Non-unital modular operads}

The non-unital modular operad monad $\TT = (T, \mu^\TT, \eta^\TT)$ on $\GS$ was described in \cite[Section~5]{Ray20}. In this paper, $\TT$ is used to denote 
the non-unital $\E$-valued modular operad monad on $\GSE$ where $ \E$ is an arbitrary enriching category with all finite limits and sufficient colimits. 

The endofunctor $\T\colon\GSE \to \GSE$ is described, for all graphical species $S$ in $\E$, and all finite sets $X$, by
\begin{equation}\label{eq. T def}
TS_X \defeq \mathrm{colim}_{\X \in X\Griso}S(\X),\end{equation} where $X\Griso \subset X \Grisok$ is the full subgroupoid of connected $X$-graphs. 

By \cite[Corollary~5.19]{Ray20} (which follows \cite[Lemma~1.5.12]{Koc16}), if $\G$ is a connected graph, and $\Gg \colon \elG \to \Gr$ is a non-degenerate $\G$-shaped graph of graphs such that $\Gg(b)$ is connected for all $(\C, b) \in \elG$, then the colimit $ \Gg(\G)$ is also connected. Hence, the triple $\TT = (T, \mu^\TT, \eta^\TT)$, where $\mu^{\TT}, \eta^{\TT}$ are the obvious (co)restrictions, defines a monad on $\GSE$. 

\begin{prop}\label{prop. Talg}
	The EM category $\GSE^\TT$ of algebras for $\TT$ on $\GSE$ is canonically equivalent to $\nuCSME$.  
\end{prop}
\begin{proof}
The proof for general $\E$ closely resembles the proof of \cite[Proposition~5.29]{Ray20} for the case $\E = \Set$. 

The inclusion $X \Griso \hookrightarrow X \Grisok$ induces a monomorphism $TS_X \to  \Tk S_X$. In particular, since $\X^{x \ddagger y}$ is connected for all connected $X$-graphs $\X$ and all pairs $x, y$ of distinct elements of $X$, if $(A, h)$ is an algebra for $\TT$, then $A$ admits a contraction $\zeta$
\[ \zeta_X^{x \ddagger y} \colon A_X^{x \ddagger y} = A(\C_X^{x \ddagger y}) \xrightarrow{h} TA_{X\setminus \{x,y\}}, \] that satisfies (M2) by the proof of \cref{unpointed modular operad}.

Moreover, 
there is a multiplication $\diamond$ on $A$ such that, for all finite sets $X$ and $Y$ and all elements $x \in X$, and $y \in Y$, $\diamond_{X,Y}^{x \ddagger y}$ is induced by 
\[ A(\mathcal D_{X,Y}^{x \ddagger y} ) \to TA_{(X \amalg Y) \setminus \{x,y\}} \xrightarrow {h} A_{(X \amalg Y) \setminus \{x,y\}} .\]

To show that $(A, \diamond, \zeta)$ satisfy the axioms (M1)-(M4), we construct (as in the proof of \cref{unpointed modular operad} and \cite[Proposition~5.29]{Ray20}), for each axiom, a pair of graphs of graphs that have the same colimit in $\Gr$. 

For example, to prove Axiom (M4), let $X$ and $Y$ be finite sets with distinct elements $x_1, x_2 \in X$ and $y_1 , y_2 \in Y$. Then the diagram 
\[\scriptsize{\xymatrix@C=6pt{ && (A_{\scriptscriptstyle{X}} \times A_{\scriptscriptstyle{Y}} )^{\scriptscriptstyle{x \ddagger y_1, x_2\ddagger y_2}}\ar[dl]_-{=} \ar[d]^{=}\ar[dr]^-{=}&&\\
		A_{\scriptscriptstyle{X \amalg Y \setminus \{x_1, y_1\}}}^{\scriptscriptstyle{ x_2\ddagger y_2}} \ar@{=}[d]&\ar@{-->}[l]_-{\diamond }A(\mathcal D_{\scriptscriptstyle{X,Y}}^{\scriptscriptstyle{x_1\ddagger y_1}})^{\scriptscriptstyle{x_2\ddagger y_2}}\ar[r]^-{=}\ar[d]& A(\C_X \amalg \C_Y )^{\scriptscriptstyle{x_1 \ddagger y_1, x_2\ddagger y_2}}\ar[dd]&  A(\mathcal D_{\scriptscriptstyle{X,Y}}^{\scriptscriptstyle{x_2 \ddagger y_2}})^{\scriptscriptstyle{ x_1\ddagger y_1}}\ar[d] \ar[l]_-{=}\ar@{-->}[r]^-{ \diamond}& A_{\scriptscriptstyle{ X \amalg Y\setminus \{x_2,y_2\}}} ^{\scriptscriptstyle{x_1 \ddagger y_1}}\ar@{=}[d]\\
		A(\C_{\scriptscriptstyle{X \amalg Y \setminus\{x_1, y_1\}}} ^{\scriptscriptstyle{x_2\ddagger y_2}})  \ar[d]&TA (\C_{\scriptscriptstyle{X \amalg Y \setminus\{x_1, y_1\}} }^{\scriptscriptstyle{x_2\ddagger y_2}}) \ar[l]_-{h}\ar[d]&& TA (\C_{\scriptscriptstyle{X \amalg Y \setminus\{x_2, y_2\}}} ^{\scriptscriptstyle{x_1\ddagger y_1}}) \ar[d]\ar[r]^-{h}\ar[d]&A(\C_{X \amalg Y \setminus\{x_2, y_2\}} ^{\scriptscriptstyle{x_1\ddagger y_1}})\ar[d]\\
		TA_{\scriptscriptstyle{X \amalg Y\setminus \{x_1, x_2, y_1, y_2\}}}\ar[rrdd]^-{h}& \ar[l]_{Th}T^2A_{X \amalg Y \setminus \{x_1, x_2, y_1, y_2\}}\ar[r]^-{ \mu^{\TTk A}}& TA_{\scriptscriptstyle{X \amalg Y\setminus\{x_1, x_2, y_1, y_2\}}}\ar[dd]^-{h}& \ar[l]_-{ \mu^{\TTk A}}T^2A_{\scriptscriptstyle{X \amalg Y \setminus\{x_1, x_2, y_1, y_2\}}}\ar[r]^-{Th} &  TA_{\scriptscriptstyle{X \amalg Y \setminus\{x_1, x_2, y_1, y_2\}}}\ar[ddll]^-{h}\\
		&&&&\\
		&& A_{\scriptscriptstyle{X \amalg Y \setminus \{x_1, x_2, y_1, y_2\}}}&&}}\]
commutes, whereby it follows that 
	$\diamond$ and $\zeta$ satisfy (M4).
	
	Axioms (M1)-(M3) follow by the same method. (Of course, (M2) is proved in \cref{unpointed modular operad}, and, when $\E = \Set$, the proof that $(A,h)$ satisfies (M1) is described in detail in \cite{Ray20}.)

The converse follows exactly the same method as the proof of \cite[Proposition~5.29]{Ray20} and \cref{unpointed modular operad}.	
\end{proof}

\subsection{The monad $\DD$}\label{ssec. D monad}
The modular operad monad on $\GSE$ will be described in terms of a distributive law for composing $\TT$ and the \textit{pointed graphical species monad} $\DD = (D, \mu^{\DD}, \eta^{\DD})$. This was defined on graphical species in $\Set$ in \cite{Ray20}, and is generalised here to a monad on $\GSE$, for $\E$ with sufficient limits and colimits. 

	The endofunctor $D $ adjoins \textit{formal (contracted) units} to a graphical species $S$ in $\E$ according to $DS_\S = S_\S$ with $DS_\tau = S_\tau$ and, for finite sets $X$: 
	\[ DS_X   = \left \{ \begin{array}{ll}
	S_X & X \not \cong \two, \text{ and } X \not \cong \nul,\\
	S_\two \amalg S_\S& X = \two,\\
	S_\nul \amalg \widetilde{S_\S} & X = \nul,
	\end{array}  \right . \]
	where $ \widetilde{S_\S}$ is the coequaliser of the identity and $S_\tau$ on $S_\S$

There are canonical natural transformations $\eta^\DD \colon 1_{\GSE} \Rightarrow D$, provided by the induced morphism $S \hookrightarrow DS$, and $\mu^\DD  \colon D^2 \Rightarrow D$ induced by the canonical projections $D^2S_\two \to DS_\two$, so that $\DD = (D, \mu^\DD, \eta^\DD)$ defines the \textit{pointed graphical species monad} on $\GSE$. 

Algebras for $\DD$ -- called \textit{pointed graphical species in $\E$} -- are graphical species $S$ in $\E$ equipped with 
	a distinguished unit-like morphism (\cref{defn: formal connected unit}) $\epsilon \colon S_\S \to  S_\two,$ 
	and a distinguished morphism $o \colon S_\S \to S_\nul $ that factors through $ \widetilde{S_\S}$.

Let $\fisinvp$ be category  
 obtained from $\fisinv$ by formally adjoining morphisms
$u \colon  \two \to \S$ and $ z\colon  \nul \to \S$, subject to the relations:
\begin{enumerate}[(i)]
	\item 
	$u \circ ch_1 =  id \in\fisinv (\S, \S)$ and  $ u \circ ch_2= \tau \in \fisinv (\S, \S) $;
	\item $\tau \circ u = u \circ \sigma_{\mathbf 2}\in \fisinv (\two, \S)$ ;
	\item 
	$z = \tau \circ z \in \fisinv (\nul, \S)$.
	\label{relations}
\end{enumerate}

It is straightforward to verify that $\fisinvp$ is completely described by:
\begin{enumerate}[]
	\item $\fisinvp(\S, \S) = \fisinv (\S, \S)$ and $\fisinvp (Y,X)  = \fisinv(Y,X)$ whenever $Y \not \cong \nul$ and $ Y \not \cong \two$;
	\item $\fisinvp (\nul, \S) = \{z\}$, and $\fisinvp (\nul, X) = \fisinv (\nul, X) \amalg \{ ch_x\circ z\}_{x \in X}$;
	\item $\fisinvp (\two, \S) = \{u, \tau \circ u\}$, and $\fisinvp (\two, X) = \fisinv (\nul, X) \amalg \{ ch_x\circ u, ch_x \circ \tau \circ u\}_{x \in X}$. 
\end{enumerate}

Hence:
\begin{lem}\label{lem: elp presheaves}
The EM category $\GSEp$ of algebras for $\DD$ is the category of $\E$-presheaves on $\fisinvp$. In other words, the following are equivalent:
	\begin{enumerate}
		\item $S_*$ is a presheaf on $ \fisinvp$ that restricts to a graphical species $S$ on $ \fisinv$;
		\item $(S, \epsilon, o)$, with $\epsilon  = S_*(u)$ and $o = S_*(z)$ is a pointed graphical species.
	\end{enumerate}
	\label{CGSp}
\end{lem}

As a consequence of \cref{lem: elp presheaves}, the notation $ S_*$ and $ (S, \epsilon, o)$ will be used interchangeably to denote the same pointed graphical species.

\subsection{ Pointed graphs}\label{ss. pointed}

The category $\Grp$ of connected graphs and \textit{pointed \'etale morphisms}, obtained in the bo-ff factorisation of the functor $\Gr \to \GSp$ is described in detail in \cite[Section~7.2]{Ray20}. Here, $\Grp$ is extended to the category $\Gretp$ of \textit{all} graphs and pointed \'etale morphisms. In other words, $\Gretp$ is the category defined by the bo-ff factorisation of the functor $ \Gret \xrightarrow{\yet} \GS \to \GSp$, as in the following commuting diagram of functors:
\begin{equation} \label{defining Grp}
\xymatrix{ 
	\fisinvp \ar[rr]^-{\text{ }}_-{\text{ f.f.}}	\ar@/^2.0pc/[rrrr]_-{\text{ f.f.} }			&& \Gretp \ar [rr]_-{\text{ f.f.} }^-{\yetp}				&& \GSp \ar@<2pt>[d]^-{\text{forget}^{\DD}}\\ 
	\fisinv \ar@{^{(}->} [rr]\ar[u]^{\text{b.o.}}	&& \Gret \ar@{^{(}->} [rr]_-{\yet} \ar[u]^{\text{b.o.}}	&& \GS.  \ar@<2pt>[u]^-{\text{free}^{\DD}}
}\end{equation}

The inclusion $\fisinvp \to \Gretp$ is fully faithful (by uniqueness of bo-ff factorisation), and the induced nerve $\yetp\colon \Gretp \to \GSp$ is fully faithful by construction

Let $\G\in \Gret$ be a graph. By \cref{ssec. D monad}, for each edge $e \in E$, the $ch_e$-coloured unit for $ \yetp \G$ is given by $ \epsilon^\G_e  =  ch_e \circ u \in \Gretp (\C_\two, \G)$, and the corresponding contracted unit is given by $o^\G_{\tilde e} = ch_e \circ z \in \Gretp (\C_\nul, \G)$.

Since the functor $\yetp$ embeds $\Gretp$ as a full subcategory of $\GSp$, I will write $\G$, rather than $\yetp \G$, where there is no risk of confusion. In particular, the element category $\elpG[\yetp \G]$ is denoted simply by $\elpG$ and called the 
\emph{category of pointed elements of a graph $\G$}.

For all pointed graphical species $S_*$, the forgetful functor $\GSp \to \GS$ induces an inclusion $ \ElS[S]\hookrightarrow \elpG[S_*]$, and, for all graphs $\G$, the inclusion $\elG \hookrightarrow \elpG$ is final (see \cite[Section~IX.3]{Mac98}) by \cite[Lemma~7.8]{Ray20}. Hence,  
\begin{equation}
\label{eq: pointed sheaves}  \mathrm{lim}_{ (\C, b) \in \elpG}S(\C)  = \mathrm{lim}_{ (\C, b) \in \elG}S(\C) = S(\G) \end{equation} for all pointed graphical species $S_* = (S, \epsilon, o)$ in $\E$. 

Moreover, by \cite[Section~7.3]{Ray20}, the induced inclusion $\GSEp \hookrightarrow\prE{\Gretp}$ is fully faithful. (In fact, $\GSE$ is equivalent to the category of $\E$-sheaves for the obvious extension of the canonical \'etale topology on $\Gret$ to $\Gretp$.)

To better understand  morphisms in $\Gretp$, let $\G$ be a graph and 
let $W \subset V_0 \amalg V_2$ a subset of isolated and bivalent vertices of $\G$. 

The \emph{vertex deletion functor (for $W$)} is the $ \G$-shaped graph of graphs ${\Gdg}^\G_{\setminus W}\colon  \elG \to \Gret $ given by 
\[{\Gdg}^\G_{\setminus W} (\CX, b)  = \left \{ \begin{array}{ll} 
(\shortmid)& \text{ if } (\CX, b) \text{ is a {neighbourhood} of } v \in W\ (\text{so }  X \cong \nul \text{ or } X \cong \two),\\
\CX &\text{ otherwise. }  
\end{array} \right.  \]   

By \cref{colimit exists}, this has a colimit $\Gnov$ in $\Gret$, and, for all $(\C, b) \in \elG$, ${\Gdg}^\G_{\setminus W}$ describes a morphism $\C \to {\Gdg}^\G_{\setminus W} (b)$ in $\fisinvp$. Hence, ${\Gdg}^\G_{\setminus W}$ induces a morphism $\delW \in \Gretp (\G, \Gnov)$.

\begin{defn}\label{def: vertex deletion}

The morphism $\delW\in \Gretp (\G,\Gnov)$ induced by ${\Gdg}^\G_{\setminus W}\colon  \elG \to \Gret $ is called the  \emph{vertex deletion morphism corresponding to $W \subset V_0 \cong V_2$}. 

\end{defn}

\begin{ex} If $\G = \C_\nul$ is an isolated vertex $*$ and $W = \{*\}$, then $\delW$ is precisely $ z \colon \C_\nul \to (\shortmid). $	

\end{ex}

\begin{ex}\label{line deletion} For $\G = \C_\two$ and $ W = V = \{*\}$, ${\Gdg}^\G_{\setminus W}$ is the constant functor $\elG \to (\shortmid)$:
	\begin{equation}
	\xymatrix{ 
		(\shortmid) \ar[rr]^-{ch_1}\ar[drr]_-{id_{(\shortmid)}}  && \C_\two \ar[d]^{u} && (\shortmid) \ar[ll]_-{ch_2 \circ \tau} \ar[dll]^{id_{(\shortmid)}}  \\
		&&(\shortmid)&&}
	\end{equation}
	and hence $\Gnov = (\shortmid)$ and $\delW  = u \in \Gretp (\C_\two, \shortmid)$.

	More generally, if $\G = {\Lk}$, and $ W  = V$, then ${\Gdg}^\G_{\setminus W}$ is also the constant functor to $(\shortmid)$ and hence $\Gnov = (\shortmid)$. The induced morphism is denoted by $u^k \defeq \delW\colon \Lk \to \Gretp$:
	\[
	\xymatrix{ 
		(\shortmid) \ar[rr]^-{ch_1}\ar[drr]_-{id_{(\shortmid)}}  && \C_\two \ar[d]^{u} &&{ \text{ \dots }} \ar[ll]_-{ch_2 \circ \tau} \ar[rr]^-{ch_1}
		&& \C_\two \ar[d]^-{u}&& \ar[ll]_-{ch_2 \circ \tau}\ar[dll]^{id_{(\shortmid)}} (\shortmid)\\
		&&(\shortmid)\ar@{=}[r]&&{ \text{ \dots }} &&\ar@{=}[l](\shortmid).&&}\]
	In particular, $u^1 = u \colon  \C_\two \to (\shortmid)$ and $u^0$ is just the identity on $(\shortmid)$.

\end{ex}

\begin{ex}\label{wheel deletion} 
	Let $\G = \W$ be the wheel graph with one vertex $v$. Then $\W_{/\{v\}} = \mathrm{colim}_{ \elG[\W]} {\Gdg}^\W_{/\{v\}} $ exists and is isomorphic to $(\shortmid)$ in $ \Gret$.  The induced morphism $\kappa \defeq \delW[\{v\}]\colon  \W \to (\shortmid)$ is described in (\ref{epsilon}).

	Hence, there are precisely two morphisms $\kappa$ and $ \tau \circ \kappa$ in $ \Gretp (\W, \shortmid)$: 
	
	\begin{minipage}[t]{0.5\textwidth}
		\begin{equation} \label{epsilon}
		\xymatrix{
			& \W & \\
			(\shortmid) \ar[ur]^{ch_a} \ar@<-2pt>[rr]_{ch_2\circ \tau}  \ar@<2pt>[rr]^{ ch_{1} }  \ar@{=}[dr] && \C_\two \ar[ul]_{ 1_{\C_\two} \mapsto  a} \ar[dl]^{u}\\
			& (\shortmid)& }\end{equation}
		
	\end{minipage}
	\begin{minipage}[t]{0.5\textwidth}
		\begin{equation} \label{tau epsilon}
		\xymatrix{
			& \W & \\
			(\shortmid) \ar[ur]^{ch_a} \ar@<-2pt>[rr]_{ch_2\circ \tau}  \ar@<2pt>[rr]^{ ch_1 } \ar@{=}[d] && \C_\two \ar[ul]_{ 1_{\C_\two} \mapsto  a} \ar[d]^{\sigma_\two}\\
			(\shortmid) \ar@<-2pt>[rr]_{ch_1\circ \tau}  \ar@<2pt>[rr]^{ ch_2}  \ar@{=}[dr] && \C_\two  \ar[dl]^{u}\\
			& (\shortmid).& }\end{equation}
		
	\end{minipage}
	
	More generally, let $ \Wl$ be the wheel graph with $m$ vertices $(v_i) _{i = 1}^m$, and let $\iota \in \Gr(\Lk[m-1], \Wl)$ be an \'etale inclusion. If $W$ is the image of $ V(\Lk[m-1])$ in $V(\Wl)$, then $ V(\Wl) \cong W \amalg \{*\}$, and by (\ref{epsilon}) and \cref{line deletion}, 
	there are two distinct pointed morphisms, $\kappa^m  $ and $\tau \circ \kappa^m$,  in $ \Gretp (\Wl, \shortmid)$. Hence, for all $\G$, 
	\[\Gretp (\Wl, \G) =  \Gret(\Wl, \G)\amalg \{ch_{e} \circ \kappa^m\}_{ e\in E(\G)} \cong \Gretp(\W, \G).\]  \end{ex}

\begin{rmk}
	Recall from \cref{rmk. pushout} that the composition $\bigcirc = \cup \circ \cap$ of Brauer diagrams is not described by a pushout of cospans as in (\ref{eq. cospan pushout}). Instead, the pushout (\ref{eq. cospan pushout}) for $(\cup \circ \cap)$ 
	is induced by 
	the trivial diagram $(\shortmid) \xrightarrow{id}(\shortmid) \xleftarrow{id} (\shortmid)$, and therefore agrees with $\kappa \colon \W \to (\shortmid)$ described in \cref{wheel deletion}. 
	
	(See \cite[Section~6]{Ray20} for more details on the combinatorics of contractions.)
	\end{rmk}

For convenience, let us define $\Wm[0] \defeq  \Wm_{/V(\Wm)} = (\shortmid) $ ($m \geq 1$). Then for all $k \geq 0$, $u^k \colon \Lk \to (\shortmid)$ factors canonically as $ \Lk \to \Wm[k] \xrightarrow{\kappa^k} (\shortmid)$.

\begin{ex}\label{ex. admissible deletion}

Assume that $\G \not \cong \C_\nul$ is connected (so $V_0$ is empty). Let $W \subset V_2$ be a subset of bivalent vertices of $\G$. Unless $W = V(\G)$ and $\G = \Wm$ for some $m \geq 1$, the graph $\Gnov$ may be intuitively described as `$\G$ with a subset of bivalent vertices deleted' as in \cref{fig: vertex deletion}.

In this case, the graph $\Gnov$ is described explicitly by:
\[\ \Gnov = \ \xymatrix{
	*[r] { \Enov}\ar@(ul,dl)[]_{\taunov} && { \Hnov} \ar[ll]_-{\snov} \ar[rr]^-{\tnov}&& \Vnov,}\] where
\[ \begin{array}{ll}\Vnov&=V \setminus W,\\
\Hnov &= H \setminus \left (\coprod_{v \in W} \vH \right),\\

&=  E \setminus\left (\coprod_{v \in W} \vE \right), 
 	\end{array}\]
and $\snov, \tnov$ are just the restrictions of $s$ and $t$. The description of the involution $\taunov$ is more complicated than needed here. (The interested reader is referred to \cite[Section~7.2]{Ray20} which contains full descriptions of the vertex deletion morphisms). 

\begin{figure}[htb!]
	\begin{tikzpicture}[scale = .3]

	\draw[->] (-15,0)--(-11,0);
	\node at (-13,.5){\scriptsize{$\delW$}};

	\node at (-25,0){\begin{tikzpicture}[scale = .25]
		
		\draw [thick, cyan]
		(-3,2).. controls (-2.5,1.5) and (-2.5,0.5).. (-3,0)
		(-3,0).. controls (-3.5,-1).. (-5,-1)
		(-5,3) --(-2,4)--(1,4)--(3,3)
		
		(1,-4)--(4, -3.5)--(5,-1)--(5,2)
		(-1,-5) -- (-3,-5)--(-5, -6)--(-6, -9)
		;
		
		\draw[blue, fill = cyan]
		(-3,0) circle (8pt)
		(4, -3.5) circle (8pt)
		(5,-1)circle (8pt)
		(5,2)circle (8pt)
		(-2,4)circle (8pt)
		(1,4)circle (8pt)
		(-3,-5)circle (8pt)(-5, -6)circle (8pt)(-5.5,-7.5)circle (8pt)
		;

		\draw[ thick] 	
		(-6,1).. controls (-6,2).. (-5,3)
		(-5,3).. controls (-4.5,3) and (-3.5,3).. (-3,2)
		(-5,-1).. controls (-6,-0) .. (-6,1)
		(-5,-1)-- (-3,2)
		(-3,2)--(-6,1)
		(-6,1)..controls ( -7.5,1) and (-9, 0.5) ..(-11,0.5)
		(-1,-5)..controls (-0.5,-3.5) and ( 0.5, -3.5)..(1,-4)
		(-1,-5)..controls (-0.5,-5) and ( 0.5, -5)..(1,-4)
		(4,1)--(5,2)
		(5,2)--(3,3)
		(3,3)--(4,1)
		(3,3).. controls (4,3.5) and (5,2.5)..(5,2)
		(3,3).. controls (3,4) and(4,6)  ..(8,6)
		(-3,2).. controls (1,-1.5) and (5,-1).. (5,2)
		;
		
		\filldraw [black]	
		(-6,1) circle (8pt)
		(-5,3) circle (8pt)
		(-3,2) circle (8pt)
		(-5,-1) circle (8pt)
		(1,-4) circle (8pt)
		(-1,-5) circle (8pt)
		(4,1) circle (8pt)
		(5,2) circle (8pt)
		(3,3) circle (8pt);
		
		\draw[black, fill = white]
		(-4,.5)circle (8pt)
		(4,1) circle (8pt);
		

		(0,6 ) .. controls (2,6) and (7,4) ..(7,0)
		( 7,0) .. controls (7,-2) and (2,-8) ..(0,-8)--(0,0)--(-7.5,0);
		\end{tikzpicture}};
	
	\node at (-2,0){\begin{tikzpicture}[scale = .25]
		
		\draw [thick, red]
		(-3,2).. controls (-2.5,1.5) and (-2.5,0.5).. (-3,0)
		(-3,0).. controls (-3.5,-1).. (-5,-1)
		(-5,3)..controls (-2,4) and (1,4) .. (3,3)
		
		(1,-4)..controls (2, -4) and (6,-2).. (5,2)
		(-1,-5) .. controls (-2.5,-5) and (-5, -4.5) ..(-6, -9)
		;

		\draw[ thick] 	
		(-6,1).. controls (-6,2).. (-5,3)
		(-5,3).. controls (-4.5,3) and (-3.5,3).. (-3,2)
		(-5,-1).. controls (-6,-0) .. (-6,1)
		(-5,-1)-- (-3,2)
		(-3,2)--(-6,1)
		(-6,1)..controls ( -7.5,1) and (-9, 0.5) ..(-11,0.5)
		(-1,-5)..controls (-0.5,-3.5) and ( 0.5, -3.5)..(1,-4)
		(-1,-5)..controls (-0.5,-5) and ( 0.5, -5)..(1,-4)
		(4,1)--(5,2)
		(5,2)--(3,3)
		(3,3)--(4,1)
		(3,3).. controls (4,3.5) and (5,2.5)..(5,2)
		(3,3).. controls (3,4) and(4,6)  ..(8,6)
		(-3,2).. controls (1,-1.5) and (5,-1).. (5,2)
		;
		
		\filldraw [black]	
		(-6,1) circle (8pt)
		(-5,3) circle (8pt)
		(-3,2) circle (8pt)
		(-5,-1) circle (8pt)
		(1,-4) circle (8pt)
		(-1,-5) circle (8pt)
		(4,1) circle (8pt)
		(5,2) circle (8pt)
		(3,3) circle (8pt);
		
		\draw[black, fill = white]
		(-4,.5)circle (8pt)
		(4,1) circle (8pt);
		

		(0,6 ) .. controls (2,6) and (7,4) ..(7,0)
		( 7,0) .. controls (7,-2) and (2,-8) ..(0,-8)--(0,0)--(-7.5,0);
		\end{tikzpicture}};

\end{tikzpicture}
\caption{Vertex deletion $\delW\colon  \G \to \Gnov$. } 
\label{fig: vertex deletion}
\end{figure}
\end{ex}

\begin{defn}
	\label{def: similar}
	The \emph{similarity category} $\Grsimp\hookrightarrow \Gretp$ is the identity-on-objects subcategory of $\Gretp$ whose morphisms are generated under composition by $z\colon \C_\nul \to (\shortmid)$, the vertex deletion morphisms, and graph isomorphisms. Morphisms in $\Grsimp$ are called \emph{similarity morphisms}, and connected components of $\Grsimp$ are \emph{similarity classes}. Graphs in the same connected component of $\Grsimp$ are \emph{similar}. 
\end{defn}

The following proposition summarises some important basic properties of vertex deletion morphisms. Detailed proofs may be found in \cite[Section~7.2]{Ray20}.

\begin{prop}
	\label{prop. vertex deletion}
	For all graphs $\G$ and all subsets $W \subset V_2$ of bivalent vertices of $\G$,
	\begin{enumerate}
		\item $\delW \colon \G \to \Gnov$ preserves connected components, and may be defined componentwise;
		\item	if $\G$ is connected and $W = V$, then $\G \cong \Lk$ or $\G \cong \Wm$ for $k, m \geq 0$ and $\Gnov \cong (\shortmid)$;
		\item 	unless $\G \cong \Wm$ and $W = V$ for $m \geq 1$, if $\G$ is connected, then $\delW$ induces an identity on ports: $\delW \colon E_0 (\G) \xrightarrow{=}E_0 (\Gnov)$.
	
	\end{enumerate} 
The pair $(\Grsimp, \Gret)$ of subcategories of $\Gretp$ defines an orthogonal 
factorisation system on $\Gretp$.%

\end{prop}

\begin{ex}\label{ex. DS}
	For all graphical species $S$ in $\E$, and all graphs $\G$ with no isolated vertices,
	\begin{equation} \label{eq. D on graphs}
	DS(\G) = \coprod_{W \subset V_2} S(\Gnov).
	\end{equation} 
\end{ex}

More generally, since $DS_\nul  = DS (\C_\nul) = S_\nul \amalg \widetilde{S_\S}$ :  \begin{equation}\label{eq. D on graphs general} DS(\G) \cong \mathrm{colim}_{(\H,f) \in \G \ov \Gretsimp} S(\H), \quad  \text{  for all graphs } \G \in \Gret.\end{equation}

\begin{ex}
	\label{wheel and line structure notation} Let $S_*  = (S, \epsilon, o)$ be a pointed graphical species in $\E$. For $k,m \geq 0$, 
	there are distinguished monomorphisms in $\E$: 
	\begin{equation}
	\label{eq. unit graph definitions}
 S_*(u^k)\colon S_\S \to S(\Lk), \ \text{ and } S_*(\kappa ^m)\colon S_\S \to S(\Wm).
	\end{equation}
\end{ex}

\subsection{The distributive law $TD \Rightarrow DT$}\label{ssec. Similar connected X-graphs}

For the remainder of this section, all graphs will be assumed to be connected, unless stated otherwise. 

 Recall from \cref{def. X graph} that an $X$-graph $\X$ is said to be admissible if it has no stick components.
For any finite set $X$, we may enlarge the category $X\Griso$ of admissible connected $X$-graphs and port-preserving isomorphisms, to form a category $X\XGrsimp$ that includes vertex deletion morphisms $\delW$ and $z \colon \C_\nul \to (\shortmid)$ out of $X$-graphs together with any labelling of ports from the domain:

By \cref{prop. vertex deletion}, any similarity morphism that is not of the form $z \colon \C_\nul \to (\shortmid)$ or $\kappa^m \colon \Wm \to (\shortmid)$ is boundary preserving. So, let $\X = (\G , \rho)$ be an admissible connected $X$-graph such that $\X \not \cong \Wm$, $\X \not \cong \C_\nul$, and let $f \in \Grsimp (\G, \G')$ be a similarity morphism. Then $f$ and the $X$-labelling $\rho \colon E_0 \to X$ for $\X$, induce an $X$-labelling on $\G'$. 
The category $X\XGrsimp$ is obtained by adjoining to the category $X\Griso$ of {(admissible)} connected $X$-graphs and port-preserving isomorphisms, all similarity morphisms from objects of $X\Griso$, 
and their codomains, equipped with the induced labelling. 
\begin{itemize}
	\item If $X \not \cong \nul$, $X \not \cong \two$, then $X \XGrsimp$ is the category whose objects are connected $X$-graphs and whose morphisms are similarity morphisms 
	that preserve the labelling of the ports.
	
	\item For $X = \two $, $\two\XGrsimp$ contains the morphisms $ \delW[V] \colon \Lk \to (\shortmid)$, and hence the labelled stick graphs $(\shortmid, id)$ and $(\shortmid, \tau)$. There are no non-trivial morphisms out of these special graphs, and $\X$ is in the same connected component as $(\shortmid, id)$ if and only if $\X = \Lk$ (with the identity labelling) for some $k \in \N$. In particular, $\tau \colon (\shortmid) \to (\shortmid)$ does not induce a morphism in $\two \XGrsimp$.
	
	\item Finally, when $X = \nul$, the morphisms $\delW[V] \colon \Wm \to (\shortmid)$, and $z \colon \C_\nul \to (\shortmid)$ are not boundary-preserving, and, in particular, do not equip $(\shortmid)$ with any labelling of its ports. So, the objects of $\nul\XGrsimp$ are the admissible $\nul$-graphs, and $(\shortmid)$. 
	
	In particular, $\Wm, \C_\nul$ and $(\shortmid)$ are in the same connected component of $\nul\XGrsimp$. Since $(\shortmid)$ is not admissible, there are no non-trivial morphisms in $\nul \XGrsimp$ with $(\shortmid)$ as domain.

\end{itemize}
	To simplify notation in what follows, we write ${\C_{\nul}}_{\setminus V} \defeq (\shortmid)$ and $ \delW[V] = z \colon \C_\nul \to (\shortmid)$.

The following lemma is immediate from the definitions:
\begin{lem}\label{lem. terminal X}
	For all finite sets $X$ and all connected $X$-graphs $\X$, the connected component of $X\XGrsimp$ containing $\X$ has a terminal object $\X^\bot$ -- without bivalent or isolated vertices -- given by $\X_{ /(V_2\amalg V_0)}$.
	
If $\X \cong \Lk$ or $\X \cong \Wm $, $\X \cong \C_\nul$ for some $k, m \geq 0$, then $\X^\bot \cong (\shortmid)$. Otherwise, $\X^\bot$ is an admissible, connected $X$-graph.

\end{lem}

Let $S$ be a graphical species in $\E$. 
By (\ref{eq. D on graphs general}), if $\X$ is a connected $X$-graph and $X\XGrsimp^\X$ is the connected component of $\X$ in $X\Grsimp$, then there is a canonical isomorphism
\[ DS(\X)  \cong  \mathrm{colim}_{\X' \in X\XGrsimp^\X}S(\X').\]

In particular, $TDS_X$ is described by a colimit --  indexed by pairs $(\X, W)$, of a connected $X$-graph $\X$ and a subset $W \subset V_0 \amalg V_2 $ of vertices of $(\X)$, and isomorphisms $\X \xrightarrow{\cong} \X'$ in $X \Griso$ --  of objects $S(\X_{\setminus W})$ in $\E$. The distributive law $\lambda_{\DD, \TT} \colon TD \Rightarrow DT$ is induced by forgetting the pair $(\X, W)$, and remembering only $\X_{\setminus W}$, in the evaluation of  $S(\X_{\setminus W})$.

Precisely, for $X \not \cong \nul$, 
\[TD S_X= \mathrm{colim}_{ \X \in X\Griso} DS(\X)= \mathrm{colim}_{ \X \in X\Griso} \coprod_{W \subset V_2(\X)} S(\X_{\setminus W}).\]

Since vertex deletion morphisms preserve ports when $X \neq \nul$, each $\X_{\setminus W}$ is either an admissible $X$-graph or $\X_{\setminus W} \cong (\shortmid)$, and hence there are canonical morphisms
\[ TD S_X 
\to \left \{ \begin{array}{llll}
\mathrm{colim}_{ \X \in X\Griso}  S(\X)  &= 
TS_X& = DTS_X & \text{ when } X \not \cong \nul, X \not \cong \two,\\
 \left( \mathrm{colim}_{ \X \in X\Griso}  S(\X) \right) \amalg  S_\S &  = 
TS_\two \amalg S_\S &= DTS_\two &\text{ when } X = \two
\end{array} \right. \]
(since $(\shortmid)$ is not a $\two$-graph by convention). These describe the morphisms $ \lambda_{\DD\TT} S_X \colon TDS_X \to DTS_X$ in $\E$ for $X \not \cong \nul$. 

When $X = \nul$, 
\[\begin{array}{lll}
TDS_\nul & = &\mathrm{colim}_{ \X \in \nul\Griso} DS(\X)\\
& = & \left(\mathrm{colim}_{\overset{ \X \in \nul \Griso}{\X \not \cong \C_\nul}} \coprod_{W \subset V_2(\X)} S(\X_{\setminus W}) \right) \amalg DS(\C_\nul) \\ 
& = &  \left(\mathrm{colim}_{\overset{ \X \in \nul \Griso}{\X \not \cong \C_\nul}} \coprod_{W \subset V_2(\X)} S(\X_{\setminus W}) \right) \amalg S_\nul \amalg \widetilde{S_\S}.
\end{array}\]
Vertex deletion maps $\delW \colon \G \to \Gnov$ from $\nul$-graphs preserve ports unless $\delW \cong z \colon \C_\nul \to (\shortmid)$ or $\delW \cong \kappa ^m \colon \Wm \to (\shortmid)$. So, $\lambda_{\DD\TT} S _\nul \colon TDS_\nul \to DTS_\nul$ is given by the canonical morphism 
\[ TDS_\nul = \left(\mathrm{colim}_{\overset{ \X \in \nul \Griso}{\X \not \cong \C_\nul}} \coprod_{W \subset V_2(\X)} S(\X_{\setminus W}) \right) \amalg S_\nul \amalg \widetilde{S_\S} \to  \mathrm{colim}_{ \X \in \nul \Griso} S(\X) \amalg \widetilde{S_\S}  = DTS_\nul.  \]

The verification that $\lambda_{\DD\TT}$ satisfies the four distributive law axioms \cite{Bec69} follows by a straightforward application of the definitions. To help the reader gain familiarity with the constructions, I describe just one here, namely that the following diagram of endofunctors on $\GSE$ commutes:
\begin{equation}\label{eq. dist 1 DT}
\xymatrix{T^2 D \ar@{=>}[rr]^-{ T \lambda_{ \DD \TT}}\ar@{=>}[d]_-{ \mu^{\TT}D} && TDT \ar@{=>}[rr]^-{ \lambda_{ \DD \TT}T}&& DT^2 \ar@{=>}[d]^-{ D \mu^\TT}\\
	TD\ar@{=>}[rrrr]^-{ \lambda_{ \DD \TT}}&&&& DT.}
\end{equation}

Recall, from (\ref{eq. T mult}) and (\ref{eq. D on graphs}), that, for all graphical species $S$ in $\E$ and all finite sets $X$, \[\begin{array}{lll}T^2 D S_X &=&\mathrm{colim}_{ \X \in X \Griso} \mathrm{lim}_{(\C,b) \in \elG[\X]} \mathrm{colim}_{ \Y^b \in [b]\Griso} \coprod_{W^b \subset V_2}  S(\Y^b_{\setminus W^b})\\ 
&= &\mathrm{colim}_{ \X \in X \Griso} \mathrm{lim}_{ \Gg \in \GrG[\X]} \coprod_{W^\Gg \subset V_2(\Gg(\X))} S (\Gg (\X)_{\setminus W^{\Gg}}),
\end{array}\]
where, as usual, $\esv \colon \Cv \to \X$ is the essential morphism. So, to describe the two paths in (\ref{eq. dist 1 DT}), 
 let $\X$ be a connected $X$-graph, $\Gg$ a non-degenerate $\X$-shaped graph of connected graphs, and let $W^\Gg \subset (V_0 \amalg V_2)(\Gg(\X))$ be a subset of (bivalent or isolated) vertices of the colimit $\Gg(\X)$ of $\Gg$ in $\Gr$. 
 
 For each $(\C, b) \in \elG[\X]$, let 
 \[ W^b \defeq V(\Gg(b)) \cap W^\Gg, \ \text{ so } W^{\Gg}  \cong \coprod_{v \in V(\X)} W^{\esv}.\] 
 
 If $(\X, \Gg, W) = (\C_\nul, \Gid[\C_{\nul}], V(\C_\nul))$, then $\Gg(\X) = \C_\nul$, and $(\C_\nul, b)$ is the unique object of $\elG[\X]$. 
 
  Otherwise, since $\X$ is connected, it must be the case that $W^{\Gg} \subset V_2 (\Gg(\X))$.

In each case we may define $W \subset V(\X)$ to be the set of vertices $v$ of $\X$ such that $W^{\esv} = V(\Gg(\esv))$. 
Then, 
there is a unique non-degenerate $\X_{\setminus W}$-shaped graph of connected graphs $\Gg_{\setminus W}$ given by $\Gg_{\setminus W} (\delW \circ b) = \Gg(b)_{\setminus W^b}$, and with colimit $\Gg_{\setminus W} (\X_{\setminus W})$ in $\Gretp$:
\begin{equation}\label{eq. colimits agree}
\Gg_{\setminus W} (\X_{\setminus W}) = \Gg(\X)_{\setminus W^\Gg}.
\end{equation}

	\begin{figure}[htb!]
		\begin{tikzpicture} [scale = .85]
		\node at(-6,0){ 
			\begin{tikzpicture}[scale = .35]
			\begin{pgfonlayer}{above}
			\node [dot, red] (0) at (0, 0) {};
			\node [dot, red] (1) at (1.5, -1) {};
			\node [dot, red] (2) at (0, -1) {};
			\node [dot, red] (3) at (-3, 0) {};
			\node [dot, red] (4) at (-2.5, -1) {};
			\node [dot, red] (5) at (-4, -1) {};
			\node [dot, green] (6) at (-1.5, -2.5) {};
			\node [dot, green] (7) at (-0.5, -3) {};
			\node [dot, green] (8) at (1, -3) {};
			\node [dot, green] (9) at (0, 2) {};
			\node  (10) at (-6, -4) {};
			\node  (11) at (3, -5) {};
			\node  (12) at (3, 4) {};
			\draw[draw = blue, fill = blue, fill opacity = .2] (.65,-.5) circle (1.6cm);
			\draw[draw = blue, fill = blue, fill opacity = .2] (-3.25,-.5) circle (1.6cm);
			\draw[draw = green, fill = green, fill opacity = .2] (0,2) circle (.7cm);
			\draw[rotate around={80:(-.5,-2.8)}, draw = green, fill = green, fill opacity = .2] (-.5,-2.8) ellipse (.7cm and 2cm);
			\end{pgfonlayer}
			\begin{pgfonlayer}{background}
			\draw [bend right, looseness=1.25] (0) to (2);
			\draw [bend right=60, looseness=1.25] (2) to (1);
			\draw [bend right] (0) to (1);
			\draw [in=165, out=-120, loop] (0) to ();
			\draw [bend right] (3) to (5);
			\draw [bend right=45, looseness=1.25] (5) to (4);
			\draw [in=120, out=-165, loop] (4) to ();
			\draw [in=135, out=60] (3) to (0);
			\draw [bend left] (3) to (9);
			\draw [bend left=60, looseness=0.75] (0) to (1);
			\draw [in=150, out=-60] (4) to (6);
			\draw [bend left=15] (6) to (7);
			\draw [bend right=15, looseness=1.25] (7) to (8);
			\draw [in=-75, out=30] (9) to (12.center);
			\draw [in=135, out=0, looseness=1.25] (8) to (11.center);
			\draw (5) to (10.center);
			\end{pgfonlayer}
			\end{tikzpicture}
			
		};
		
		\node at(0,0){ 
			\begin{tikzpicture}[scale = .35]
			\begin{pgfonlayer}{above}
			\node [dot, red] (0) at (0, 0) {};
			\node [dot, red] (1) at (1.5, -1) {};
			\node [dot, red] (2) at (0, -1) {};
			\node [dot, red] (3) at (-3, 0) {};
			\node [dot, red] (4) at (-2.5, -1) {};
			\node [dot, red] (5) at (-4, -1) {};
			\node (6) at (-1.5, -2.5) {};
			\node  (7) at (-0.5, -3) {};
			\node  (8) at (1, -3) {};
			\node  (9) at (0, 2) {};
			\node  (10) at (-6, -4) {};
			\node  (11) at (3, -5) {};
			\node  (12) at (3, 4) {};
			\draw[draw = blue, fill = blue, fill opacity = .2] (.65,-.5) circle (1.6cm);
			\draw[draw = blue, fill = blue, fill opacity = .2] (-3.25,-.5) circle (1.6cm);
			\draw[draw = green, fill = green, fill opacity = .8] (0,2) circle (.7cm);
				\draw[draw = green, fill = green, fill opacity = .8] (7) circle (.7cm);
			\end{pgfonlayer}
			\begin{pgfonlayer}{background}
			\draw [bend right, looseness=1.25] (0) to (2);
			\draw [bend right=60, looseness=1.25] (2) to (1);
			\draw [bend right] (0) to (1);
			\draw [in=165, out=-120, loop] (0) to ();
			\draw [bend right] (3) to (5);
			\draw [bend right=45, looseness=1.25] (5) to (4);
			\draw [in=120, out=-165, loop] (4) to ();
			\draw [in=135, out=60] (3) to (0);
			\draw [bend left] (3) to (9.center);
			\draw [bend left=60, looseness=0.75] (0) to (1);
			\draw [in=150, out=-60] (4) to (6.center);
			\draw [bend left=15] (6.center) to (7.center);
			\draw [bend right=15, looseness=1.25] (7.center) to (8.center);
			\draw [in=-75, out=30] (9.center) to (12.center);
			\draw [in=135, out=0, looseness=1.25] (8.center) to (11.center);
			\draw (5) to (10.center);
			
			\end{pgfonlayer}
			\end{tikzpicture}
			
		};
		\node at(6,0){
			\begin{tikzpicture}[scale = .35]
			\begin{pgfonlayer}{above}
			\node [dot, red] (0) at (0, 0) {};
			\node [dot, red] (1) at (1.5, -1) {};
			\node [dot, red] (2) at (0, -1) {};
			\node [dot, red] (3) at (-3, 0) {};
			\node [dot, red] (4) at (-2.5, -1) {};
			\node [dot, red] (5) at (-4, -1) {};
			\node (6) at (-1.5, -2.5) {};
			\node  (7) at (-0.5, -3) {};
			\node  (8) at (1, -3) {};
			\node  (9) at (0, 2) {};
			\node  (10) at (-6, -4) {};
			\node  (11) at (3, -5) {};
			\node  (12) at (3, 4) {};
			\draw[draw = blue, fill = blue, fill opacity = .2] (.65,-.5) circle (1.6cm);
			\draw[draw = blue, fill = blue, fill opacity = .2] (-3.25,-.5) circle (1.6cm);
			
			\end{pgfonlayer}
			\begin{pgfonlayer}{background}
			\draw [bend right, looseness=1.25] (0) to (2);
			\draw [bend right=60, looseness=1.25] (2) to (1);
			\draw [bend right] (0) to (1);
			\draw [in=165, out=-120, loop] (0) to ();
			\draw [bend right] (3) to (5);
			\draw [bend right=45, looseness=1.25] (5) to (4);
			\draw [in=120, out=-165, loop] (4) to ();
			\draw [in=135, out=60] (3) to (0);
			\draw [bend left] (3) to (9.center);
			\draw [bend left=60, looseness=0.75] (0) to (1);
			\draw [in=150, out=-60] (4) to (6.center);
			\draw [bend left=15] (6.center) to (7.center);
			\draw [bend right=15, looseness=1.25] (7.center) to (8.center);
			\draw [in=-75, out=30] (9.center) to (12.center);
			\draw [in=135, out=0, looseness=1.25] (8.center) to (11.center);
			\draw (5) to (10.center);
			
			\end{pgfonlayer}
			\end{tikzpicture}
			
		};
		
		\draw  [->, line width = 1.22, draw= gray, dashed](-4,0)--(-2,0);
		\node at (-3,.5){
			$T\lambda_{ \DD \TT}$
		};
		\draw  [->, line width = 1.22, draw= gray, dashed](2,0)--(4,0);
		\node at (3,.5){
			$\lambda_{ \DD \TT}T$
		};
		\draw  [->,dashed,  line width = 1.2, draw= gray](-5,-2)--(-5,-3);
		\node at (-5.8,-2.5){$D\mu^{\TT}S$
		};

		\draw  [->,dashed,  line width = 1.2, draw= gray](5,-2)--(5,-3);
		
		\node at (5.8,-2.5){
			$D\mu^{\TT}$
		};
		
		\node at (-6,-4.5){ 
			\begin{tikzpicture}[scale = .35]
			\begin{pgfonlayer}{above}
			\node [dot, red] (0) at (0, 0) {};
			\node [dot, red] (1) at (1.5, -1) {};
			\node [dot, red] (2) at (0, -1) {};
			\node [dot, red] (3) at (-3, 0) {};
			\node [dot, red] (4) at (-2.5, -1) {};
			\node [dot, red] (5) at (-4, -1) {};
			\node [dot, green] (6) at (-1.5, -2.5) {};
			\node [dot, green] (7) at (-0.5, -3) {};
			\node [dot, green] (8) at (1, -3) {};
			\node [dot, green] (9) at (0, 2) {};
			\node  (10) at (-6, -4) {};
			\node  (11) at (3, -5) {};
			\node  (12) at (3, 4) {};
		
			\end{pgfonlayer}
			\begin{pgfonlayer}{background}
			\draw [bend right, looseness=1.25] (0) to (2);
			\draw [bend right=60, looseness=1.25] (2) to (1);
			\draw [bend right] (0) to (1);
			\draw [in=165, out=-120, loop] (0) to ();
			\draw [bend right] (3) to (5);
			\draw [bend right=45, looseness=1.25] (5) to (4);
			\draw [in=120, out=-165, loop] (4) to ();
			\draw [in=135, out=60] (3) to (0);
			\draw [bend left] (3) to (9);
			\draw [bend left=60, looseness=0.75] (0) to (1);
			\draw [in=150, out=-60] (4) to (6);
			\draw [bend left=15] (6) to (7);
			\draw [bend right=15, looseness=1.25] (7) to (8);
			\draw [in=-75, out=30] (9) to (12.center);
			\draw [in=135, out=0, looseness=1.25] (8) to (11.center);
			\draw (5) to (10.center);
			\end{pgfonlayer}
			\end{tikzpicture}
			
		};
		\draw  [->, line width = 1.2, draw= gray, dashed](-2.5,-4)--(2.5,-4);
		\node at (0,-4.5){
			$\lambda_{ \DD \TT}$
		};
		\node at (6,-4){ 
			\begin{tikzpicture}[scale = .35]
			\begin{pgfonlayer}{above}
			\node [dot, red] (0) at (0, 0) {};
			\node [dot, red] (1) at (1.5, -1) {};
			\node [dot, red] (2) at (0, -1) {};
			\node [dot, red] (3) at (-3, 0) {};
			\node [dot, red] (4) at (-2.5, -1) {};
			\node [dot, red] (5) at (-4, -1) {};
			\node (6) at (-1.5, -2.5) {};
			\node  (7) at (-0.5, -3) {};
			\node  (8) at (1, -3) {};
			\node  (9) at (0, 2) {};
			\node  (10) at (-6, -4) {};
			\node  (11) at (3, -5) {};
			\node  (12) at (3, 4) {};
		
			\end{pgfonlayer}
			\begin{pgfonlayer}{background}
			\draw [bend right, looseness=1.25] (0) to (2);
			\draw [bend right=60, looseness=1.25] (2) to (1);
			\draw [bend right] (0) to (1);
			\draw [in=165, out=-120, loop] (0) to ();
			\draw [bend right] (3) to (5);
			\draw [bend right=45, looseness=1.25] (5) to (4);
			\draw [in=120, out=-165, loop] (4) to ();
			\draw [in=135, out=60] (3) to (0);
			\draw [bend left] (3) to (9.center);
			\draw [bend left=60, looseness=0.75] (0) to (1);
			\draw [in=150, out=-60] (4) to (6.center);
			\draw [bend left=15] (6.center) to (7.center);
			\draw [bend right=15, looseness=1.25] (7.center) to (8.center);
			\draw [in=-75, out=30] (9.center) to (12.center);
			\draw [in=135, out=0, looseness=1.25] (8.center) to (11.center);
			\draw (5) to (10.center);
			
			\end{pgfonlayer}
			\end{tikzpicture}
			
		};

	\end{tikzpicture}

	\caption{Diagram (\ref{eq. dist 1 DT}) commutes. Here the vertices in $W$ are marked in green.}
	\label{fig. DTdist}
\end{figure}

Then, the two paths (illustrated in \cref{fig. DTdist}) described by the diagram (\ref{eq. dist 1 DT}) are induced by the following canonical maps (where, in each case, the vertical arrows are the defining universal morphisms):\\
{	\bf{Top-right}}
		\[\xymatrix @R =.25cm@C = .75cm{
			&(\X, (\Gg(b)_{\setminus W^b})_{b}, W) \ar@{..>}[dd] \ar@{..>}[rr]^-{\tiny{\text{ delete } W \text{ in } \X}} && (\X_{\setminus W}, \Gg_{\setminus W}) \ar@{..>}[dd] \ar@/^1pc/@{..>}[rd]^-{\tiny{\text{evaluate colimit } \Gg_{\setminus W} }}&\\
			(\X, (\Gg(b))_{b}, (W^b)_b) \ar@{..>}[dd]\ar@/^1pc/@{..>}[ur]^-{\tiny{ \text{ delete } W^b \text{ in } \Gg(b)}}&&&&\Gg_{\setminus W}(\X_{\setminus W}) \ar@{..>}[dd] \\
	&\mathrm{lim}_{\elG[\X]}S(\Gg(b)_{\setminus W^b})\ar[rr]\ar[dd]&&S(\Gg_{\setminus W}(\X_{\setminus W})\ar@/^1pc/[dr]\ar[dd]&\\
	\mathrm{lim}_{\elG[\X]}	S(\Gg(b)_{\setminus W^b})\ar@/^1pc/[ur]\ar[dd]	&&&&S(\Gg_{\setminus W}(\X_{\setminus W}))\ar[dd] \\
&TDTS_X \ar[rr]_-{\lambda_{\DD\TT} T S_X}	&& DT^2 S_X \ar@/^1pc/[rd]_-{ D \mu^\TT S_X }&\\
T^2 DS_X \ar@/^1pc/[ur]_-{T\lambda_{\DD\TT}S_X}&&&& DTS_X }
\]

{\bf{Left-bottom}}
	\[	\xymatrix{(\X, \Gg, W^{\Gg}) \ar@{..>}[d]\ar@{..>}[rrr]^-{\tiny{\text{ evaluate colimit } \Gg}}&&&(\Gg(\X), W^{\Gg})\ar@{..>}[d] \ar@{..>}[rrr]^-{\tiny{\text{ delete } W^{\Gg} \text{ in } \Gg(\X)}}&&&\Gg(\X)_{\setminus W^{\Gg}}\ar@{..>}[d]\\ 
	S(\Gg(\X)_{\setminus W^{\Gg}})\ar@{=}[rrr]\ar[d]&&& S(\Gg(\X)_{\setminus W^{\Gg}})\ar@{=}[rrr] \ar[d]&&&S(\Gg(\X)_{\setminus W^{\Gg}})\ar[d]\\
		T^2 DS_X \ar[rrr]_-{\mu^{\TT}DS_X}&&&TDS_X \ar[rrr]_-{\lambda_{\DD\TT}  S_X}	&&& DT S_X.}\]

These are equal since $S(\Gg_{\setminus W}(\X_{\setminus W})) = S(\Gg(\X)_{\setminus W^{\Gg}})$ by construction, and hence (\ref{eq. dist 1 DT}) commutes.

The proofs that $\lambda_{\DD\TT}$ satisfies the remaining three distributive law axioms follow similarly (and are somewhat simpler than the proof that (\ref{eq. dist 1 DT}) commutes.

 By \cite[Section~3]{Bec69}, there is a monad $\TTp= (\Tp, \mu^{\TTp}, \eta^{\TTp})$ on $\GSEp$ whose EM category of algebras is equivalent to the EM category of algebras for $\DD\TT$ on $\GSE$. 
In case $\E = \Set$, $\TTp$ has been described in detail in \cite[Section~7]{Ray20}, and the description of $\TTp$ given in \cite[Corollary~7.43]{Ray20} for graphical species in $\Set$, holds for graphical species in a general category $\E$ with sufficient (co)limits. 

Precisely, the endofunctor $\Tp$ is the quotient of $T$ given by
\begin{equation}\label{eq. Tp} \Tp S_\S = S_\S, \text { and } \Tp S_X = \mathrm{colim}_{\X \in X\XGrsimp} S(\X)  \end{equation}
for all pointed graphical species $S_* = (S, \epsilon, o)$ in $\E$. The unit $\epsilon ^{\Tp S}\defeq \Tp (\epsilon)  \colon S_\S \to \Tp S_\two$ is described by the obvious composite 
\[  \xymatrix{S_\S \ar[r]^-{\epsilon} &S_\two\ar[r]^- {\eta^\TT S}& TS_\two \ar@{->>}[r]& \Tp S_\two,}  \] and the contracted unit $o ^{\Tp S}\defeq \Tp (o)  \colon S_\S \to \Tp S_\nul$ is described by
\[  \xymatrix{S_\S \ar[r]^-{o}& S_\nul \ar[r]^-{\eta^\TT S} & TS_\nul \ar@{->>}[r]& \Tp S_\nul.} \]
In particular, $\epsilon^{\Tp S}$ is defined by maps $S(u^k) \colon S_\S \to S(\Lk)$ (for $k \geq 1$), and $o^{\Tp S}$ by maps $S(\kappa^m)\colon S_\S \to S(\Wm)$, $m \geq 1$ and $S( o) \colon S_\S \to S(\C_\nul)$ such that $o \circ S_\tau = o$.

The multiplication $\mu^{\TTp}$ for $\TTp$ is the one induced by $\mu^\TT$, that forgets the pair $(\X, \Gg)$ -- of an $X$-graph $\X$, and a non-degenerate $\X$-shaped graph of connected graphs $\Gg$ -- and replaces it with the colimit $\Gg(\X)$. The unit $\etap$ for $\TTp$ is defined by the maps $S_X =  S(\C_X) \mapsto \mathrm{colim}_{ \X \in X\XGrsimp} S(\X)$.

It is then straightforward to derive the following theorem (that was the main result, Theorem~7.46, of \cite{Ray20} in the case $\E = \Set$) for general $\E$.
\begin{thm}\label{CSM monad DT}
	The EM category $\GS^{\DD\TT}$ of algebras for $\DD\TT$ is canonically isomorphic to $\CSM$.
\end{thm}
\begin{proof}
	Let $(A,h)$ be a $\DD \TT$-algebra in $\E$. Then $h$ induces a $\TT$-algebra structure $h^\TT \colon TA \to A$ by restriction, and a $\DD$-algebra structure $h^\DD = h \circ D \eta^{\TT}A \colon DA \to A$. In particular, $h^\TT$ equips $A$ with the a multiplication $\diamond$ and contraction $\zeta$ such that $(A, \diamond , \zeta)$ is a non-unital modular operad in $\E$, and $h^\DD$, equips $A$ with the structure of a pointed graphical species $A_* = (A, \epsilon, o) $ in $\E$. To prove that $(A, h)$ describes a modular operad, we must show that $\epsilon$ is a unit for $\diamond.$ 
	
This is straightforward since $(A, h)$ is a $\DD\TT$ algebra and therefore $(A, \epsilon , o)$, together with $h^\TT$, describe a $\TTp$-algebra. In particular, $\epsilon$ is a unit for $\diamond$, by (\ref{eq. Tp}).

Conversely, if $(S, \diamond, \zeta, \epsilon)$ is a modular operad in $\E$, then $(S, \diamond, \zeta)$ defines a $\TT$-algebra $p^\TT \colon TS \to S$ by \cref{prop. Talg}, and $(S, \epsilon, \zeta \epsilon)$ defines a $\DD$-algebra by \cref{ssec. D monad}. In particular, since $\epsilon$ is a unit for $\diamond$, and $p^\TT$ is induced by iterations of $\diamond$ and $\zeta$, $p^\TT$ induces a $\TTp$-algebra structure on $(S, \epsilon, \zeta\epsilon)$ by (\ref{eq. Tp}), whereby $(S, \diamond, \zeta, \epsilon)$  describes an algebra for $\DD\TT$.
\end{proof}

\section{A monad for circuit operads}  \label{sec. monad CO}

Despite the similarity between the definitions of the monads $\TT$ and $\TTk$ on $\GSE$, it is not possible to modify $\lambda_{\DD\TT}$ by replacing $\TT$ with $\TTk$ to obtain a distributive law for $\DD$ and $\TTk$: 
Namely, for all graphical species $S$, $DS_\two\cong S_\two \amalg S_\S$ by definition. Hence there is a canonical morphism \[S(\shortmid \amalg \shortmid) = S_\S \times S_\S   \to DS_\two \times DS_\two =  DS(\C_\two \amalg \C_\two) \to TDS_{ \two \amalg \two}.  \] 
However, there is, in general, no morphism
\[ (S_\two \amalg S_\S)\times (S_\two \amalg S_\S) \to T S_{ \two \amalg \two}  = DTS_{\two \amalg \two}. \] Therefore $\lambda_{ \DD \TT} \colon TD \Rightarrow DT$ does not extend to a distributive law $\Tk D \Rightarrow D \Tk$.

The problem of finding a circuit operad monad on $\GSE$ is easily solved by observing that the monad $\TTk$ is itself a composite: the free graded monoid monad on $\pr{\Sigma}$ induces a monad $\LL = (L, \mu^\LL, \eta^\LL)$ on $\GSE$, (that will be described in \cref{ssec. LL}), whose algebras are graphical species in $\E$ that are equipped with an external product. Moreover, there is a distributive law $\lambda_{ \LL \TT} \colon TL \Rightarrow LT$, and the induced composite monad is $\TTk = \LL \TT$. 

It is then straightforward to check that, for all (unpointed) graphical species $S$ in $\E$, $LDTS$ describes a circuit operad in $\E$. By \cite{Che11}, there is a distributive law $\lambda_{ \LL(\DD\TT)} \colon (DT)L \Rightarrow L(DT)$, and hence a composite monad $\LL \DD \TT$ on $\GSE$, if and only if there are pairwise distributive laws $\lambda_{ \LL \TT} \colon TL \Rightarrow LT$, and $\lambda_{ \LL\DD} \colon DL \Rightarrow LD$, such that the following \textit{`Yang-Baxter' diagram} of endofunctors on $\GSE$ commutes:
	\begin{equation}\label{YB}
	\xymatrix{
		&& DTL \ar[rr]^-{D\lambda_{\LL\TT}} && DLT \ar[drr]^-{\lambda_{\LL\DD}T}&& \\
		TDL \ar[urr]^-{\lambda_{\DD\TT}L}\ar[drr]_-{T\lambda_{\LL\DD}}&&  && && LDT.\\
		&& TLD \ar[rr]_-{\lambda_{\LL\TT}D}&& LTD \ar[urr]_-{L\lambda_{\DD\TT}}&& }
	\end{equation}

In this section, we construct the monad $\LL$ and show that there is a distributive law $\lambda_{\LL\TT}: TL \Rightarrow LT$ such that the induced composite monad is precisely $\TTk$. 

The following section gives a description of the distributive law $\lambda_{\LL \DD}: DL \Rightarrow LD $ such that the diagram (\ref{YB}) commutes, and proves that the EM category of algebras for the composite monad $\LL\DD \TT$ on $\GSE$ is canonically equivalent to the category $\COE$ of circuit operads in $\E$.

\subsection{The free external product monad $\LL$}\label{ssec. LL}

For any finite set $X$, recall (\cref{ssec. presheaves}) that $\core{X \ov \fin}$ is the maximal groupoid in $X \ov \fin$ whose objects are pairs $(Y, f)$ with $f \in \fin (X, Y)$ a morphism of finite sets. 

As usual, let $\E$ be a category with sufficient (co)limits. The free graded monoid endofunctor $L$ on $\GSE$ is given by 
\begin{equation}\label{eq. defn L}\begin{array}{lll}
LS_\S &= & S_\S \\ 
LS _X  &= & \mathrm{colim}_{(Y,f) \in \core{X \ov \fin}} \prod_{y \in Y} S_{  f^{-1}(y)}
\end{array}\end{equation}
for all finite sets $X$. This underlies a monad $\LL  = (L, \mu^\LL, \eta^\LL)$  with monadic multiplication $\mu^\LL$ induced by concatenation of lists, and monadic unit $\eta^\LL$ induced by inclusion of singleton lists. 

The following lemma is immediate:
\begin{lem}\label{lem. L free monoid}
The EM category of algebras for the monad $\LL$ on $\GSE$ is equivalent to the category of graphical species with unital external product in $\E$.
	\end{lem}

In order to describe a distributive law $\lambda_{\LL \TT}\colon TL \Rightarrow$ for $\LL$ and $\TT$, observe first that $L \subset \Tk$ is a subfunctor since, if $S$ is a graphical species in $\E$, and $X$ a finite set, then 
\[ \begin{array}{lll}
LS_X  & =  & \mathrm{colim}_{(Y,f) \in \core{X \ov \fin}} \prod_{y \in Y}S_{  f^{-1}(y)}\\
& = & \mathrm{colim}_{\X \in  \Coriso} S(\X)
\end{array}\] where $\Coriso\subset X\Grisok$ is the groupoid of $X$-graphs $\X_{\mathrm{cor}}$ that are disjoint unions of corollas $\X_{\mathrm{cor}}= \coprod_{ i \in I} \CX[X_i]$ (with $\coprod_{i \in I} X_i = X$). In particular, from this point of view, the multiplication $\mu^\LL$ for $\LL$ is just the restriction of the multiplication $\mu^{\TTk}$ for $\TTk$ (\cref{unpointed modular operad}) to $\X_{\mathrm{cor}}$-shaped graphs of disjoint unions of corollas, 
and the unit $\eta^\LL$ for $\LL$, is the obvious corestriction. 

Given a finite set $X$, and a connected $X$-graph $\X \in X \Griso$, let $\CorGg$ denote the full subgroupoid of $\Gret^{(\X)}$ whose objects are non-degenerate $ \X$-shaped graphs of graphs $\Gg$ such that, for all $(\C_{X_b}, b) \in \elG[\X]$, $ \Gg (\C_{X_b}, b) \in \Coriso$ is a disjoint union of corollas. Then, \[ \begin{array}{lll}
TL S_X& = &\mathrm{colim}_{ \X \in X \Griso}  \mathrm{lim}_{(\C,b) \in \elG[\X]} LS(\C)\\
&=& \mathrm{colim}_{ \X \in X \Griso}  \mathrm{lim}_{ \Gg \in\CorGg} S(\Gg(\X)).
\end{array}\]

Since, for all $\Gg \in \CorGg$, the colimit $\Gg(\X)$ has the structure of an admissible (in general, not connected) $X$-graph, there is a canonical morphism $TLS_X \to \mathrm{colim}_{\X' \in X \Grisok} S(\X')$. 
And, therefore, since every graph is the coproduct in $\Gret$ of its connected components, $S(\X')= \prod_{k \in K}S(\X'_k)$ describes an element of $LTS_X$ for all $X$-graphs $\X'=  \coprod_{k \in K} \X'_k$.

Let $\lambda_{\LL \TT}\colon TL \Rightarrow LT$ be the natural transformation so obtained. In other words, $\lambda_{ \LL \TT}$ is induced by the canonical morphisms, 
\[\mathrm{colim}_{ \X \in X \Griso } \left( \mathrm{colim}_{ \Gg \in \CorGg} S (\Gg(\X)) \right)\to \mathrm{colim}_{ \X' \in X \Griso } S (\X') \  (\text{where } S\in \GSE, \text{ and } X \in \fiso),\] that forgets pairs
 $(\X, \Gg)$ where $\X$ is a connected admissible $X$-graph and $\Gg \in\CorGg$, and remembers only the $X$-graphs given by the colimit $\Gg(\X) = \coprod_{k \in K} \Gg(\X)_k$, in the evaluation of $S(\Gg(\X))$.

\begin{lem}\label{lem. LT is Tk}
The natural transformation $\lambda_{\LL \TT} \colon TL \Rightarrow LT$ satisfies the four axioms of a distributive law, and $\LL \TT = \TTk$ on $\GSE$. Hence, algebras for $\LL\TT$ are non-unital circuit operads.
\end{lem}
\begin{proof}
Since a graph is the disjoint union of its connected components, $LT = \Tk  \colon \GSE \to \GSE$. 

Moreover, $L \subset \Tk$ and $T \subset \Tk$ so $TL \subset ({\Tk})^2$ is a subfunctor, and $\lambda_{\LL\TT}$ is just the restriction to $TL$ of $\mu^{\TTk} \colon ({\Tk})^2 \Rightarrow \Tk = LT$ on $\GSE$. The four distributive law axioms then follow immediately from the monad axioms.
\end{proof}

\subsection{An iterated distributive law for circuit operads}\label{sec. iterated}
So far, we have defined three monads $\TT, \DD$ and $\LL$ on $\GSE$, and two distributive laws $\lambda_{\DD\TT}\colon TD \Rightarrow DT$ and $\lambda_{\LL\TT}\colon TL\Rightarrow LT$. 

Observe also that there is an obvious choice of natural transformation $\lambda_{ \LL \DD} \colon DL \Rightarrow LD$ that defines the distributive law for $\LL$ and $\DD$: By definition of $D$, we may extend the natural transformation 
$L \eta^{\DD} \colon L \Rightarrow LD$ to a natural transformation $\lambda_{\LL \DD} \colon DL \Rightarrow LD$ by 
\[ S_\S \to DS_\two \xrightarrow {\eta^\LL DS} LDS_\two, \ \text{ and } S_\S \to DS_\nul \xrightarrow {\eta^\LL DS} LDS_\nul, \ \text{ for all } S \in \GSE.\]

Informally, $\lambda_{\LL \DD}$ takes a tuple $(S_{X_i})_{i = 1}^n $ of objects in the image of $S$, and regards it as a tuple in the image of $DS$, and takes the image of $S_\S$ in $DS_\nul$ or $DS_\two$ and regards it as a 1-tuple in the image of $DS$.

This construction is natural in $S$, and it is simple to show that $\lambda_{\LL \DD}$ so obtained describes a distributive law for $\LL$ and $\DD$.

It remains to show that there is an iterated distributive law for the three monads, and that the resulting composite monad $\LL \DD \TT $ on $\GSE$ is the desired circuit operad monad. 

\medspace

Observe first that, for any graph $\G$ and any subset $W \subset V_2 $ of bivalent vertices of $\G$, there is a canonical full embedding $\Gret^{(\Gnov)}\hookrightarrow \Gret^{(\G)}$ of categories of non-degenerate graphs of graphs, that preserves similarity classes of colimits:   

Namely, if  $\Gg_{\setminus W} \colon \elG [\Gnov]\to \Gret$ is a non-degenerate $\Gnov$-shaped graph of graphs, then there is an induced non-degenerate $\G$-shaped graph of graphs $\Gg^W \colon \elG \to \Gret$ defined by
\begin{equation}
\label{eq. GgW}
(\C,b) \mapsto \left \{ \begin{array}{ll}
\C_\two & (\C,b) \text{ is a neighbourhood of } v \in W\\
\Gg_{\W} (\delW \circ b) & \text{ otherwise,}
\end{array} \right .
\end{equation} 
and such that $\Gg^W$ induces an inclusion $W \hookrightarrow V_2(\Gg^W (\X))$ (see \cref{fig. GgW}), and   
 \begin{equation}
 \label{eq. deleted colims} \Gg^W(\X)_{\setminus W} = \Gg_{\setminus W}(\X_{\setminus W}) .
 \end{equation} 

\begin{figure}[htb!]
	\begin{tikzpicture} [scale = .8]
	\node at (-2,-4){$\Gg_{\setminus W}$};
	\node at(0,-5){ 
		\begin{tikzpicture}[scale = .3]
\begin{pgfonlayer}{above}
\node [dot, red] (0) at (2.5, -0.5) {};
\node [dot, red] (1) at (3, -1) {};
\node [dot, red] (3) at (-1, -0.25) {};
\node [dot, red] (4) at (-0.25, -1) {};
\node [dot, red] (5) at (-1.75, -1) {};
\node  (7) at (1.5, -3) {};
\node  (9) at (2, 2) {};
\node  (10) at (-4, -4) {};
\node  (11) at (5, -5) {};
\node  (12) at (5, 4) {};
\node  (13) at (7, 0) {};
\draw [draw = blue, fill = blue, fill opacity = .2](2.7,-.75) circle (1.3cm);
\draw [draw = blue, fill = blue, fill opacity = .2](-1, -.5) circle (1.3cm);
\end{pgfonlayer}
\begin{pgfonlayer}{background}
\draw [in=165, out=-120, loop] (0) to ();
\draw [bend right=45, looseness=1.25] (5) to (4);
\draw [in=135, out=60] (3) to (0);
\draw [bend left] (3) to (9.center);
\draw [bend right, looseness=1.25] (9.center) to (12.center);
\draw [bend right] (5) to (10.center);
\draw [bend left] (4) to (7.center);
\draw [in=150, out=-90, looseness=0.75] (7.center) to (11.center);
\draw [in=180, out=60] (1) to (13.center);
\end{pgfonlayer}
		\end{tikzpicture}
		
	};
	\node at (6,-4){$\Gg^W$};
\node at(8,-5){ 
	\begin{tikzpicture}[scale = .3]
	\begin{pgfonlayer}{above}
	\node [dot, red] (0) at (2.5, -0.5) {};
	\node [dot, red] (1) at (3, -1) {};
	\node [dot, red] (3) at (-1, -0.25) {};
	\node [dot, red] (4) at (-0.25, -1) {};
	\node [dot, red] (5) at (-1.75, -1) {};
	\node [dot, green] (7) at (1.5, -3) {};
	\node [dot, green] (9) at (2, 2) {};
	\node  (10) at (-4, -4) {};
	\node  (11) at (5, -5) {};
	\node  (12) at (5, 4) {};
	\node  (13) at (7, 0) {};
	\draw [draw = blue, fill = blue, fill opacity = .2](2.7,-.75) circle (1.3cm);
	\draw [draw = blue, fill = blue, fill opacity = .2](-1, -.5) circle (1.3cm);
		\draw [draw = green, fill = green, fill opacity = .2](7) circle (.6cm);
			\draw [draw = green, fill = green, fill opacity = .2](9) circle (.6cm);
	\end{pgfonlayer}
	\begin{pgfonlayer}{background}
	\draw [in=165, out=-120, loop] (0) to ();
	\draw [bend right=45, looseness=1.25] (5) to (4);
	\draw [in=135, out=60] (3) to (0);
	\draw [bend left] (3) to (9.center);
	\draw [bend right, looseness=1.25] (9.center) to (12.center);
	\draw [bend right] (5) to (10.center);
	\draw [bend left] (4) to (7.center);
	\draw [in=150, out=-90, looseness=0.75] (7.center) to (11.center);
	\draw [in=180, out=60] (1) to (13.center);
	\end{pgfonlayer}
	\end{tikzpicture}
	
};
	\node at (-2,1){$\Gnov$};
	\node at(0,0){ 
	\begin{tikzpicture}[scale = .3]
	\begin{pgfonlayer}{above}
	\node [dot, blue] (0) at (2.5, -0.5) {};
	\node [dot, blue] (1) at (2.5, -0.5) {};
	\node [dot, blue] (3) at (-1, -0.5) {};
	\node [dot, blue] (4) at (-1, -0.5) {};
	\node [dot, blue] (5) at (-1, -0.5) {};
	\node  (7) at (1.5, -3) {};
	\node  (9) at (2, 2) {};
	\node  (10) at (-4, -4) {};
	\node  (11) at (5, -5) {};
	\node  (12) at (5, 4) {};
	\node  (13) at (7, 0) {};
	\end{pgfonlayer}
	\begin{pgfonlayer}{background}
	\draw [in=135, out=60] (3) to (0);
	\draw [bend left] (3) to (9.center);
	\draw [bend right, looseness=1.25] (9.center) to (12.center);
	\draw [bend right] (5) to (10.center);
	\draw [bend left] (4) to (7.center);
	\draw [in=150, out=-90, looseness=0.75] (7.center) to (11.center);
	\draw [in=180, out=60] (1) to (13.center);
	\end{pgfonlayer}
	\end{tikzpicture}	
};
	\node at (6,1){$\G$};
	\node at(8,0){ 
	\begin{tikzpicture}[scale = .3]
	\begin{pgfonlayer}{above}
	\node [dot, blue] (0) at (2.5, -0.5) {};
	\node [dot, blue] (1) at (2.5, -0.5) {};
	\node [dot, blue] (3) at (-1, -0.5) {};
	\node [dot, blue] (4) at (-1, -0.5) {};
	\node [dot, blue] (5) at (-1, -0.5) {};
	\node [dot, green] (7) at (1.5, -3) {};
	\node  [dot, green](9) at (2, 2) {};
	\node  (10) at (-4, -4) {};
	\node  (11) at (5, -5) {};
	\node  (12) at (5, 4) {};
	\node  (13) at (7, 0) {};
	\end{pgfonlayer}
	\begin{pgfonlayer}{background}
	\draw [in=135, out=60] (3) to (0);
	\draw [bend left] (3) to (9.center);
	\draw [bend right, looseness=1.25] (9.center) to (12.center);
	\draw [bend right] (5) to (10.center);
	\draw [bend left] (4) to (7.center);
	\draw [in=150, out=-90, looseness=0.75] (7.center) to (11.center);
	\draw [in=180, out=60] (1) to (13.center);
	\end{pgfonlayer}
	\end{tikzpicture}	
};
\end{tikzpicture}

\caption{The colimits of the $\G$-shaped graph of graphs $\Gg^W$ described in (\ref{eq. GgW}), and the $\Gnov$-shaped graph of graphs $\Gg_{\setminus W}$ are similar in $\Gretp$. }
	\label{fig. GgW}
\end{figure}
 In particular, the embedding $\Gret^{(\Gnov)}\hookrightarrow \Gret^{(\G)}$ restricts to an inclusion  $\mathsf{Cor}_\boxtimes^{(\Gnov)}\hookrightarrow \mathsf{Cor}_\boxtimes^{(\G)}$. 
 
 This will be used to prove:

\begin{prop}\label{prop. iterated law}
	The triple of distributive laws $\lambda_{ \DD \TT}\colon TD \Rightarrow DT$, $\lambda_{ \LL \TT}\colon TL \Rightarrow LT$, and $\lambda_{ \LL \DD }\colon DL  \Rightarrow LD$ on $\GS$, describe a composite monad $\LL \DD\TT$ on $\GSE$. 

\end{prop}
\begin{proof}
By \cite[Theorem~1.6]{Che11}, we must check that the diagram (\ref{YB}) commutes for $\lambda_{ \DD \TT}, \lambda_{ \LL \TT}$, and $\lambda_{ \LL \DD }$:

For all graphical species $S$ in $\E$ and all finite sets $X$, 
\[
TDLS_X  = \mathrm{colim}_{ \X \in X \Griso} \coprod_{W \subset (V_0 \amalg V_2)(\X)} \mathrm{lim}_{ \Gg_{\setminus W}\in \mathsf{Cor}_\boxtimes^{\scriptscriptstyle{(\X_{\setminus W})}}} S(\Gg_{\setminus W}(\X_{\setminus W})) \] 
is indexed by triples 
$(\X, W, \Gg _{\setminus W})$, where $\X$ is a connected $X$-graph, $W \subset V_2(\X)$ (or $\X \cong \C_\nul$ and $W= V(\C_\nul)$) and $\Gg_{\setminus W} $ is a non-degenerate $\X_{\setminus W}$-shaped graph of graphs such that $\Gg_{\setminus W} (b)$ is a disjoint union of corollas for each $(\C_{X_b}, b) \in \elG[\X_{\setminus W}]$.

Observe first that, if $ W = V(\X)$, then either $\X \cong \C_\nul$ or $\X \cong \Wm$ in $\nul \Griso$, or $\X \cong \Lk $ in $ \two \Griso$. In each case $\X_{\setminus W} \cong (\shortmid)$ and hence $\Gg_{\setminus W}$ is the trivial graph of graphs $id \mapsto (\shortmid)$. In particular, the maps $S(\Gg_{\setminus W} (\X_{\setminus W})) \to TDLS_X \to LDTS_X$ induced by each path $TDL \Rightarrow  LDT$ in (\ref{YB}) coincide, since both factor through $S_\S$. 

Assume therefore that $W \neq V(\X)$. The two paths in (\ref{YB}) are then described by the following data:

\begin{description}
	\item[Top]
		\[\small{\xymatrix@C = .8cm{
			&	{	\begin{array}{c}
				(\X_{\setminus W}, \Gg_{\W}),\\ \X_{\setminus W}\in X \Griso
				\end{array}}
			\ar@{..>}[rrr]^-{\tiny{\text{ take colimit of } \Gg_{\setminus W}}}\ar@{..>}[d]&&& 	{	\begin{array}{c}
			\Gg_{\setminus W}(\X_{\setminus W})\\ \in X \Grisok
				\end{array}}	\ar@/^1.0pc/@{..>}[rd]^-{L\eta^{\DD}T}\ar@{..>}[d]&\\
		{	\begin{array}{c}
	(\X, W,  \Gg_{\setminus W}),\\ \X \in X \Griso
			\end{array}}\ar@/^1.0pc/@{..>}[ur]^-{\tiny{\text{delete } W \text{ in } \X}}\ar@{..>}[d] &S(\Gg_{\setminus W}(\X_{\setminus W}))\ar[d]\ar[rrr]&&&
		S(\Gg_{\setminus W}(\X_{\setminus W}))\ar[d]\ar@/^1.0pc/[rd]& {	\begin{array}{c}
			\Gg_{\setminus W}(\X_{\setminus W}) \\  \in X \Grisok
			\end{array}} \ar@{..>}[d]
			\\
		 S(\Gg_{\setminus W}(\X_{\setminus W}))\ar[d]	\ar@/^1.0pc/[ru]& DTLS_X\ar[rrr]_-{D\lambda_{\LL\TT}S} &&& DLTS_X \ar@/^1.0pc/[rd]_-{\lambda_{\LL\DD}TS}& S(\Gg_{\setminus W}(\X_{\setminus W}))\ar[d]\\
			TDLS_X	\ar@/^1.0pc/[ur]_-{\lambda_{ \DD \TT}LS} &&&&  & LDTS_X}}\]
where, in each case, we evaluate $S(\Gg_{\setminus W}(\X_{\setminus W}))$ and take the appropriate colimit of such objects.
	\item[Bottom]
The bottom path of (\ref{YB}) gives
	\[ \small{\xymatrix@C = .8cm{ {	\begin{array}{c}
					(\X, W,  \Gg_{\setminus W}) ,\\ \X\in X \Griso
				\end{array}}
	\ar@{..>}[d]	\ar@/_1.0pc/@{..>}[dr]^-{\tiny{\text{ forget } W}}
&&&&&{	\begin{array}{c}
\Gg^W(\X)_{\setminus W} \\ \in X \Grisok
	\end{array}} 
 \ar@{..>}[d] \\
		S( \Gg_{\setminus W}(\X_{\setminus W}))\ar[d]	\ar@/_1.0pc/[dr]&   {	\begin{array}{c}
			(\X, \Gg^{W}, W)  ,\\ \X\in X \Griso
			\end{array}}  \ar@{..>}[d]  \ar@{..>}[rrr]^-{\tiny{\text{ take colimit of } \Gg^W}}  &&& {	\begin{array}{c}
		(\Gg^W(\X), W)  ,\\ \Gg^W(\X)\in X \Grisok
			\end{array}}  \ar@/_1.0pc/@{..>}[ru]^-{\tiny{\text{ delete } W \text{ in } \Gg^W(\X)}} \ar@{..>}[d]&
	S(\Gg^W(\X)_{\setminus W})\ar[d]	\\
	TDLS_X\ar@/_1.0pc/[dr]_{T \lambda_{\DD\LL}S} & S(\Gg^W(\X)_{\setminus W})\ar[d] \ar[rrr]&&&S(\Gg^W(\X)_{\setminus W})\ar@/_1.0pc/[ru]\ar[d]&LDTS_X \\
	& TLDS_X\ar[rrr]_-{\lambda_{\LL\TT}DS}&&& LTDS_X \ar@/_1.0pc/[ru]_-{L \lambda_{\DD\TT} S}&}}\]
where $\Gg^W$ is the non-degenerate $\X$-shaped graph of graphs in $\mathsf{Cor}_\boxtimes^{(\X)}$ obtained from $\Gg_{/W}$ as in (\ref{eq. GgW}).
\end{description}

So, $ \Gg_{\setminus W}(\X_{\setminus W}) = \Gg^W(\X)_{\setminus W}$ by (\ref{eq. deleted colims}). Hence, the diagram (\ref{YB}) commutes, and $(\lambda_{\LL\DD}, \lambda_{\LL\TT}, \lambda_{\DD\TT})$ describe an iterated distributive law for $\LL, \DD$ and $\TT$ on $\GSE$.
\end{proof}

It therefore remains to prove:
\begin{thm}\label{iterated law}
The EM category of algebras for the composite monad $\LL \DD\TT$ on $\GSE$ is isomorphic to the category $\COE$ of circuit operads in $\E$. \end{thm}

\begin{proof}
By \cite{Bec69}, for a composite monad is equipped with an algebra structure for each of the component monads.  Hence, an algebra $(A, h)$ for the composite monad $\LL \DD \TT$ on $\GSE $ has an $\LL \TT$-algebra structure $h^{\LL\TT} \colon LTA \to A$, and also a $\DD \TT$ algebra structure $h^{\DD\TT} \colon DTA \to A$. Each structure is given by the corresponding restriction $LTA \to LDT A \xrightarrow{h} A$, respectively $DTA \to LDT A \xrightarrow{h} A$. Therefore, by \cref{lem. LT is Tk}, $h$ induces an external product $\boxtimes$, and contraction $\zeta$ on $A$ such that $(A, \boxtimes, \zeta)$ is a non-unital circuit operad. And, by \cref{CSM monad DT}, $h$ induces a unital multiplication $(\diamond, \epsilon)$ on $A$, such that 
\[\diamond_{X,Y}^{x \ddagger y}  = \zeta_{X \amalg Y}^{x \ddagger y}\boxtimes_{X, Y}\colon (A_X \times A_Y)^{x\ddagger y } \to A_{X \amalg Y \setminus \{x, y\}}.\] Hence, 
\[  \xymatrix{ A_X \ar[rr]^-{ (id \times ch_x)\circ \Delta}&&( A_X \times A_\S )^{x \ddagger 2}\ar[rr] ^-{ id \times \epsilon} &&(A_X \times A_\two)^{x \ddagger 2} \ar[rr]^-{\zeta^{x \ddagger 2}_{X \amalg \two}\circ \boxtimes}&&
	A_X }\] is the identity on $A_X$. And therefore, $(A, h))$ has a canonical circuit operad structure.

Conversely, a circuit operad $(S, \boxtimes, \zeta, \epsilon)$ in $\E$ has an underlying modular operad structure $(S, \diamond, \zeta, \epsilon)$, where $\diamond_{X,Y}^{x \ddagger y}  = \zeta_{X \amalg Y}^{x \ddagger y}\boxtimes_{X, Y}$ and hence $(S, \boxtimes, \zeta, \epsilon)$ defines a $\DD \TT$-algebra structure $h^{\DD\TT} \colon DTS \to S$ by \cref{CSM monad DT}. By \cref{lem. L free monoid}, $(S, \boxtimes)$ has an $\LL$-algebra structure where $h^\LL \colon LS \to S$ is given by $\boxtimes_{X,Y} \colon S_X \times S_Y$ on $S_X \times S_Y$. It follows immediately from the circuit operad axioms that $h^\LL \circ L h^{\DD\TT} \colon LDT S \to S$ gives $S$ the structure of an $\LL\DD \TT$-algebra.
                                                                   
These maps are functorial by construction and extend to an equivalence $\COE \simeq \GSE^{\LL\DD\TT}$.\end{proof}

\begin{thm}\label{thm. CO CA}
For each palette $(\CCC, \omega)$, there is an equivalence of the categories $\CCA \simeq \CCO$ of $(\CCC, \omega)$-coloured circuit algebras and circuit operads in $\Set$. This extends to an equivalence  $\CA\simeq \CO$ of the categories of all circuit algebras and all circuit operads in $\Set$. 
\end{thm}
\begin{proof}
By \cref{prop. CA CO}, there is a faithful functor $\CA \to \CO$, that restricts to a functor $ \CCA \to \CCO$ for all palettes $(\CCC, \omega)$. 

For the converse, let $(A, h)$ be an algebra for $\LL \DD\TT$, with $(A_\S, A_{\tau}) = (\CCC, \omega)$, and let $(A, \boxtimes, \zeta, \epsilon)$ be the underlying circuit operad, with contracted unit $ o = \zeta \circ \epsilon \colon \CCC \to A_{\nul}$.  

Let $\ddd  = (d_1, \dots, d_n) \in  \CCC^{\nn}$, and, for $k \in \N$, and each $1 \leq i \leq k$, $\ccc_i = (c_{i, 1}, \dots, c_{i,m_i}) \in \CCC^{\mm_i}$, be elements of $\listm \CCC$, and let $(\overline f, \lambda)$ in $\CWD(\ccc_1, \dots, \ccc_k;\ddd)$ be a wiring diagram with underlying uncoloured Brauer diagram $f = (\tau,\kcl) \in \BD(\sum_{i = 1}^k m_i, n)$. 

By (\ref{eq. wiring graph}), the open part $(\tau, 0)$ of $\overline f$ describes an $\nn$-graph $\mathring \X = \mathring \X (\overline f)$ with $k$ ordered vertices $(v_i)_{ i  = 1}^k$ and a distinguished \'etale neighbourhood $(\C_{\mm_i}, b_i) \in \elG[\mathring \X]$ of each $v_i$. 
Moreover, the colouring $\lambda$ of $ \overline f$ describes an involution-preserving map $\colop \colon E (\mathring \X) \to \CCC$, such $j \mapsto d_j$ for each port $1 \leq j \leq n$ of $\mathring \X$, and $b_i(l_i) \mapsto c_{i,l_i}$ for each port $1 \leq l_i \leq m_i$ of $\C_{\mm_i}$, $1 \leq i \leq k$. 

So, for any $k$-tuple $(\phi_1, \dots, \phi_k) \in A_{\ccc_1} \times \dots \times A_{\ccc_k}$ there is an element $\alpha \in A(\mathring \X)$ defined by $ A(b_i) (\alpha)= \phi_i \in A(\C_{\mm_i})$ for each $v_i$, and $ A(ch_e) (\alpha)= \colop (e) $ for each edge $e \in E(\mathring \X)$. Since $LTA\subset LDTA$, the assignment 
$ (\phi_1, \dots, \phi_k )\mapsto \alpha \in A(\mathring \X) $ extends to a map 
\[ A( \fopen, \lambda )\colon A_{\ccc_1} \times \dots \times A_{\ccc_k} \to A (\mathring \X) \to LTA_{\nn} \to LDTA_{\nn} \xrightarrow {h} A_{\nn}.\]

Moreover, the restriction of $\colcl \colon \pi_0(f) \to \widetilde{\CCC}$ to the set of $\kcl$ closed components of $f$ defines a $\kcl$-tuple $(\tilde c^1, \dots, \tilde c^{\kcl})$ in $ {\widetilde \CCC}$, and hence an element 
\[ o(\kcl, \colcl) \defeq  o_{\tilde c^1} \boxtimes \dots \boxtimes o_{\tilde c^{\kcl}}  \in A_{\nul}.\] 

Hence, there is a morphism of sets $A (\overline f, \lambda) \colon A_{\ccc_1}\times \dots \times A_{\ccc_k}\to A_{\ddd }$ given by 
\[ (\phi_1, \dots, \phi_k) \mapsto \left(A(\overline \fopen)(\phi_1, \dots, \phi_k)\right) \boxtimes   o(\kcl, \colcl)  \in A(\ddd).\]
Let $\X = \X(\overline f) \defeq \mathring \X \amalg \W \amalg \dots \amalg \W$ be the disjoint union of $ \mathring \X$ with $\kcl$ copies of the wheel graph $\W$. Since $o_{\tilde c} = h (A (\kappa) (c))$ for all $ c \in \CCC$,
\[o(\kcl, \colcl)  =  h \left (A(\kappa \amalg \dots \amalg \kappa ) (c^1 \boxtimes \dots \boxtimes c^{\kcl} )\right ) \text { with } A(\kappa \amalg \dots \amalg \kappa ) (c^1 \boxtimes \dots \boxtimes c^{\kcl} ) \in A (\W \amalg \dots \amalg \W).\]
Therefore $ A (\overline f, \lambda)$ is precisely the composite
\[  A_{\ccc_1} \times \dots \times A_{\ccc_k} \xrightarrow{A( \fopen, \lambda ) \amalg \left(\coprod_{i = 1}^\kcl \W(e_{c^i})\right) } A (\X) \longrightarrow LDT(A)_{\nn} \xrightarrow {h} A_{\nn}\] that restricts to $ A (\fopen, \lambda)$ on $A (\mathring \X) $.

The distinguished element $\cup_c \in \CWD( (\omega c,  c); \varnothing_{\CCC})$ has $\X (\cup_c) = \W$, and induces a morphism $A({\cup_c}) \colon A_{ (\omega c, c)} \to A_\nul$ by 
\[ A_{\cup_c} \colon A_{ (\omega c, c)} \longrightarrow  A(\W) \longrightarrow LDT A(\W) \xrightarrow {h} A_\nul.\]
Slightly more care is needed for $\cap_c \in \CWD(-; (c, \omega c))$, with $ \X (\cap_c) = (\shortmid)$. Then, the morphism $ A(\cap_c) \colon I \to A _{ (c, \omega c)}$ is induced by $ c \in A(\shortmid) \xrightarrow{\epsilon_c} A _\two$.

If $(\overline f, \lambda) \in \CWD (\ccc_1, \dots, \ccc_k; \ddd)$ and $\mathring \X$ are as before, let $(\overline f^i, \lambda^i)\in \CWD (\bbb_{i, 1}, \dots, \bbb_{i, k_i};\ccc_i)$ with $\bbb_{i,j} = (b_{i,j,1}, \dots, b_{i, j, K_{i,j} }) \in \CCC^{K_{i,j}}$, for $1 \leq i \leq k $, $1 \leq j \leq k_i$. 

For each $ 1\leq i \leq k$, let
\[ A (\overline f^i, \lambda^i) \colon A_{\bbb_{i,1} }\times \dots \times A_{\bbb_{i,k_i}} \to A (\X^i) \to LT(A)_{\mm_i} \to LDT(A)_{\mm_i} \xrightarrow {h} A_{\mm_i}\] be the induced morphism with image in $A_{\ccc_i} \subset  A_{\mm_i}$. 
	Then the collection of composable wiring diagrams $\left((\overline f, \lambda), (\overline f^i, \lambda^i)_i\right)$, describes an $\X$-shaped graph of graphs $\Gg\colon \elG[\X] \to \Gret$, $(\C_{\mm_i}, b^i) \mapsto \X^i = \X(\overline f^i)$ (that is non-degenerate if and only if none of the graphs $\X^i $ have stick components). By the monad algebra axioms, and composing with $\epsilon$ where necessary, $ \Gg$ induces an element of $(LDT)^2A_{\ddd}$. Hence there are morphisms 
	\[A \left(\gamma( \overline f, \lambda),( \overline f^i, \lambda^i)_{ i}\right) \colon 
	\prod_{i = 1}^k \prod_{ j = 1}^{k_i}A_{\bbb_{i,j}} \longrightarrow \prod_{i = 1}^k A _{\ccc_i} \longrightarrow A_{\ddd} .\]
	In particular, the composition $\bigcirc_{\tilde c} = \cup_c \circ \cap_{\omega c}$ corresponds to the morphism $I \to A_{(c, \omega c)} \to A_\nul$ given by $o_{\tilde c }  = \zeta \epsilon_c \colon I \to A_{(c, \omega c)} \to A_\nul$. 
	 
	The monad algebra axioms, and the definition of the distributive law $\lambda_{\DD\TT}\colon  	TD \Rightarrow DT$,  imply that this defines a circuit algebra structure on the collection $(A_{\ccc})_{\ccc \in \CCC}$.

It is straightforward to verify that this assignment is natural in $(A, h)$, and hence extends to a functor $\CO \to \CA$ such that, moreover, the pair of functors $\CA \leftrightarrows \CO$ in fact define an (adjoint) equivalence of categories. 
\end{proof}

 \section{A nerve theorem for circuit operads}\label{sec. nerve}

This section recalls the nerve theorem for $\Set$-valued modular operads \cite[Theorem~8.2]{Ray20}, and gives a proof, by the same methods and using \cref{thm. CO CA}, of an analogous result for $\Set$-valued circuit algebras.

\subsection{Overview of the method}
Let $\Klgr$ be the category, whose objects are indexed by connected graphs, obtained in bo-ff factorisation of the functor $\Gr \to \GS \to \CSM$. Let $\Grp \hookrightarrow \GSp$ be the full subcategory of connected graphs and pointed \'etale morphisms from \cref{ss. pointed}, that is obtained in the bo-ff factorisation of the functor $\Gr \to \GS \to \GSp$, and, as in \cref{ssec. Similar connected X-graphs}, let $\TTp$ be the monad on $\GSp$ whose EM category of algebras is equivalent to the category $\CSM$ of modular operads (in $\Set$).

So, there is a commuting diagram of functors
 
\begin{equation} \label{eq: MO nerve big picture}
\xymatrix{ 
	&&\Klgr\ar@{^{(}->} [rr]_-{\text{f.f.}}						&& \CSM\ar@<2pt>[d]^-{\text{forget}^{\TTp}} \ar [rr]^-N				&&\pr{\Klgr}\ar[d]^{j^*}	\\
	\fisinvp \ar@{^{(}->} [rr]	^-{\text{}}_-{\text{ f.f.}}		&& \Grp \ar@{^{(}->} [rr] ^-{\text{}}_-{\text{ f.f.}} \ar[u]^{j}_{\text{b.o.}}			&& \GSp \ar@<2pt>[d]^-{\text{forget}^{\DD}} \ar@{^{(}->} [rr]^-{\text{   }}_-{\text{f.f.}}	 \ar@<2pt>[u]^-{\text{free}^{\TTp}}	&& \pr{\Grp}\ar[d] \\
	\fisinv\ar@{^{(}->} [rr]^-{\text{}}_-{\text{ f.f.}}\ar[urr]_-{\text{}}\ar[u]^{\text{b.o.}}	&& \Gr \ar@{^{(}->} [rr]^-{\text{}}_-{\text{ f.f.}}	 \ar[u]_{\text{b.o.}}	&& \GS \ar@{^{(}->} [rr]_-{\text{   }}_-{\text{f.f.}}	 \ar@<2pt>[u]^-{\text{free}^{\DD}}		&& \pr{\Gr}.}
\end{equation}

In \cite[Section~8]{Ray20}, it was shown that that, though the modular operad monad $\DD\TT$ on $\GS$ does not, itself, satisfy the conditions of \cite[Sections~1~\&~2]{BMW12}, the monad $\TTp$ on $\GSp$ does, whereby the following theorem holds: 

\begin{thm}{\cite[Theorem~8.2]{Ray20}}\label{MO nerve}
The monad $\TTp$, and the category $\Grp \hookrightarrow \GSp$, satisfy the conditions of \cite{BMW12} ($\TTp$ has arities $\Grp$). Hence, the nerve $N\colon \CSM \to \pr{\Klgr} $ is fully faithful. 
Its essential image consists of those presheaves $P$ on $\Klgr$ such that 
\[P(\G)  = \mathrm{lim}_{(\C,b) \in \elG} P(\C) \text{ for all graphs } \G.\]

\end{thm}

The abstract nerve theory conditions of \cite{BMW12} are also not satisfied by the the circuit operad monad $\LL \DD \TT$ on $ \GS$, and the aim of this section is to modify the proof of \cref{MO nerve} 
to the circuit operad case. 

By (\ref{eq. cube}) there exists a monad $\widetilde \LL$ on $\GSp$, and 
a distributive law $\Tp \widetilde L \Rightarrow \Tpk$ for the monads $\widetilde \LL$ and $\TTp$, such that the EM category of algebras for the composite $\TTpk$ on $\GSp$ is canonically isomorphic to the EM category of algebras for $\LL \DD \TT$ on $\GS$, which itself is just the category $\CO\simeq \CA$ of circuit operads in $\Set$. 

The idea now is to show that $\TTpk$ together with a suitable subcategory of $\GSp$, satisfies the conditions of \cite[Sections~1~\&~2]{BMW12} and hence that the induced nerve is fully faithful. 

Let $\Klgrt$ be the category obtained in the bo-ff factorisation of $\Gret \to \GS \to \CO$. By \cref{ss. pointed} (see also  \cite[Section~7.3]{Ray20}), the full subcategory $\Gretp \subset \GSp$ of all graphs and pointed morphisms induces a fully faithful inclusion $\GS \hookrightarrow \pr{\Gretp}$. 

Hence, we may consider the following diagram of functors in which the left and middle squares commute, and the right squares commute up to natural isomorphism: 
\begin{equation} \label{eq: CO Gr diag picture}
\xymatrix{ 
	&&{\Klgrt}\ar@{^{(}->} [rr]_-{\text{f.f.}}						&& \CSM\ar@<2pt>[d]^-{\text{forget}^{\widetilde \LL\TTp}} \ar [rr]^-{N_{{\Klgrt}}	}			&&\pr{{\Klgrt}}\ar[d]^{j^*}	\\
	\fisinvp \ar@{^{(}->} [rr]	^-{\text{}}_-{\text{ f.f.}}		&& \Gretp \ar@{^{(}->} [rr] ^-{\text{}}_-{\text{ f.f.}} \ar[u]^{j}_{\text{b.o.}}			&& \GSp \ar@<2pt>[d]^-{\text{forget}^{\DD}} \ar@{^{(}->} [rr]^-{\text{   }}_-{\text{f.f.}}	 \ar@<2pt>[u]^-{\text{free}^{\widetilde \LL\TTp}}	&& \pr{\Gretp}\ar[d] \\
	\fisinv\ar@{^{(}->} [rr]^-{\text{}}_-{\text{ f.f.}}\ar[urr]_-{\text{dense}}\ar[u]^{\text{b.o.}}	&& \Gret \ar@{^{(}->} [rr]^-{\text{}}_-{\text{ f.f.}}	 \ar[u]_{\text{b.o.}}	&& \GS \ar@{^{(}->} [rr]_-{\text{   }}_-{\text{f.f.}}	 \ar@<2pt>[u]^-{\text{free}^{\DD}}		&& \pr{\Gr}.}
\end{equation}

The remainder of this paper is dedicated to proving the following theorem using a modification of the proof of \cref{MO nerve}  (\cite[Theorem~8.2]{Ray20}):

\begin{thm}\label{nerve theorem} 	\label{thm: CO nerve}The functor $N\colon\CO \to\pr{\Klgrt}$ is full and faithful. Its essential image consists of precisely those presheaves $P$ on $\Klgrt$ whose restriction to $\pr{\Klgr}$ are graphical species. In other words, 	
	\begin{equation}
	\label{eq. Segal nerve sec}P(\G)  = \mathrm{lim}_{(\C,b) \in \elG} P(\C) \text{ for all graphs } \G.
	\end{equation}
\end{thm}

\begin{rmk}
	\label{rmk. grp no arities} Unlike the operad case (\cref{MO nerve}), the category $\Grp \subset \GSp$ of connected graphs and pointed morphisms cannot be used to prove an abstract nerve theorem for circuit operads (see \cref{rmk: Gr no arities}). 
\end{rmk}
\subsection {The composite monad $\widetilde \LL\TTp$ on $\GSEp$}

The first step is to understand the monad $\widetilde \LL\TTp$ on $\GSEp$. Here $\E$ is, as usual, an arbitrary category with sufficient limits and colimits. From \cref{ss. graphical category}, we will restrict to the case $\E = \Set$. 

For all pointed graphical species $S_* = (S, \epsilon, o)$ in $\E$, the lift $\widetilde L (S_*)$ is given by $\widetilde L (S_*) = (LS, \eta^\LL \epsilon ,\eta^\LL o)$ where $\mu^{\widetilde \LL}, \eta^{\widetilde \LL}$ are induced by $\mu^\LL$ and $\eta^\LL$ in the obvious way.

The categories $ X\Grsimp$ of similar connected $X$-graphs were defined in \cref{ssec. Similar connected X-graphs}. Similarity morphisms preserve connected components, and hence the category $X\XGretsimp$ of \textit{all} $X$-graphs and similarity morphisms, may be defined as follows: 
\begin{equation}\label{eq. XGretsimp} X\XGretsimp \defeq \mathrm{colim}_{(Y,f) \in \core{X\ov\fin} }
\coprod_{y \in Y} \left(f^{-1}(y)\right) \XGrsimp.\end{equation} 

The monad $\TTp = (\Tp, \mu^{\TTp}, \eta^{\TTp})$ with  $\Tp S_X  = \mathrm{colim}_{\X \in X \Grsimp} S(\X)$ was discussed in \cref{ssec. Similar connected X-graphs}. 

Let $S_* = (S, \epsilon, o)$ be a pointed graphical species in $\E$ and let $X$ be a finite set. Then, the endofunctor $\Tpk$ admits is described by the obvious quotient of $\Tk $: \begin{equation}\label{eq. lift LT}
\begin{array}{lll}{\Tpk} S_X& = &\mathrm{colim}_{(Y,f) \in \core{X\ov\fin} }
\left( \prod_{y \in Y}  \Tp S_{f^{-1}(y)} \right)\\
& = &\mathrm{colim}_{(Y,f) \in \core{X\ov\fin} } \left( \prod_{y \in Y}
\left(\mathrm{colim}_{(\G, \rho) \in (f^{-1}(y))\XGrsimp} S(\G)\right)\right)\\
& = &\mathrm{colim}_{\X \in X\XGretsimp} S(\X).

\end{array}\end{equation}

And, 
\[ \begin{array}{lll}
\Tp \tilde L S_X& = &\mathrm{colim}_{ \X \in X \Grsimp}  \mathrm{lim}_{(\C,b) \in \elG[\X]} LS(\C)\\
&=& \mathrm{colim}_{ \X \in X \Grsimp}  \mathrm{lim}_{ \Gg \in\CorGg} S(\Gg(\X)).
\end{array}\] 
The distributive law $\Tp \widetilde L \Rightarrow \Tpk$ is induced by the distributive law $\lambda_{\LL \TT} \colon TL \Rightarrow LT = \Tk$, which is compatible with similarity morphisms by construction: So, for each finite set $X$, a colimit of $S$ indexed by pairs $(\X, \Gg)$ of a connected admissible $X$-graph $\X$, and $\Gg \in \CorGg$, is replaced by the 
colimit indexed over 
(in general not connected) admissible $X$-graphs $\Gg(\X)$. 

\subsection{A graphical category for circuit operads}\label{ss. graphical category}

Assume $\E = \Set$ for the remainder of the paper. 

The category $\Gretp$ obtained in the bo-ff factorisation of $\Gret \to \GS \to \GSp$ has been discussed in \cref{ss. pointed}. Let $\Klgrt$ be the category obtained in the bo-ff factorisation of $\Gret \to \GS \to \CO$ (and $\Gretp \to \GSp \to \CO$). This is the full subcategory of the Kleisli category ${\GSp}_{\TTpk}$ of $\TTpk$ on free circuit operads of the form $\Tpk \yetp \G = LDT \yet \G$ for graphs $\G \in \Gret$. So, we may regard objects of $\Klgrt$ as graphs. Morphisms $ \beta \in \Klgrt(\G, \H)$ are given by morphisms in $\GSp (\yetp \G, \Tpk \yetp \H)$.

To understand $\Klgrt$, the first step is therefore to describe $\Tpk \yetp \H$ for all graphs $\H$.

Since each connected component of $X\XGretsimp $ contains an (in general, not connected) admissible $X$-graph $\X$, it follows from (\ref{eq. lift LT}) that, for all finite sets $X$, elements of $ (\Tpk \yetp \H)_X \cong \Klgrt(\CX, \H) $ are represented by pairs $(\X, f)$ where $\X$ 
is an admissible $X$-graph and $f \in \Gretp(\X, \H)$. In particular, $f$ factors uniquely as a morphism $\X \to \X_{\setminus W_f}$ in $\Gretsimp$ followed by a morphism $f^\bot \in \Gret(\X_{\setminus W_f}, \H)$.
Pairs $(\X^1, f^1)$ and $(\X^1, f^2)$ represent the same element $[\X,f]_* \in  (\Tpk \yetp \H)_X  $ if and only if there is an object $\X_f \in X\XGretsimp$ and a commuting diagram 
\begin{equation}\label{connected graph factors}\xymatrix{ 
	\X^1 \ar[rr]^-{g^1} \ar[drr]_{f^1}& & \X^f  \ar[d]^{f^\bot }&&\ar[ll]_-{g^2} \X^2 \ar[dll]^{f^2} \\&&\H&&}\end{equation} in $\Gretp$ such that $g^j$ are similarity morphisms in $X\XGretsimp$ for $j = 1,2$, and $f^\bot \in \Gret (\X^f, \H)$ is an (unpointed) \'etale morphism.

In particular, for all $\H$, $ e \in E(\H)$, and all $m \geq 1$, the following special case of (\ref{connected graph factors}) commutes in $\Gretp$:
\begin{equation}\label{0 connected} \xymatrix{ \C_\nul \ar[rr]^-{  z} \ar[drr]_{ch_e \circ z } & &(\shortmid)\ar[d]^{ch_e}&& \Wl \ar[ll]_{\kappa } \ar[dll]^-{ch_e \circ \kappa}\\
	&&\H.&&}\end{equation} 

Since $\Klgrt (\G, \H) \cong \Tpk \yetp \H (\G)   = \mathrm{lim}_{ (\C, b) \in \elG} \Tpk \yetp \H(\C)$, morphisms $\gamma \in \Klgrt(\G, \H)$ are represented by pairs $(\Gg, f)$, where $\Gg: \elG \to \Gret$ is a non-degenerate $\G$-shaped graph of graphs, and $f \in \Gretp(\Gg(\G), \H)$ is a morphism in $\Gretp$ (see \cite[Section~8.1]{Ray20}). 

The following lemma generalises \cite[Lemma~8.9]{Ray20}, and is proved in exactly the same manner.

 \begin{lem}\label{lem: Klgr representatives} Let $\G$ and $\H$ be graphs and, for $i = 1,2$, let $\Gg^i$ be a non-degenerate $\G$-shaped graph of graphs with colimit $\Gg^i(\G)$, and $f^i \in \Gretp(\Gg^i(\G), \H)$.
	Then $(\Gg^1, f^1), (\Gg^2, f^2)$ represent the same element $\gamma$ of $\Klgrt(\G, \H)$ if and only if there is a representative $(\Gg, f)$ of $\gamma$, and a commuting diagram in $\Gretp$ of the following form: 
	\begin{equation}\label{eq: well-defined kleisli}\xymatrix{\Gg^1(\G) \ar[rr] \ar[drr]_{f^1} && \Gg(\G)\ar[d]^{f} &&\ar[ll] \Gg^2(\G) \ar[dll]^{f^2}\\&&\H,&&}\end{equation} where the morphisms in the top row are vertex deletion morphisms (including $z \colon \C_\nul \to (\shortmid)$) in $\Gretsimp$, and $f \in \Gret (\Gg(\G), \H)$. \end{lem}

 The following terminology is from \cite{Koc16}. 
\begin{defn}\label{free} A \emph{(pointed) free morphism in $\Klgrt$} is a morphism in the image of the inclusion $\Grp \hookrightarrow \Klgrt$. An \emph{unpointed free morphism} in $\Klgrt$ is a morphism in the image of $\Gr \hookrightarrow \Grp \hookrightarrow \Klgrt$. A \emph{generic morphism $[\Gg]$ in $\Klgrt$} is a morphism in $\Klgrt(\G,\H)$ with a representative of the form $(\Gg, id_{\H})$. \end{defn}

So, a free morphism in $\Klgrt(\G, \H)$ has a representative of the form $(id, f)$ where $id \colon \elG \to \Gret$ is the trivial identity $\G$-shaped graph of graphs with colimit $\G$ and $f \in \Gretp(\G, \H)$. Generic morphisms $[\Gg]$ in  $\Klgrt(\G, \H)$ are described by non-degenerate $\G$-shaped graph of graphs $\Gg$ with colimit $\H$, and hence induce isomorphisms $E_0(\G) \xrightarrow{\cong} E_0(\H)$ on boundaries. 

In particular, morphisms in $\Klgrt(\G, \H)$ factor as $\G \xrightarrow{[\Gg]} \Gg(\G) \xrightarrow{\delW} \Gg(\G)_{\setminus W}\xrightarrow {f} \H$. Here $\Gg\colon \elG \to \Gret$ is a non-degenerate $\G$-shaped graph of graphs with colimit $\Gg(\G) \in \Gret$, 
and $f \in \Gret(\Gg(\G)_{\setminus W}, \H)$ is an (unpointed) \'etale morphism $\Gret$.

\subsection{Factorisation categories and the nerve theorem}

Let ${\GSp}_{\TTpk}$ be the Kleisli category of the monad $\TTpk$ on $\GSp$. 
By \cite[Proposition~2.5]{BMW12}, the monad $\TTpk$ together with the category $\Gretp$ satisfy the conditions of the nerve theorem \cite[Theorem~1.10]{BMW12} (that is, $\TTpk$ has arities $\Gretp$), and therefore the nerve functor $N \colon \CO \to \pr{\Klgrt}$ is fully faithful 
if certain categories -- associated to morphisms in ${\GSp}_{\TTpk}$ -- are connected.

Let $S_*$ be a pointed graphical species in $\Set$, $ \G$ a graph in $\Gret$, and let $\beta \in {\GSp}_{\TTpk}(\G, S_*) $ be a morphism in the Kleisli category of $\TTpk$. By the Yoneda lemma, \[{\GSp}_{\TTpk}(\G, S_*) \cong \GSp (\G, \Tpk S_*) \cong \Tpk S_* (\G) = \mathrm{\mathrm{colim}_{(\C,b) \in \elG}} \Tpk S_*(\C) \] where
\[ \Tpk S_*(\C)\cong \left \{ \begin{array}{ll}
S_\S & \C \cong (\shortmid),\\
\mathrm{colim}_{\X \in X\XGretsimp} S (\X) & \C \cong \CX.
\end{array}\right.\]

So, $\beta \in{\GSp}_{\TTp}(\G, S_*)$ is represented by pairs $(\Gg, \alpha)$ of
  \begin{itemize}
  	\item a non-degenerate $\G$-shaped graph of graphs $\Gg$ with colimit $\Gg(\G)$, viewed as a  {generic} morphism $[\Gg] \in \Klgrt (\G, \Gg(\G))  = {\GSp}_{\TTp}(\G, \Gg(\G))$, 
  	\item an element $\alpha \in S(\Gg (\G))$ viewed, by the Yoneda lemma, as a morphism $\alpha \in {\GSp}_{\TTp}(\yetp\Gg(\G), S_*)$ 
  \end{itemize}
such that the induced composition $\G \xrightarrow {[\Gg]} \Gg(\G) \xrightarrow {\alpha}  S(\Gg(\G)) \twoheadrightarrow (\Tpk S_*)(\G)$ is precisely $\beta$.

\begin{defn}\label{def. factcat}
	Objects of the  \emph{factorisation category} $\factcat$ of $\beta$ are pairs $(\Gg, \alpha)$ that represent $\beta$ as above, and morphisms in $\factcat((\Gg^1, \alpha^1), (\Gg^2, \alpha^2))$ are commuting diagrams in ${\GSp}_{\TTpk}$
\begin{equation}\label{factcat mor}
\xymatrix{&& {[\Gg^1]}(\G)\ar[dd]_{g} \ar[drr]^{ \alpha^1}&&\\
	\G \ar[urr]\ar[drr]&&&& S_*\\
	&&{[\Gg^2]}(\G) \ar[urr]_{ \alpha^2}&&}
\end{equation} such that $g $ is a morphism in $ \Gretp\hookrightarrow {\GSp}_{\TTpk}$.

\end{defn}

\begin{lem}\label{connected} For all pointed graphical species $S_*$ in $\Set$, all graphs $\G \in \Gret$, and all $ \beta \in \GSp(\G, \Tpk S)$, the category $\factcat$ is connected.
\end{lem}

\begin{proof} 
If $\G \cong (\shortmid)$, then $\Gg$ is isomorphic to the identity $(\shortmid)$-shaped graph of graphs. 
 If $\G \cong \CX$ for some finite set $X$, and $\beta \in \Tpk S_X \cong  {\GSp}_{\TTp}(\G, S_*)$, then objects of $\factcat$ have the form $(\X, \alpha)$ where $\X\in X \Grisok$ is an admissible (in general not connected) $X$-graph, $\alpha \in S_*(\X)$ and $\beta$ is the image of $\alpha$ under the universal morphism $S_*(\X) \to (\Tpk S_* )_X$. Since $(\Tpk S_* )_X = \mathrm{colim}_{ \X' \in X\XGretsimp}S(\X')$, if $(\X', \alpha')$ also represents $\beta$, then $\X$ and $\X'$ are connected in $X\Gretsimp$ and hence they are similar in $\Gretp$. 
Hence, if $\G = \C$ is in the image of $\fisinvp \subset \Gretp$, then $\factcat$ is connected for all $\beta \in{\GSp}_{\TTp}(\C, S_*) \cong \Tpk S(\C)$.

For general $\G \in \Gretp$, and $\beta \in{\GSp}_{\TTp}(\G, S_*)$, observe first that, since colimits of graphs of graphs are computed componentwise, we may assume that $\G$ is connected. Assume also that $\G \not \cong \C_\nul$. Let $(\Gg^i, \alpha^i)$, $i = 1, 2$ be objects of $\factcat$. Then, for each $(\CX, b) \in \elG$, $\Gg^1(b)$ and $\Gg^2(b)$ are in the same connected component of $X \XGretsimp$, since the lemma holds when $\G \cong \CX$. Therefore, $\Gg^1(\G)$ and $\Gg^2(\G)$ are similar in $\Gretp$. 
\end{proof}

\cref{nerve theorem} now follows from \cite[Sections~1~\&~2]{BMW12}.

\begin{proof}[Proof of \cref{nerve theorem}]
The induced nerve functor $N\colon \CO \to \pr{\Klgrt}$ is fully faithful by \cref{connected}, \cite[Proposition~2.5]{BMW12}, and \cite[Propositions~ 1.5~\&~1.9]{BMW12}.

Moreover, by \cite[Theorem~1.10]{BMW12} its essential image is the subcategory of those presheaves on $\Klgrt$ whose restriction to $\Gretp$ are in the image of the fully faithful embedding $\GSp \hookrightarrow \pr{\Gretp}$. So, a presheaf $P$ on $\Klgrt$ described the nerve of a circuit operad if and only if 
$ P(\G)  \cong \mathrm{lim}_{(\C,b) \in \elpG} P(\C)$, and hence, by \cite[Lemma~7.8]{Ray20} (finality of $\elG \subset \elpG$, see also \cite[Section~IX.3]{Mac98}), 
\[ P(\G)  = \mathrm{lim}_{(\C,b) \in \elG} P(\C). \] Therefore, the Segal condition (\ref{eq. Segal nerve sec}) is satisfied and \cref{thm: CO nerve} is proved. 
\end{proof}

\begin{rmk}\label{rmk: Gr no arities} 
Let $S_*$ be a pointed graphical species and $\G$ a connected graph. A morphism $\beta \in {\GSp}_{\TTpk}(\G, S_*) $ in the Kleisli category of $\TTpk$ is represented by a pair $(\Gg, \underline\alpha)$ of 
\begin{itemize}
	\item a non-degenerate $\G$-shaped graph of connected graphs $\Gg \colon \elG \to \Gr$ with (connected) colimit $\Gg(\G)$;
	\item an element $\underline \alpha \in \widetilde L S_*(\Gg(\G))$ such that, for each $(\CX, b) \in \elG[\Gg(\G)]$, $L S_*(b)(\underline \alpha)$ is given by a tuple $(\phi_i)_{i = 1}^k \in  LS_X$ (with $\phi _i \in S_{X_i}$, and $\coprod_{i = 1}^k X_i = X$) of elements of $S$. 
\end{itemize}
	
	Such pairs $(\Gg, \underline \alpha)$ are the objects of the \textit{connected factorisation category} $\factcat_c$ of $\beta$. Morphisms in $\factcat_c ((\Gg, \underline \alpha), (\Gg', \underline \alpha'))$ are pointed morphisms $g \in \Grp(\Gg(\G), \Gg'(\G))$ of connected graphs, such that the following diagram commutes in ${\GSp}_{\TTpk}$:
\begin{equation}
\label{eq. fact grp co}
	\xymatrix{&& {\Gg}(\G)\ar[dd]_{g} \ar[drr]^{\underline \alpha}&&\\
		\G \ar[urr]\ar[drr]&&&& S_*.\\
		&&{\Gg'}(\G) \ar[urr]_{ \underline \alpha}&&}
\end{equation}

By \cite[Proposition~2.5]{BMW12}, the full subcategory of $\Klgrt$ on the connected graphs induces a fully faithful nerve functor from $\CO$ (i.e.~$\Grp$ provides arities for $\TTpk$) if and only if $\factcat_c$ is connected for all connected graphs $\G$, all pointed graphical species $S_*$, and all $\beta \in {\GSp}_{\TTpk}(\G, S_*)$. 

To see that this is not the case, 
let $S_* = \yetp (\shortmid \amalg \shortmid)$, and let $\G = \C_\nul$. A $\C_\nul$-shaped graph of connected graphs is just a connected graph $\H$ with empty boundary. Recall (from \cref{ssec. LL}) that $\mathsf {Cor}_{\boxtimes}^{(\H)}$ denotes the category of non-degenerate $\H$-shaped graphs of disjoint unions of corollas. An element $ \underline \alpha\in \widetilde L \yetp (\shortmid \amalg \shortmid)(\H)$ is just an $\H$-shaped graph of disjoint unions of corollas $\Gg_\boxtimes \in \mathsf {Cor}_{\boxtimes}^{(\H)}$ with colimit $\Gg_{\boxtimes}(\H) \in \Gretp$, followed by a morphism $f \in \Gretp(\Gg_\boxtimes (\H), (\shortmid \amalg \shortmid))$. 
In particular, each component of $\Gg_{\boxtimes}(\H)$ is either an isolated vertex, or of the form $\Lk$ or $\Wm$ ($k,m \geq 0$). 

So, let $\beta \in {\GSp}_{\TTpk}(\C_\nul,  (\shortmid \amalg \shortmid)) $ be the morphism represented by the unique $\C_\nul$-shaped graph of graphs with colimit $\C_\nul \amalg \C_\nul$, together with the componentwise morphism $(z \amalg z) \in \Gretp (\C_\nul \amalg \C_\nul, \shortmid \amalg \shortmid) $.

Then $\beta$ is also represented by $(\H,\Gg_\boxtimes, f)$, illustrated in \cref{fig. counter}, where 
$\H = \C_{\mathbf{4}}^{1 \ddagger 2, 3 \ddagger 4}$ is the graph with two loops at one vertex, $\Gg_\boxtimes$ is given by $ (\C_{\mathbf 4}, b)\mapsto \C_\two \amalg \C_\two $ with colimit $\Gg_\boxtimes(\H) = \W \amalg \W$, and $f =( \kappa \amalg \kappa )\in \Gretp (\W \amalg \W, \shortmid \amalg \shortmid)$.

\begin{figure}	[htb!]
	
	\begin{tikzpicture}
	
	\node at (-6,-1.5) {\begin{tikzpicture}[scale = .5]
		\filldraw[blue] (0,0) circle (.2cm);
		\end{tikzpicture}
	};
\draw[gray, line width = 1.2, dashed, ->] (-5.5,-1.25)--(-4,-.2);
\node at(-5,-.5){$=$};
\draw[gray, line width = 1.2, dashed, ->] (-5.5,-1.8)--(-5,-2.2);
	\node at (7,-1.5) {\begin{tikzpicture}[scale = .5]
	\draw[ thick] (0,-2.2)--(0,-.8)
	(.5,-2.2)--(.5,-.8);
	\end{tikzpicture}
};
\draw[gray, line width = 1.2, dashed,<-] (6.5,-1.25)--(4,-.2);
\node at(6.5,-.3){$(z \amalg z)$};
\draw[gray, line width = 1.2, dashed, <-] (6.5,-1.8)--(5,-2.5);
\node at(6.5,-2.8){$(\kappa \amalg \kappa)$};
		\node at (0,0){
			\begin{tikzpicture}
				\node at (-1.8,0){
					\begin{tikzpicture}[scale = .5]
						\draw[draw = green, fill = green, fill opacity = .2] (-.2,0) circle (.5cm);
						\filldraw [red] (0,0) circle (2pt);
						\filldraw [red] (-.4,0) circle (2pt);
					\end{tikzpicture}
				};
		\draw[gray, line width = 1.2, dashed, ->] (-4.2,0)--(-2.8,0);
		\draw[gray, line width = 1.2, dashed, ->] (-.8,0)--(.8,0);
		\node at(-3.5,.3){$\C_\nul \amalg \C_\nul$};	
			\node at (-5,-.2){
				\begin{tikzpicture}[scale = .5]
				\filldraw[green] (0,0) circle (.3cm);
				\node at (0, -1){$\C_\nul$};
				\end{tikzpicture}
			};
		
		\node at (1.5,0){
			\begin{tikzpicture}[scale = .5]
			\filldraw [red] (0,0) circle (3pt);
			\filldraw [red] (-.4,0) circle (3pt);
			\end{tikzpicture}
		};
			\end{tikzpicture}
		};
			\node at (0,-3){
		\begin{tikzpicture}
		\node at (-.8,0){
			\begin{tikzpicture}[scale = .5]
			\draw (1,0) circle (1cm);
				\draw (-1.4,0) circle (1cm);
			\draw[draw = green, fill = green, fill opacity = .2] (-.2,0) circle (.5cm);
			\filldraw [red] (0,0) circle (2pt);
			\filldraw [red] (-.4,0) circle (2pt);
			
			\end{tikzpicture}
		};
		\draw[gray, line width = 1.2, dashed, ->] (-3.2,0)--(-2.2,0);
	\draw[gray, line width = 1.2, dashed, ->] (.7,0)--(1.7,0);
	\node at(-2.7,-.3){$\Gg_\boxtimes$};		
		\node at (-4.4,.2){
			\begin{tikzpicture}[scale = .5]
			\draw (1,0) circle (1cm);
				\draw (-1,0) circle (1cm);
\filldraw[green] (0,0) circle (.3cm);
	\node at (0,1.5){$\H$};
			\end{tikzpicture}
		};
		
		\node at (3.2,0){
			\begin{tikzpicture}[scale = .5]
			\draw (1,0) circle (1cm);
				\draw (-1.4 ,0) circle (1cm);
			\filldraw [red] (0,0) circle (3pt);
			\filldraw [red] (-.4,0) circle (3pt);
			\end{tikzpicture}
		};
	\end{tikzpicture}
};
	\end{tikzpicture}

\caption{There is no zigzag of morphisms $\C_\nul \leadsto \H$ in $\Grp$ that connect $(\C_\nul, \C_\nul \amalg \C_\nul)$ and $(\H, \Gg_\boxtimes)$ in $\factcat_c$.}\label{fig. counter}
\end{figure}

Since morphisms in $\Grp$ factor as vertex deletion morphisms followed by morphisms in $\Gr$, and the single vertex of $\H$ has valency four, $\factcat_c$ is connected if there is a morphism 
of the form $ch_e\circ z \in \Gr(\C_\nul, \H)$ making (\ref{eq. fact grp co}) commute. But then the image of $\C_\nul$ under the induced morphism $\C_\nul \to \H \to (\shortmid \amalg \shortmid)$ factors through an inclusion $(\shortmid) \hookrightarrow (\shortmid \amalg \shortmid)$, so the diagram (\ref{eq. fact grp co}) does not commute. 

\end{rmk}

\begin{rmk}
	\label{rmk. No weak version}
Both the nerve theorem \cite[Theorem~8.2]{Ray20} for modular operads on which \cref{thm: CO nerve} is built, as well as Hackney, Robertson and Yau's nerve theorem \cite[3.6]{HRY19b}, lead to well-defined models of $(\infty, 1)$-modular operads in which the fibrant objects in the appropriate category of presheaves are precisely those that satisfy the weak Segal condition: in the case of \cite{Ray20} it followed from the results of \cite{CH15} that these were given by $P \colon \Klgr^{\mathrm{op}} \to \sSet$ such that, for all (connected) graphs $\G \in \Klgr$,
\[P(\G) \simeq \mathrm{lim}_{(\C,b) \in \elG} P(\C). \]

However, describing a similar model for circuit operads (algebras) is more challenging.  The presence of unconnected graphs means that, neither the method of \cite{HRY19b}, nor that of \cite[Corollary~8.14]{Ray20} may be directly extended to a model structure on functors from ${\Klgrt}^{\mathrm{op}}$ (or some subcategory thereof) to $\sSet$ that gives a good description of weak circuit operads. {For this reason, a description of a model for $(\infty, 1)$-circuit operads is deferred to a future work.}

\end{rmk}

\bibliography{Compactbib}{}
\bibliographystyle{plain}

\end{document}